\documentclass{article}
\usepackage{graphicx} % Required for inserting 
\usepackage{amssymb}
\usepackage{amsmath}
\usepackage{verbatim}

\usepackage{empheq}

\usepackage{xurl}
\usepackage{hyperref}

\usepackage{tikz}
\usetikzlibrary{calc,arrows.meta,intersections,through,decorations.markings}

\usepackage[
    backend=biber,% bibtex
    url=false,
    eprint=true,
    maxbibnames=120,
    giveninits=true,
    block=space, 
    style=numeric,
    citestyle=numeric,
]{biblatex}
\numberwithin{equation}{section}

\newtheorem{theorem}{Theorem}[section]
\newtheorem{lemma}[theorem]{Lemma}
\newtheorem{prop}[theorem]{Proposition}
\newtheorem{corollary}[theorem]{Corollary}
\newtheorem{defi}[theorem]{Definition}
\newtheorem{cond}[theorem]{Condition}

\newtheorem{rem}[theorem]{Remark}
\newtheorem{note}[theorem]{Note}
\newtheorem{example}[theorem]{Example}

\newenvironment{geometricinterpretation}
{\par
 \noindent\textbf{Geometric interpretation (optional).}\itshape}
{\par}

\newcommand{\RR}{\mathbb{R}}
\newcommand{\C}{\mathbb{C}}
\newcommand{\Z}{\mathbb{Z}}
\newcommand{\N}{\mathbb{N}}
\newcommand{\T}{{\mathbb{T}}}
\newcommand{\F}{\mathcal{F}}

\newcommand{\dx}{\,dx}
\newcommand{\dy}{\,dy}

\newcommand{\la}{\lambda}

\newcommand{\te}{\theta}
\newcommand{\al}{\alpha}

\newcommand{\ep}{\varepsilon}

\newcommand{\Dcal}{\mathcal{D}}
\newcommand{\Pcal}{\mathcal{P}}
\newcommand{\Lcal}{\mathcal{L}}

\newcommand{\pa}{\partial}

\newcommand{\om}{\omega}

\newcommand{\g}{\gamma}

\newcommand{\intpi}{\int_{-\pi}^\pi}

\newcommand{\sgn}{{\rm sgn}}
\renewcommand{\Re}{{\rm Re}\thinspace}
\renewcommand{\Im}{{\rm Im}\thinspace}
\newcommand{\spa}{{\rm span}\thinspace}

\renewcommand{\emptyset}{\varnothing}

\newcommand{\PU}{P_{\mathcal{U}_1}}
\newcommand{\PV}{P_{\mathcal{V}_1}}
\newcommand{\PUl}[1]{P_{\mathcal{U}_1(#1)}}
\newcommand{\PVl}[1]{P_{\mathcal{V}_1(#1)}}

\newcommand{\norm}[1]{\Vert#1\Vert}
\newcommand{\bnorm}[1]{\big\Vert#1\big\Vert}

\newcommand{\scalednorm}[1]{\left\Vert#1\right\Vert}

\newcommand{\ti}[1]{\widetilde{#1}}
\newcommand{\ha}[1]{\widehat{#1}}

\newif\ifhideproofs 
\ifhideproofs

\else   
    \newenvironment{proof}{
    \begin{trivlist} \item[] {\bf Proof:}
    }{\hfill $\Box$\end{trivlist}}
\fi %________________________________________________|

\newenvironment{proofsth}[1]{\begin{trivlist} \item[] {\textbf{Proof of #1:}}}{\hfill $\Box$
                       \end{trivlist}} % Used

\newcommand{\proofcompassi}{
\begin{trivlist}
\item[] {\bf Proof:}
 The proof is computer-assisted and the code can be found in the supplementary material. We refer to the Appendix~\ref{sec:compassisted} for details about the implementation.
\hfill $\Box$ \end{trivlist}}

\title{Linear instability of a Burgers--Hilbert traveling wave}

\author{\'Angel Castro, Javier G\'omez-Serrano and Miguel M.G. Pascual-Caballo}           

\newcommand{\D}{\mathbb{D}}

\newcommand{\XL}{{L^2(\T)}}
\newcommand{\XLinf}{{{L^{\infty}(\T)}}}

\newcommand{\XH}[1]{{H^{#1}(\T)}}
\newcommand{\XHhol}[1]{{\mathcal{H}^{#1}}}
\newcommand{\XHholZMi}[1]{{\dot{\mathcal{H}}^{#1 }}}

\newcommand{\XHholo}{{\mathcal{H}}}
\newcommand{\XHholZM}{{\dot{\mathcal{H}}}}

\newcommand{\XHeven}[1]{{H^{#1}_{\rm even}}}
\newcommand{\XHodd}[1]{{H^{#1}_{\rm odd}}}
\newcommand{\XHevod}[1]{{H^{#1}_{\textrm{odd or (even)}}}}

\newcommand{\XLIpi}{{L^2([-\pi,\pi])}}

\newcommand{\XLinfinterval}{{L^\infty([-\pi,\pi])}}

\newcommand{\LOpExisTwap}{{
L_{\twap}
}}

\newcommand{\LOpExisTwapInv}{L^{-1}_{\twap}}

\newcommand{\LOpExisTw}{{
L_\tw
}}

\newcommand{\LOpExis}{{
\Lcal^{\rm ap}_{\rm exis}
}}

\newcommand{\LOpExisHol}{{
\Lcal_{\rm exis}^{\rm ap,\, hol}
}}

\newcommand{\LinearOpStability}[1]{{
\Lcal_{\rm stab}^{#1}
}}

\newcommand{\LStabthe}{\Lcal^{\rm ap}_\te}
\newcommand{\LStabtheTW}{\Lcal_\te}
\newcommand{\LOpStabHol}[1]{{
\Lcal^{\rm ap,\, hol}_{\te,#1}
}}

\newcommand{\DiffOpExpNonExp}{{
\Dcal_\te
}}

\newcommand{\TopExis}{{
T_{\rm exis}
}}
\newcommand{\TopStab}{{
T^\te_{\rm stab}
}}

\newcommand{\twap}{v^{\rm ap}}
\newcommand{\tw}{v}
\newcommand{\twdiff}{\om}
\newcommand{\twxi}{\xi_{\twap}}
\newcommand{\Procomp}{P_{ub}}

\newcommand{\EigenValue}{-\la}
\newcommand{\Eigdiff}{{\eta}}
\newcommand{\laap}{{\la^{\rm ap}}}

\newcommand{\EigenVecto}{f_{\EigenValue}}
\newcommand{\Eigveap}{f^{\rm ap}}
\newcommand{\Fvrmap}{{F_v^{\rm ap}}}

\newcommand{\Eigres}{
\xi_{\Eigveap}
}

\newcommand{\PiExis}{{
\Pi
}}

\newcommand{\MExisHol}{{
M^{\rm hol}_{\rm exis}
}}
\newcommand{\MExisTor}{{
M^{\rm tor}_{\rm exis}
}}
\newcommand{\MExisTorApprox}{{
M^{\rm thin}_{\rm exis}
}}

\newcommand{\MStabHol}{{
M^{\rm hol}_{\rm stab}
}}
\newcommand{\MStabTor}{{
M^{\rm tor}_{\rm stab}
}}
\newcommand{\MStabTorApprox}{{
M^{\rm thin}_{\rm stab}
}}

\newcommand{\II}{
{I\!I}
}
\newcommand{\III}{
{I\!I\!I}
}

\newcommand{\Iexis}{
I^{\rm exis}
}

\newcommand{\IIexis}{
\II^{\rm exis}
}

\newcommand{\Istab}{
I^{\rm stab}
}

\newcommand{\IIIexis}{
\III^{\rm exis}
}

\newcommand{\IIIstab}{
\III^{\rm stab}
}

\newcommand{\IIstab}{
\II^{\rm stab}
}

\makeatletter

\newcommand{\setvar}[2]{%
  \expandafter\def\csname var@#1\endcsname{#2}%
}

\newcommand{\getvar}[1]{%
  \csname var@#1\endcsname
}

\makeatother

\setvar{speed}{1.109643839000099}
\setvar{Nexplicit}{1000}
\setvar{la_reg}{1.035i}
\setvar{la_app}{0.142128757204011-0.0707987316896479i}

\setvar{resi_exis_L2}{8.00338\cdot 10^{-14}}
\setvar{Lexis_inv}{3.13328\cdot 10^4}
\setvar{rad_exis}{2.50805\cdot 10^{-9}}
\setvar{lip_exis}{3\cdot 10^{-4}}

\setvar{Dthe_op}{6.78140\cdot 10^{-9}}
\setvar{fvap_stab_L2_sq}{0.526963}
\setvar{fvap_stab_H1_sq}{4.54832}
\setvar{resi_stab_L2_sq}{1.60881\cdot 10^{-20}}
\setvar{Lstab_inv}{1014.41}
\setvar{rad_stab}{10^{-3}}
\setvar{lip_stab}{0.12}
\setvar{lare}{0.142127}

\setvar{svd1_exis}{6.82641\cdot 10^{-4}}

\setvar{beta_max}{19.9056}
\setvar{kappa1}{4.03105}
\setvar{beta_mod}{12107.9}
\setvar{kappa2}{77441.4}
\setvar{JH1H2}{354.490}
\setvar{hkH2}{1.02110\cdot 10^6}

\setvar{svd_stab1}{10^{-7}}
\setvar{svd_stab2}{1.31011\cdot 10^{-3}}

\setvar{kappa3}{14.0518}
\setvar{kappa4p}{524.203}
\setvar{kappa4m}{782.958}

\setvar{J1p}{42.8011}
\setvar{J1m}{47.8870}

\setvar{hkL2p}{9.61596}
\setvar{hkL2m}{9.59380}
\setvar{hkH1p}{247.212}
\setvar{hkH1m}{283.513}

\setvar{vthep}{0.637881}
\setvar{vthem}{0.770135}

\setvar{Jfvu1}{4.54415\cdot 10^{-2}}

\setvar{radii_roots}{10^{-41}}

\setvar{github_link}{
\url{https://github.com/MiguelMGPascualCaballo/bhtw}
}
\setvar{zenodo_link}{
\url{https://doi.org/10.5281/zenodo.19250315}
}
\begin{document}

\maketitle

\begin{abstract}
    We study the stability of traveling wave solutions to the Burgers--Hilbert equation on $\mathbb{T}$ in the regime of small frequency $\omega$ and large wave speed $c$. For $\omega = 3$ and $c \approx 1.1$, we show that the linearized operator around these solutions has an eigenvalue with negative real part, indicating spectral instability. Our approach is computer-assisted: we reduce the problem to a finite-dimensional system and solve it rigorously using interval arithmetic. The Burgers--Hilbert equation arises as a quadratic approximation of the vortex patch problem for the two-dimensional Euler equations. In this setting, our results point to the instability of threefold symmetric V-states.
\end{abstract}

\tableofcontents

\section{Introduction}
In \cite{MW}, Marsden and Weinstein introduced the Burgers--Hilbert (BH) equation,
\begin{equation}\label{BHEQN}
    \left\{
    \begin{array}{l}
        \pa_t u+Hu+u\pa_x u=0,\quad u:\T\times[0,\infty)\to\RR,\\
        u(x,0)=u_0(x),
    \end{array}
    \right.
\end{equation}
as a quadratic approximation of the vortex patch problem for the 2D Euler equations.

 A traveling wave $\tw:\T\to\RR$, with speed $c\in\RR$, is a solution $u$ of \eqref{BHEQN} of the form
\begin{align*}
    u(x,t)=\tw(x+ct),
\end{align*}
and hence
\begin{equation}\label{eqn_BH_TW}
    c\pa_x \tw+H\tw+\tw\pa_x\tw=0.
\end{equation}
We impose the additional condition 
\begin{align*}
    \intpi \tw(x)\dx=0.
\end{align*}

In this paper, we focus on the stability of traveling waves of \eqref{BHEQN} which leads to studying the spectral properties of the operator $L_vf:=c\pa_xf+Hf+\pa_x(vf),$ which will also be denoted $L_v:=c\pa_x+H+\pa_x(v\,\cdot)$.

\vspace{0.5cm}

We briefly review some of the main results concerning \eqref{BHEQN}. In \cite{BielloHunter}, the validity of \eqref{BHEQN} for describing surface waves on a planar discontinuity in vorticity in two-dimensional inviscid flow was studied. Later, in \cite{HMRSJZ}, it was proved that the approximation remains valid on cubically nonlinear timescales. 

Several works have addressed the question of whether the Hilbert transform can prevent the formation of singularities in \eqref{BHEQN}. Numerical simulations were carried out in \cite{BielloHunter,Hunter2016,KleinSaut2015}, loss of $C^{1+\delta}$ regularity was proved in \cite{CastroCordobaGancedo2010}, shock formation was established in \cite{SautWang}, and the existence of asymptotically self-similar shocks was shown in \cite{Yang}. 

Bressan and Nguyen established the global existence of weak solutions in \cite{BressanNguyen2014} for initial data $f_0\in L^2(\RR)$, with $f(x,t)\in L^\infty(\RR)\cap L^2(\RR)$ for all $t > 0$. In \cite{BressanZhang2017}, piecewise continuous solutions are constructed locally in time. Such solutions were shown to be unique in \cite{KrupaVasseur2020}.

Branches of traveling waves $\tw_n$, bifurcating from zero, were constructed in \cite{Hunter2016}. These solutions are $C^\infty$, $\frac{2\pi}{n}$-periodic, and have speeds $c_n=\frac{1}{n}$. Later, in \cite{CCZ2}, it was shown that these solutions are analytic. By introducing the asymptotic expansions
\begin{align}\label{cvexpan}
    v(x)=\ep \cos(x)+\sum_{k=2}^\infty u_k(x)\ep^k, && 
    c=1+\sum_{k=1}^\infty c_k \ep^k,
\end{align}
it was shown in \cite{CCZ2} that the branch of solutions could be extended beyond $\ep=0.22$ and that the series \eqref{cvexpan} cannot converge for $\ep>\frac2e$. The existence of a maximal wave with speed $c\in [1.1,1.121]$ and with an $x\log(|x|)$ singularity at the origin was proved in \cite{DahGose2023}.

The stability of these traveling waves was also studied in \cite{CCZ2}. In particular, the authors showed that near the beginning of the branches, namely, for small $\ep$ in \eqref{cvexpan}, the eigenvalues of the operator $c\pa_x+H+\pa_x(v\,\cdot)$ are purely imaginary. Moreover, this phenomenon was used to enhance the lifespan of perturbations of a traveling wave. For lifespan enhancement of perturbations of the zero solution, see \cite{HunterIfrim2012} and \cite{HunterIfrimTataruWong2015}. 

The main results of this paper are Theorems \ref{thmexistencetw} and \ref{theo:stab}.
Taken together, they imply the following:

\begin{theorem}\label{theo:compboth}
    There exist a speed $\ti{c}$ and a real, smooth, even, and zero-mean function $\ti{\tw}\neq 0$, with minimal period $\frac{2\pi}{3}$, solving \eqref{eqn_BH_TW}, such that the linear operator $\ti{c}\pa_x+H+\pa_x(\ti{\tw}\,\cdot)$ has eigenvalues with both negative and positive real parts.
\end{theorem}

Let us emphasize that the stability of traveling waves for related models was considered by Maspero and Radakovic in \cite{MasperoRadakovic2025}. They work on $\RR$ instead of $\T$ and consider the equation
\begin{align}\label{maspero}
    \pa_tu +u \pa_xu +\mathcal{M}\pa_xu=0,
\end{align}
where $\ha{\mathcal{M}}=m(\xi)$, with $m(\xi)\in C^r(\RR)$, $r\geq 3$, real-valued, even, satisfying, for some $m\geq -1$, that $C_1\langle\xi\rangle^m\leq |m(\xi)|\leq C_2\langle\xi\rangle^m$, where $\xi\gg1$ and $\inf_{n\in \N^+-\{1\}}|m(n)-m(1)|>c_0$.

The stability of small traveling waves of \eqref{maspero} is studied in \cite{MasperoRadakovic2025} when subjected to perturbations with very small frequency (modulational instability).  Roughly speaking, this corresponds to the analysis of the stability of $\frac{2\pi}{m}$-periodic traveling waves for very large $m$ in our setting. Maspero and Radakovic completely characterize the spectrum near the origin of the complex plane of the operator $c\pa_x +\mathcal{M}\pa_x+\pa_x(v\,\cdot)$, and, in particular, find an unstable eigenvalue (see also \cite{johnson2013stability,hur2014modulational,hur2015modulational,hur2016modulational,angulo2017stability}). In contrast, we focus on a large-amplitude traveling wave and small values of $m$, focusing only on the case $m=3$.

\subsection{Computational aspects of the proof}

We emphasize that part of the proof of Theorem~\ref{theo:compboth} is computer-assisted. There has been an emergence of proofs using those techniques in the last decade. The main idea is to substitute floating point numbers on a computer by rigorous bounds, which are then propagated through every operation that the computer performs taking into account any error made throughout the process. 

More concretely, we use the computer as part of the proof in the following settings:
\begin{itemize}
    \item Bounds for the approximate traveling wave and the approximate eigenfunction, including estimates for their residuals and norms. See Lemmas~\ref{lemma:residue_exis_L2}, \ref{lemma:norm_fv}, and \ref{lemma:eignot0}.

    \item Non-degeneracy checks, including the non-vanishing of the Fourier coefficients of the traveling wave approximation in Lemma~\ref{lemma:twapprox_nonzero_coeffs}, validated enclosures of the roots of the associated polynomial $z^N P(z)$ in Lemma~\ref{lemma:enclose_poly_roots}, and the exclusion of certain integer obstructions in Lemmas~\ref{lemma:Proots1} and \ref{lemma:Proots2}.

    \item Bounds for certain explicit norms involving $\beta=(c+\twap)^{-1}$. See Lemmas~\ref{lemma:betamax}, \ref{lemma:kappa1}, \ref{lemma:betamod_L2}, \ref{lemma:kappa2}, \ref{lemma:kappa3}, and \ref{lemma:kappa4}. 

    \item Upper bounds for the sum of explicit quantities, obtained in Lemmas~\ref{lemma:hk_funcs_exis} and \ref{lemma:hk:stab}.

    \item Working with piecewise polynomial approximations of certain explicit functions, and deriving estimates for the approximation errors. These estimates are obtained in Lemmas~\ref{lemma:ode_exis}, \ref{lemma:ode_stab:regu}, \ref{lemma:ode_stab:sing}, and \ref{lemma:ode_stab:Fv}.

    \item Singular value estimates for the finite-dimensional matrices $\MExisTor$ and $\MStabTor$, which are used throughout the argument. In particular, they enter the proofs of Lemmas~\ref{lemma:exis_singval}, \ref{lemma:stab_is_regu}, \ref{lemma:stab_singval}, and \ref{lemma:prod_fv_u}.
\end{itemize}

In a broader context (computer-assisted proofs in PDE), the reduction of the proof of existence to a fixed-point argument has also been particularly successful in the context of \textit{radii polynomials}, developed in \cite{vandenBerg-Lessard:chaotic-braided-solutions-swift-hohenberg,Day-Lessard-Mischaikow:validated-continuation-equilibria-pde,Gameiro-Lessard-Mischaikow:validated-continuation-large-parameter-ranges-pde} and later used in  \cite{vandenBerg-Breden-Lessard-vanVeen:periodic-orbits-ns,Castelli-Gameiro-Lessard:rigorous-numerics-ill-posed-pde}.

There are several other works that incorporate computer-assisted proofs.  We highlight the following representative examples: the Euler--Poisson system \cite{Guo-Hadzic-Jang-Schrecker:gravitational-collapse-stars-self-similar}, extreme Stokes waves \cite{Kobayashi:global-uniqueness-stokes}, the De Gregorio model \cite{Chen-Hou-Huang:blowup-degregorio}, the SQG equation \cite{Castro-Cordoba-GomezSerrano:global-smooth-solutions-sqg}, and Boussinesq-type models, including applications to the 3D Euler equations \cite{ElgindiPasqualotto2025_invertibility_boussinesq,Chen-Hou:nearly-self-similar-blowup-boussinesq-euler-analysis,Chen-Hou:nearly-self-similar-blowup-boussinesq-euler-numerics,Chen-Hou:singularity-formation-3d-euler-smooth-data}. Further applications include Whitham-type equations \cite{Cadiot:proofs-existence-stability-capillary-gravity-whitham}, the Kuramoto--Sivashinsky equation \cite{Arioli-Koch:cap-stationary-ks,Figueras-DeLaLLave:cap-periodic-orbits-kuramoto,Gameiro-Lessard:periodic-orbits-ks,Figueras-Gameiro-Lessard-DeLaLLave:framework-cap-invariant-objects,Zgliczynski:periodic-orbit-kuramoto,Zgliczynski-Mischaikow:rigorous-numerics-kuramoto}, compressible Euler and Navier--Stokes equations \cite{Buckmaster-CaoLabora-GomezSerrano:implosion-compressible}, and the incompressible Navier--Stokes equations \cite{Arioli-Gazzola-Koch:uniqueness-bifurcation-ns,Bedrossian-PunshonSmith:chaos-stochastic-2d-galerkin-ns}.

The supplementary material, including the code, is available at \getvar{github_link}. Larger data files are available at \getvar{zenodo_link}.

\subsection{The vortex patch problem}
Let us briefly discuss the motion of vortex patches for the 2D Euler equations, which provides the main motivation for our study. These solutions correspond to vorticities of the form
\begin{align*}
    \om({\bf x},t)=\left\{\begin{array}{cc}\varpi&{\bf x}\in \Omega(t)\\0 & {\bf x}\in \Omega^c(t)\end{array}\right.,
\end{align*}
where $\Omega(t)$ is a time-dependent domain transported by the flow. It was first proved by Chemin in \cite{Chemin1993,Chemin1995} that if the boundary of $\Omega(0)$ is $C^{1+\delta}$, then there exists a global patch solution to the 2D Euler equations whose boundary remains of class $C^{1+\delta}$. A different proof of this result can be found in \cite{BertozziConstantin1993}. 

The existence of traveling waves for the BH equation is closely related to the existence of V-states for the vortex patch problem. A V-state is a simply connected patch of vorticity whose evolution consists of a rigid rotation with constant angular velocity. The properties of such solutions have been extensively studied in the mathematical literature. Their existence was first observed numerically in \cite{DeemZabusky1978} and later proved analytically by Burbea \cite{Burbea1982} and Hmidi, Mateu, and Verdera in \cite{HmidiMateuVerdera2013}. In fact, there exist branches of $m$-fold symmetric solutions bifurcating from the circle at angular velocity $\Omega_m=\varpi\frac{m-1}{2m}$ for every $m\geq 2$. $C^\infty$ regularity of the boundary was established in \cite{HmidiMateuVerdera2013}, and analyticity in \cite{CastroCordobaGomezSerrano2016}. There is a huge literature dealing with patches for 2D Euler, including many extensions of these results. For further references, see \cite{Fraenkel2000, GomezSerranoParkShiYao2021, Hmidi2015trivial, HassainiasMasmoudiWheeler2020, delaHozHassainiaHmidiMateu2016, delaHozHmidiMateuVerdera2016, HmidiMateu2016degenerate, CaoLaiZhan2021, CaoWanWangZhan2021, ChoiJeong2022, ElgindiJeong2023, ElgindiJeong2020, Garcia2020karman, Garcia2021choreography, GarciaHaziot2023, GarciaHmidiSoler2020, HassainiaHmidi2021, HassainiaWheeler2022, HmidiMateu2017, Turkington1985, Park2022, GomezSerranoParkShi2025,ElgindiJo2025,garcia2025dynamicsvortexcapsolutions,GarciaHmidiMateu2024,RaduStevenson2025}. More recently, KAM theory has been applied to construct quasi-periodic solutions, see \cite{CrouseillesFaou2013, EncisoPeraltaSalasTorres2023, BaldiMontalto2021, BertiHassainiaMasmoudi2023, HassainiaRoulley2025, HassainiaHmidiRoulley2024, HassainiaHmidiMasmoudi2025,HassainiaHmidiRoulley2024desingularization,garcia2026timeperiodicleapfroggingvortexrings}. Interval arithmetic has already been used for patches. In fact, the existence of non-convex V-states was proved in \cite{CastroGomezSerrano2025NonConvex} (see also \cite{DahGose2023}).

A classical example of a V-state are Kirchhoff ellipses \cite{Kirchhoff1874}, which correspond to the branch $m=2$. The ellipses rotate with angular velocity
\begin{align*}
    \Omega=\varpi \frac{\la}{2(1+\la)},
\end{align*}
where $\la$ denotes the aspect ratio. 

Remarkably, thanks to the explicit form of these solutions, Guo, Hallstrom, and Spirn in \cite{GuoHallstromSpirn2004} proved that for some $\lambda$, the Kirchhoff ellipses become unstable. In particular, as part of the proof, they showed that the operator arising from the linearization of the boundary evolution equation for a patch around a rotating ellipse, that is, the analogue of the operator $L_v$, has an eigenvalue with negative real part.

However, the validity of this result has not been verified for higher symmetries. That is, it is expected that for $m\geq 3$, the corresponding branches of solutions also become unstable at some point, but no such result is currently available. For more information on the properties of Kirchhoff ellipses, see \cite{CastroCordobaGomezSerrano2016,HmidiMateu2016degenerate,HmidiMateu2016Kirchhoff,WangXuZhou2022} .

In this paper, we focus on the stability of $\frac{2\pi}{3}$-periodic traveling waves for the BH equation, which correspond to threefold symmetric V-states for the 2D Euler equations.

\subsection{Organization of the paper}
The paper is organized as follows. In Section~\ref{sec:tw_exis}, we introduce a traveling wave with speed given by \eqref{val:speed} and postpone Proposition~\ref{prop:inv_exis} for later proof. In Section~\ref{sec:tw_stab}, we prove the existence of an eigenvalue with nonzero real part, while postponing Lemmas~\ref{lemma:stab_ballintoitself}--\ref{lemma:real_eigen_control}. In Section~\ref{sec:prop:inv_exis}, we establish Proposition~\ref{prop:inv_exis}, and in Section~\ref{sec:prop:inv_stab} we prove Lemmas~\ref{lemma:stab_ballintoitself}--\ref{lemma:real_eigen_control}.

Appendix~\ref{sec:clasineq} contains classical estimates on the torus, Appendix~\ref{sec:BracketHiCo} presents computations involving the Hilbert transform, and Appendix~\ref{sec:polpaly} describes an explicit polynomial arising in the analysis. Appendix~\ref{sec:ode} describes the approach to certain explicit functions, and Appendix~\ref{sec:compassisted} details the computer-assisted arguments and the supplementary data.

\subsection{Functional setting and notation}

We consider complex-valued functions. When needed, we will explicitly indicate that a function is real-valued.

In this paper, we will use two conventions for the $L^2$ norm, defined by
\begin{align*}
    \norm{f}_\XL^2:=\sum_{k\in\Z}|\ha{f}_k|^2,&&
    \norm{f}_\XLIpi^2:=\intpi |f(x)|^2\dx,
\end{align*}
where 
\begin{align*}
    \ha{f}_k:=\frac{1}{2\pi}\intpi f(x)e^{-ikx}\dx.
\end{align*}

Throughout the paper, we will work with functions $f \in C^1([-\pi,\pi])$ which are not necessarily periodic, and hence not in $C(\T)$. In this case, we write $\norm{\pa_x f}_\XLIpi$ to emphasize that we consider the classical derivative of $f$ on $(-\pi,\pi)$. 

If $f$ admits a weak derivative on the torus belonging to $L^2(\T)$, we instead write $\norm{\pa_x f}_\XL$.

Moreover, the $\norm{\cdot}_{\XLIpi}$ and $\norm{\cdot}_{\XL}$ differ by a multiplicative constant. We have,
\begin{align*}
    \norm{f}_\XL^2 = 
    \frac{1}{2\pi} 
    \intpi 
        |f(x)|^2\dx.
\end{align*}

We denote by $\XH{s}$ the Sobolev space defined as the closure of $C^\infty(\T)$ with respect to the norm
\begin{align*}
    \norm{f}_\XH{s}^2 := 
    \frac{1}{2\pi} 
    \intpi |f(x)|^2 + |\Lambda^sf(x)|^2\dx
    =
    \sum_{k\in\Z} (1+|k|^{2s}) |\ha{f}_k|^2
\end{align*}

In the next section, we work with the following Sobolev spaces
\begin{align*}
    \XHeven{s}&=\big\{ f\in \XH{s}:\, f(x)=f(-x) \,\, \text{a.e. in $\T$} \,\, \text{and}\,\, \ha{f}_0=0\big\},
    \\
    \XHodd{s}&=\big\{ f\in \XH{s}:\, f(x)=-f(-x) \,\, \text{a.e. in $\T$}\big\}.
\end{align*}
We emphasize that both $\XHeven{s}$ and $\XHodd{s}$ consist of zero-mean functions. 

We endow these spaces with the following norm
\begin{align*}
    \norm{f}_{H_{\text{odd (or even)}}^s}^2:= \sum_{k\geq 1}
    k^{2s} |\ha{f}_k|^2=\frac{1}{4\pi}\int_{-\pi}^\pi |\Lambda^sf(x)|^2\dx,
\end{align*}

A function $f\in\XHeven{s}$ can be represented by
\begin{align*}
    f(x)=\sum_{k\geq 1} a_k \cos(kx),
\end{align*}
and then
\begin{align*}
    \norm{f}_\XHeven{s}^2=\frac{1}{4}\sum_{k\geq 1}k^{2s}|a_k|^2.
\end{align*}

The operator $\pa_x^{-1}:\, H^k_{\text{odd (or even)}}\to H^{k+1}_{\text{even (or odd)}}$ is given by
\begin{align*}
    \pa_x^{-1}f(x)=
    \int_0^x f(y)\dy 
    -
    \frac{1}{2\pi}\intpi \int_{0}^x f(y)\dy \dx.
\end{align*}

$\XHholZM$ denotes the space of holomorphic functions on the unit disk $\D$ that vanish at zero. We note that if $F\in \XHholZM$ then $F$ admits the expansion
\begin{align}\label{taylorexpansion}
    F(z)=\sum_{k\geq 1}F_k z^k
\end{align}
with $F_k\in \mathbb{C}.$

In addition, we introduce the Hardy spaces 
\begin{align}\label{def:XHholZMi}
    \XHholZMi{s}
    =\{ F\in \XHholZM:\, \left. F\right|_\T\in \XH{s}\},
\end{align}
where $\left.F\right|_\T= F(e^{ix})$. These spaces are endowed with the norm
\begin{align*}
    \norm{F}_\XHholZMi{s}^2:=
    \sum_{k\geq 1}|F_k|^2 k^{2s}=
    \sum_{k\geq 1}|\ha{f}_k|^2 k^{2s}.
\end{align*}
We will use the following notation  
\begin{align*}
    f^{+} := \tfrac{1}{2}(f + i Hf),&&
    f^{-} := \tfrac{1}{2}(f^\sharp + i Hf^\sharp),
\end{align*}
where $f^{\sharp}(x):=f(-x)$. 
\begin{comment}
It will also be convenient to use 
\begin{align*}
    \Pcal^\pm f:=f^{\pm}
\end{align*}
since sometimes we will need to apply $\Pcal^\pm$ to equations rather than functions.
\end{comment}

Notice that for $f\in\XH{s}$, with $s>\frac{1}{2}$, the Fourier series of $f^+$ and $f^-$ are
\begin{align*}
    f^+(x) = \sum_{k\geq 1} \ha{f}_k e^{ikx},&&
    f^-(x) = \sum_{k\geq 1} \ha{f}_{-k} e^{ikx}.
\end{align*}

Furthermore, replacing $e^{ix}$ by $z$, we also find holomorphic extensions $F^+,F^-\in
\XHholZMi{s}$ of $f^+$ and $f^-$ to the complex unit disk,
\begin{align}
    F^+(z) = \sum_{k\geq 1} \ha{f}_k z^k,&&
    F^-(z) = \sum_{k\geq 1} \ha{f}_{-k} z^k.
    \label{def:F+F-}
\end{align}

\begin{rem}
    The definition in \eqref{def:F+F-} will slightly change in Definition~\ref{def:FplusFminu} from Section~\ref{sec:stab:invert}.
\end{rem}

For every $F\in\XHholZMi{s}$, there exists $f\in \XHodd{s}$ such that $f^+=F|_\T$. This can be checked by taking $f(x)=F|_\T(x)-F|_\T(-x)$ and the following computation
\begin{align*}
    Hf=HF|_\T(x)-HF|_\mathbb{T}^\sharp(x)=HF|_\T(x)+HF|_\mathbb{T}(-x)=-i(F|_\T(x)+F|_\T(-x)).
\end{align*}

For $n\geq 2$ and $l=1,...,n-1$, we will also use the following Fourier projectors
\begin{align*}
    \ha{\Pcal_{l,n}f}_k=\ha{f}_{nk+l}.
\end{align*}

We view vectors as \textit{column vectors}. In addition, for each $v\in\C^d$, $v^\ast$ will be $\overline{v}^t$, that is a \textit{row vector} whose entries are the complex conjugates of those of $v$. In this way, we can write, for $v_1$, $v_2\in \C^d$, the scalar product $(v_1,v_2)$ as $v_1^\ast v_2$.

Finally, given a matrix $A\in \C^{d\times d}$, we denote by $\sigma_1(A)\leq \sigma_2(A)\leq \cdots \leq \sigma_d(A)$ its singular values, arranged in \textit{increasing order}. This convention differs from the more common decreasing ordering.

\section{Traveling wave existence}\label{sec:tw_exis}

In this section, we prove the existence of a smooth traveling wave  $\tw$ for the Burgers--Hilbert equation \eqref{BHEQN}, traveling  at a fixed speed $c$. In particular, we will take 
\begin{align}\label{val:speed}
    c := \getvar{speed}.
\end{align}
More precisely, we shall prove the following theorem.
\begin{theorem}\label{thmexistencetw}
    There exists a smooth, even, real, and zero-mean traveling wave $\tw\neq 0$ of \eqref{BHEQN} with speed $c$. In addition, $\tw$ is not $\frac{2\pi}{n}$-periodic for any $n\geq 2$. Since $\frac{1}{n}(\tw(nx),\,c)$ also solves \eqref{eqn_BH_TW},  we obtain traveling waves of minimal period $\frac{2\pi}{n}$, for every $n\in\N^+$.
\end{theorem}

Section~\ref{sec:tw_exis} is devoted to the proof of Theorem~\ref{thmexistencetw}. From now on, the speed $c$ will remain fixed.

First of all, let us clarify that the existence of a branch of traveling wave $\frac{2\pi}{n}$-periodic solutions bifurcating from zero can be proved by using the classical Crandall--Rabinowitz theorem (see \cite{CCZ2}).
Moreover, it is shown in \cite{CCZ2} that the branch exists on $[1,c^\ast)$, for some finite $c^\ast$.\footnote{Numerical simulations suggest that $c^\ast\simeq 1.01393404560815$.}
However, this value is not sufficient to find instabilities and we need to go further up to $c$. 

Our strategy for proving Theorem \ref{thmexistencetw} consists in showing that there exists a unique traveling wave $\tw$ close to an explicit approximation $\twap$ that we have constructed using numerical simulations. This is based on a fixed-point argument. 

We write $\tw$ as a perturbation of the explicit approximation $\twap$,
\begin{align}\label{def_tw_as_approx_plus_diff}
    \tw=\twap+\twdiff.
\end{align}

Here, $\twap$ is explicit and defined as
\begin{align}\label{tildev}
    \twap(x):=\sum_{j=1}^{N}\twap_j\cos(jx).
\end{align}
where the real coefficients $\twap_j$, $1\leq j\leq N$, with $N=\getvar{Nexplicit}$, can be found in the supplementary material. 
The graph of $c+\twap$ is shown in Figure~\ref{fig:twap}.

\begin{figure}[htbp]
    \centering
    \includegraphics[width=8.5cm]{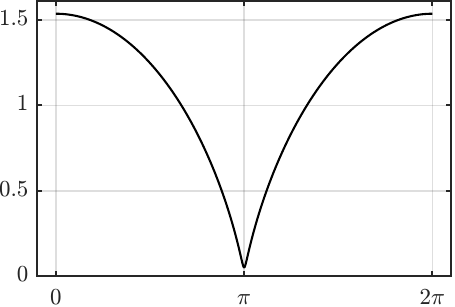}
    \caption{Graph of $c+\twap$.}
    \label{fig:twap}
\end{figure}

Plugging \eqref{def_tw_as_approx_plus_diff} into \eqref{eqn_BH_TW}, we obtain a new equation for $\twdiff$
\begin{align}
    0&=
    c\pa_x\tw+
    H\tw+
    \tw\pa_x\tw\nonumber\\
    &=
    (c\pa_x\twap+
    H\twap+
    \twap\pa_x\twap) +
    (c\pa_x\twdiff+
    H\twdiff+
    \twap\pa_x\twdiff+
    \twdiff\pa_x\twap) +
    \twdiff\pa_x\twdiff
    \nonumber\\
    &=
    (c\pa_x\twap+
    H\twap+
    \twap\pa_x\twap) +
    (c\pa_x\twdiff+
    H\twdiff+
    \pa_x(\twap\twdiff)) +
    \tfrac{1}{2}\pa_x(\twdiff^2)
    \label{eqn_traveling_diff_aux1}
\end{align}
We denote by
\begin{align*}
    \twxi&:=
    c\pa_x\twap+
    H\twap+
    \twap\pa_x\twap,\\
    \LOpExisTwap\twdiff&:=
    c\pa_x\twdiff+
    H\twdiff+
    \pa_x(\twap\twdiff),\\
    Q(\twdiff_1,\twdiff_2)&:=\tfrac{1}{2}\pa_x(\twdiff_1\twdiff_2).
\end{align*}
With this notation, equation \eqref{eqn_traveling_diff_aux1} can be written as
\begin{align}\label{eqn_traveling_diff}
    \twxi+
    \LOpExisTwap\twdiff+
    Q(\twdiff,\twdiff)=
    0.
\end{align}

In order to prove the existence of $\om\in\XHeven{1}$, one key step is to invert the linear operator $\LOpExisTwap$, so that we can write \eqref{eqn_traveling_diff} in the form 
\begin{align*}
    \om=-\LOpExisTwapInv
    \big(
    \twxi+Q(\om,\om)
    \big).
\end{align*}

Nevertheless, this may not be possible. Observe that, if $\tw$ is an even solution of equation \eqref{eqn_BH_TW}, taking a derivative, we get
\begin{align*}
    c\pa_x^2 \tw+
    H\pa_x \tw+\pa_x(\tw\pa_x\tw)=0.
\end{align*}
Then, $\pa_x\tw$ is an odd function in the kernel of the operator $\LOpExisTw$ given by
\begin{align}\label{Lv}
    \LOpExisTw f:=
    c\pa_x f+Hf+\pa_x(\tw f).
\end{align}
This may create some difficulties when inverting $\LOpExisTwap$. To avoid them, we will work in spaces with defined parity. Furthermore, we also impose that $\twdiff$ has mean zero, as $\twap$. This allows us to deduce that $\tw$ is also even and zero-mean.

We will be able to prove the following proposition which is one of the main ingredients in our proof of the existence of traveling waves.

\begin{prop}\label{prop:inv_exis}
    Let $\twap$ be the function given in \eqref{tildev} and 
    \begin{align*}
        \LOpExisTwap: \XHeven{1} \to \XHodd{0}
    \end{align*}
    the operator defined in \eqref{Lv}. Then, $\LOpExisTwap$ is invertible, and its inverse
    \begin{align*}
        \LOpExisTwapInv: \XHodd{0} \to \XHeven{1}
    \end{align*}
    satisfies
    \begin{align*}
        \norm{\LOpExisTwapInv}_{\XHodd{0}\to \XHeven{1}} < \getvar{Lexis_inv}.
    \end{align*}
\end{prop}
\begin{proof}
    Section~\ref{sec:prop:inv_exis} is devoted to proving Proposition~\ref{prop:inv_exis}.
\end{proof}

We establish all the estimates required to apply the fixed-point argument in Proposition~\ref{PropTravelingWaveExistence}.
We want to show that there exists $\om$ satisfying equation \eqref{eqn_traveling_diff}. Using \eqref{eqn:LOpExis1}, the problem is equivalent to finding a fixed point for $\TopExis$, where
\begin{align}\label{def:ContOpExis}
    \TopExis \om :=
    - 
    \LOpExisTwapInv\big(
        \twxi
        +
        Q(\om,\om)
    \big).
\end{align}
 Therefore, we aim to obtain bounds for $\twxi$ and $Q$ to be able to apply Proposition~\ref{PropTravelingWaveExistence}.

We first compute a bound for $\twxi$.
\begin{lemma}\label{lemma:residue_exis_L2}
Let $\twxi$ be the odd function given by
\begin{align*}
    \twxi:=
    c\pa_x\twap
    +
    H\twap
    +
    \twap\pa_x \twap.
\end{align*}
Then, $\twxi$ satisfies
\begin{align*}
    \norm{\twxi}
    _\XHodd{0}
    <
    \getvar{resi_exis_L2}.
\end{align*}
\end{lemma}
\proofcompassi

\begin{lemma}\label{Lemma_Bound_Exis_Q}The following estimate holds
    \begin{align*}
        \norm{Q}_{\XHeven{1}\times\XHeven{1}\to\XHodd{0}}\leq 2^{1/2}\zeta(2)^{1/2}.
    \end{align*}
\end{lemma}
\begin{proof}
    We have
    \begin{align*}
        \norm{Q(\om,\om)}_{\XHodd{0}}=
        \frac{1}{2}\norm{\om^2}_{\XHeven{1}}\leq 
        2^{1/2}\zeta(2)^{1/2}\norm{\om}_\XHeven{1}^2,
    \end{align*}
    where we have used the estimate in Corollary~\ref{coro:apriori:BanachAlg:H1Even}.
\end{proof}

\begin{prop}\label{PropTravelingWaveExistence}
    The operator $\TopExis:\XHeven{1}\to \XHeven{1}$ is a contraction on
    \begin{align*}
        X:=\{g\in \XHeven{1}:\ \norm{g}_\XHeven{1}\leq 
        \getvar{rad_exis}
        \}.
    \end{align*}
    Consequently, there exists a traveling wave $\tw$ with speed $c$ for the Burgers--Hilbert equation. Moreover, the difference between $\tw$ and our approximation $\twap$, denoted by $\twdiff := \tw - \twap$, satisfies
    \begin{align*}
        \norm{\twdiff}_{\XHeven{1}}
        <
        \getvar{rad_exis}.
    \end{align*}
\end{prop}

\begin{proof}
    We first prove that $\TopExis(X) \subset X$, and then we derive an upper bound smaller than $1$ for the Lipschitz constant $\ell_\TopExis$ of $\TopExis$ on $X$. These two facts imply that $\TopExis$ is a contraction on $X$. Since $X \neq \emptyset$ is a closed subset of the complete metric space $\XHeven{1}$, it follows that $\TopExis$ admits a fixed point $\twdiff$ in $X$.

    We now verify both assumptions. Observe that
    \begin{align}
        \norm{\TopExis \om}_{\XHeven{1}}&=
        \norm{\LOpExisTwapInv(\twxi+Q(\om,\om))}_{\XHeven{1}}
        \nonumber \\
        &\leq 
        \norm{\twxi+Q(\om,\om)}_{\XHodd{0}}
        \norm{\LOpExisTwapInv}
        _{\XHodd{0}\to\XHeven{1}}
        \nonumber \\
        &\leq 
        \big(\norm{\twxi}_{\XHodd{0}}+
        \norm{\om}_{\XHeven{1}}^2
        \norm{Q}_{\XHeven{1}\times\XHeven{1}\to\XHodd{0}}\big)
        \norm{\LOpExisTwapInv}
        _{\XHodd{0}\to\XHeven{1}}.
        \label{InequalityExistenceNormContinuity}
    \end{align}
    
    We have shown in Lemma~\ref{lemma:residue_exis_L2}, Lemma~\ref{Lemma_Bound_Exis_Q}, and Proposition~\ref{prop:inv_exis} that
    \begin{subequations}\label{eqn:subs_lemmas_contexis}
    \begin{align}
        \norm{\twxi}
        _{\XHodd{0}}&\leq
        \getvar{resi_exis_L2},
        \label{InequalityExistenceNormContinuity_explicit_norms_xi}
        \\
        \norm{Q}
        _{\XHeven{1}\times\XHeven{1}\to\XHodd{0}}&\leq
        2^{1/2}\zeta(2)^{1/2},
        \label{InequalityExistenceNormContinuity_explicit_norms_Q}
        \\
        \norm{\LOpExisTwapInv}
        _{\XHodd{0}\to\XHeven{1}}&<
        \getvar{Lexis_inv}.
        \label{InequalityExistenceNormContinuity_explicit_norms_L}
    \end{align}\end{subequations}
    
    Now, plugging \eqref{eqn:subs_lemmas_contexis} into \eqref{InequalityExistenceNormContinuity}, we obtain for every $\om\in X$ that 
    \begin{align}
        \norm{
            \TopExis
            \om
        }_\XHeven{1}
        <
        \getvar{rad_exis}.
        \label{ExistenceContractive1}
    \end{align}
    Consequently, we deduce that $\TopExis X\subset X$.

    We now derive a bound for the Lipschitz constant $\ell_\TopExis$. Let $\om_1, \om_2 \in X$, then
    \begin{align}
        \TopExis \om_1-\TopExis \om_2 &= 
        \LOpExisTwapInv[
        \twxi+
        Q(\om_2,\om_2)
        ]-
        \LOpExisTwapInv[
        \twxi+
        Q(\om_1,\om_1)
        ]\nonumber\\
        &= 
        \LOpExisTwapInv[
        Q(\om_2,\om_2)-
        Q(\om_1,\om_1)]\nonumber\\
        &= 
        \LOpExisTwapInv[
        Q(\om_2,\om_2)-Q(\om_2,\om_1)]+
        \LOpExisTwapInv[
        Q(\om_2,\om_1)-Q(\om_1,\om_1)].
        \label{EqnExistenceLipschitzAuxiliar}
    \end{align}
    
    We denote by
    \begin{align*}
        I_1:=\LOpExisTwapInv[
        Q(\om_2,\om_2)-Q(\om_2,\om_1)],&&
        I_2:=\LOpExisTwapInv[
        Q(\om_2,\om_1)-Q(\om_1,\om_1)].
    \end{align*}
    Observe that
    \begin{equation}\label{EqnExistenceLipschitzAuxiliar2}
        \left.\begin{array}{l}
        \norm{I_1}_\XHeven{1}\leq 
        \norm{\LOpExisTwapInv}
        _{\XHodd{0}\to\XHeven{1}}
        \norm{Q}
        _{\XHeven{1}\times\XHeven{1}\to\XHodd{0}}
        \norm{\om_2}_{\XHeven{1}}\norm{\om_2-\om_1}_{\XHeven{1}},\\
        \norm{I_2}
        _\XHeven{1}\leq 
        \norm{\LOpExisTwapInv}
        _{\XHodd{0}\to\XHeven{1}}
        \norm{Q}
        _{\XHeven{1}\times\XHeven{1}\to\XHodd{0}}
        \norm{\om_1}_{\XHeven{1}}\norm{\om_2-\om_1}_{\XHeven{1}}.
        \end{array}\right\}
    \end{equation}

    Plugging \eqref{InequalityExistenceNormContinuity_explicit_norms_Q} and \eqref{InequalityExistenceNormContinuity_explicit_norms_L} into \eqref{EqnExistenceLipschitzAuxiliar2}, we derive from \eqref{EqnExistenceLipschitzAuxiliar} that
    \begin{align*}
        \norm{\TopExis \om_1-\TopExis \om_2}_\XHeven{1}
        &<
        \getvar{lip_exis}
        \norm{\om_2-\om_1}_\XHeven{1}.
    \end{align*}

    Consequently, we found a bound for $\ell_\TopExis$ smaller than $1$,
    \begin{align*}
        \ell_\TopExis<
        \getvar{lip_exis}
        <1.
    \end{align*}
\end{proof}

We now conclude the proof of Theorem~\ref{thmexistencetw}.
First, observe that $\tw$ is $2\pi$-periodic, but not $\frac{2\pi}{n}$-periodic for any $n \geq 2$. 
To see this, if this were the case, then we would have
\begin{align*}
     \frac{1}{\pi} \intpi 
     \tw(x) \cos(x) \dx = 0.
\end{align*}
However, we compute
\begin{align*}
    \frac{1}{\pi}\intpi \tw(x)\cos(x)\dx
    &= \frac{1}{\pi}\intpi \twap(x)\cos(x)\dx 
    + \frac{1}{\pi}\intpi \twdiff(x)\cos(x)\dx\\
    &\geq
    \twap_1 - 
    \frac{1}{\sqrt{\pi}}
    \norm{\om}_\XLIpi > 
    0.54 - 
    2\cdot
    \getvar{rad_exis}
    >0,
\end{align*}
where we have used that the first cosine coefficient of $\twap$ satisfies $\twap_1>0.54$ and the Cauchy--Schwarz inequality.
Hence the nonvanishing first Fourier mode of $\tw$ implies that the minimal period is $2\pi$. In addition, since $\twap_1> 0.54$, we also have
\begin{align*}
    \norm{\tw}_\XHeven{1} \geq
    \norm{\twap}_\XHeven{1} - \norm{\twdiff}_\XHeven{1} >
    \frac{1}{2}\twap_1
    -
    \getvar{rad_exis} > 0,
\end{align*}
so that $\tw \neq 0$.

Second, in Proposition~\ref{PropTravelingWaveExistence}, we show the existence of a fixed point $\twdiff$ of $\TopExis$ in $X\subset\XHeven{1}$. Hence, $\tw = \twap + \twdiff \in \XHeven{1}$ solves \eqref{eqn_BH_TW}.
    
However, this traveling wave is in fact smooth.
This follows from Lemma~\ref{lemma:betamax}, which provides an upper bound for $\beta := \frac{1}{c+\twap}$. Therefore,
\begin{align*}
    c+\twap > \frac{1}{\getvar{beta_max}}.
\end{align*}

In particular,
\begin{align*}
    c+\tw \geq c+\twap - \norm{\twdiff}_\XLinf >
    \frac{1}{\getvar{beta_max}} - 2\zeta(2)^{1/2}\getvar{rad_exis}>
    0,
\end{align*}
where we have used Lemma~\ref{lemma:apriori:LInfty_H1odd}.

Since $\tw \in \XHeven{1}$ solves \eqref{eqn_BH_TW}, we have
\begin{align*}
    \pa_x \tw = -\frac{1}{c+\tw} H\tw.
\end{align*}
A bootstrap argument then shows that all derivatives of $\tw$ belong to $L^2$, and hence are continuous.

Finally, notice that the traveling wave $\tw$ is real-valued, since $\TopExis$ is a real operator.

\section{Traveling wave stability}\label{sec:tw_stab}
In Section~\ref{sec:tw_exis}, we proved that there exists a traveling wave $\tw$ with speed $c$ for the Burgers--Hilbert equation \eqref{BHEQN}, where $c$ is given by \eqref{val:speed}.

Here, we plug in the ansatz
\begin{align*}
    u(x,t)=\tw(x+ct) + w(x+ct,t)
\end{align*}
into \eqref{BHEQN}. Therefore, we obtain the following equation for $w$
\begin{align*}
    \pa_t w + w\pa_x w +\pa_x(\tw w)+Hw+c \pa_x w=0
\end{align*}
Integrating, we find
\begin{align*}
    \pa_t f +
    \frac{1}{2}(\pa_x f)^2+
    \tw\pa_x f+ 
    Hf+
    c\pa_x f=
    C(t).   
\end{align*}
where $\pa_x f = w$.
If we want to consider $f$ to have mean zero, we simply take
\begin{align*}
    C(t)=\frac{1}{2\pi}
    \intpi
    \Big(\frac{1}{2} 
    \big(\pa_x f(x)\big)^2 + \tw\pa_x f(x)\Big)\dx.
\end{align*}
At the linear level, we get that
\begin{align}
    \label{eqn:lindyn}
    \pa_t f+ 
    \LinearOpStability{c,\tw} f=0, && \LinearOpStability{c,\tw} f:=(c+\tw)\pa_x f+ Hf -\frac{1}{2\pi}\intpi \tw(x)\pa_x f(x)\dx.
\end{align}

\begin{rem}
    $\LOpExis$, defined in \eqref{def:LOpExis}, and $\LinearOpStability{c,\twap}$ have the same expression, but we restrict $\LOpExis$ to odd functions.
\end{rem}

From $\tw$ we can construct new traveling waves $\tw^n$. Indeed, $\tw^n(x):=\frac{1}{n}\tw(nx)$ is a solution of \eqref{eqn_BH_TW} with speed $\frac{c}{n}$.

The main theorem of the paper is as follows
\begin{theorem}
\label{theo:stab}
    The operator $\LinearOpStability{c/3,\tw^3}$ admits an eigenvalue $\EigenValue$ with negative real part and a smooth associated eigenfunction $\EigenVecto$.

    Hence, $(\EigenValue, \pa_x \EigenVecto)$ is an eigenpair of $L_\tw$.
\end{theorem}

A numerical approximation of $\EigenVecto$ is shown in Figure~\ref{fig:eige}.

\begin{figure}[htbp]
    \centering
    \includegraphics[width=8.5cm]{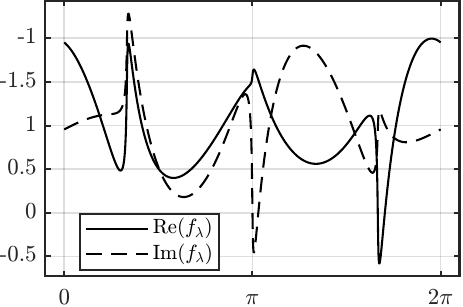}
    \caption{Numerical approximation of $\EigenVecto$.}
    \label{fig:eige}
\end{figure}

\begin{rem}
    If $f$ is an eigenfunction of $\LinearOpStability{c/3,\tw^3}$ associated with the eigenvalue $\lambda$, then the function $x \mapsto f(-x)$ is also an eigenfunction of $\LinearOpStability{c/3,\tw^3}$ associated with the eigenvalue $-\lambda$.
\end{rem}

\begin{rem}
    Notice that $\LinearOpStability{c/3,\tw^3}$ is a real operator. Therefore, if $f$ is an eigenfunction associated with the eigenvalue $\lambda$ and $\Re(\lambda)<0$, then
    \begin{align*}
        (t,x)\mapsto\Re\bigl(e^{-\lambda t} f(x)\bigr)
    \end{align*}
    is an unstable solution of \eqref{eqn:lindyn}.
\end{rem}

\begin{rem}
    Let us emphasize that we cannot guarantee that the solution $\big(\frac{c}{3},\tw^3\big)$ belongs to the branch of $\frac{2\pi}{3}$-periodic functions bifurcating from $0$. However, nonrigorous numerical evidence indicates that this is the case.
\end{rem}
The rest of Section~\ref{sec:tw_stab} is devoted to proving Theorem~\ref{theo:stab}.

In the following lemma, we analyze the Fourier structure of $\LinearOpStability{c/n,\tw^n}$.

\begin{lemma}\label{lemma:unfold_stability}
    Let $n\in\N$ with $n\geq 2$ and let $f,g:\T\to\C$ be smooth functions with
    \begin{align*}
        \ha{f}_{nk}=\ha{g}_{nk}=0,\ \forall k\in\Z.
    \end{align*}

    Let $V:\T\to\C$ be a smooth function and $V^n:=\tfrac{1}{n}V(nx)$.
    
    Then, $f$ and $g$ satisfy
    \begin{align}
    \label{eqn:unfold_aux0}
        g=V^n\pa_x f+ Hf -\frac{1}{2\pi}\intpi V^n(x)\pa_x f(x)dx
    \end{align}
    
    if and only if for every $1\leq l\leq n-1$, $f^{l,n}:=\Pcal_{l,n} f$ and $g^{l,n}:=\Pcal_{l,n} g$ with $\te:=\tfrac{l}{n}$, satisfy 
    \begin{align}
    \label{eqn:unfold_}
        g^{l,n}= V \pa_x f^{l,n} + i\te V f^{l,n} + H f^{l,n} - i\ha{f^{l,n}}_0.
    \end{align}
\end{lemma}

\begin{rem}
    We exclude the case $l = 0$. Its behavior is the same as for $n = 1$, and we currently have no results concerning its stability.
    However, we conjecture that the eigenvalues of $\tw$ with speed $1$ remain stable up to the critical speed.
\end{rem}

The numerical spectra in Figures~\ref{fig:twapLn} and~\ref{fig:twapLt} suggest that considering such a high value of the speed $c$ is necessary to detect the instability.

\begin{figure}[htbp]
    \centering
    \includegraphics[width=0.9\textwidth]{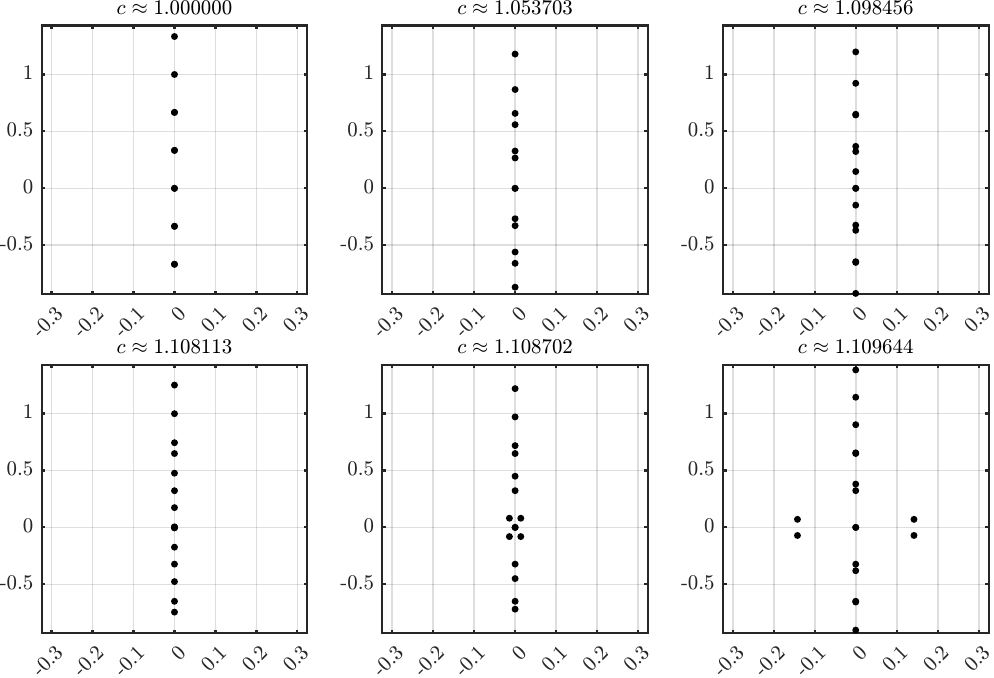}
    \caption{Numeric approximation of the spectrum of $\LinearOpStability{c/3,\tw^3}$ for different values of $c$.}
    \label{fig:twapLn}
\end{figure}

\begin{figure}[htbp]
    \centering
    \includegraphics[width=0.9\textwidth]{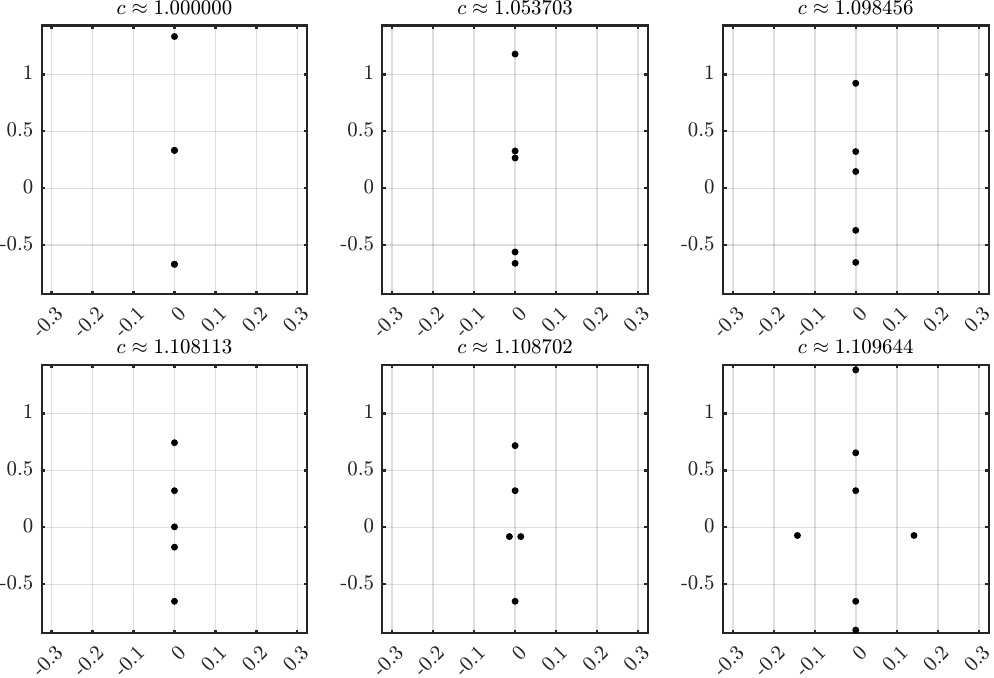}
    \caption{Numeric approximation of the spectrum of $\LinearOpStability{1/3}$ for different values of $c$.}
    \label{fig:twapLt}
\end{figure}

\begin{proof}
    We project \eqref{eqn:unfold_aux0} onto $e^{ikx}$ with $k\in\Z$ such that $n$ does not divide $k$,
    \begin{align}
        \ha{g}_k &=
        \frac{1}{2\pi} \intpi
        e^{-ikx}\big(
        V^n(x)\pa_x f(x)+
        Hf(x)
        \big)\dx
        \nonumber
        \\
        &=
        \frac{1}{2\pi} \intpi
        e^{-ikx}
        \frac{1}{n}
        \sum_{j\in\Z}
        \ha{V}_{j}e^{injx}
        \sum_{l\in\Z}
        il\ha{f}_le^{ilx}
        \dx
        -i\sgn(k)\ha{f}_k
        \nonumber
        \\
        &=
        \frac{i}{n}
        \sum_{j\in\Z}
        \sum_{l\in\Z}
        l
        \ha{V}_{j}
        \ha{f}_l
        \delta_{k,nj+l}
        -i\sgn(k)\ha{f}_k
        \nonumber
        \\
        &=
        \frac{i}{n}\sum_{j\in\Z} 
        (k-nj)\ha{V}_j\ha{f}_{k-nj}
        -i\sgn(k)\ha{f}_k.\label{eqn:unfold_aux1}
    \end{align}
    Fourier modes of $\Pcal_{l,n}f$ are defined as
    \begin{align*}
        \ha{\Pcal_{l,n}f}_k:=\ha{f}_{nk+l}.
    \end{align*}
    Then, fixing $\te:=\frac{l}{n}$, denoting by $f^\te:=\Pcal_{l,n}f$, and using identity \eqref{eqn:unfold_aux1}, we obtain
    \begin{align}
        \ha{g^{l,n}}_k&=
        \frac{i}{n}\sum_{j\in\Z} (nk+l-nj)\ha{V}_j\ha{f}_{nk+l-nj}
        -i\sgn(nk+l)\ha{f}_{nk+l}
        \nonumber
        \\
        &=
        i
        \sum_{j\in\Z} 
        (k-j+\te)\ha{V}_j \ha{f^{l,n}}_{k-j}
        -i\sgn(nk+l) \ha{f^{l,n}}_{k}
        \nonumber
        \\
        &=
        i
        \sum_{j\in\Z} 
        (j+\te)\ha{V}_{k-j} \ha{f^{l,n}}_{j}
        -i(\sgn(k) + \delta_{k,0}) \ha{f^{l,n}}_{k}.\label{eqn:unfold_aux3}
    \end{align}
    In the last step, we used that $\sgn(nk+l)-\sgn(k)=\delta_{k,0}$ since $1\leq l\leq n-1$.

    Finally, we note that \eqref{eqn:unfold_aux3} is the projection of \eqref{eqn:unfold_} onto $e^{ikx}$.
\end{proof}
By Lemma~\ref{lemma:unfold_stability} we can reduce our eigenvalue problem 
\begin{align*}
    \LinearOpStability{c/n,\tw^n}f=\la f
\end{align*}
to 
\begin{align}
    \label{Eigen_problem_theta}\mathcal{L}_\te f=\la f
\end{align}
where
\begin{align}\label{def:LStabtheTW}
     \LStabtheTW f := (c+\tw) \pa_x f + i\te (c+\tw) f + H f - i \ha{f}_0.
\end{align}

In order to find eigenpairs $(\la,f)$ for \eqref{Eigen_problem_theta}, we introduce a new operator, which approximates $\LStabtheTW$:
\begin{align}\label{def:LStabthe}
    \LStabthe f := (c+\twap) \pa_x f + i\te (c+\twap) f + H f - i\ha{f}_0.
\end{align}
The reason for introducing this operator is that it contains $\twap$ instead of $v$, and $\twap$ is known explicitly. We can then construct a suitable approximation $(\laap, \Eigveap)$ satisfying
\begin{align}\label{residual_stability}
    \LStabthe \Eigveap - \laap \Eigveap = 
    \Eigres.
\end{align}
Here $\laap$ is given by 
\begin{align}\label{def:laapvalue}
    \laap = \getvar{la_app},
\end{align}
and the expression of $\Eigveap$ can be found in the supplementary material. A numerical approximation of $\Eigveap$ is shown in Figure~\ref{fig:fvap}.

\begin{figure}[htbp]
    \centering
    \includegraphics[width=8.5cm]{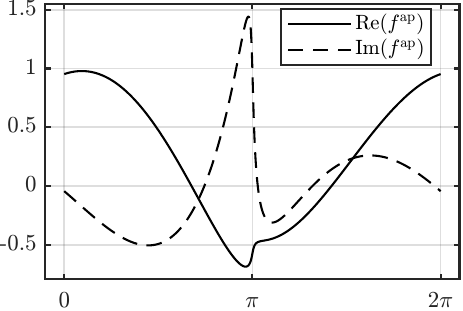}
    \caption{Numerical approximation of $\Eigveap$.}
    \label{fig:fvap}
\end{figure}

We now denote 
\begin{align*}
    \DiffOpExpNonExp:=\LStabtheTW-\LStabthe.
\end{align*}
Introducing the following ansatz into \eqref{Eigen_problem_theta}: 
\begin{align*}
    \la \mapsto \laap+\Eigdiff,&&
    f \mapsto \Eigveap+f,
\end{align*}
we obtain
\begin{align}\label{eqn:aux_cont_op2}
    0=
    (\LStabthe \Eigveap - \laap \Eigveap)+
    (\LStabthe f - 
    \laap f) +
    \DiffOpExpNonExp
    (\Eigveap+f) -
    \Eigdiff (\Eigveap+f).
\end{align}

Hence, from \eqref{eqn:aux_cont_op2} and \eqref{residual_stability}, we finally obtain
\begin{equation}\label{residual_equation}
    (\LStabthe - 
    \laap I) f = 
    (\Eigdiff I -\DiffOpExpNonExp)
    (\Eigveap+f) - \Eigres.
\end{equation}

\subsection{Fixed point argument to prove linear instability}

In this section, we derive the estimates required to apply the fixed-point argument in Proposition~\ref{prop:twStab}. Our goal is to show the existence of $\Eigdiff$ and $f$ satisfying the equation~\eqref{residual_equation}. This problem can be reformulated as finding a fixed point of the operator $\TopStab$, defined by
\begin{align}\label{def:ContOpStab}
    \TopStab f :=
    (\LStabthe - 
    \laap I)^{-1}\!\big[(\Eigdiff(f) I -\DiffOpExpNonExp)
    (\Eigveap+f) - \Eigres\big].
\end{align}
where $\Eigdiff(f):\XH{1}\to\C$ will be defined later in Section~\ref{sec:eta_estimates}. 

The operator $(\LStabthe - \laap I)^{-1}$ is nearly non-invertible. However, with a particular choice of  $\Eigdiff$ in terms of $f$, we can overcome this problem. With this choice, $\TopStab$, with $\te=\frac{1}{3}$, will become a contraction on an appropriate closed ball of $\XH{1}$.

Accordingly, we aim to obtain bounds for $\DiffOpExpNonExp$, $\Eigveap$, $\Eigres$, and $(\LStabthe - \laap I)^{-1}$, and to verify that the latter operator is indeed well defined.

\begin{lemma}\label{lemma:DiffOpExpNonExp} The following estimate holds
    \begin{align*}
        \norm{\DiffOpExpNonExp}_{\XH{1}\to\XL}\leq 
        2\zeta(2)^{1/2}(1+\te^2)^{1/2}
        \norm{\tw-\twap}_\XHeven{1}.
    \end{align*}
\end{lemma}
\begin{proof}
    By definition, the operator $\DiffOpExpNonExp$ is given by
    \begin{align*}
        \DiffOpExpNonExp f&:=
        \LStabtheTW f-\LStabthe f\\
        &=
        \big[
        (c+\tw) \pa_x f + i\te (c+\tw) f + H f - i\ha{f}_0\big]
        -
        \big[
        (c+\twap) \pa_x f + i\te (c+\twap) f + H f - i\ha{f}_0\big]\\
        &=
        (\tw-\twap) (\pa_x f + i\te f).
    \end{align*}

    Taking norms,
    \begin{align*}
        \norm{\DiffOpExpNonExp f}_\XL\leq
        \norm{\tw-\twap}_\XLinf(
        \norm{\pa_x f}_\XL+
        \te 
        \norm{f}_\XL).
    \end{align*}

    Note that we can apply Peter--Paul inequality in the optimal way
    \begin{align*}
        (\norm{\pa_x f}_\XL+
        \te 
        \norm{f}_\XL)^2 = &
        \norm{\pa_x f}_\XL^2+
        \te^2\norm{f}_\XL^2+
        2\te\norm{\pa_x f}_\XL\norm{f}_\XL
        \\ 
        \leq &
        (1+\vartheta)\norm{\pa_x f}_\XL^2+
        (1+\vartheta^{-1})\te^2\norm{f}_\XL^2.
    \end{align*}
    Taking $\vartheta=\te^2$, we get
    \begin{align*}
        (\norm{\pa_x f}_\XL+
        \te 
        \norm{f}_\XL)^2 
        \leq 
        (1+\te^2) 
        (\norm{\pa_x f}_\XL^2 + \norm{f}_\XL^2)=
        (1+\te^2)
        \norm{f}_{\XH{1}}^2.
    \end{align*}

    Finally, by Lemma~\ref{lemma:apriori:LInfty_H1odd}
    \begin{align*}
        \norm{\DiffOpExpNonExp f}_\XL\leq
        2
    \zeta(2)^\frac{1}{2}
        (1+\te^2)^\frac{1}{2}
        \norm{\tw-\twap}_\XHeven{1}
        \norm{f}_{\XH{1}}.
    \end{align*}
\end{proof}

\begin{corollary}
\label{coro:DiffOpExpNonExp}
    For $\te=\frac{1}{3}$, we have
    \begin{align*}
        \norm{\DiffOpExpNonExp}_{\XH{1}\to\XL}\leq
        \getvar{Dthe_op}.
    \end{align*}
\end{corollary}
\begin{proof} We apply Lemma~\ref{lemma:DiffOpExpNonExp} 
    and Proposition~\ref{PropTravelingWaveExistence}.
\end{proof}

\begin{lemma}
\label{lemma:norm_fv}
    The following estimates hold
    \begin{align*}
        \norm{\Eigveap}_\XL^2 <
        \getvar{fvap_stab_L2_sq},
        &&
        \norm{\Eigveap}_\XH{1}^2<
        \getvar{fvap_stab_H1_sq},&&
        \norm{\Eigres}_\XL^2<
        \getvar{resi_stab_L2_sq}.
    \end{align*}
\end{lemma}
\proofcompassi

\begin{prop}\label{prop:twStab}
    For $\theta=\frac{1}{3}$,
    the operator $\TopStab:\XH{1}\to\XH{1}$ is a contraction on
    \begin{align*}
        X:=\{g\in\XH{1}:\ \norm{g}_{\XH{1}}\leq
        \getvar{rad_stab}
        \}.
    \end{align*}
    Therefore, the operator $\LStabtheTW$ admits an eigenvalue $\la$ with positive real part. In fact,
    \begin{align*}
        \Re(\la) > \getvar{lare}.
    \end{align*}
\end{prop}

\begin{proof}   
    The proof is based on the following lemmas:
    \begin{lemma}
    \label{lemma:stab_ballintoitself}
        $T^\te_{\text{stab}}: \XH{1}\to\XH{1}$ is a continuous operator that maps 
        \begin{align*}
            X=\{f\in \XH{1}:\, \norm{f}_\XH{1}\leq \getvar{rad_stab}\}
        \end{align*}
        into itself.
    \end{lemma}
    
    \begin{proof} 
        The proof will be given in Section~\ref{sec:prop:inv_stab}.
    \end{proof}
    
    \begin{lemma}
    \label{lemma:stab_lipschitz}
    The following estimates hold 
        \begin{align*}
            \norm{
                \TopStab f_1
                -
                \TopStab f_2
            }_\XH{1}
            \leq 
            \getvar{lip_stab} 
            \norm{f_1-f_2}_\XH{1},
        \end{align*}
        for all $f_1, f_2\in X.$
    \end{lemma}
    \begin{proof}
        The proof will be given in Section~\ref{sec:prop:inv_stab}.
    \end{proof}
       
    Then, we can apply the Banach fixed-point theorem. Hence, there exists $f_0\in \XH{1}$, a fixed point for $\TopStab$, i.e.
    \begin{align*}
        (\LStabthe-\laap I)f_0=
        [(\Eigdiff(f_0)I-\DiffOpExpNonExp)
        (f_0+\Eigveap)-\Eigres]. 
    \end{align*}
    Consequently, $\la:=\laap+\Eigdiff(f_0)$ and $\Eigveap+f_0$ solve the equation \eqref{Eigen_problem_theta}. 

    In the following lemma, we compute a bound for $\Eigdiff(f_0)$ which ensures that $\la$ still has positive real part.
    \begin{lemma}\label{lemma:real_eigen_control}
        For all $f\in X$, it holds that
        \begin{align*}
            \Re(\laap) - |\Eigdiff(f)| > \getvar{lare}.
        \end{align*}
    \end{lemma}
    \begin{proof}
        The proof will be given in Section~\ref{sec:prop:inv_stab}. 
    \end{proof}

    Finally, the following lemma shows that $\Eigveap+f_0\neq 0$,
    \begin{lemma}\label{lemma:eignot0}
        It holds that
        \begin{align*}
            \norm{\Eigveap}_\XH{1} > \getvar{rad_stab}.
        \end{align*}
    \end{lemma}
    \proofcompassi
\end{proof}

\subsection{Conclusion of the proof of Theorem~\ref{theo:stab}}
    Proposition~\ref{prop:twStab} states that the operator $\LStabtheTW$ admits an eigenpair $(\la,f)$ with
    \begin{align*}
        \Re(\la) > \getvar{lare}.
    \end{align*}

    Applying Lemma~\ref{lemma:unfold_stability} with $n=3$ and $l=1$, hence $\te=\frac{1}{3}$, we deduce that $\la$ is an eigenvalue of $\LinearOpStability{c/3,\tw^3}$. We construct the associated eigenfunction $\ti{f}$ as follows
    \begin{align*}
        \ti{f}(x) = e^{ix}f(3x).
    \end{align*}
    Finally, substituting $x\mapsto -x$, we find that $\EigenValue$ is an unstable eigenvalue for $\LinearOpStability{c/3,\tw^3}$ with eigenfunction $\EigenVecto(x)=\ti{f}(-x)$. 

    Again, as at the end of Section~\ref{sec:tw_exis}, a bootstrap argument shows that $\EigenVecto$ is, in fact, smooth.
    
\section{Proof of Proposition \ref{prop:inv_exis}}
\label{sec:prop:inv_exis}
We define the operator $\LOpExis: \XHodd{2}\to \XHeven{1}$ as 
\begin{align}\label{def:LOpExis}
    \LOpExis f(x)=
    H\!f(x)+
    \big(c+\twap(x)\big)\pa_x f(x)
    -\frac{1}{2\pi}
    \int_{-\pi}^\pi
    \twap(x)\pa_x f(x)\dx.
\end{align}
Then, for any smooth function $f:\T\to\RR,$ we find that
\begin{align}\label{eqn:LOpExis1}
    \pa_x \LOpExis f=\LOpExisTwap \pa_x f.
\end{align}

Therefore,
\begin{align}\label{eqn:LOpExis2}
    \LOpExisTwap = \pa_x \LOpExis \pa_x^{-1}.
\end{align}

Let us explain a possible strategy to study the equation $\LOpExis f = g$. If we apply the Hilbert transform to the equation and consider the function $f^+ := \frac{1}{2} (f + i Hf)$, as we will see later, one can find a first-order ODE for $f^+$ which depends on $N$ Fourier coefficients of $f^+$. This equation can be easily integrated. Computing the Fourier transform of the solution, one can arrive at a linear system that must be satisfied by these Fourier coefficients. One may think that the problem has been reduced to showing that this system is invertible. 
However, there are additional conditions to check in order to guarantee that the function is periodic and all non-positive Fourier modes vanish. To overcome this problem, we will use a slightly different strategy that requires the introduction of a new operator.

The next lemma, Lemma~\ref{lemma:exis:td}, states that studying the invertibility of $\LOpExis$  is equivalent to solving a local problem in the unit complex disk.

We first define the operator $\LOpExisHol:\XHholZM\to\XHholZM$ as
\begin{align}\label{def:LOpExisHol}
    \LOpExisHol F (z):= izP(z)\pa_z F(z) - i F(z) 
        -\sum_{k=1}^N F_k\mathcal{G}_k(z),
\end{align}
where we are using \eqref{taylorexpansion}, and
$P(z)$ and $\mathcal{G}(F)(z)$ are given by
\begin{align}
    P(z)&:=\
    c+\frac{1}{2}\sum_{j=1}^N \twap_j\big(z^{j}+z^{-j}\big),
    \label{def:P_}\\
    \mathcal{G}_k(z) &:=        
        \frac{i}{2}
        k\twap_k-
        \frac{i}{2}\sum_{m=k+1}^N
        k\twap_m
        \big(z^{m-k}-z^{-(m-k)}\big).
        \label{def:Gkexis}
\end{align}

The operator $\LOpExisHol$ is well defined. Indeed, for any $F\in\XHholZM$, the expression in \eqref{def:LOpExisHol} defines a holomorphic function vanishing at the origin, since all terms involving non-positive powers of $z$ cancel out.

\begin{lemma}\label{lemma:exis:td}
    The operator $\LOpExisHol:\XHholZMi{2}\to\XHholZMi{1}$ is invertible if and only if 
    $\LOpExis:\XHodd{2}\to\XHeven{1}$ is invertible.

    Furthermore, assuming both are invertible, the following identity holds:
    \begin{align*}
        \norm{(\LOpExisHol)^{-1}}
        _{\XHholZMi{1}\to\XHholZMi{2}}
        =
        \norm{(\LOpExis)^{-1}}
        _{\XHeven{1}\to\XHodd{2}}.
    \end{align*}
\end{lemma}

\begin{proof}
    First, we show that if $f \in \XHodd{2}$ satisfies 
    \begin{align}\label{eq:fg}
        Hf(x)+
        \big(c+\twap(x)\big)\pa_x f(x)
        -\frac{1}{2\pi}
        \intpi
        \twap(x)\pa_x f(x)\dx=g,
    \end{align}
    
    then $f^+$ also satisfies 
    \begin{equation}\label{eqn:exis:td2}
                \big(c+\twap(x)\big) 
                \pa_x f^+(x)
                =
                if^+(x)+\mathcal{J}(x) +g^+(x)
    \end{equation}
    with
    \begin{align*}
        \mathcal{J}(x)&:=
        \frac{i}{2}\sum_{j=1}^Nj\twap_j\ha{f^+}_j-
        \frac{i}{2}\sum_{j=1}^N\twap_j\sum_{m=1}^{j-1}
        m\ha{f^+}_m 
        (e^{i(j-m)x} - e^{-i(j-m)x}).
    \end{align*}
    
    From \eqref{eq:fg}, we obtain for $f^+=\frac{1}{2}(f+iHf)$ that
    \begin{equation}\label{eqn:exis:td12}
        -if^+(x)+
        \big(c+\twap(x)\big)
        \pa_x f^+(x)+
        \frac{i}{2}[H,\twap]\pa_x f(x)
        -\frac{1}{4\pi}
        \int_{-\pi}^\pi
        \twap(x)\pa_x f(x)\dx=g^+.
   \end{equation}

   We now compute the integral term in \eqref{eqn:exis:td12}:
   \begin{align*}
        \frac{1}{4\pi}
        \intpi
        \twap(x)\pa_x f(x)\dx
        &=
        -\frac{1}{4\pi}
        \int_{-\pi}^\pi
        \pa_x\twap(x) f(x)\dx\\
        &=
        \frac{1}{4\pi}
        \int_{-\pi}^\pi
        \sum_{j=1}^Nj\twap_j\sin(jx) f(x)\dx
    \end{align*}

    By construction, $f$ is odd, so $\ha{f}_k=-\ha{f}_{-k}$ and we obtain
    \begin{align}
    \label{eqn:exis:td22}
        \frac{1}{4\pi}
        \intpi
        \twap(x)
        \pa_x f(x)\dx
        =
        \frac{i}{2}
        \sum_{j=1}^N
        j \twap_j
        \ha{f^+}_j.
    \end{align}
    
    The computation of $[H,\twap]\pa_x f$ is given in Corollary~\ref{coro:BracketHiTw:df}. Hence, as $f$ is odd, computations yield
    \begin{align}
        \frac{i}{2}[H,\twap]\pa_x f(x)&=
        \frac{i}{2}\sum_{j=1}^N
        \twap_j
        \Big(-\frac{j}{2}(\ha{f}_j+\ha{f}_{-j})
        -\sum_{m=1}^{j-1}
        m(\ha{f}_me^{-i(j-m)x}+\ha{f}_{-m}e^{i(j-m)x})
        \Big)\nonumber\\
        &=
        \frac{i}{2}\sum_{j=1}^N
        \twap_j
        \sum_{m=1}^{j-1}
        m\ha{f}_m(e^{i(j-m)x}-e^{-i(j-m)x}).
        \label{eqn:exis:td1}
    \end{align}

    Expressions \eqref{eqn:exis:td1} and \eqref{eqn:exis:td22}, together with \eqref{eqn:exis:td12}, yield \eqref{eqn:exis:td2}.
    
    Conversely, if $F|_\T$, for $F\in\XHholZMi{2}$, satisfies  
    \begin{align}\label{eq:F|G|}
    (c+\twap(x)) \pa_x F|_\T(x)=
                iF|_\T(x)+\mathcal{J}(x) +G|_\T(x),
    \end{align}
    with
    \begin{align*}
            \mathcal{J}(x)&:=
            \frac{i}{2}
            \sum_{j=1}^N
            j\twap_j
            \ha{F|_\T}_j-
            \frac{i}{2}\sum_{j=1}^N\twap_j\sum_{m=1}^{j-1}
            m\ha{F|_\T}_m 
            (e^{i(j-m)x} - e^{-i(j-m)x}).        
    \end{align*}
    and $G\in \XHholZMi{1}$, then the function $f(x)=F|_\T(x)-F|_\T(-x)$ satisfies \eqref{eq:fg} with $g(x)=G|_\T(x)+G|_\T(-x)\in \XHeven{1}$. To check this, we just note that
    \begin{align*}
        \pa_x f(x)&=
        \pa_x F|_\T(x)+\pa_xF|_\T(-x),\\
        Hf(x)&=-i\big(F|_\T(x)+F|_\T(-x)\big),\\
        \mathcal{J}(x)+\mathcal{J}(-x)&=
        \frac{1}{2\pi}
        \intpi \twap(x) \pa_x f(x)\dx,
    \end{align*}
    where we have used \eqref{eqn:exis:td22} and $\ha{f^+}_n=\ha{f}_n=(\ha{F|_\T})_n$, for $n\geq 1$, in the identity for $\mathcal{J}$.
    
    Let us now assume that $\LOpExisHol$ is invertible and that $g\in\XHeven{1}$. We define $G\in\XHholZMi{1}$ as the analytic extension of $g^+$ to the interior of the unit disk. Then, there exists a unique $F\in\XHholZMi{2}$ such that $\LOpExisHol F =G$. 
    
    The restriction $F|_\T$ will satisfy \eqref{eq:F|G|} and therefore $f=F|_\T(x)-F|_\T(-x)\in \XHodd{2}$ solves 
    \begin{align*}
        \LOpExis f=G|_\T(x)+G|_\T(-x)=g(x).
    \end{align*}
    Furthermore, this solution must be unique. Otherwise, if there exists $f_1\neq f$ solving $\LOpExis f_1=g$, the holomorphic extension of $f_1^+$, $F_1$ satisfies $\LOpExisHol F_1 =G$ and then $F_1=F$, which yields $f_1=f.$
    
    Finally, let us assume that $\LOpExis$ is invertible and that $G\in\XHholZMi{1}$. Take $g=G|_\T(x)+G|_\T(-x)$ and solve 
    \begin{align*}
        \LOpExis f = g,
    \end{align*}
    with $f\in\XHodd{2}$. Then $f^+$ solves \eqref{eqn:exis:td2} and its holomorphic extension $F\in\XHholZMi{2}$ solves $\LOpExisHol F=g^+=G.$ This solution must  also be unique since, if there exists another solution $F_1$, then $f_1=F_1|_\T(x)-F_1|_\T(-x)$ solves $\LOpExisHol f_1=g$ and therefore $f_1=f$, which implies $F_1=F.$

    Moreover, we have
    \begin{align*}
        \norm{g}_\XHeven{1}^2
        &= \sum_{k\geq 1}k^2|\ha{g}_k|^2
        = \norm{G}_{\XHholZMi{1}}^2,
        \\
        \norm{f}_\XHodd{2}^2
        &= \sum_{k\geq 1}k^4|\ha{f}_k|^2
        = \norm{F}_{\XHholZMi{2}}^2.
    \end{align*}
    Hence, the norms of both operators coincide.
\end{proof}

We study $\LOpExisHol$ instead of $\LOpExis$. 
In Section~\ref{sec:OpExis}, we analyze $\LOpExisHol$. 
First, in Section~\ref{sec:fuchs}, we establish some theoretical results in Fuchsian theory. 
Then, in Section~\ref{sec:exis:invert}, we prove that $\LOpExisHol$ admits an inverse when considered between appropriate functional spaces. 
In Section~\ref{sec:exis:bounds}, we derive suitable bounds. 
Finally, we combine these results to prove Proposition~\ref{prop:inv_exis} in Section~\ref{sec:exis:laststep}.

\subsection{The linear operator $\LOpExisHol$}
\label{sec:OpExis}

In order to study the operator $\LOpExisHol$, given in \eqref{def:LOpExisHol}, our first idea is to momentarily disregard the Taylor coefficients $F_j$ in expression \eqref{def:LOpExisHol} and replace them with general constants.

\begin{lemma}\label{lemma:exis:holo_implies}
    Let $a_k \in \C$ for $k = 1, \ldots, N$, and let $F\in\XHholZM$ and $G\in\XHholZM$.
    
    Assume that $F$ and $G$ satisfy
    \begin{align}\label{eqn:Exis_Holo_Implies_aux0}
        i z P(z) \pa_z F(z)
        = i F(z)
        + \frac{i}{2} \sum_{j=1}^N j \twap_j a_j
        -\frac{i}{2} \sum_{j=2}^N \twap_j
            \sum_{m=1}^{j-1}
            m \big(z^{j-m} - z^{-(j-m)}\big) a_m
        + G(z).
    \end{align}
    Then, for every $k = 1, \ldots, N$, we have
    \begin{align*}
        F_k:=\frac1{k!}\left.\frac{d^kF}{dz^k}\right|_{z=0} = a_k.
    \end{align*}
\end{lemma}

\begin{proof}
    We argue by comparing Taylor coefficients at $z=0$.

    Multiplying \eqref{eqn:Exis_Holo_Implies_aux0} by $z^N$ and using that $F(0)=G(0)=0$, we obtain
    \begin{align}\label{eqn:Exis_Holo_Implies_coeff0}
        i z^{N+1}P(z) \pa_z F(z)
        =\frac{i}{2} z^N \sum_{j=1}^N j \twap_j a_j
        +\frac{i}{2}\sum_{j=2}^N \twap_j\sum_{m=1}^{j-1} m z^{N-(j-m)}a_m
        +\mathcal{O}(z^{N+1}).
    \end{align}

    Expanding terms in \eqref{eqn:Exis_Holo_Implies_coeff0} yields
    \begin{align}\label{eqn:Exis_Holo_Implies_coeff1}
        \frac{i}{2}
        \sum_{j=1}^N 
        \sum_{l=0}^{j}
        \twap_j l F_l z^{N+l-j}
        = 
        \frac{i}{2} z^N \sum_{j=1}^N j \twap_j a_j
        +
        \frac{i}{2}\sum_{j=2}^N \twap_j\sum_{m=1}^{j-1} m z^{N-(j-m)}a_m
        +\mathcal{O}(z^{N+1}).
    \end{align}

    Fix $k\in\{1,\ldots,N-1\}$, apply $\frac{1}{k!}\frac{d^{k}}{dz^{k}}$ to \eqref{eqn:Exis_Holo_Implies_coeff1}, and evaluate at $z=0$. We have
    \begin{align}
        \frac{i}{2}
        \sum_{j=1}^N 
        \sum_{l=0}^{j-1}
        l\twap_j F_l \delta_{k,N+l-j}
        &= 
        \frac{i}{2}\sum_{j=2}^N 
        \twap_j\sum_{m=1}^{j-1} m a_m \delta_{k,N+m-j}
        \ \Rightarrow\nonumber
        \\
        0
        &= 
        \sum_{j=N+1-k}^N 
        (k+j-N)\twap_j (F_{k+j-N}-a_{k+j-N})
        \nonumber
        \\
        &= 
        \sum_{j=1}^{k} 
        j\twap_{j+N-k} (F_{j}-a_{j}).
        \label{eqn:Exis_Holo_Implies_coeff2}
    \end{align}

    We now use the following lemma
    \begin{lemma}\label{lemma:twapprox_nonzero_coeffs}
        It holds for every $j=1,\ldots,N$ that
        \begin{align*}
            \twap_j\neq 0.
        \end{align*}
    \end{lemma}
    \proofcompassi

    Since $\twap_j\neq 0$ for every $j=1,\ldots,N$ and \eqref{eqn:Exis_Holo_Implies_coeff2} is a triangular system, we conclude that $F_k=a_k$ for $k=1,\ldots,N-1$.

    Finally, we apply $\frac{1}{N!}\frac{d^N}{dz^N}$ to \eqref{eqn:Exis_Holo_Implies_coeff1} and evaluate at $z=0$: 
    \begin{align*}
        \frac{i}{2}
        \sum_{j=1}^N 
        j \twap_j F_j
        = 
        \frac{i}{2} \sum_{j=1}^N j  \twap_j a_j.
    \end{align*}

    Since $F_j=a_j$ for $j=1,\ldots,N-1$, we conclude that
    \begin{align*}
        0 = N\twap_N(F_N-a_N),
    \end{align*}
    and thus $F_N=a_N$ because $\twap_N\neq 0$.
\end{proof}

Our goal is to reduce the study of $\LOpExisHol$ to a finite dimensional problem. 
The next section is devoted to developing some tools we need to prove that the operator $\LOpExisHol:\XHholZMi{2}\to\XHholZMi{1}$ is one-to-one.

\subsubsection{Fuchsian theory}\label{sec:fuchs}
The tools in this subsection allow us to determine when the solutions of a singular first-order ODE are holomorphic in a simply connected domain $\Omega \subset \C$. We focus on the affine case. See \cite{ince1956ordinary,whittaker1996modern} for background on linear equations with regular singularities.

Let $\Omega \subset \C$ be open, connected, and simply connected, and let $A,B:\Omega\to\C$ be holomorphic (and continuous on $\overline{\Omega}$ when needed). Fix $z_0\in\Omega$. We study conditions on $A$ and $B$ ensuring that a solution $F$ of
\begin{align}\label{eqn:ODE_sing}
    (z - z_0)\pa_z F(z) + A(z) F(z) = B(z)
\end{align}
extends holomorphically across $z_0$.

For any $\om_0\in\Omega$, $\om_0\neq z_0$, equation \eqref{eqn:ODE_sing} is regular at $\om_0$ and admits a one-parameter family of holomorphic solutions in a neighborhood of $\om_0$. Any singular behavior can only occur at $z_0$. By analytic continuation, solutions defined away from $z_0$ can be extended along paths in $\Omega-\{z_0\}$, which allows us to analyze their behavior near $z_0$.
\begin{comment}
\begin{geometricinterpretation}
    Readers unfamiliar with this geometric language may safely skip this paragraph, since the rest of the subsection is self-contained.
    
    Geometrically, equations \eqref{eqn:ODE_sing} and \eqref{eqn:fuchs_crit_disk} define rank-one flat connections with regular singularities. Our non-integrality assumptions on the local exponents (namely $A(z_0)\notin\Z$ for \eqref{eqn:ODE_sing}, and $\frac{A(z_k)}{Q'(z_k)}\notin\Z$ for \eqref{eqn:fuchs_crit_disk}) imply that the local monodromy around each singular point is nontrivial: after one full turn, a nonzero solution of the homogeneous equation cannot return to its original value. Alternatively, the affine equation admits a unique local holomorphic solution near each singular point (Lemma~\ref{lemma:fuchs_sing_part}).
    
    Lemma~\ref{lemma:fuchs_criteria} characterizes when this canonical local solution extends to a global holomorphic solution on the unit disk with finitely many singular points. This happens if and only if the associated monodromy representation of $\pi_1(\D-\{z_1,\ldots,z_n\})$ is trivial, a condition that can be detected by transporting an initial value $c_0$ along the generating loops $\tau_k$. Equivalently, the functions $f_k$ encode parallel transport along $\tau_k$, and the conditions $f_k(2\pi+2)=c_0$ express trivial monodromy. We nevertheless present the proof in explicit analytic terms, avoiding additional geometric notation.
\end{geometricinterpretation}
\end{comment}

We now return to the analytic development. We begin with the regular (non-singular) case, which also illustrates the analytic continuation mechanism.

\begin{lemma}\label{lemma:fuchs_regular}
    Let $\Omega \subset \C$ be open and simply connected, and let $A,B$ be holomorphic on $\Omega$. Then for any $\om_0 \in \Omega$ and any $c_0\in\C$, there exists a unique holomorphic function $F:\Omega\to\C$ solving
    \begin{equation}\label{eqn:fuchs_regular}
    \left\{
    \begin{array}{l}
        \pa_z F + A F = B,\\[4pt]
        F(\om_0) = c_0.
    \end{array}
    \right.
    \end{equation}
\end{lemma}

\begin{proof}
    Define
    \begin{align*}
        F_h(z) := \exp\!\left(-\int_{\om_0}^{z} A(y) \dy\right),
    \end{align*}
    where the integral is taken along any smooth path in $\Omega$ joining $\om_0$ to $z$. Since $\Omega$ is simply connected and $A$ is holomorphic, the integral is path-independent and $F_h$ is well defined and holomorphic on $\Omega$, with $F_h(\om_0)=1$.

    Note that $\pa_z \!\left(\tfrac{1}{F_h}\right)=A\tfrac{1}{F_h}$. Multiplying \eqref{eqn:fuchs_regular} by $\tfrac{1}{F_h}$ gives
    \begin{align*}
        \frac{B}{F_h}
        = \pa_z F\frac{1}{F_h} + F\pa_z\!\bigg(\frac{1}{F_h}\bigg)
        = \pa_z\!\bigg(\frac{F}{F_h}\bigg).
    \end{align*}
    Integrating and using $F(\om_0)=c_0$ yields
    \begin{align*}
        F(z)=c_0F_h(z)+F_h(z)\int_{\om_0}^{z}\frac{B(y)}{F_h(y)}\dy.
    \end{align*}
    This function is holomorphic on $\Omega$ and is the unique solution to \eqref{eqn:fuchs_regular}.
\end{proof}

Although solutions to \eqref{eqn:ODE_sing} need not be holomorphic at $z_0$, there are special cases (for instance, the trivial solution $F= 0$ when $B= 0$).
Our goal is to give criteria that guarantee holomorphicity at the singular point.

\begin{lemma}\label{lemma:fuchs_sing_homo}
    Let $\Omega \subset \C$ be open and simply connected, and let $A:\Omega\to\C$ be holomorphic.
    Fix $z_0\in\Omega$ and consider the homogeneous equation
    \begin{align}\label{eqn:ODE_sing_hom}
        (z-z_0)\pa_z F_h(z)+A(z)F_h(z)=0.
    \end{align}
    Then for any simply connected set $U\subset\Omega-\{z_0\}$, every solution on $U$ is a complex multiple of
    \begin{align}\label{def:F_h}
        F_h(z):=(z-z_0)^{-A(z_0)}\psi(z),
    \end{align}
    where $\psi$ admits a holomorphic extension to $\Omega$ and satisfies $\psi(z_0)=1$, and where
    \begin{align*}
        (z-z_0)^{-A(z_0)}:=\exp\!\big(-A(z_0)\log(z-z_0)\big)
    \end{align*}
    is defined using a branch of $\log(z-z_0)$ on $U$.
    
    In particular, if $A(z_0)\notin \Z^- \cup \{0\}$, the only solution that extends holomorphically to $z_0$ is the trivial zero solution.
\end{lemma}

\begin{proof}
    Write
    \begin{align*}
        A(z)=A(z_0)+(z-z_0)A^{(1)}(z),
    \end{align*}
    where $A^{(1)}$ is holomorphic on $\Omega$. Fix a simply connected $U\subset\Omega-\{z_0\}$ and a branch of $\log(z-z_0)$ on $U$. Solving by separation of variables gives
    \begin{align*}
        F_h(z)
        &=\exp\!\left(-\int_{\om_0}^{z}\frac{A(y)}{y-z_0}\dy\right)\\
        &=\exp\!\left(-A(z_0)\int_{\om_0}^{z}\frac{dy}{y-z_0}\right)
        \exp\!\left(-\int_{\om_0}^{z}A^{(1)}(y)\dy\right),
    \end{align*}
    where $\om_0\in U$ is fixed and all integrals are taken along paths in $U$. Using the chosen branch of $\log(z-z_0)$,
    \begin{align*}
        \int_{\om_0}^{z}\frac{dy}{y-z_0}=\log(z-z_0)-\log(\om_0-z_0).
    \end{align*}
    Hence
    \begin{align*}
        F_h(z)=C(z-z_0)^{-A(z_0)}\psi(z),
    \end{align*}
    where $C\in\C$ is constant and
    \begin{align}\label{def:psi_fuchs}
        \psi(z):=\exp\!\left(-\int_{z_0}^{z}A^{(1)}(y)\dy\right)
    \end{align}
    is holomorphic on $\Omega$ with $\psi(z_0)=1$ (the integral is path-independent in $\Omega$). The last assertion follows because $(z-z_0)^{-A(z_0)}$ is holomorphic at $z_0$ if and only if $-A(z_0)\in\Z^+\cup\{0\}$.
\end{proof}

We now turn to the affine equation \eqref{eqn:ODE_sing}. When $A(z_0)\notin\Z^-\cup\{0\}$, one can construct a holomorphic solution near $z_0$ by a convergent power series, and then extend it holomorphically to all of $\Omega$ by analytic continuation.

\begin{lemma}\label{lemma:fuchs_sing_part}
    Let $\Omega\subset\C$ be open and simply connected and let $A,B:\Omega\to\C$ be holomorphic. Fix $z_0\in\Omega$ and assume $A(z_0)\notin\Z^-\cup\{0\}$. Then there exists a holomorphic function $F_p:\Omega\to\C$ solving \eqref{eqn:ODE_sing}.
\end{lemma}

\begin{proof}
    Let $\psi$ be as in \eqref{def:psi_fuchs} and set
    \begin{align*}
        \psi_B(z):=\frac{B(z)}{\psi(z)}.
    \end{align*}
    Since $\psi$ never vanishes, $\psi_B$ is holomorphic on $\Omega$. Choose $\delta>0$ such that $D_{\delta}(z_0)\subset\Omega$ and write the Taylor expansion
    \begin{align*}
        \psi_B(z)=\sum_{k\geq 0}\psi_{B,k}(z-z_0)^k
    \end{align*}
    on $D_{\delta}(z_0)$.

    Define, for $z\in D_{\delta}(z_0)$,
    \begin{align}\label{def:aff_fuchs_sol}
        F_p(z):=\psi(z)\sum_{k\geq  0}\frac{\psi_{B,k}}{k+A(z_0)}(z-z_0)^k.
    \end{align}
    Since $A(z_0)\notin\Z^-\cup\{0\}$, all denominators are nonzero and the series has the same radius of convergence as $\psi_B$. In particular, $F_p$ is holomorphic on $D_{\delta}(z_0)$.

    We claim that $F_p$ solves \eqref{eqn:ODE_sing} on $D_{\delta}(z_0)$.
    Using
    \begin{align*}
        (z-z_0)\pa_z\psi(z)=-(A(z)-A(z_0))\psi(z),
    \end{align*}
    we compute
    \begin{align*}
        (z-z_0) \pa_z F_p(z)&=
        (z-z_0)\Big[
        \pa_z \psi(z)
        \sum_{k\geq 0}\frac{\psi_{B,k}}{k+A(z_0)}(z-z_0)^{k}+
        \psi(z)
        \sum_{k\geq 1}\frac{\psi_{B,k}}{k+A(z_0)} k (z-z_0)^{k-1}
        \Big]\\ &= 
        -(A(z) - A(z_0))\psi(z)
        \sum_{k\geq 0}\frac{\psi_{B,k}}{k+A(z_0)}(z-z_0)^{k}+
        \psi(z)\sum_{k\geq 1}\frac{\psi_{B,k}}{k+A(z_0)} k (z-z_0)^{k}\\
        &=
        -A(z)F_p(z) + A(z_0)\psi(z)
        \sum_{k\geq 0}\frac{\psi_{B,k}}{k+A(z_0)}(z-z_0)^{k}+
        \psi(z)\sum_{k\geq 1}\frac{\psi_{B,k}}{k+A(z_0)} k (z-z_0)^{k}\\
        &=
        -A(z)F_p(z) + \psi(z) \Big(\psi_{B,0} +
        \sum_{k\geq 1}\psi_{B,k} (z-z_0)^{k}\Big)
        \\
        &=
        -A(z)F_p(z) + \psi(z) \psi_B(z)
        \\
        &=
        -A(z)F_p(z) + B(z).
    \end{align*}
    which is exactly \eqref{eqn:ODE_sing}.

    We define the open and simply connected sets $U_1,U_2\subset\Omega$ as follows.
    First, if $\Omega=\C$, we set
    \begin{align*}
        U_1&:=\{z_0+re^{ix}:\, 0<r,\, x\in(0,2\pi)\},\\
        U_2&:=\{z_0+re^{ix}:\, 0<r,\, x\in(-\pi,\pi)\}.
    \end{align*}

    Otherwise, if $\Omega\neq\C$, we apply the Riemann mapping theorem (see \cite[Chapter~6, \S1]{ahlfors1979complex}), so there exists a biholomorphic function $\varphi:\D\to\Omega$, with $\varphi(0)=z_0$. Then we define $U_1$ and $U_2$ by
    \begin{align*}
        U_1&:=\varphi\big(\{re^{ix}:\, 0<r<1,\, x\in(0,2\pi)\}\big),\\
        U_2&:=\varphi\big(\{re^{ix}:\, 0<r<1,\, x\in(-\pi,\pi)\}\big).
    \end{align*}
    In particular, $D_{\delta}(z_0)\cap U_1\cap U_2\neq\emptyset$, and $D_{\delta}(z_0)\cup (U_1\cap U_2)$ is connected.

    Consider any point $\om_0\in U_1\cap U_2 \cap D_\delta(z_0)-\{z_0\}$ and set $c_0:=F_p(\om_0)$. Since $U_1$ and $U_2$ are simply connected, and $\frac{A}{z-z_0}$ and $\frac{B}{z-z_0}$ are holomorphic in both domains, Lemma~\ref{lemma:fuchs_regular} yields unique holomorphic solutions
    \begin{align*}
        F_p^{(1)}:U_1\to\C,\qquad F_p^{(2)}:U_2\to\C,
    \end{align*}
    of \eqref{eqn:ODE_sing} satisfying $F_p^{(1)}(\om_0)=F_p^{(2)}(\om_0)=c_0$. By uniqueness, $F_p^{(j)}$ coincides with $F_p$ on $D_\delta(z_0)\cap U_j$. Consequently, $F_p^{(j)}$ admits a holomorphic extension to $D_\delta(z_0)\cup U_j$, which will be denoted again by $F_p^{(j)}$. These functions $F_p^{(1)}$ and $F_p^{(2)}$ coincide on $D_\delta(z_0)\cup (U_1\cap U_2)$ since they already coincide on $D_\delta(z_0)$ and $D_\delta(z_0)\cup (U_1\cap U_2)$ is connected.
    
    Therefore, $F_p^{(1)}$ and $F_p^{(2)}$ glue together to define a single holomorphic function, extending $F_p$ to $D_\delta(z_0)\cup U_1\cup U_2 =\Omega$, which solves \eqref{eqn:ODE_sing}.
\end{proof}

\begin{corollary}\label{coro:fuchs_particular}
    Let $\Omega,A,B,z_0$ be as in Lemma~\ref{lemma:fuchs_sing_part}. Then for any simply connected $U\subset\Omega-\{z_0\}$, every holomorphic solution $F$ of \eqref{eqn:ODE_sing} on $U$ can be written as
    \begin{align*}
        F(z)=C F_h(z)+F_p(z),
    \end{align*}
    where $C\in\C$, $F_h$ is given by \eqref{def:F_h}, and $F_p$ is the holomorphic solution given in Lemma~\ref{lemma:fuchs_sing_part}.
    In particular, there exists a unique holomorphic $F:\Omega\to\C$ solving \eqref{eqn:ODE_sing}, corresponding to $C=0$.
\end{corollary}

\begin{proof}
    If $F$ solves \eqref{eqn:ODE_sing} on $U$, then $F-F_p$ solves the homogeneous equation \eqref{eqn:ODE_sing_hom} on $U$, hence $F-F_p=C F_h$ by Lemma~\ref{lemma:fuchs_sing_homo}.
    If $F$ extends holomorphically across $z_0$, then the homogeneous term must also extend holomorphically. Under $A(z_0)\notin\Z^-\cup\{0\}$, Lemma~\ref{lemma:fuchs_sing_homo} implies $C=0$. Uniqueness follows.
\end{proof}

Now assume $A(z_0)\notin\Z$. We give a useful criterion in terms of the restriction to a small circle around $z_0$.

\begin{lemma}\label{lemma:fuchs_chara}
    Let $\Omega\subset\C$ be open and simply connected and let $\ti{A},\ti{B}:\Omega\to\C$ be holomorphic. Fix $z_0\in\Omega$ and assume $\ti{A}(z_0)\notin\Z$.
    Let $\delta>0$ be such that $\overline{D}_{\delta}(z_0)\subset\Omega$, and define
    \begin{align*}
        \g(x)=z_0+\delta e^{i(x+x_0)},\quad x\in[0,2\pi],\ x_0\in\RR.
    \end{align*}
    Consider the ODE on $[0,2\pi]$:
    \begin{align}\label{eqn:ODE_circ}
        \pa_x f(x)+ i\ti{A}\big(\g(x)\big)f(x)=i\ti{B}\big(\g(x)\big).
    \end{align}
    Let $f$ be a solution of \eqref{eqn:ODE_circ} and set $c_0:=f(0)$. Let $F_p:\Omega\to\C$ be the unique holomorphic solution of \eqref{eqn:ODE_sing} given by Corollary~\ref{coro:fuchs_particular}, with $A=\ti{A}$ and $B=\ti{B}$.
    Then $f=F_p\circ\g$ if and only if $f(2\pi)=f(0)$.
\end{lemma}

\begin{proof}
    Let $F_p:\Omega\to\C$ be the unique holomorphic function solving \eqref{eqn:ODE_sing}, see Corollary~\ref{coro:fuchs_particular}, and define $f:=F_p\circ\g$. It follows that $f$ satisfies \eqref{eqn:ODE_circ} and
    \begin{align*}
        f(2\pi)=F_p(\g(2\pi))=F_p(\g(0))=f(0).
    \end{align*}
    
    Now assume that $f$ solves \eqref{eqn:ODE_circ} and satisfies $f(2\pi)=f(0)$. We define the open and simply connected set $U\subset \Omega - \{z_0\}$ given by
    \begin{align*}
        U:=\{z_0 + r e^{i(x+x_0)}:\ r\in(0,\delta),\ x\in(0,2\pi)\}.
    \end{align*}
    Fix the branch of $\log(z-z_0)$ on $U$ determined by
    \begin{align*}
        \log\bigl(re^{i(x+x_0)}\bigr)=\log(r)+i(x+x_0).
    \end{align*}
    With this choice, $\log(z-z_0)$ is holomorphic on $U$ and has well-defined boundary limits as $x\to0^+$ and $x\to2\pi^-$.

    By Corollary~\ref{coro:fuchs_particular} we find that  every solution to \eqref{eqn:ODE_sing} in $U$ is written as
    \begin{align*}
        F = C F_h + F_p,
    \end{align*}
    where $C\in\C$, $F_p:\Omega\to\C$ is holomorphic and $F_h:U\to\C$ is given in \eqref{def:F_h} with $A=\ti{A}$.
    
    Fix any $x_1\in(0,2\pi)$ and set $\om_1:=\g(x_1)=z_0+\delta e^{i(x_1+x_0)}\neq\g(0)$. Choose $C$ so that $F(\om_1) = f(x_1)$, and hence $f = F\circ\g$. Our goal is to show that this choice forces $C=0$, so that $F$ is the unique holomorphic solution given by Corollary~\ref{coro:fuchs_particular}.
    
    Since $F_p$ is holomorphic on a neighborhood of $\overline{D}_\delta(z_0)$, the function $F_p\circ\g$ extends continuously to $[0,2\pi]$ and satisfies $F_p(\g(2\pi))=F_p(\g(0))$. Consequently
    \begin{align*}
        0=f(2\pi)-f(0)& =
        \lim_{x\to 2\pi^-} F\big(z_0 + \delta e^{i(x+x_0)}\big) -
        \lim_{x\to 0^+} F\big(z_0 +  \delta e^{i(x+x_0)}\big)\\
        &= C \Big(
        \lim_{x\to 2\pi^-} F_h\big(z_0 + \delta e^{i(x+x_0)}\big)-
        \lim_{x\to 0^+} F_h\big(z_0 +  \delta e^{i(x+x_0)}\big)
        \Big).
    \end{align*}

    Therefore, it follows for every $z\in U$:
    \begin{align*}
        F_h(z)=(z-z_0)^{-\ti{A}(z_0)} \psi(z)=\exp\!\left(-\ti{A}(z_0)\int_{\om_1}^{z} \frac{dy}{y-z_0}\right)
        \psi(z),
    \end{align*}
    where the integral is defined along paths contained in $U$. 

    Indeed, working with $\g$ yields
    \begin{align*}
        0
        &= C \Big(
        \lim_{x\to 2\pi^-} F_h(z_0 + \delta e^{i(x+x_0)})-
        \lim_{x\to 0^+} F_h(z_0 + \delta e^{i(x+x_0)})
        \Big)\\
        &=
        C \left[\lim_{x\to 2\pi^-}
        \exp\!\left(-i\ti{A}(z_0)\int_{x_1}^{x} \,ds\right)
        -
        \lim_{x\to 0^+}
        \exp\!\left(-i\ti{A}(z_0)\int_{x_1}^{x} \,ds\right)\right]
        \psi(z_0+\delta e^{ix_0})
        \\
        &=
        C \big(e^{-i\ti{A}(z_0)2\pi}-1\big)
        \psi(z_0+\delta e^{ix_0}).
    \end{align*}
    
    By definition (see \eqref{def:psi_fuchs}), $\psi$ is the exponential of a holomorphic function. Moreover, since $\ti{A}(z_0) \notin \Z$ by assumption, we have $e^{-i \ti{A}(z_0) 2\pi} \neq 1$. Therefore, the only possibility is $C = 0$, and hence $F = F_p$ is holomorphic and $f = F_p \circ \g$.
\end{proof}

We introduce the following notation, which will be used in the next lemma. Let $z_0,z_1\in\C$ and let $r$ satisfy $0<r<|z_0-z_1|$. We define the \emph{racket set} associated to $z_0$, $z_1$ and $r$ by
\begin{align}\label{def:raqueta}
    R_{z_0,z_1}(r):=
    \big\{z\in\C:\, |z-z_1|=r\big\}\cup
    \big\{(1-x)z_0 + x\big(z_1 + r \tfrac{z_0-z_1}{|z_0-z_1|}\big):\, x\in[0,1]\big\}.
\end{align}
Note that the point
\begin{align*}
    z_1 + r \frac{z_0-z_1}{|z_0-z_1|}
\end{align*}
is the unique intersection point between the circle of center $z_1$ and radius $r$ and the line segment joining $z_0$ and $z_1$. For a geometric illustration of the set $R_{z_0,z_1}(r)$, see Figure~\ref{fig:raqueta}.

\begin{figure}[ht]
\centering
\begin{tikzpicture}[scale=3]
    
    \def\r{0.25}      % racket radius
    \def\d{1.3}      % distance |z0-z1|
    \def\te{30}    % angle arg(z1-z0) in degrees
    
    \coordinate (z0) at (0,0);
    \coordinate (z1) at (\te:\d);
    
    \coordinate (p) at ($(z1)+(\te+180:\r)$);
    
    \draw (z1) circle (\r);
    \draw (z0) -- (p);
    
    \draw[->, >=Stealth] (z1) -- ++(0:\r);
    
    \fill (z0) circle (0.6pt);
    \fill (z1) circle (0.6pt);

    \node[above] at ($(z1)+(0.5*\r,0)$) {$r$};
    \node[below left] at (z0) {$z_0$};
    \node[above left] at (z1) {$z_1$};

\end{tikzpicture}
\caption{The racket set $R_{z_0,z_1}(r)$.}
\label{fig:raqueta}
\end{figure}
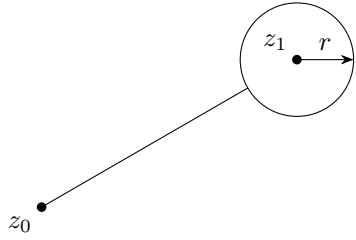

The next lemma is technical and will be used to prove the main result of this section.

\begin{lemma}\label{lemma:raquetas}
    Let $\D\subset\C$ be the open unit disk and let $\{z_j\}_{j=1}^n\subset\D$ be $n$ distinct points. Then there exists a point $\om_0\in\partial\D$ and a collection of open and convex sets $\{C_j\}_{j=1}^n$ covering $\D$ such that
    \begin{align*}
        C_j \cap \{z_1,\ldots,z_n\}=\{z_j\}.
    \end{align*}
    
    Moreover, $\om_0$ is an accumulation point for every $C_j$.

    Additionally, there exists $r>0$ such that the racket set $R_{\om_0,z_j}(r)$ satisfies
    \begin{align*}
        R_{\om_0,z_j}(r)\subset C_j\cup\{\om_0\}\quad\text{for all } j
    \end{align*}
\end{lemma}

\begin{proof}
    We first choose a boundary point $\om_0\in\partial\D$ which is not collinear with any pair of distinct points $z_j,z_k$.
    For each pair $j\neq k$, the line through $z_j$ and $z_k$ meets $\partial\D$ in at most two points, so the set of such \emph{bad} boundary points is finite. Fix $\om_0$ outside this set.
    
    For each $j=1,\ldots,n$, there exist unique numbers $t_j>0$ and $\te_j\in(0,\pi)$ such that
    \begin{align*}
        z_j=\om_0\bigl(1+i t_j e^{i\te_j}\bigr).
    \end{align*}
    Each $z_j\in\D$ satisfies $\Re(\overline{\om_0}z_j-1)<0$, so $\overline{\om_0} z_j-1$ has a unique polar form $\overline{\om_0}z_j-1=i t_j e^{i\te_j}$ with $t_j>0$ and $\te_j\in(0,\pi)$.

    By the choice of $\om_0$, the angles $\te_1,\ldots,\te_n$ are pairwise distinct. Relabel the points $z_j$ so that
    \begin{align}\label{eqn:th_sorted}
        0<\te_1<\te_2<\ldots<\te_n<\pi,
    \end{align}
    and set $\te_0:=0$ and $\te_{n+1}:=\pi$.
    
    For each $j=1,\ldots,n$, define the set $C_j$ as the intersection of an open cone with vertex at $\om_0$ and the $\D$:
    \begin{align}\label{def:Cj}
        C_j:=\bigl\{\om_0(1+i t e^{i\te}):\, t>0,\ \te\in(\te_{j-1},\te_{j+1})\bigr\}\cap\D.
    \end{align}
    These sets $\{C_j\}$ cover $\D$. Each $C_j$ is open and convex, as it is the intersection of open and convex sets, and contains $z_j$.
    
    Finally, let $\delta_j>0$ such that 
    \begin{align*}
        D_{\delta_j}(z_j)\subset C_j.
    \end{align*}

    We set $r:=\min_j\{\delta_j\}$, so that
    \begin{align*}
        R_{\om_0,z_j}(r)\subset C_j\cup\{\om_0\}\qquad\text{for all } j.
    \end{align*}

    For a geometric illustration of the set $C_j$ and $R_{\om_0,z_j}(r)$ for a fixed $j$, see Figure~\ref{fig:raquetaCj}.

    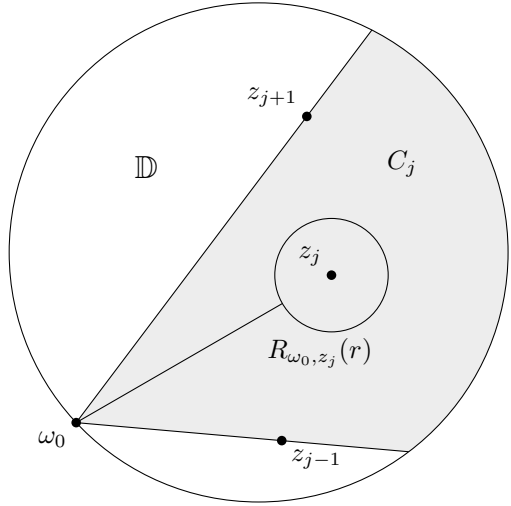
\begin{figure}[ht]
    \centering
    \begin{tikzpicture}[scale=3, >=Stealth]
    
        \def\r{0.25}
        \def\d{1.3}
        \def\theta{30}
        \def\thetar{-30}
        \def\thL{-5}
        \def\thR{53}
        \def\pazja{0.7}
        \def\pazjb{1.3}
        \def\RayLen{3}
        \def\rc{1.1}
        \def\thc{43}
        
        \coordinate (c) at (\thc:\rc);

        \def\Rclip{1.05}               % radio del clipeo
        
        \coordinate (w0) at (0,0);
        \coordinate (zj) at (\theta:\d);
        \coordinate (zja) at (\thL:{\pazja*\d});
        \coordinate (zjb) at (\thR:{\pazjb*\d});
        \coordinate (p) at ($(zj)+(\te+180:\r)$);
        
        \begin{scope}
            \clip (c) circle (\rc);
            \fill[black!7]
            (w0) -- (\thL:\RayLen) -- (\thR:\RayLen) -- cycle;
    
            \draw[thin] (w0) -- (\thL:\RayLen);
            \draw[thin] (w0) -- (\thR:\RayLen);
        \end{scope}
    
        \draw (c) circle (\rc);
        \draw (zj) circle (\r);
        \draw[thin] (w0) -- (p);

        \fill (w0) circle (0.6pt);
        \fill (zj) circle (0.6pt);
        \fill (zja) circle (0.6pt);
        \fill (zjb) circle (0.6pt);
        
        \node[below left] at (w0) {$\omega_0$};
        \node[above left] at (zj) {$z_j$};
        \node[below right] at (zja) {$z_{j-1}$};
        \node[above left] at (zjb) {$z_{j+1}$};
        \node[left] at ($(c)+(180-\thc:\rc*0.5)$) {\large $\D$};
        \node[above left] at ($(zj)+(\thc:\r*2.3)$) {$C_j$};
        \node[below left] at ($(zj)+(\thc-90:\r*1.3)$) {$R_{\om_0,z_j}(r)$};

    \end{tikzpicture}
    \caption{$R_{\om_0,z_j}(r)$ and $C_j$.}
    \label{fig:raquetaCj}
    \end{figure}
\end{proof}

In order to prove the main lemma of this section, we introduce the following notation. Let $a,b,c,d\in\RR$ with $a<b$ and $c<d$. Let $\g_1:[a,b]\to\C$ and $\g_2:[c,d]\to\C$ be continuous paths such that $\g_1(b)=\g_2(c)$. We denote by $\g_1\ast\g_2:[a,b+d-c]\to\C$ and by $\overline{\g}_1:[a,b]\to\C$ the continuous paths defined by
\begin{align}
    \g_1\ast\g_2(x) =
    \begin{cases}
    \g_1(x), & x\in[a,b],\\
    \g_2(x-b+c), & x\in[b,b+d-c],
    \end{cases}
    \qquad
    \overline{\g}_1(x) = \g_1(b+a-x), \quad \forall x\in[a,b].
    \label{def:path_opers}
\end{align}
This parametrization avoids re-scaling. Moreover, the operation $\ast$ is associative.\\

We now state the main result of Section~\ref{sec:fuchs}.

\begin{lemma}\label{lemma:fuchs_criteria}
    Let $\D$ denote the open unit disk, and let $\ti{Z}:=\{z_j\}_{j=1}^n\subset\D$ be a finite set of pairwise distinct points. Define
    \begin{align*}
        Q(z):=\prod_{j=1}^n (z-z_j).
    \end{align*}
    
    Let $A,B:\D\to\C$ be holomorphic functions extending continuously to $\overline{\D}$, and assume that
    \begin{align}\label{eqn:cond_proots_general}
        \frac{A(z_j)}{\pa_z Q(z_j)}\notin\Z
        \qquad \text{for all } j=1,\ldots,n.
    \end{align}
    
   Let $\om_0$, $\{C_j\}_{j=1}^n$, $r$, and $R_{\om_0,z_j}(r)$ be as in Lemma~\ref{lemma:raquetas}. For each $j=1,\ldots,n$, define the paths given by $\g_j:[0,1]\to\C$ and $\al_j:[0,2\pi]\to\C$ by
    \begin{align*}
        \g_j(x)
        &= (1-x)\om_0
           + x\Bigl(z_j + r\frac{\om_0-z_j}{|\om_0-z_j|}\Bigr),\\
        \al_j(x)
        &= z_j + r e^{ix}\frac{\om_0-z_j}{|\om_0-z_j|}.
    \end{align*}
    Using the notation of \eqref{def:path_opers}, set
    \begin{align}\label{def:tauj}
        \tau_j:=
        \g_j\ast\al_j\ast\overline{\g}_j:[0,2\pi+2]\to\C.
    \end{align}
    
    For each $j$, consider the problem on $(0,2\pi+2)-\{1,2\pi+1\}$ given by
    \begin{align}\label{eqn:fuchs_crit_raqueta}
        \frac{Q\circ\tau_j}{\pa_x\tau_j}\pa_x f_j
        + (A\circ\tau_j) f_j
        = B\circ\tau_j .
    \end{align}
    
    We also consider the equation on $\D$
    \begin{align}\label{eqn:fuchs_crit_disk}
        Q\pa_z F + AF = B .
    \end{align}

    Let $c_0\in\C$.
    For each $j=1,\ldots,n$ let $f_j:[0,2\pi+2]\to\C$ be the unique continuous function with $f_j(0)=c_0$, solving \eqref{eqn:fuchs_crit_raqueta}. Then there exists a holomorphic function $F:\D\to\C$, whose derivative extends continuously to $\overline{\D}$, solving \eqref{eqn:fuchs_crit_disk} and satisfying $F(\om_0)=c_0$ if and only if
    \begin{align*}
        f_j(2\pi+2)=c_0
        \qquad \text{for all } j=1,\ldots,n.
    \end{align*}
    Moreover, such a function $F$, if it exists, is unique.
\end{lemma}

\begin{rem}
    $\g_j$ and $\al_j$ have been defined so that
    \begin{align*}
        \g_j([0,1])\cup\al_j([0,2\pi]) = R_{\om_0,z_j}(r).
    \end{align*}
\end{rem}

\begin{proof}
    Assume first that there exists a holomorphic function $F:\D\to\C$, with derivative continuous on $\overline{\D}$, solving \eqref{eqn:fuchs_crit_disk} and satisfying $F(\om_0)=c_0$. Fix $j\in\{1,\ldots,n\}$ and define
    \begin{align*}
        g_j(x):=F(\tau_j(x)),
        \qquad x\in[0,2\pi+2].
    \end{align*}
    Then $g_j$ is continuous and, by change of variables along $\tau_j$, it satisfies \eqref{eqn:fuchs_crit_raqueta} on $(0,2\pi+2)-\{1,2\pi+1\}$. Since $g_j(0)=F(\om_0)=c_0=f_j(0)$, uniqueness of solutions to \eqref{eqn:fuchs_crit_raqueta} yields
    \begin{align*}
        g_j=f_j \quad \text{on } [0,2\pi+2].
    \end{align*}
    In particular,
    \begin{align*}
        f_j(2\pi+2)
        = g_j(2\pi+2)
        = F(\tau_j(2\pi+2))
        = F(\om_0)
        = c_0 ,
    \end{align*}
    which proves the necessity.
    
    Conversely, assume that for each $j=1,\ldots,n$, the continuous solution $f_j:[0,2\pi+2]\to\C$ of \eqref{eqn:fuchs_crit_raqueta} satisfies that $f_j(0)=f_j(2\pi+2)=c_0$. Fix $j$. We first claim that
    \begin{align*}
        f_j(1)=f_j(2\pi+1).
    \end{align*}
    To see this, on $(0,1)$ the equation reads
    \begin{align*}
        \frac{Q\circ\g_j}{\pa_x\g_j}\pa_x f_j
        + (A\circ\g_j)f_j
        = B\circ\g_j,
    \end{align*}
    while on $(2\pi+1,2\pi+2)$ it reads
    \begin{align*}
        \frac{Q\circ\overline{\g}_j}{\pa_x\overline{\g}_j}\pa_x f_j
        + (A\circ\overline{\g}_j)f_j
        = B\circ\overline{\g}_j.
    \end{align*}

    Because $\g_j([0,1])\subset R_{\om_0,z_j}(r)\subset C_j$, we know that $Q$ never vanishes on this curve. Hence both equations admit unique solutions once an initial condition is prescribed.
    
    Define $\phi_j^0:(0,1)\to\C$ by $\phi_j^0(x):=f_j(2\pi+2-x)$. A direct change of variables shows that $\phi_j^0$ satisfies the same ODE as $f_j$ on $(0,1)$. Since $\phi_j^0(0)=f_j(2\pi+2)=c_0=f_j(0)$, uniqueness yields $f_j=\phi_j^0$ on $(0,1)$, hence $f_j(1)=f_j(2\pi+1)$.
    
    By Lemma~\ref{lemma:raquetas}, we know that the set $C_j$ is convex, hence simply connected, and by the choice of $r$ we have
    \begin{align*}
        C_j\cap\{z_1,\ldots,z_n\}=\{z_j\},
    \end{align*}
    Here, we assume the points $z_j$ have been relabeled so \eqref{eqn:th_sorted} holds.
    
    Define $\phi_j^1:[0,2\pi]\to\C$ by $\phi_j^1(x):=f_j(x+1)$ and observe that $Q$ has exactly one zero in $C_j$, which is simple. Note that \eqref{eqn:cond_proots_general} is equivalent to
    \begin{align*}
        \lim_{z\to z_j} (z-z_j)\frac{A(z)}{Q(z)} = \frac{A(z_j)}{\pa_z Q(z_j)}\notin\Z\qquad \text{for all } j.
    \end{align*}

    Then, applying Lemma~\ref{lemma:fuchs_chara} with  $\Omega=C_j$ and $z_0=z_j$, and
    \begin{align*}
        \ti{A}(z) = A(z)
        \prod_{\substack{l=1\\l\neq j}}^n (z-z_l)^{-1}, 
        \qquad 
        \ti{B}(z) = B(z)
        \prod_{\substack{l=1\\l\neq j}}^n (z-z_l)^{-1},
    \end{align*}
    we obtain that the unique holomorphic function $F_j:C_j\to\C$ satisfies 
    \begin{align*}
        (F_j\circ\al_j)(x)
        = \phi_j^1(x)
        = f_j(x+1),
        \qquad x\in[0,2\pi],
    \end{align*}
    where $F_j$ is given by Corollary~\ref{coro:fuchs_particular} and solves \eqref{eqn:fuchs_crit_disk}.
    
    For every $\ep>0$, define
    \begin{align*}
        V_\ep:=D_{\ep}(\om_0)\cap\D.
    \end{align*}    
    Let $\ep>0$ be small enough so that $z_j\notin V_\ep$ for all $j$. In particular, $Q$ has no zeros on $V_\ep$. As in the proof of Lemma~\ref{lemma:fuchs_regular}, we solve \eqref{eqn:fuchs_crit_disk} on $V_\ep$ by integration. For $z\in V_\ep$, define
    \begin{align*}
        F_h^\ep(z)
        := \exp\!\left(
            -\int_{\om_0}^z \frac{A(y)}{Q(y)}\dy
        \right),
    \end{align*}
    and
    \begin{align*}
        F^\ep(z)
        := c_0 F_h^\ep(z)
        + F_h^\ep(z)\int_{\om_0}^z
          \frac{B(y)}{Q(y)F_h^\ep(y)}\dy,
    \end{align*}
    where the integrals are taken along any smooth path in $V_\ep$. Although $\om_0\not\in V_\ep$, these integrals are well defined because $A$ and $B$ are continuous at $\om_0$. 
    Since $V_\ep$ is simply connected and $A$ and $B$ holomorphic, the integrals are path independent, and $F^\ep$ is holomorphic on $V_\ep$ and solves \eqref{eqn:fuchs_crit_disk} there.
    
    For each $j$, the functions $F^\ep$ and $F_j$ coincide on $\g_j([0,1])\cap V_\ep$, since both solve the same initial value problem along $\g_j$ with initial condition $F(\om_0)=c_0$. By uniqueness,
    \begin{align*}
        F^\ep = F_j
        \quad \text{on } C_j\cap V_\ep.
    \end{align*}

    Consequently, each function $F_j$ extends holomorphically to the set $C_j\cup V_\ep$. We continue to denote this extension by $F_j$. Now let $j\neq k$. The set $V_\ep\cap (C_j\cup C_k)$ is connected, since every point of $C_j$ admits a path contained in $C_j\cup V_\ep$ joining it to $V_\ep$. Since $F_j$ and $F_k$ coincide on $V_\ep$, it follows by analytic continuation that they coincide on the whole set $V_\ep\cap (C_j\cup C_k)$. In particular, $F_k$ extends holomorphically to the domain of $F_j$, and vice versa.

    Hence there exists a holomorphic function $F:\D\to\C$ such that
    \begin{align*}
        F|_{C_j}=F_j,
        \qquad
        F|_{V_\ep}=F^\ep .
    \end{align*}
    
    Since $A$ and $B$ are continuous on $\overline{\D}$ and holomorphic on $\D$, we find that $F$ admits a continuous extension to $\overline{\D}$. Finally, the equation \eqref{eqn:fuchs_crit_disk} implies
    \begin{align*}
        \pa_z F(z)
        = \frac{B(z)-A(z)F(z)}{Q(z)},
        \qquad z\in\D.
    \end{align*}
    Because $Q$ has no zeros on $\partial\D$, the right-hand side extends continuously to $\partial\D$, and therefore $\pa_z F$ extends continuously to $\overline{\D}$. Uniqueness follows from Corollary~\ref{coro:fuchs_particular}.
\end{proof}

\subsubsection{Invertibility of $\LOpExisHol$}
\label{sec:exis:invert}
In order to show that $\LOpExisHol:\XHholZMi{2}\to\XHholZMi{1}$ is invertible, we first need to understand the roots of the polynomial $z^N P(z)$.

\begin{lemma}\label{lemma:Proots1}
    Let $P$ be the function defined in \eqref{def:P_}, and denote by $Z$ the set of roots of $z^N P(z)$ lying inside the unit disk $\D$.
    Then $Z$ contains exactly $N$ simple roots. Moreover, for every root $\al\in Z$, we have
    \begin{align}\label{eqn:auxProots1}
        A(\al)\prod_{\substack{\al'\in Z\\ \al'\neq \al}}
        (\al-\al')^{-1}
        \notin \Z,
    \end{align}
    where the function $A$ is defined by
    \begin{align*}
        A(z):=-2\frac{z^{N-1}}{\twap_N}\prod_{\al\in Z}(z-\al^{-1})^{-1}.
    \end{align*}
\end{lemma}
\proofcompassi

Lemma~\ref{lemma:fuchs_criteria} will be used to establish the invertibility of $\LOpExisHol$. We introduce some notation that will be used to prove the next proposition.

\begin{defi}\label{def:exis_hol_notation}
        
    Let $Z=\{z_j\}_{j=1}^N$ be the set of roots of $z^N P(z)$ lying inside the open unit disk, and let $\om_0$, $Q$, and the paths $\{\tau_j\}$ be as in Lemma~\ref{lemma:fuchs_criteria} with $\ti{Z}:=Z$.
    
    For $G\in\XHholo\cup\{\mathcal{G}_k:\ k=1,\ldots,N\}$ and $j=1,\ldots,N$, we define $J_{\tau_j}G$ as the unique continuous solution of \eqref{eqn:fuchs_crit_raqueta} with initial condition $J_{\tau_j}G(0)=0$ and coefficients
    \begin{align*}
        A(z) &:= -2\frac{z^{N-1}}{\twap_N}
        \prod_{\al\in Z}(z-\al^{-1})^{-1},\\
        B(z) &:=-2i\frac{z^{N-1}}{\twap_N}
        G(z)\prod_{\al\in Z}(z-\al^{-1})^{-1}.
    \end{align*}
    
    For each $j=1,\ldots,N$, we also define $\Pi_{\tau_j}$ as the unique continuous solution to \eqref{eqn:fuchs_crit_raqueta} with initial condition $\Pi_{\tau_j}(0)=1$ and
    \begin{align*}
        A(z) &:= -2\frac{z^{N-1}}{\twap_N}
        \prod_{\al\in Z}(z-\al^{-1})^{-1},\\
        B(z) &:= 0.
    \end{align*}
    
    Define $\g_{\om_0}(x):=\om_0 e^{ix}$ and
    \begin{align*}
        \beta_{\om_0} := \frac{1}{P\circ\g_{\om_0}}.
    \end{align*}
    
    For $G\in C(\T)$, we define $J_{\g_{\om_0}}G$ as the unique continuous
    solution with $J_{\g_{\om_0}}G(0)=0$ of
    \begin{align*}
        \pa_x J_{\g_{\om_0}}G
        =
        i\beta_{\om_0}J_{\g_{\om_0}}G
        +
        \beta_{\om_0}(G\circ\g_{\om_0}),
    \end{align*}
    and $\Pi_{\g_{\om_0}}$ as the unique continuous solution with $\Pi_{\g_{\om_0}}(0)=1$ of
    \begin{align*}
        \pa_x \Pi_{\g_{\om_0}}
        =
        i\beta_{\om_0}\Pi_{\g_{\om_0}}.
    \end{align*}

\end{defi}

At the end of the section, we prove that $\LOpExisHol:\XHholZMi{2}\to\XHholZMi{1}$ admits a bounded inverse. We will also need the following lemma, which provides a criterion for the invertibility of $\LOpExisHol$ in a different class of function spaces and is based on Lemma~\ref{lemma:fuchs_criteria}.

\begin{prop}\label{prop:exis_hol}
    Using the notation from Definition~\ref{def:exis_hol_notation} and $\mathcal{G}_k$ from~\eqref{def:Gkexis}, define the $(N+1)\times(N+1)$ matrix $\MExisHol$ by
    \begin{align}
        (\MExisHol)_{jk}=
        \begin{cases}
            J_{\tau_j}\mathcal{G}_k(2\pi+2)-\underbrace{J_{\tau_j}\mathcal{G}_k(0)}_{0}, & 
            1\leq j\leq N,\ 1\leq k\leq N,\\
            \Pi_{\tau_j}(2\pi+2)-\underbrace{\Pi_{\tau_j}(0)}_{1}, & 
            1\leq j \leq N,\ k = N+1,\\
            \displaystyle \frac{1}{2\pi}\intpi 
            J_{\g_{\om_0}}\mathcal{G}_k(x)\dx, & 
            j=N+1,\ 1\leq k\leq N,\\
            \displaystyle \frac{1}{2\pi}\intpi \Pi_{\g_{\om_0}}(x)\dx, & 
            j=N+1,\ k=N+1.
        \end{cases}
        \label{def:MExisHol}
    \end{align}

    Then the matrix $\MExisHol$ is nonsingular if and only if, for every holomorphic function $G:\D\to\C$ that is continuous on $\overline{\D}$ and satisfies $G(0)=0$, there exists a unique holomorphic function $F:\D\to\C$ whose derivative $\pa_z F$ is continuous on $\overline{\D}$, satisfying $F(0)=0$ and solving
    \begin{align}\label{eqn:exis_finite_hol_main}
        \LOpExisHol F = G.
    \end{align}

\end{prop}

\begin{proof}
    Let $G:\D\to\C$ be holomorphic in $\D$ and continuous on $\overline{\D}$, with $G(0)=0$. We study equation \eqref{eqn:exis_finite_hol_main} and seek a holomorphic function $F:\D\to\C$ whose derivative $\pa_z F$ is continuous on $\overline{\D}$ and satisfies $F(0)=0$.
        
    Writing the equation explicitly, we have
    \begin{align*}
        izP(z)\pa_z F(z) =
        iF(z) +
        \sum_{k=1}^N 
            F_k\mathcal{G}_k(z) 
            +
            G(z),
    \end{align*}
    where $F(z)=\sum_{k\geq 1} F_k z^k$.
    
    Multiplying by $-iz^{N-1}$, we obtain
    \begin{align}\label{eqn:exis_finite_hol_affine}
        z^N P(z)\pa_z F(z)=
        z^{N-1}F(z)-
        iz^{N-1}\Big(\sum_{k=1}^N F_k\mathcal{G}_k(z)+G(z)\Big).
    \end{align}
    
    Since $z^{N-1}\sum_{k=1}^N F_k \mathcal{G}_k$ is holomorphic in $\D$, the singular points of \eqref{eqn:exis_finite_hol_affine} inside $\D$ are precisely the roots
    \begin{align*}
        Z=\{z_j\}_{j=1}^N
    \end{align*}
    of $z^N P(z)$ lying in the open unit disk.

    We introduce auxiliary unknowns $a_1,\ldots,a_N\in\C$ and consider the following affine problem
    \begin{align}\label{eqn:exis_finite_hol_affine_coeffs}
        izP(z)\pa_z F(z)=
        iF(z)+\sum_{k=1}^N a_k \mathcal{G}_k(z)+G(z).
    \end{align}
    
    The equation \eqref{eqn:exis_finite_hol_affine_coeffs} reduces to a linear first-order equation of the form \eqref{eqn:fuchs_crit_disk} with
    \begin{align*}
        A(z) &:= -2\frac{z^{N-1}}{\twap_N}
        \prod_{\al\in Z}(z-\al^{-1})^{-1},\\
        B(z) &:=-2i\frac{z^{N-1}}{{\twap_N}}
        \Big(\sum_{k=1}^N a_k\mathcal{G}_k(z) + G(z)\Big)
        \prod_{\al\in Z}(z-\al^{-1})^{-1}.
    \end{align*}
    
    For each fixed $j=1,\ldots,N$, solving \eqref{eqn:exis_finite_hol_affine_coeffs} along $\tau_j$ yields a unique continuous solution $\phi_j$. This solution can be expressed as 
    \begin{align}\label{eqn:exis_finite_hol_decomp_k}
        \phi_j=
        F(\om_0)\Pi_{\tau_j}+
        \sum_{k=1}^N a_k J_{\tau_j}\mathcal{G}_k+
        J_{\tau_j}G.
    \end{align}

    Lemma~\ref{lemma:Proots1} implies that condition~\eqref{eqn:cond_proots_general} holds. Moreover, $A$ is continuous on $\overline{\D}$, since it is a finite product of terms of the form
    \begin{align*}
        (z-\al^{-1})^{-1}\qquad \text{where } |\al|<1.
    \end{align*}

    Therefore, Lemma~\ref{lemma:fuchs_criteria} applies, and the equation~\eqref{eqn:exis_finite_hol_affine_coeffs} admits a holomorphic solution $F$ in $\D$, with $\pa_z F$ continuous on $\overline{\D}$, if and only if each function $\phi_j$ satisfies the condition
    \begin{align}\label{eqn:exis_finite_hol_close}
        \phi_j(2\pi+2)=\phi_j(0),
        \qquad j=1,\ldots,N.
    \end{align}
    Using \eqref{eqn:exis_finite_hol_decomp_k}, these conditions yield $N$ affine linear equations in the unknowns $a_1,\ldots,a_N,F(\om_0)$.
    
    Restricting \eqref{eqn:exis_finite_hol_affine_coeffs} to $\g_{\om_0}$ yields the equation
    \begin{align*}
        \pa_x \phi = i \beta_{\om_0} \phi + \beta_{\om_0} 
        \Big[
        \sum_{k=1}^N
        a_k(\mathcal{G}_k\circ\g_{\om_0})+
        (G\circ\g_{\om_0})\Big].
    \end{align*}

    The solutions are of the form
    \begin{align*}
        \phi = \Pi_{\g_{\om_0}} F(\om_0) +
        \sum_{k=1}^N a_k J_{\g_{\om_0}} \mathcal{G}_k + 
        J_{\g_{\om_0}} G.
    \end{align*}
    
    Imposing the mean-zero condition
    \begin{align*}
        \frac{1}{2\pi}\int_0^{2\pi} \phi(x)\dx=0
    \end{align*}
    yields one additional affine linear equation.
    
    Collecting the $N$ closing conditions \eqref{eqn:exis_finite_hol_close} and the mean condition above, we obtain a linear system of $N+1$ equations in the $N+1$ unknowns $(a_1,\ldots,a_N,F(\om_0))$, whose coefficient matrix is exactly $\MExisHol$.

    If $\MExisHol$ is nonsingular, the system admits a unique solution $(a_1,\ldots,a_N,F(\omega_0))$. Consequently, there exists a unique holomorphic function $F:\D\to\C$ such that $\pa_z F\in C(\overline{\D})$,
    \begin{align*}
        F(0)=\frac{1}{2\pi}\int_0^{2\pi}F(e^{ix})\dx=0,
    \end{align*}
    and $F$ solves \eqref{eqn:exis_finite_hol_affine_coeffs}. By Lemma~\ref{lemma:exis:holo_implies}, the coefficients satisfy
    \begin{align*}
        a_k = F_k \qquad \text{for all } k,
    \end{align*}
    and therefore $F$ is the unique solution of \eqref{eqn:exis_finite_hol_main}.
    
    Conversely, if for every admissible $G$ there exists a unique solution $F$ of \eqref{eqn:exis_finite_hol_main}, then the homogeneous system associated with $\MExisHol$ admits only the trivial solution, and hence $\MExisHol$ is nonsingular.
\end{proof}

We have observed numerically that for large $N$ (our case applies), the singular values of $\MExisHol$ cluster near zero.

This behavior is problematic for our analysis using interval arithmetic techniques, since it becomes difficult to certify that the smallest singular value does not vanish.

Moreover, the norm of $(\MExisHol)^{-1}$, which we need to know to establish Proposition~\ref{prop:inv_exis}, would be extremely large. To overcome these difficulties, we introduce an alternative $(N+1)\times (N+1)$ matrix, denoted by $\MExisTor$, whose singular values accumulate near $1$. We then show that
\begin{align}\label{eqn:exis_tor_imp_hol}
    \det(\MExisTor)\neq 0\ \Rightarrow\
    \det(\MExisHol)\neq 0.
\end{align}
In addition, this new matrix allows us to compute an explicit bound for the inverse norm appearing in Proposition~\ref{prop:inv_exis}.

We introduce the following notation.
Let $\beta:\RR\to\RR$ and $\PiExis:\RR\to\C$ be defined by
\begin{align}
    \beta(x)&:=\frac{1}{c+\twap(x)},\label{def:beta}\\
    \PiExis(x)&:=\exp\!\left(i\int_0^x \beta(y)\dy\right)\label{def:PiExis}.
\end{align}

Both functions $\beta$ and $\PiExis$ are continuous on $\RR$, since $c+\twap(x)$ has no zeros on the real line, since $c+\twap(x) = P(e^{ix})$ and the polynomial $z^N P(z)$ does not vanish on the unit circle.

Let $J$ denote the operator defined by
\begin{align}
    \label{def:Jexis}
    Jg(x):= \PiExis(x) \int_0^x 
    \beta(y) \frac{g(y)}{\PiExis(y)}
    \dy.
\end{align}

We now introduce the new criterion. 

\begin{lemma}\label{lemma:exis_tor}
    Let $\beta$, $\PiExis$, and $J$ be defined in
    \eqref{def:beta}, \eqref{def:PiExis}, and \eqref{def:Jexis}, respectively.

    For every $k=1,\ldots,N$, define
    \begin{align}\label{def:gk_exis}
        g_k (x) := 
        \frac{i}{2}
        k\twap_k+
        \sum_{m=1}^{N-k}
        k \twap_{k+m}
        \sin(m x).
    \end{align}
    
    For each $k=1,\ldots,N+1$, set
    \begin{align}\label{def:hk}
        h_{k}&:=
        \begin{cases}
            J g_k,& 1 \leq k\leq N,\\
            \PiExis,& k = N+1.
        \end{cases}
    \end{align}

    Finally, define the square matrix $\MExisTor$ of order $N+1$ by
    \begin{align}
        (\MExisTor)_{jk}=
            (\ha{h_k})_{j-1} -\delta_{k,j-1},
            \qquad 1\leq j,k\leq N+1.
            \label{def:Mexistor}
    \end{align}
    
    Then the implication \eqref{eqn:exis_tor_imp_hol} holds.
    
\end{lemma}

\begin{comment}
\begin{rem}
    The matrix $\MExisTor$ comes from fixing the Fourier coefficients of a solution $f$ of \eqref{eqn:exis_finite_hol_main} restricted to the unit circle.
    
    More precisely, let $f:\RR\to\C$ satisfy
    \begin{align*}
        (c+\twap(x)) \pa_x f(x)
        = i f(x) + \sum_{k=1}^N g_k(x) \ha{f}_k  + g(x).
    \end{align*}
    We impose that $\ha{f}_k = a_k$ for $k=1,\ldots,N$ and that $\ha{f}_0=0$. However, these conditions do not ensure that $f$ is periodic, and hence that it corresponds to a genuine inverse.
    
    The matrix $\MExisHol$ in Proposition~\ref{lemma:exis_hol} is introduced only to guarantee that the function $f$ obtained in this way is indeed periodic. The rest of the argument is based on the analysis of the matrix $\MExisTor$.

\end{rem}
\end{comment}

\begin{proof}
    Assume that $\det(\MExisHol)=0$ and fix $G=0$. Then, by Proposition~\ref{prop:exis_hol}, there exists a nontrivial subspace $V$ of holomorphic functions on $\D$ with continuous derivative on $\overline{\D}$, such that for every $F\in V$ we have that $F$ (with $G=0$) satisfies \eqref{eqn:exis_finite_hol_main}. 

    Fix $0\neq F\in V$ and set the function $f:\RR\to\C$ by
    \begin{align*}
        f(x) := F(e^{ix}).
    \end{align*}

    A change of variables shows that $f$ satisfies
    \begin{align*}
        \bigl(c+\twap(x)\bigr)\pa_x f(x)=if(x)+
        \sum_{k=1}^N
        \ha{f}_kg_k(x).
    \end{align*}

    Therefore, using Duhamel's formula, we find that
    \begin{align}
        f(x)&=
            \PiExis(x)f(0)+
            \PiExis(x)\int_0^x
            \frac{\beta(y)}{\PiExis(y)}
            \sum_{k=1}^N 
            \ha{f}_k
            g_k(y)
            \dy
        \nonumber\\
        &=
        \PiExis(x)f(0)+
        \sum_{k=1}^N 
        \ha{f}_k 
        Jg_k(x).\label{eqn:exis_Duhamel}
    \end{align}
    
    We project \eqref{eqn:exis_Duhamel} onto the modes $e^{ijx},\ j=0,\ldots,N$,
    \begin{align*}
        \ha{f}_j = 
        f(0) (\ha{h_{N+1}})_j +
        \sum_{k=1}^N 
        \ha{f}_k (\ha{h_k})_j.
    \end{align*}

    Rearranging, we obtain for every $j=0,\ldots,N$ that
    \begin{align}\label{eqn:exis_TCLassPlusI_TCLass_aux2}
        0=f(0) (\ha{h_{N+1}})_j +
        \sum_{k=1}^N 
        \ha{f}_k ((\ha{h_k})_j-\delta_{k,j}).
    \end{align}

    Since $\ha{f}_0=0$, taking $j=0$ in \eqref{eqn:exis_TCLassPlusI_TCLass_aux2} yields
    \begin{align}\label{eqn:exis_TCLassPlusI_TCLass_aux3}
        0=f(0) (\ha{h_{N+1}})_0 +
        \sum_{k=1}^N 
        \ha{f}_k (\ha{h_k})_0.
    \end{align}
    
    Equations \eqref{eqn:exis_TCLassPlusI_TCLass_aux2} and \eqref{eqn:exis_TCLassPlusI_TCLass_aux3} are equivalent to 
    \begin{align*}
        \MExisTor
        \begin{pmatrix}
        \ha{f}_1\\ \vdots\\ \ha{f}_N\\ f(0)
        \end{pmatrix}=0.
    \end{align*}

    Finally, note that $f(0)=0$ and $\ha{f}_k=0$ for every $k=1,\ldots,N$ is not possible by equation \eqref{eqn:exis_Duhamel}. Hence, we have $\det(\MExisTor)=0$.
\end{proof}

\begin{lemma}\label{lemma:exis_singval}
    The smallest singular value $\sigma_1$ of $\MExisTor$ satisfies
    \begin{align*}
        \sigma_1^2>
        \getvar{svd1_exis}.
    \end{align*}
\end{lemma}
\proofcompassi

\begin{corollary}\label{coro:exis_singval} 
    For every holomorphic function $G:\D \to \C$ that is continuous on $\overline{\D}$ and satisfies $G(0)=0$, there exists a unique holomorphic function $F:\D \to \C$ with $F(0)=0$ and $\pa_z F$ continuous on $\overline{\D}$, such that~\eqref{eqn:exis_finite_hol_main} holds.
\end{corollary}

\begin{proof}
    By definition, the singular values of $\MExisTor$ are the square roots of the eigenvalues of the Hermitian matrix
    \begin{align*}
        \MExisTor^\ast
        \MExisTor.
    \end{align*}
    
    Since $\sigma_1>0$, all singular values are positive and therefore
    \begin{align*}
        |\det(\MExisTor)|^2
        =
        \prod_{k=1}^{N+1} \sigma_k^2
        > 0.
    \end{align*}
    In particular, $\det(\MExisTor)\neq 0$.
    
    Lemma~\ref{lemma:exis_tor} then implies $\det(\MExisHol)\neq 0$, and Proposition~\ref{prop:exis_hol} applies.
\end{proof}

\subsubsection{Estimates of the explicit functions and operators}\label{sec:exis:bounds}

This section is devoted to proving explicit bounds for some functions that will allow us to obtain an estimate for the norm of $(\LOpExisHol)^{-1}$ from $\XHholZMi{1}$ to $\XHholZMi{2}$. We recall that the spaces $\XHholZMi{1}$ and $\XHholZMi{2}$ are defined in~\eqref{def:XHholZMi}.

We collect such bounds in Lemmas~\ref{lemma:betamax}--\ref{lemma:hk_funcs_exis}.

\begin{lemma}\label{lemma:betamax}
The function $\beta:\RR\to\RR$ given in \eqref{def:beta} satisfies
    \begin{align*}
        0 <  \beta(x)< \getvar{beta_max}\quad \text{for all }x\in[-\pi,\pi].
    \end{align*}
\end{lemma}
\proofcompassi

\begin{lemma}\label{lemma:kappa1}
Let $\kappa_1$ be defined by
\begin{align*}
    \kappa_1(x):=\int_0^x \beta(y)^2\dy.
\end{align*}

Then
    \begin{align*}
        \norm{\kappa_1}_{L^1([0,\pi])}<
        \getvar{kappa1}.
    \end{align*}
\end{lemma}
\proofcompassi

\begin{lemma}\label{lemma:betamod_L2}
The following estimate holds:
    \begin{align*}
        \norm{\pa_x\beta + i\beta^2}_\XLIpi^2
        <
        \getvar{beta_mod}.
    \end{align*}
\end{lemma}
\proofcompassi

\begin{lemma}\label{lemma:kappa2}
    Let $\kappa_1$ as in Lemma~\ref{lemma:kappa1} and $\kappa_2$ be defined by
    \begin{align*}
        \kappa_2(x) :=
        \big[(\pa_x\beta(x))^2+\beta(x)^4\big]
        \kappa_1(x).
    \end{align*}

    Then
    \begin{align*}
        \norm{\kappa_2}_{L^1([0,\pi])}
        <
        \getvar{kappa2}.
    \end{align*}
    
\end{lemma}

\proofcompassi

\begin{lemma}\label{lemma:J_L2L2_exis}
    Let $g\in\XLIpi$. Then
    \begin{align*}
        \norm{Jg}_\XLIpi^2 \leq
        \getvar{kappa1}
        \,
        \norm{g}_\XLIpi^2.
    \end{align*}
\end{lemma}

\begin{proof}
    It suffices to prove the estimate for $g\in C(\T)$, since $C(\T)$ is dense in $\XLIpi$.
    
    We note that $|\PiExis(x)|=1$ for all $x\in\RR$. Hence,
    \begin{align}
        \norm{Jg}_\XLIpi^2
        \nonumber
        &= 
        \intpi
        \Big|\PiExis(x)
        \int_0^x\beta(y)
        \frac{g(y)}{\PiExis(y)}
        \dy\Big|^2
        \dx
        \nonumber
        \\
        &= 
        \intpi
        \Big|\int_0^x
        \beta(y)
        \frac{g(y)}{\PiExis(y)}
        \dy\Big|^2
        \dx
        \nonumber
        \\
        &\leq  
        \int_{0}^\pi
            \norm{\beta g}_{L^1([0,x])}^2
        \dx
        +
        \int_{-\pi}^0
            \norm{\beta g}_{L^1([x,0])}^2
        \dx.
        \label{eqn2:J_L2L2_exis:aux1}
    \end{align}

    We focus on the first integral. The computations for the second one are similar since $\beta$ is an even function. Applying the Cauchy--Schwarz inequality yields
    \begin{align}
        \int_{0}^\pi
        \norm{\beta g}_{L^1([0,x])}^2
        \dx & \leq 
        \int_{0}^\pi
        \norm{\beta}_{L^2([0,x])}^2
        \norm{g}_{L^2([0,x])}^2
        \dx
        \label{eqn2:J_L2L2_exis:aux111}.
    \end{align}

    Let $\kappa_1$ be the function in Lemma~\ref{lemma:kappa1}, we obtain
    \begin{align}
        \int_{0}^\pi
        \norm{\beta}_{L^2([0,x])}^2
        \norm{g}_{L^2([0,x])}^2
        \dx
        \nonumber 
        & = 
        \int_{0}^\pi
        \kappa_1(x)
        \int_0^x 
            |g(y)|^2
        \dy
        \dx
        \nonumber
        \\ & = 
        \int_{0}^\pi
        |g(y)|^2
        \int_y^\pi
        \kappa_1(x)
        \dx\dy\nonumber
        \\ & \leq 
        \norm{\kappa_1}_{L^1([0,\pi])}
        \norm{g}_{L^2([0,\pi])}^2
        \label{eqn2:J_L2L2_exis:aux112}
    \end{align}

    Combining \eqref{eqn2:J_L2L2_exis:aux111} and \eqref{eqn2:J_L2L2_exis:aux112}, we obtain
    \begin{align*}
        \int_{0}^\pi
        \norm{\beta g}_{L^1([0,x])}^2
        \dx \leq 
        \norm{\kappa_1}_{L^1([0,\pi])}
        \norm{g}_{L^2([0,\pi])}^2.
    \end{align*}

    Therefore, we get
    \begin{align}
        \norm{Jg}_\XLIpi^2
        &\leq 
        \norm{\kappa_1}_{L^1([0,\pi])}
        \big(
            \norm{g}_{L^2([0,\pi])}^2
            +
            \norm{g}_{L^2([-\pi,0])}^2
        \big)
        \nonumber
        \\
        &\leq 
        \norm{\kappa_1}_{L^1([0,\pi])}
        \norm{g}_\XLIpi^2.
        \label{eqn2:J_L2L2_exis:last}
    \end{align}

    Finally, we substitute the explicit estimate in Lemma~\ref{lemma:kappa1} into \eqref{eqn2:J_L2L2_exis:last} to obtain the bound of the statement.
\end{proof}

\begin{lemma}\label{lemma:J_H1H2_exis}
    Let $G\in\XHholZMi{1}$ and let $g:=G|_\T$. Then
    \begin{align*}
        \norm{\pa_x^2 Jg}_\XLIpi
        \leq 
        \getvar{JH1H2}
        \norm{\pa_x g}_\XLIpi.
    \end{align*}
\end{lemma}

\begin{rem}
    Observe that $Jg$ is not necessarily continuous on $\T$, even if it is smooth on $[-\pi,\pi]$.
\end{rem}

\begin{proof}
    First, assume that $g=G|_\T\in C^1(\T)$, so that $\pa_x^2 Jg$ is well defined pointwise. We prove the estimate in this case. 
    Since smooth functions are dense in the space of traces of $\XHholZMi{1}$ and the norms involved depend continuously on $g$, it follows that the estimate extends to all $G\in\XHholZMi{1}$.

    Fix $G\in\XHholZMi{1}$ with $g:=G|_\T\in C^1(\T)$. Differentiating $Jg$ twice yields
     \begin{align*}
        \pa_x^2 Jg=
        \beta \pa_x g
        +
        (\pa_x\beta+i\beta^2) g 
        +
        (i\pa_x\beta-\beta^2) Jg,
    \end{align*}
    because $Jg$ satisfies $\pa_x Jg=\beta g+i\beta Jg$.
    
    We compute
    \begin{align}\label{eqn:J_H1H2_exis_split}
        \norm{\pa_x^2 Jg}_\XLIpi\leq I_1+I_2+I_3,
    \end{align}
    
    where
    \begin{align*}
        I_1:=
        \norm{\beta \pa_x g}_\XLIpi,
        &&
        I_2:=
        \norm{(\pa_x\beta+i\beta^2) g}_\XLIpi,
        &&
        I_3:=
        \norm{(i\pa_x\beta-\beta^2) Jg}_\XLIpi.
    \end{align*}

    We first bound $I_1$,
    \begin{align}
        I_1&=
        \norm{\beta \pa_x g}_\XLIpi
        \leq
        C_1
        \norm{\pa_x g}_\XLIpi
        \label{eqn:J_H1H2_exis_I1}
    \end{align}
    Here, $C_1:=\norm{\beta}_{\XLinf}$ is bounded by Lemma~\ref{lemma:betamax}.

    Let $C_2\in\RR$ be defined as
    \begin{align*}
        C_2&:=
        \Big(\frac{\pi}{12}\Big)^{1/2}
        \norm{\pa_x\beta+i\beta^2}_\XLIpi,
    \end{align*}
    The constant $C_2$ is bounded by Lemma~\ref{lemma:betamod_L2}. 
    
    Then, we bound $I_2$:
    \begin{align}
        I_2=
        \norm{(\pa_x\beta+i\beta^2) g}_\XLIpi
        \leq
        \norm{\pa_x\beta+i\beta^2}_\XLIpi
        \norm{g}_{\XLinf}
        \leq
        C_2
        \norm{\pa_x g}_\XLIpi.
        \label{eqn:J_H1H2_exis_I2}
    \end{align}
    Here, we have used Lemma~\ref{lemma:apriori:Infty_H1:hol_mean0} in the last step.
    
    Finally, we bound $I_3$. As in the proof of Lemma~\ref{lemma:J_L2L2_exis}, we use that $|\PiExis(x)|=1$ for all $x\in\RR$. Thus
    \begin{align*}
        I_3^2 \leq 
        \int_0^\pi 
            \big|i\pa_x\beta(x)-\beta(x)^2\big|^2 
            \norm{\beta g}_{L^1([0,x])}^2
        \dx
        +
        \int_{-\pi}^0 
            \big|i\pa_x\beta(x)-\beta(x)^2\big|^2 
            \norm{\beta g}_{L^1([x,0])}^2
        \dx.
    \end{align*}

    We study the first integral. Computations for the second one are similar since $\beta$ is even. 
    Let $\kappa_1$ and $\kappa_2$ be as in Lemma~\ref{lemma:kappa2}, applying the same strategy used in the proof of Lemma~\ref{lemma:J_L2L2_exis}, we obtain
    \begin{align}
        \int_0^\pi 
        \big|i\pa_x\beta(x)-\beta(x)^2\big|^2 
        \norm{\beta g}_{L^1([0,x])}^2
        \dx
        &=
        \int_0^\pi 
        \big[(\pa_x \beta(x))^2 + \beta(x)^4\big]
        \norm{\beta g}_{L^1([0,x])}^2
        \dx
        \nonumber
        \\
        &\leq 
        \int_0^\pi 
        \underbrace{
        \big[(\pa_x \beta(x))^2 + \beta(x)^4\big]
        \kappa_1(x)}_{\kappa_2(x)}
        \norm{g}_{L^2([0,x])}^2
        \dx
        \nonumber
        \\
       &\leq
        \int_0^\pi 
        |g(y)|^2 
        \int_y^\pi
        \kappa_2(x)
        \dx\dy
        \nonumber
        \\
        &\leq
        \norm{\kappa_2}_{L^1([0,\pi])}
        \norm{g}_{L^2([0,\pi])}^2.
        \label{eqn:J_L2L2_exis_I3}
    \end{align}

    Define $C_3$ by
    \begin{align*}
        C_3:=\norm{\kappa_2}_{L^1([0,\pi])}^{1/2}.
    \end{align*}
    The constant $C_3$ is bounded by Lemma~\ref{lemma:kappa2}.

    Therefore, we find that
    \begin{align}
        I_3^2
        &\leq
        \norm{\kappa_2}_{L^1([0,\pi])}
        \big(
        \norm{g}_{L^2([0,\pi])}^2
        +
        \norm{g}_{L^2([-\pi,0])}^2
        \big)
        \nonumber\\&
        \leq
        \norm{\kappa_2}_{L^1([0,\pi])}
        \norm{g}_\XLIpi^2
        \nonumber 
        \\
        &\leq
        C_3^2
        \norm{\pa_x g}_\XLIpi^2.
        \label{eqn2:J_H1H2_exis_I3}
    \end{align}
    In the last step, we used Parseval's identity in the following way:
    \begin{align*}
        \frac{1}{2\pi}
        \norm{g}_\XLIpi^2 =
        \sum_{k\geq 1} |\ha{g}_k|^2 \leq 
        \sum_{k\geq 1} k^2 |\ha{g}_k|^2 =
        \frac{1}{2\pi}
        \norm{\pa_x g}_\XLIpi^2,
    \end{align*}
    where we used that $G\in\XHholZM$, so that $\ha{g}_k=0$ for all $k\leq 0$.

    Therefore, substituting \eqref{eqn:J_H1H2_exis_I1}, \eqref{eqn:J_H1H2_exis_I2}, and \eqref{eqn2:J_H1H2_exis_I3} into \eqref{eqn:J_H1H2_exis_split}, we obtain
    \begin{align*}
        \norm{\pa_x^2 Jg}_\XLIpi
        &\leq 
        \big(C_1 + C_2 + C_3\big)
        \norm{\pa_x g}_\XLIpi.
    \end{align*}
\end{proof}

The next lemma is the last one needed before proving Proposition~\ref{prop:inv_exis} and is also computer-assisted.

\begin{lemma}\label{lemma:hk_funcs_exis}
    For each $k=1,\ldots,N+1,$ let $h_k$ be as in \eqref{def:hk}. Then the following estimate holds:
    \begin{align*}
        \sum_{k=1}^{N+1}
            \norm{\pa_x^2 h_k}_\XLIpi^2
            \leq
            \getvar{hkH2}.
    \end{align*}
\end{lemma}

\proofcompassi

\subsection{Conclusion}\label{sec:exis:laststep}
We are now in a position to prove Proposition~\ref{prop:inv_exis}.

By Lemma~\ref{lemma:exis:td}, the operator $\LOpExis:\XHodd{2}\to\XHeven{1}$ is invertible if and only if its holomorphic counterpart $\LOpExisHol:\XHholZMi{2}\to\XHholZMi{1}$ is invertible. In addition, the norms of the inverses coincide.

Let $G\in\XHholZMi{1}$ and set $g:=G|_\T\in\XH{1}$. 
In particular, $g$ is continuous. Hence, by Corollary~\ref{coro:exis_singval}, we know that there exists a unique holomorphic function $F:\D\to\C$ with $F(0)=0$ and $\pa_z F\in C(\overline{\D})$ that satisfies \eqref{eqn:exis_finite_hol_main}. Set $f:=F|_\T\in C^1(\T)$.

Making the change of variables $z = e^{ix}$ in \eqref{eqn:exis_finite_hol_main}, we find that $f$ and $g$ satisfy
\begin{align*}
    \big(c+\twap(x)\big) \pa_x f(x) =
    i f(x) + \sum_{k=1}^N \ha{f}_k g_k(x) + g(x),
\end{align*}
where $g_k$ is given in \eqref{def:gk_exis}.

By Duhamel's formula, we have 
\begin{align*}
    f(x)
    =
    f(0)\PiExis(x)
    +
    \sum_{k=1}^N \ha{f}_kJg_k(x)
    +
    Jg(x),
\end{align*}
where the coefficients $\ha{f}_k$, $f(0)$ and the matrix $\MExisTor$ from \eqref{def:Mexistor} satisfy the relation
\begin{align*}
    0=\MExisTor
    \begin{pmatrix}
        \ha{f}_1\\
        \vdots\\
        \ha{f}_N\\
        f(0)
    \end{pmatrix}
    +
    \begin{pmatrix}
        (\ha{Jg})_0\\
        \vdots\\
        (\ha{Jg})_N\\
    \end{pmatrix}.
\end{align*}

Hence, we obtain
\begin{align}\label{eqn:exis_final_step1}
    f(x)=-
    \begin{pmatrix}
        h_1(x) \\
        \vdots\\
        h_{N+1}(x)\\
    \end{pmatrix}^t
    (\MExisTor)^{-1}
    \begin{pmatrix}
        (\ha{Jg})_0\\
        \vdots\\
        (\ha{Jg})_N\\
    \end{pmatrix}
    +
    Jg(x),
\end{align}
where $h_k =Jg_k$ for $k=1,\ldots, N$ and $h_{N+1}=\PiExis$.

Let 
\begin{align*}
    \MExisTor = U\Sigma V^\ast
\end{align*} 
be a singular value decomposition, where $\Sigma={\rm diag}(\sigma_1,\ldots,\sigma_{N+1})$ is ordered increasingly.

Because $V$ and $U$ are unitary, for all $a,b\in\C^{N+1}$, we have
\begin{align*}
    \Big|
    \sum_{k,l,m=1}^{N+1}
        a_k
        V_{k,l}
        \frac{1}{\sigma_l} 
        \overline{U}_{m,l} 
        b_m
    \Big|^2
    &\leq
    \norm{a}_{\ell^2}^2
    \sum_{k=1}^{N+1}
        \Big|
        \sum_{l,m=1}^{N+1}
            V_{k,l}
            \frac{1}{\sigma_l} 
            \overline{U}_{m,l}  b_m
        \Big|^2 
    = 
    \norm{a}^2_{\ell^2}
    \sum_{l=1}^{N+1}
        \frac{1}{\sigma_l^2}
        \Big|
        \sum_{m=1}^{N+1}
            \overline{U}_{m,l}  
            b_m
        \Big|^2
    \\
    &\leq
    \frac{1}{\sigma_1^2}
    \norm{a}^2_{\ell^2}
    \sum_{l=1}^{N+1}
        \Big|
        \sum_{m=1}^{N+1}
            \overline{U}_{m,l}  
            b_m
        \Big|^2 
    = 
    \frac{1}{\sigma_1^2}
    \norm{a}^2_{\ell^2}
    \norm{b}^2_{\ell^2}.
\end{align*}
In particular, taking $b_m=(\ha{Jg})_{m-1},\ m=1,\ldots,N+1$ and $a_k=\pa_x^2h_k(x),\ k=1,\ldots,N+1$ for a fixed $x\in[0,2\pi]$, we obtain
\begin{align*}
    \Big|
    \sum_{k,l,m=1}^{N+1}
    \pa_x^2 h_k(x) V_{k,l}
    \frac{1}{\sigma_l} 
    \overline{U}_{m,l}  
    (\ha{Jg})_{m-1}
    \Big|^2
    &\leq
    \frac{1}{\sigma_1^2}
    \Big(
    \sum_{k=1}^{N+1} 
        |\pa_x^2 h_k(x)|^2
    \Big)
    \Big(
    \sum_{m=0}^{N} 
        |(\ha{Jg})_{m}|^2
    \Big)\\
    &\leq
    \frac{1}{\sigma_1^2}
    \Big(
    \sum_{k=1}^{N+1} 
        |\pa_x^2 h_k(x)|^2
    \Big)
    \norm{Jg}_\XL^2,
\end{align*}
where we used Parseval's identity.

Taking the $\XLIpi$ norm yields
\begin{align}
    \scalednorm{
    \begin{pmatrix}
        \pa_x^2 h_1\\
        \vdots\\
        \pa_x^2 h_{N+1}
    \end{pmatrix}^t
    V
    \Sigma^{-1}
    U^\ast
    \begin{pmatrix}
        (\ha{Jg})_0\\
        \vdots\\
        (\ha{Jg})_N
    \end{pmatrix}
    }_\XLIpi^2 
    &= 
    \intpi
    \Big|
    \sum_{k,l,m=1}^{N+1}
        \pa_x^2 h_k(x) 
        V_{k,l} 
        \frac{1}{\sigma_l} 
        \overline{U}_{m,l}
        \ha{Jg}_{m-1}
    \Big|^2
    \dx\nonumber
    \\
    &\leq
    \frac{1}{\sigma_1^2}
    \norm{Jg}_\XL^2
    \sum_{k=1}^{N+1} 
    \intpi
        |\pa_x^2 h_k(x)|^2
    \dx\nonumber
    \\
    &\leq
    \frac{1}{\sigma_1^2}
    \norm{Jg}_\XL^2
    \sum_{k=1}^{N+1} 
        \norm{\pa_x^2 h_k}_\XLIpi^2.
        \label{eqn:exis_final_step2}
\end{align}

Let $C_1\in\RR$ be defined by
\begin{align*}
    C_1^2:=
    \frac{1}{\sigma_1^2}
    \norm{J}_{\XL\to\XL}^2
    \sum_{k=1}^{N+1}\norm{\pa_x^2 h_k}_\XLIpi^2,
\end{align*}
Lemmas~\ref{lemma:exis_singval}, \ref{lemma:J_L2L2_exis}, and \ref{lemma:hk_funcs_exis} provide an explicit upper bound for $C_1$.

By Lemma~\ref{lemma:J_H1H2_exis}, there exists an explicit constant $C_2>0$ such that
\begin{align}\label{eqn:exis_final_step3}
    \norm{\pa_x^2 Jg}_\XLIpi
    \leq 
    C_2
    \norm{\pa_x g}_\XLIpi.
\end{align}

We differentiate \eqref{eqn:exis_final_step1} twice and apply the estimates \eqref{eqn:exis_final_step2} and \eqref{eqn:exis_final_step3}:
\begin{align*}
    \norm{\pa_x^2 f}_\XLIpi
    &\leq
    \scalednorm{
    \begin{pmatrix}
        \pa_x^2 h_1\\
        \vdots\\
        \pa_x^2 h_{N+1}
    \end{pmatrix}^t
    V
    \Sigma^{-1}
    U^\ast
    \begin{pmatrix}
        (\ha{Jg})_0\\
        \vdots\\
        (\ha{Jg})_N
    \end{pmatrix}
    }_\XLIpi
    +
    \norm{\pa_x^2 Jg}_\XLIpi
    \\
    &\leq 
    C_1 \norm{g}_\XL
    +
    C_2 \norm{\pa_x g}_\XLIpi.
\end{align*}

Since $G\in\XHholZMi{1}$, the trace $g=G|_\T$ has only positive Fourier modes. Hence,
\begin{align*}
    \norm{g}_\XL^2=
    \sum_{k\geq 1}
    |\ha{g}_k|^2 \leq 
    \sum_{k\geq 1}
    k^2 |\ha{g}_k|^2
    = 
    \frac{1}{2\pi}
    \norm{\pa_x g}_\XLIpi^2.
\end{align*}

Define $C_3 := \frac{C_1}{\sqrt{2\pi}} + C_2$, so we get that
\begin{align*}
    \norm{\pa_x^2 f}_\XLIpi
    &\leq 
    C_3 
    \norm{\pa_x g}_\XLIpi.
\end{align*}

Consequently, since $f\in C^1(\T)$ and $\pa_x^2 f\in \XLIpi$, it follows that $f\in\XH{2}$. Additionally, since $F$ is holomorphic in $\D$ and $F(0)=0$, the function $f$ has only positive Fourier modes. We then find that
\begin{align*}
    \norm{F}_\XHholZMi{2}^2=
    \frac{1}{2\pi}\norm{\pa_x^2 f}_\XLIpi^2
    \leq \frac{C_3^2}{2\pi}
    \norm{\pa_x g}_\XLIpi^2=
     C_3^2
    \norm{G}_\XHholZMi{1}^2.
\end{align*}

Therefore, by Lemma~\ref{lemma:exis:td} and identity \eqref{eqn:LOpExis2}, we have that
\begin{align*}
    \norm{\LOpExisTwapInv}
    _{\XHodd{0}\to\XHeven{1}}=
    \norm{(\LOpExis)^{-1}}
    _{\XHeven{1}\to\XHodd{2}}=
    \norm{(\LOpExisHol)^{-1}}
    _{\XHholZMi{1}\to\XHholZMi{2}} \leq C_3.
\end{align*}

Finally, we substitute all the bounds and find that
\begin{align*}
    \norm{\LOpExisTwapInv}
    _{\XHodd{0}\to\XHeven{1}} \leq C_3 <
    \getvar{Lexis_inv}.
\end{align*}

\section{Proof of Proposition \ref{prop:twStab}: Lemmas \ref{lemma:stab_ballintoitself}--\ref{lemma:real_eigen_control}}
\label{sec:prop:inv_stab}
In this section we shall prove Lemmas~\ref{lemma:stab_ballintoitself}, \ref{lemma:stab_lipschitz} and \ref{lemma:real_eigen_control} which yield Proposition \ref{prop:twStab}. Recall that Proposition \ref{prop:twStab} is the main ingredient which remains to show that the solution $(\frac{c}{3},\tw^3)$ is linearly unstable.

\subsection{The operator $\TopStab$}\label{sec:proof_cont_stab}
This section is devoted to the study of the operator $\TopStab$, given by
\begin{align}\label{def:ContOpStab_bis}
    \TopStab f :=
    (\LStabthe 
    - 
    \laap I)^{-1}
    \big[
        (\Eigdiff(f) I -\DiffOpExpNonExp)
        (\Eigveap+f)
        -
        \Eigres
    \big].
\end{align}

Before proceeding with the analysis of $\TopStab$, it is necessary to justify that the definition \eqref{def:ContOpStab_bis} is indeed meaningful. More precisely, the expression involves the inverse operator
\begin{align*}
    (\LStabthe - \laap I)^{-1},
\end{align*}
and \emph{a priori} it is not clear that the operator $\LStabthe - \laap I$ is invertible. 

In addition, in order to define $\Eigdiff(f)$, we will need to further understand the structure of $\LStabthe-\laap I$. Indeed, since $\laap$ is an approximation of the eigenvalue $\la$ of $\LStabtheTW$ (whose existence we aim to prove), and $\LStabthe$ is close to $\LStabtheTW$, the operator $\LStabthe-\laap I$ is expected to be close to being non-invertible. 

Nevertheless, by choosing $\Eigdiff(f)$ appropriately, we will remove the nearly singular direction, yielding uniform bounds for the inverse.

\subsection{The linear operator $\LStabthe$}
\label{sec:proof_inv_stab}

Our first goal in this section is to reduce the study of the operator $\LStabthe-\laap I$, acting on periodic function spaces, to the study of an operator acting on Hardy spaces. Here, $\LStabthe$ is defined in \eqref{def:LStabthe}.

To this end, we define the operator 
$\LOpStabHol{\la}:\XHholo\oplus\XHholo\to\XHholo\oplus\XHholo$, depending on two parameters $\te\in(0,1)$ and $\la\in\C$, by 
\begin{align*}
    \LOpStabHol{\la}(F^+,F^-):=
    \big(
    \LOpStabHol{\la}(F^+,F^-)^1
    ,
    \LOpStabHol{\la}(F^+,F^-)^2
    \big),
\end{align*}
where
\begin{subequations}\label{system:stab_hol}
        \begin{empheq}[left=\empheqlbrace]{align}
        \LOpStabHol{\la}(F^+,F^-)^1
        :=
            izP(z)\pa_z F^+(z)-
            (i+\la)F^+(z)+i\te^+ P(z)F^+(z) -\mathcal{G}^{\te^+}\!(F^+,F^-)(z),
        \label{system:stab_hol1}
        \\
        -\LOpStabHol{\la}(F^+,F^-)^2
        :=
            izP(z)\pa_z F^-(z)-
            (i-\la)F^-(z)+i\te^- P(z)F^-(z) -\mathcal{G}^{\te^-}\!(F^-,F^+)(z).
        \label{system:stab_hol2}
        \end{empheq}
\end{subequations}
Recall that $P$ is defined in \eqref{def:P_}.
Here we set $\te^+:=\te$ and $\te^-:=1-\te$. The meromorphic functions $\mathcal{G}^{\te^\pm} (A,B)$ are defined for $A,B\in\XHholo$ by
\begin{align}\label{def:Gcalthe}
    \mathcal{G}^{\te}(A,B):=
    \sum_{k=1}^N
    \big(
    A_{k-1}
    \mathcal{G}_{k,1-\te}^-
    +
    B_{k-1}
    \mathcal{G}_{k,\te}^+ 
    \big)\quad \text{for }\te\in(0,1),
\end{align}
where
\begin{align}\label{def:Gk_stab}
    \mathcal{G}_{k,\te}^-(z):=
    (k-\te)\frac{i}{2}
    \sum\limits_{m=0}^{N-k}    
    \twap_{m+k}
     z^{-(1+m)}
    \qquad
    \text{and}
    \qquad 
    \mathcal{G}_{k,\te}^+(z):=
    (k-\te)\frac{i}{2}
    \sum\limits_{m=0}^{N-k}    
    \twap_{m+k}
     z^m.
\end{align}

The operator $\LOpStabHol{\la}(F^+,F^-)$ is well defined. Indeed, for any $(F^+,F^-)\in\XHholo\oplus\XHholo$, the expressions in \eqref{system:stab_hol} define holomorphic functions, since all terms involving negative powers of $z$ cancel out.

\begin{rem}\label{rem:Gthe}
    With this notation, we have
    \begin{align*}
        \mathcal{G}^{\te^+}(F^+,F^-)(z):=
        \frac{i}{2}
        \sum_{j=1}^N\twap_j
        \sum_{m=1}^j
        \big[
        (m-\te^-) F^+_{m-1} z^{-(1+j-m)}
        +
        (m-\te^+) F^-_{m-1} z^{j-m}
        \big],
        \\
        \mathcal{G}^{\te^-}(F^-,F^+)(z):=
        \frac{i}{2}
        \sum_{j=1}^N\twap_j
        \sum_{m=1}^j
        \big[
        (m-\te^+) F^-_{m-1} z^{-(1+j-m)}
        +
        (m-\te^-) F^+_{m-1} z^{j-m}
        \big].
    \end{align*}
\end{rem}

\begin{defi}\label{def:FplusFminu}
    Let $f\in \XH{1}$. We define $F^+,F^-\in\XHhol{1}$ by
    \begin{align*}
        F^+(z) := \sum_{k\geq 0} \ha{f}_k z^k,
        \qquad
        F^-(z) := \sum_{k\geq 0} \ha{f}_{-k-1} z^k.
    \end{align*}

\end{defi}

\begin{rem}
    Unlike \eqref{def:F+F-}, here we do not assume $\ha{f}_0=0$. Thus, $F^+$ contains the Fourier modes $k\geq  0$, while $F^-$ contains the modes $k\leq  -1$.
\end{rem}

\begin{lemma}\label{lemma:equiv_stab}
    Let $\te\in(0,1)$ and $\la\in\C$.
    Let $f\in \XH{2}$ and $g\in\XH{1}$, and define $F^\pm\in\XHhol{2}$ and $G^\pm\in\XHhol{1}$ as in Definition~\ref{def:FplusFminu}. Assume that $f$ and $g$ satisfy
    \begin{align}\label{eqn:equiv_stab}
        \LStabthe f-\la f = g,
    \end{align}

    Then, $F^\pm$ and $G^\pm$ satisfy \eqref{system:stab_hol}.
    
\end{lemma}

\begin{proof}
    Assume that $f$ and $g$ satisfy \eqref{eqn:equiv_stab}, that is
    \begin{align}\label{eqn:LStabthe}
        (c+\twap) (\pa_x f + i\te f) +
        H f - i\ha{f}_0
        - \la f
        = g.
    \end{align}

    Define
    \begin{align*}
        f^+(x):=
        \sum_{k\geq 0}\ha{f}_k e^{ikx}
        \qquad
        \text{and}
        \qquad
        f^-(x):=
            \sum_{k\geq 0}\ha{f}_{-k-1} e^{ikx}
            .
    \end{align*}

    Then, setting $f^\sharp(x):=e^{-ix}f(-x)$, we obtain
    \begin{align*}
        f^+(x)=\frac{1}{2}
            \big(f(x)+iHf(x)\big)
            +
            \frac{1}{2}\ha{f}_0, \qquad\text{and}
            \qquad
        f^-(x)=\frac{1}{2}
        \big(f^\sharp(x)+iHf^\sharp(x)\big)
            +
            \frac{1}{2}\ha{f}_{-1}
        .
    \end{align*}

    We first prove that $f^+$, $f^-$, $g^+$, and $g^-$ satisfy \eqref{system:stab_hol} restricted to $\T$, namely
    \begin{subequations}\label{system:stab_tor}
        \begin{empheq}[left=\empheqlbrace]{align}
        (c+\twap)
        (\pa_x f^+ + i \te^+ f^+)   &=
        (i+\la)f^+
        +
        g^+
        +\mathcal{J}^{\te^+}(f^+,f^-),\label{system:stab_tor1}
        \\
        (c+\twap)
        (\pa_x f^-+i\te^- f^-)   &=
        (i-\la)f^-
        -g^-
        +\mathcal{J}^{\te^-}(f^-,f^+),\label{system:stab_tor2}
        \end{empheq}
    \end{subequations}

    where 
    \begin{align}
        \mathcal{J}^{\te}(f^+,f^-)(x) :=
        \frac{i}{2}
        \sum_{j=1}^N\twap_j
            \sum_{m=1}^j
            (m+\te-1) \ha{f^+}_{m-1} e^{-i(1+j-m)x}+
            (m-\te) \ha{f^-}_{m-1} e^{i(j-m)x}.
            \label{def:Jcalthe}
    \end{align}

    Taking the mean of \eqref{eqn:LStabthe} yields
    \begin{align*}
        \frac{1}{2\pi}
        \intpi
        \big(c+\twap(x)\big)
        \big(\pa_x f(x) + i\te f(x)\big) dx 
        - i\ha{f}_0
        - \la \ha{f}_0
        = 
        \ha{g}_0.
    \end{align*}
    Hence,
    \begin{align}
        i\ha{f}_0 + \la \ha{f}_0 + \ha{g}_0
        &=
        \frac{1}{2\pi}
        \intpi
        \big(c+\twap(x)\big)
        \big(\pa_x f(x) + i\te f(x)\big) dx\nonumber \\ &=
        -\frac{1}{2\pi}
        \intpi
        \pa_x \twap(x)
        f(x)\dx
        +
        i\te
        \frac{1}{2\pi}
        \intpi
        \big(c+\twap(x)\big)
        f(x)\dx 
        \nonumber
        \\
        &=
        \frac{1}{2\pi}
        \int_{-\pi}^\pi
        \sum_{j=1}^N
        j\twap_j
        \sin(jx) f(x)\dx
        +
        i\te
        \frac{1}{2\pi}
        \intpi
        \Big(c+
        \sum_{j=1}^N 
        \twap_j \cos(jx)
        \Big)
        f(x)\dx 
        \nonumber
        \\
        &=\frac{i}{2}\sum_{j=1}^N
        j\twap_j (\ha{f}_j - \ha{f}_{-j})
        +
        i\te \Big(
        c \ha{f}_0+
        \frac{1}{2}
        \sum_{j=1}^N
        \twap_j
        (\ha{f}_{j} + \ha{f}_{-j})
        \Big)
        \nonumber
        \\
        &=
        i\te 
        c \ha{f}_0
        +
        \frac{i}{2}\sum_{j=1}^N
            \twap_j 
            \big((j+\te)\ha{f}_j 
            -
            (j-\te)\ha{f}_{-j}\big).
        \label{eqn:equiv_stab:mean}
    \end{align}
    
    In addition, applying $\frac{1}{2}(I+iH)$ to \eqref{eqn:LStabthe} gives
    \begin{align}\label{eqn:equiv_stab:aux1}
        \frac{1}{2}
        (I+iH)\big[
        (c+\twap)
        (\pa_x f + i\te f)
        \big]
        \nonumber
        &=\frac{1}{2}
        (I+iH)\big[
        g -Hf + i\ha{f}_0 + \la f
        \big]
        \nonumber
        \\
        &=
        \big(g^+ - \frac{1}{2}\ha{g}_0\big)
        +
        i(f^+ - \ha{f}_0)
        +\frac{i}{2}\ha{f}_0 +
        \la 
        \big(f^+-\frac{1}{2}\ha{f}_0\big)
        \nonumber\\
        &=
        g^+ + (i+\la) f^+
        - \frac{1}{2}
        \big(
            i\ha{f}_0 + \la \ha{f}_0 + \ha{g}_0
        \big),
    \end{align}
    where we have used that
    \begin{align*}
        -\frac{1}{2}(I+iH)[Hf](x)
        &=-\frac{1}{2}\sum_{k\in\Z}
        (1+\sgn(k)) (-i\sgn(k)) \ha{f}_k
        e^{ikx}\\
        &=\frac{i}{2}\sum_{k\in\Z}
        (\sgn(k)+\sgn(k)^2) \ha{f}_k e^{ikx}=i(f^+(x)-\ha{f}_0).
    \end{align*}

    Define
    \begin{align*}
        \Xi^+:=(c+\twap) 
        (\pa_x f^+ + i\te f^+) - (i+\la) f^+
        -g^+.
    \end{align*}
    
    Using the definition of $f^+$, we write
    \begin{align}
        \Xi^+ &=
        \frac{1}{2}
        (c+\twap)(I+iH)
        (\pa_x f + i\te f)
        +
        \frac{i\te}{2}(c+\twap)\ha{f}_0
        - (i+\la) f^+-g^+
        \nonumber\\
        &=
        \frac{1}{2}
        (I+iH)\big[
        (c+\twap)
        (\pa_x f + i\te f)
        \big]
        -\frac{i}{2}[H,c+\twap]
        (\pa_x f + i\te f)\nonumber\\
        &\qquad+
        \frac{i\te}{2}
        (c+\twap)
        \ha{f}_0
        - 
        (i+\la) f^+ - g^+.
        \label{eqn:equiv_stab:aux2}
    \end{align}
    
    Now, plugging \eqref{eqn:equiv_stab:aux1} into \eqref{eqn:equiv_stab:aux2}, we obtain
    \begin{align}\label{eqn:equiv_stab:aux3}
        \Xi^+
        &=
        - \frac{1}{2}
        (
            i\ha{f}_0 + \la \ha{f}_0 + \ha{g}_0
        )
        -
        \frac{i}{2}[H,c+\twap]
        (\pa_x f + i\te f)+
        \frac{i\te}{2}
        (c+\twap)
        \ha{f}_0
        \nonumber\\
        &=
        -\frac{i}{4}\sum_{j=1}^N
            \twap_j 
            \big((j+\te)\ha{f}_j 
            -
            (j-\te)\ha{f}_{-j}\big)
        -\frac{i}{2}[H,c+\twap]
        (\pa_x f + i\te f)+
        \frac{i\te}{2}
        \twap
        \ha{f}_0,
    \end{align}
    where we used \eqref{eqn:equiv_stab:mean} in the last step.

    The expression for the commutator is given in Corollary~\ref{coro:BracketHiTw:f}. Substituting it into \eqref{eqn:equiv_stab:aux3} yields
    \begin{align*}
        \Xi^+(x)&=
        -
        \frac{i}{4}\sum_{j=1}^N
            \twap_j 
            \big((j+\te)\ha{f}_j 
            -
            (j-\te)\ha{f}_{-j}\big)
        +
        \frac{i\te}{2}
        \twap(x)
        \ha{f}_0
        \\
        &\qquad
        +\frac{i\te}{4}
        \ha{f}_0
        \sum_{j=1}^N
            \twap_j
            \big(e^{-ijx}
            -
            e^{ijx}\big)
        +
        \frac{i}{4}
        \sum_{j=1}^N
        \twap_j
            \Big(
                (j+\te)
                \ha{f}_j
                +
                (j-\te)
                \ha{f}_{-j}
            \Big)
        \\
        &\qquad +\frac{i}{2}
        \sum_{j=1}^N
        \twap_j
        \sum_{m=1}^{j-1}
            \big(
                (m+\te)
                \ha{f}_m e^{-i(j-m)x}
                +
                (m-\te)
                \ha{f}_{-m} e^{i(j-m)x}
            \big)
        \\
        &=
        \frac{i}{2}
        \sum_{j=1}^N
        \twap_j
        (j-\te)
        \ha{f}_{-j}
        +
        \ha{f}_0
        \frac{i\te}{2}
        \sum_{j=1}^N
            \twap_j
            e^{-ijx}
        \\
        &\qquad +\frac{i}{2}
        \sum_{j=1}^N
        \twap_j
        \sum_{m=1}^{j-1}
            \big(
                (m+\te)
                \ha{f}_m e^{-i(j-m)x}
                +
                (m-\te)
                \ha{f}_{-m} e^{i(j-m)x}
            \big)
        \\
        &=\frac{i}{2}
        \sum_{j=1}^N
        \twap_j
        \sum_{m=1}^{j}
            \big(
                (m-1+\te)
                \ha{f}_{m-1} e^{-i(1+j-m)x}
                +
                (m-\te)
                \ha{f}_{-m} e^{i(j-m)x}
            \big)\\
        &=\mathcal{J}^{\te^+}(f^+,f^-).
    \end{align*}
    
    We now derive the equation satisfied by $f^-$. Changing variables $x\mapsto -x$ in \eqref{eqn:LStabthe} and multiplying by $e^{-ix}$, we obtain
    \begin{align}\label{eqn:LOpStabThetamult}
        e^{-ix} \big(c+\twap(x)\big)
        \big(\pa_x f(-x) + i\te f(-x)\big) +
        e^{-ix} Hf(-x) - i\ha{f}_0 e^{-ix}
        - \la f^\sharp(x)
        = g^\sharp(x).
    \end{align}
    where $f^\sharp(x):=e^{-ix}f(-x)$ and $g^\sharp(x):=e^{-ix}g(-x)$. Using that
    \begin{align}\label{eqn:stab:dxfsharp}
        e^{-ix} \big(\pa_x f(-x) + i\te f(-x)\big) &= 
        -\pa_x f^\sharp (x) + i (\te-1) f^\sharp(x),\\
        e^{-ix} Hf(-x) &= -Hf^\sharp(x)
        +i\ha{f}_{-1} + i\ha{f}_0 e^{-ix},
        \nonumber
    \end{align}
    and multiplying \eqref{eqn:LOpStabThetamult} by $-1$, we obtain
    \begin{align}\label{eqn:LOpStabThetasharp}
        \big(c+\twap(x)\big)
        \big(\pa_x f^\sharp + i\te^- f^\sharp\big) +
         Hf^\sharp - i\ha{f}_{-1} 
        + \la f^\sharp
        = -g^\sharp.
    \end{align}

    Since $\ha{f^\sharp}_0=\ha{f}_{-1}$, we may repeat the previous argument to conclude that $f^-$ satisfies \eqref{system:stab_tor2}.

    Finally, the functions $f^+$ and $f^-$, and their derivatives extend holomorphically to $\D$, and these extensions coincide with $F^+$ and $F^-$, respectively. The same holds for $g$, hence extending \eqref{system:stab_tor} from $\T$ to $\D$ yields \eqref{system:stab_hol}.
\end{proof}

Looking at the Fourier modes of $f$ in Definition~\ref{def:FplusFminu}, we see that the mapping $f\in\XH{1}\to (F^+,F^-)\in\XHholo^1\oplus\XHholo^1$ is a bijection. In fact, we can recover $f$ via
\begin{align}\label{eqn:recover_fstab}
    f(x) := F^+|_\T(x) + e^{-ix}F^-|_\T(-x)
    = F^+(e^{ix}) + e^{-ix}F^-(e^{-ix}).
\end{align}
Here, the Sobolev exponent is chosen as $1$ so that $f$ is continuous.

In the next lemma, we prove that the converse of Lemma~\ref{lemma:equiv_stab} also holds.

\begin{lemma}\label{lemma:equiv_stab2}
    Let $\te\in(0,1)$ and $\la\in\C$.
    Let $F^\pm\in \XHhol{2}$ and $G^\pm\in \XHhol{1}$, and define $f\in\XH{2}$ and $g\in \XH{1}$ via \eqref{eqn:recover_fstab}.
    
    If $(F^\pm,G^\pm)$ solves \eqref{system:stab_hol}, then $(f,g)$ satisfies \eqref{eqn:equiv_stab}.
\end{lemma}

\begin{proof}
    Since $F^\pm$ and $G^\pm$ satisfy \eqref{system:stab_hol}, we also know that $F^\pm|_\T$ and $G^\pm|_\T$ satisfy \eqref{system:stab_tor}. Therefore, we shall deduce \eqref{eqn:LStabthe} from \eqref{system:stab_tor}.

    First, define
    \begin{align*}
        {F|_\T^-}^\sharp(x) := e^{-ix} F^-|_\T(-x)
        \quad 
        \text{and}
        \quad 
        {G|_\T^-}^\sharp(x) := e^{-ix} G^-|_\T(-x).
    \end{align*}
    It follows from the second equation in \eqref{system:stab_tor2} and \eqref{eqn:stab:dxfsharp} that
    \begin{align*}
        (c+\twap)
        ( \pa_x {F|_\T^-}^\sharp 
            +
        i\te^+ {F|_\T^-}^\sharp
        )   
        &=
        (\la-i){F|_\T^-}^\sharp
        +
        {G|_\T^-}^\sharp
        -
        \mathcal{J}^{\te^-}(F^-|_\T,F^+|_\T)^\sharp.
    \end{align*}

    We now add the first equation of \eqref{system:stab_tor1}, since $f=F|_\T^+ + {F|_\T^-}^\sharp$, we obtain
    \begin{align}
        (c+\twap)
        ( \pa_x f
            +
        i\te^+ f
        ) 
        -
        \la f - g
        =
        i (F|_\T^+ - {F|_\T^-}^\sharp)
        +
        \mathcal{J}^{\te^+}(F|_\T^+,F|_\T^-)
        -
        \mathcal{J}^{\te^-}(F^-|_\T,F^+|_\T)^\sharp
        \label{eqn:equiv_stab:auxsharp}
    \end{align}

    Note that
    \begin{align*}
        \mathcal{J}^{\te^+}(F|_\T^+,F|_\T^-)(x)
        &=
        e^{-ix}\mathcal{J}^{\te^-}(F^-|_\T,F^+|_\T)(-x)
        =
        \mathcal{J}^{\te^-}(F^-|_\T,F^+|_\T)^\sharp(x).
        \\
        Hf(x) &= \sum_{k\in\Z} -i\sgn(k)\ha{f}_k 
        =
        \sum_{k\geq 1} -i\ha{f}_k +
        \sum_{k\leq -1} i\ha{f}_k 
        = - i F|_\T^+ + i \ha{f}_0
        + i {F|_\T^-}^\sharp.
    \end{align*}

    Therefore, \eqref{eqn:equiv_stab:auxsharp} is equivalent to \eqref{eqn:equiv_stab}.
\end{proof}

As we did in Section~\ref{sec:OpExis}, we temporarily ignore the Taylor coefficients $F^+_k$ and $F^-_k$ appearing in \eqref{def:Gcalthe} and replace them with arbitrary constants.

\begin{lemma}\label{lemma:stab:holo_implies}
    Let $\te\in(0,1)$ and $\la\in\C$.
    Let $a_k^+,a_k^-\in\C$ for $k=0,\ldots,N-1$. Let $F^+,F^-,G^+,G^-\in\XHholo$.
    
    Assume that the following identities hold
    \begin{align*}
        izP(z)\pa_z F^+(z)&=
        (i+\la)F^+(z)-i\te^+ P(z) F^+(z)+G^+(z)+
        \sum_{k=1}^N
        \big(
            a^+_{k-1}
            \mathcal{G}_{k,\te^-}^-(z)
            +
            a^-_{k-1}
            \mathcal{G}_{k,\te^+}^+(z)
        \big)\\
        %%%%%%%%%%%
        izP(z)\pa_z F^-(z)&=
        (i-\la)F^-(z)-i\te^- P(z) F^-(z)-G^-(z)+
        \sum_{k=1}^N
        \big(
            a^-_{k-1}
            \mathcal{G}_{k,\te^+}^-(z)
            +
            a^+_{k-1}
            \mathcal{G}_{k,\te^-}^+(z)
        \big).
    \end{align*}
    
    Then, for every $k=0,\ldots,N-1$, we have
    \begin{align*}
        F_k^+:=\frac{1}{k!}
        \left.\frac{d^k}{dz^k}\right|_{z=0}
        F^+=a_k^+,\qquad
        F_k^-:=\frac{1}{k!}
        \left.\frac{d^k}{dz^k}\right|_{z=0}
        F^-=a_k^-.
    \end{align*}
    
\end{lemma}

\begin{proof}
    We prove the statement for $F^+$, since the argument for $F^-$ is identical.

    Recalling the definition of $\mathcal{G}_{k,\te}^\pm$, the first equation is
    \begin{align*}
        izP(z)\pa_z F^+(z)&=
        (i+\la)F^+(z)-i\te^+ P(z) F^+(z)+G^+(z)
        \\
        &\qquad+\frac{i}{2}
        \sum_{j=1}^N\twap_j
        \sum_{m=1}^j
        (m-\te^-) z^{-(1+j-m)}a^+_{m-1}+
        (m-\te^+) z^{j-m}a^-_{m-1},\\
    \end{align*}
    
    Multiply by $z^N$:
    \begin{align}\label{eqn:stab:mutpzN}
        i z^{N+1}P(z) \pa_z F^+(z)
        &= (i+\la) z^N F^+(z)
        - i\te^+ z^N P(z)F^+(z)
        + z^N G^+(z) \\
        &\qquad + \frac{i}{2}
        \sum_{j=1}^N\twap_j
        \sum_{m=1}^j
        (m-\te^-) z^{N-(1+j-m)}a^+_{m-1}
        + \mathcal{O}(z^N).\nonumber
    \end{align}

    We now compare coefficients of $z^k$ for $k=0,\ldots,N-1$. Since $z^N G^+(z)=\mathcal O(z^N)$, and $z^N F^+(z)=\mathcal O(z^N)$, none of these terms contribute to the coefficient of $z^k$. From, \eqref{eqn:stab:mutpzN} implies that
    \begin{align*}
        \frac{i}{2}
        \sum_{j=1}^N 
        \sum_{l=0}^{j-1}
        \twap_j (l+\te^+) F_l^+ z^{N+l-j}  
        = \frac{i}{2}
        \sum_{j=1}^N\twap_j
        \sum_{m=1}^j
        (m-\te^-) z^{N-(1+j-m)}a^+_{m-1}
        + \mathcal{O}(z^N).
    \end{align*}

    Consequently, for each $k=0,\ldots,N-1$ we apply $\frac{1}{k!}\frac{d^{k}}{dz^{k}}$ to both sides yielding
    \begin{align*}
        \frac{i}{2}
        \sum_{j=1}^N 
        \sum_{l=0}^{j-1}
        \twap_j (l+\te^+)
        F_l^+ \delta_{k,N+l-j}
        &= \frac{i}{2}
        \sum_{j=1}^N\twap_j
        \sum_{l=0}^{j-1}
        (l+1-\te^-) a^+_{l} \delta_{k,N+l-j}\ \Rightarrow\\
        0&=\sum_{j=N-k}^N 
        \twap_j (k+j-N+\te^+)
        (F_{k+j-N}^+ - a^+_{k+j-N})\\
        &=\sum_{j=0}^{k} 
        \twap_{j+N-k} (j+\te^+)
        (F_{k}^+ - a^+_{k} ).
    \end{align*}
    
    Since $\twap_j\neq 0$ for all $j=1,\ldots,N$ (Lemma~\ref{lemma:twapprox_nonzero_coeffs}) and $\te^+\in(0,1)$, the above system is triangular and nonsingular, hence  $F_k^+=a_k^+$ for all $k=0,\ldots,N-1$.
\end{proof}

\subsubsection{Invertibility of $\LOpStabHol{\la}$}
\label{sec:stab:invert}
In order to study the operator $\LOpStabHol{\la}$, we first need a result similar to Lemma~\ref{lemma:Proots1}, but adapted to this new operator.

\begin{cond}\label{cond:Proots}
    We say that $\te\in(0,1)$ and $\la\in\C$ satisfy Condition~\ref{cond:Proots} if for every $\al\in Z$, we have
    \begin{align*}
        \lim_{z\to \al} i\frac{z-\al}{z}
        \Big(
        \frac{i\pm\la}{P(z)}-
        i\te^{\pm}\Big)\notin\Z.
    \end{align*}
    We recall that $Z$ denotes the set of the roots of $z^N P(z)$ lying inside the unit disk $\D$.
\end{cond}

We introduce some notation, analogous to Definition~\ref{def:exis_hol_notation}.

\begin{defi}\label{def:stab_hol_notation}
    Let $\te\in(0,1)$ and $\la\in\C$.
    Let 
    \begin{align*}
        Z_0:=\{0\}\cup Z,
    \end{align*}
    where $Z=\{z_j\}$ denotes the set of roots of $z^N P(z)$ lying inside the open unit disk, and let $\om_0$, $Q$, and the paths $\{\tau_j\}$ be as in Lemma~\ref{lemma:fuchs_criteria} with $\ti{Z}=Z_0$.  
    
    Given $G\in\XHholo\cup\{\mathcal{G}_{k,\te}^\pm:\ \theta\in(0,1),\ k=1,\ldots,N\}$, define
    \begin{align*}
        B_G(z):=-2i\frac{z^{N}}{\twap_N}G(z)
        \prod_{\al\in Z}(z-\al^{-1})^{-1}.
    \end{align*}
    
    We also define
    \begin{align*}
    A_{\la,\te}(z) &:= 2i\frac{z^{N}}{\twap_N}
            \big[i+\la - i \te P(z)\big]
            \prod_{\al\in Z}(z-\al^{-1})^{-1},
    \end{align*}

    For each $j=1,\ldots,N+1$, we define $J^+_{\tau_j}G$ as the unique continuous solution to \eqref{eqn:fuchs_crit_raqueta} along $\tau_j$ with initial condition $J^+_{\tau_j}G(0)=0$ and coefficients $A=A_{\la,\te^+}$ and $B=B_G$.

    Analogously, for each $j=1,\ldots,N+1$, we define $J^-_{\tau_j}G$ as the unique continuous solution of \eqref{eqn:fuchs_crit_raqueta} along $\tau_j$ with initial condition $J^-_{\tau_j}G(0)=0$ and coefficients $A=A_{-\la,\te^-}$ and $B=B_G$.

    Moreover, for each $j=1,\ldots,N+1$, we define $\Pi^+_{\tau_j}$ as the unique continuous solution of \eqref{eqn:fuchs_crit_raqueta} along $\tau_j$ with initial condition $\Pi^+_{\tau_j}(0)=1$
    and coefficients $A=A_{\la,\te^+}$ and $B=0$.

    Finally, for each $j=1,\ldots,N+1$, we define $\Pi^-_{\tau_j}$ as the unique continuous solution of
    \eqref{eqn:fuchs_crit_raqueta} along $\tau_j$ with initial condition $\Pi^-_{\tau_j}(0)=1$
    and coefficients $A=A_{-\la,\te^-}$ and $B=0$.

\end{defi}

The next step is to establish an invertibility criterion for $\LOpStabHol{\la}$. The following result follows the same strategy used to prove Proposition~\ref{prop:exis_hol}.

\begin{prop}\label{prop:stab_hol}
    Let $\te\in(0,1)$ and $\la\in\C$ and assume Condition~\ref{cond:Proots} holds, and set $b:=2\pi+2$.
    Using the notation from Definition~\ref{def:stab_hol_notation} and $\mathcal{G}^\pm_{k,\te}$ from~\eqref{def:Gk_stab}, define the $(2N+2)\times(2N+2)$ matrix $\MStabHol(\la)$ by
    \begin{align}\label{def:MStabHol}
        \big(\MStabHol(\la)\big)_{jk}=
            \begin{cases}
                J^+_{\tau_j} 
                \mathcal{G}^-_{k,\te^-}
                (b),
                & 1 \leq j \leq N+1,\quad
                 1\leq k\leq N\\
                %%%%%%%%%%%%%%%%%%%%%%%%%%%%%%%%%%%
                J^+_{\tau_j} 
                \mathcal{G}^+_{k-N,\te^+}
                (b),
                & 1 \leq j \leq N+1,\quad
                 N+1\leq k\leq 2N\\
                %%%%%%%%%%%%%%%%%%%%%%%%%%%%%%%%%%%
                \Pi^+_{\tau_j}(b)-1,
                & 1 \leq j \leq N+1,\quad
                 k=2N+1
                \\
                %%%%%%%%%%%%%%%%%%%%%%%%%%%%%%%%%%%
                J^-_{\tau_{j-N-1}} 
                \mathcal{G}^+_{k,\te^-}
                (b),
                & N+2 \leq j \leq 2N+2,\quad
                 1\leq k\leq N\\
                %%%%%%%%%%%%%%%%%%%%%%%%%%%%%%%%%%%
                J^-_{\tau_{j-N-1}} 
                \mathcal{G}^-_{k-N,\te^+}
                (b),
                & N+2 \leq j \leq 2N+2,\quad
                 N+1\leq k\leq 2N\\
                %%%%%%%%%%%%%%%%%%%%%%%%%%%%%%%%%%%
                \Pi^-_{\tau_{j-N-1}}(b)
                -
                1,
                & N+2 \leq j \leq 2N+2,\quad
                 k=2N+2\\
                 %%%%%%%%%%%%%%%%%%%%%%%%%%%%%%%%%%%
                 0,&\text{in other case}
            \end{cases}.
    \end{align}

    Then the matrix $\MStabHol(\la)$ is nonsingular if and only if, for every holomorphic pair $(G^+,G^-)\in\XHholo\oplus\XHholo$ of functions, continuous on $\overline{\D}$, there exists a unique pair $(F^+,F^-)\in\XHholo\oplus\XHholo$ of functions, with $\pa_z F^\pm$ continuous on $\overline{\D}$, such that
    \begin{align}\label{eqn:stab_finite_hol_main}
        \LOpStabHol{\la} (F^+,F^-) = (G^+,G^-).
    \end{align}

    Moreover, assume that the kernel of $\MStabHol$ has dimension $d$. Then there exist $d$ linearly independent pairs $\{(F_k^+,F_k^-)\}_{k=1}^d$ of holomorphic functions, continuous on $\overline{\D}$, such that \begin{align}\label{eqn:stab_finite_hol_main0}
        \LOpStabHol{\la} (F^+,F^-) = 0.
    \end{align}

\end{prop}

\begin{example}\label{example:matstabhol}
    With $N=2$, the matrix $\MStabHol$ would be of order $6$, given by
    \begin{align}
        \MStabHol=
            \begin{pmatrix}
            J^+_{\tau_1} 
            \mathcal{G}^-_{1,\te^-} 
            (b) & 
            J^+_{\tau_1} 
            \mathcal{G}^-_{2,\te^-}
            (b) &
            %%%%%%%%%%%%%% 
            J^+_{\tau_1} 
            \mathcal{G}^+_{1,\te^+}
            (b) &
            J^+_{\tau_1} 
            \mathcal{G}^+_{2,\te^+}
            (b) &
            %%%%%%%%%%%%%% 
            \Pi^+_{\tau_1}(b)-1 & 0\\
            %%%%%%%%%%%%%%%%%%%%%%%%%%%%%%%%%%%%%%
            J^+_{\tau_2} 
            \mathcal{G}^-_{1,\te^-} 
            (b) & 
            J^+_{\tau_2} 
            \mathcal{G}^-_{2,\te^-}
            (b) &
            %%%%%%%%%%%%%% 
            J^+_{\tau_2} 
            \mathcal{G}^+_{1,\te^+}
            (b) &
            J^+_{\tau_2} 
            \mathcal{G}^+_{2,\te^+}
            (b) &
            %%%%%%%%%%%%%% 
            \Pi^+_{\tau_2}(b)-1 & 0\\
            %%%%%%%%%%%%%%%%%%%%%%%%%%%%%%%%%%%%%%
            J^+_{\tau_3} 
            \mathcal{G}^-_{1,\te^-} 
            (b) & 
            J^+_{\tau_3} 
            \mathcal{G}^-_{2,\te^-}
            (b) &
            %%%%%%%%%%%%%% 
            J^+_{\tau_3} 
            \mathcal{G}^+_{1,\te^+}
            (b) &
            J^+_{\tau_3} 
            \mathcal{G}^+_{2,\te^+}
            (b) &
            %%%%%%%%%%%%%% 
            \Pi^+_{\tau_3}(b)-1 & 0\\
            %%%%%%%%%%%%%%%%%%%%%%%%%%%%%%%%%%%%%%
            %%%%%%%%%%%%%%%%%%%%%%%%%%%%%%%%%%%%%%
            J^-_{\tau_{1}} 
            \mathcal{G}^+_{1,\te^-}
            (b) & 
            J^-_{\tau_{1}} 
            \mathcal{G}^+_{2,\te^-}
            (b) &
            %%%%%%%%%%%%%% 
            J^-_{\tau_{1}} 
            \mathcal{G}^-_{1,\te^+}
            (b) &
            J^-_{\tau_{1}} 
            \mathcal{G}^-_{2,\te^+}
            (b) &
            %%%%%%%%%%%%%% 
            0 & \Pi^-_{\tau_1}(b)-1\\
            %%%%%%%%%%%%%%%%%%%%%%%%%%%%%%%%%%%%%%
            J^-_{\tau_{2}} 
            \mathcal{G}^+_{1,\te^-}
            (b) & 
            J^-_{\tau_{2}} 
            \mathcal{G}^+_{2,\te^-}
            (b) &
            %%%%%%%%%%%%%% 
            J^-_{\tau_{2}} 
            \mathcal{G}^-_{1,\te^+}
            (b) &
            J^-_{\tau_{2}} 
            \mathcal{G}^-_{2,\te^+}
            (b) &
            %%%%%%%%%%%%%% 
            0 & \Pi^-_{\tau_2}(b)-1\\
            %%%%%%%%%%%%%%%%%%%%%%%%%%%%%%%%%%%%%%
            J^-_{\tau_{3}} 
            \mathcal{G}^+_{1,\te^-}
            (b) & 
            J^-_{\tau_{3}} 
            \mathcal{G}^+_{2,\te^-}
            (b) &
            %%%%%%%%%%%%%% 
            J^-_{\tau_{3}} 
            \mathcal{G}^-_{1,\te^+}
            (b) &
            J^-_{\tau_{3}} 
            \mathcal{G}^-_{2,\te^+}
            (b) &
            %%%%%%%%%%%%%% 
            0 & \Pi^-_{\tau_3}(b)-1
        \end{pmatrix}.
    \end{align}
\end{example}

\begin{proof}
    Let $G^+,G^-:\D\to\C$ be holomorphic functions, continuous on $\overline{\D}$. We study equation \eqref{eqn:stab_finite_hol_main}, and look for holomorphic functions $F^+,F^-:\D\to\C$ whose derivative $\pa_z F$ is continuous on $\overline{\D}$.

    Multiplying \eqref{eqn:stab_finite_hol_main} by $z^N$, and introducing auxiliary unknowns $a^\pm_1,\ldots,a^\pm_N\in\C$, we obtain
    \begin{subequations}\label{eqn:stab_finite_holplumin}
    \begin{empheq}[left=\empheqlbrace]{align}
    i z^{N+1}P(z)\pa_z F^+(z)
    &=
        z^N\big(i+\la - i\te^+ P(z)\big)F^+(z)
        + z^N G^+(z)
        \nonumber\\
        &\quad+
        z^N \sum_{k=1}^N
        \big(
            a^+_{k} \mathcal{G}_{k,\te^-}^-(z)
            +
            a^-_{k}\mathcal{G}_{k,\te^+}^+(z)
        \big),
    \label{eqn:stab_finite_hol_plu}
    \\
    i z^{N+1}P(z)\pa_z F^-(z)
        &=
        z^N\big(i-\la - i\te^- P(z)\big)F^-(z)
        - z^N G^-(z)
        \nonumber\\
        &\quad+
        z^N
        \sum_{k=1}^N
        \big(
            a^-_{k}\mathcal{G}_{k,\te^+}^-(z)
            +
            a^+_{k}\mathcal{G}_{k,\te^-}^+(z)
        \big).
    \label{eqn:stab_finite_hol_min}
    \end{empheq}
    \end{subequations}

    Note that the singular points of \eqref{eqn:stab_finite_hol_plu} and \eqref{eqn:stab_finite_hol_min} correspond to $Z_0$.

    The equation \eqref{eqn:stab_finite_hol_plu} reduces to a linear first order equation of the form \eqref{eqn:fuchs_crit_disk}, where $A=A_{\la,\te^+}$ and
    \begin{align*}
        B(z) :=
        B_{G^+}(z) + \sum_{k=1}^N 
        a^+_k
        B_{\mathcal{G}^-_{k,\te^-}}(z)+
        a^-_k
        B_{\mathcal{G}^+_{k,\te^+}}(z)
        .
    \end{align*}

    In addition, \eqref{eqn:stab_finite_hol_min} reduces to another linear first order equation of the form \eqref{eqn:fuchs_crit_disk}, where $A=A_{-\la,\te^-}$ and
    \begin{align*}
        B(z) :=
        -B_{G^-}(z) + \sum_{k=1}^N 
        a^+_k
        B_{\mathcal{G}^+_{k,\te^-}}(z)+
        a^-_k
        B_{\mathcal{G}^-_{k,\te^+}}(z)
        .
    \end{align*}

    Since Condition~\ref{cond:Proots} holds, it follows that \eqref{eqn:cond_proots_general} also holds. Hence, we may apply Lemma~\ref{lemma:fuchs_criteria} to both equations.

    For each fixed $j=1,\ldots,N+1$, solving \eqref{eqn:stab_finite_holplumin} along $\tau_j$ yields unique continuous solutions $\phi_j^+,\phi_j^-:[0,2\pi]\to\C$. These solutions can be expressed as 
    \begin{align*}
    \phi^+_j(x)
    &=
        F^+(\om_0)
        \Pi^+_{\tau_j}(x)
    +
        J^+_{\tau_j}
        G^+(x) 
    +
    \sum_{k=1}^N 
        a^+_k
        J^+_{\tau_j}\mathcal{G}^-_{k,\te^-}(x)+
        a^-_k
        J^+_{\tau_j}\mathcal{G}^+_{k,\te^+}(x),
    \\
    \phi^-_j(x)&=
        F^-(\om_0)
        \Pi^-_{\tau_j}(x)
    -
        J^-_{\tau_j}
        G^-(x)
    +
    \sum_{k=1}^N 
        a^+_k
        J^-_{\tau_j}\mathcal{G}^+_{k,\te^-}(x)+
        a^-_k
        J^-_{\tau_j}\mathcal{G}^-_{k,\te^+}(x).
    \end{align*}

    Lemma~\ref{lemma:fuchs_criteria} states that $F^+$ and $F^-$ are holomorphic functions on $\D$, continuous on $\overline{\D}$, and satisfying \eqref{eqn:stab_finite_holplumin}, if and only if each function $\phi^\pm_j$ satisfies the condition
    \begin{align}\label{eqn:stab_finite_hol_close}
        \phi^\pm_j(b)=\phi^\pm_j(0),
        \qquad j=1,\ldots,N.
    \end{align}

    Therefore, $F^+$ and $F^-$ are holomorphic if and only if
        \begin{align*}
        0
        &=
            F^+(\om_0)\big(
            \Pi^+_{\tau_j}(b)-1
            \big)
        +
            J^+_{\tau_j}
            G^+(b)
        +
        \sum_{k=1}^N 
            a^+_k
            J^+_{\tau_j}\mathcal{G}^-_{k,\te^-}(b)+
            a^-_k
            J^+_{\tau_j}\mathcal{G}^+_{k,\te^+}(b),
        \\
        0&=
            F^-(\om_0)\big(
            \Pi^-_{\tau_j}(b)-1
            \big)
        -
            J^-_{\tau_j}
            G^-(b)
        +
        \sum_{k=1}^N 
            a^+_k
            J^-_{\tau_j}\mathcal{G}^+_{k,\te^-}(b)+
            a^-_k
            J^-_{\tau_j}\mathcal{G}^-_{k,\te^+}(b).
    \end{align*}
    Writing these conditions gives $2N+2$ affine linear equations in the unknowns $a^\pm_1,\ldots,a^\pm,F^\pm(\om_0)$:
    \begin{align*}
        \MStabHol
        \begin{pmatrix}
            a^+_1\\\vdots\\a^+_N\\
            a^-_1\\\vdots\\a^-_N\\
            F^+(\om_0)\\F^-(\om_0)
        \end{pmatrix}+
        \begin{pmatrix}
            J_{\tau_1}^+ G^+(b)\\
            \vdots\\
            J_{\tau_{N+1}}^+ G^+(b)\\
            -J_{\tau_1}^- G^-(b)\\
            \vdots\\
            -J_{\tau_{N+1}}^- G^-(b)
        \end{pmatrix}
        =0.
    \end{align*}

    If $\MStabHol$ is nonsingular, then the system admits a unique solution. Consequently, there exists a unique pair of holomorphic functions $F^+,F^-:\D\to\C$ such that $\pa_z F^+,\pa_z F^-\in C(\overline{\D})$, and satisfy \eqref{eqn:stab_finite_holplumin}. By Lemma~\ref{lemma:stab:holo_implies}, the coefficients satisfy
    \begin{align*}
        a^\pm_k = F^\pm_{k-1} \qquad \text{for all } k,
    \end{align*}
    and therefore $(F^+,F^-)$ is the unique solution of~\eqref{eqn:stab_finite_hol_main}.
    
    Conversely, if for every admissible pair $(G^+,G^-)\in\XHholo\oplus\XHholo$ there exists a unique $(F^+,F^-)\in\XHholo\oplus\XHholo$ of \eqref{eqn:stab_finite_hol_main}, by setting $(G^+,G^-)=0$, we see that then the homogeneous system associated with $\MStabHol$ admits only the trivial solution, and hence $\MStabHol$ is nonsingular.

    Finally, assume that there exist $d$ independent terms in the kernel of $\MStabHol$: $\{v_k\}_{k=1}^d$. For each $k=1,\ldots,d$, there exists a pair of functions $(F_k^+,F_k^-)\in\XHholo\oplus\XHholo$, whose derivative extends continuously to $\D$, and  satisfying \eqref{eqn:stab_finite_hol_main0}. 
    
    Suppose that these pairs are not independent, so there exists $c\in\C^{2N+2}$, with $c\neq 0$ such that 
    \begin{align*}
        \sum_{k=1}^d c_k (F_k^+,F_k^-) = 0.
    \end{align*}
    
    Computing the difference between the evaluations at $b=2\pi+2$ and $0$, we see that 
    \begin{align*}
        \sum_{k=1}^d c_k v_k = 0,
    \end{align*}
    so we have a contradiction.
\end{proof}

Define
\begin{align}\label{def:Piplusminu}
    \Pi^+(x):=
    e^{(i+\la)\int_0^x\beta(y)\dy-i\te^+ x}
    \qquad
    \text{and}
    \qquad 
    \Pi^-(x):=
    e^{(i-\la)\int_0^x\beta(y)\dy-i\te^- x},
\end{align}
recall that $\beta$ is defined by \eqref{def:beta}. Moreover, define the formal operators
\begin{align}\label{def:Jplusminu}
    J^+ g(x):=\Pi^+(x)\int_0^x\frac{\beta(y)}{\Pi^+(y)}g(y)\dy
    \qquad
    \text{and}
    \qquad 
    J^- g(x):=\Pi^-(x)\int_0^x\frac{\beta(y)}{\Pi^-(y)}g(y)\dy.
\end{align}

\begin{rem}
    The functions $\Pi^+$ and $\Pi^-$ depend on $\la$ and $\te$, and the same holds for the operators $J^+$ and $J^-$.
\end{rem}

\begin{lemma}\label{lemma:stab_tor}
    Let $\Pi^\pm$ and $J^\pm$ be defined in
    \eqref{def:Piplusminu} and \eqref{def:Jplusminu}, respectively.

    For every $k=1,\ldots,N$, define
    \begin{align}\label{def:gkplusminu}
        g_{k,\te}^-(x):=
        (k-\te)\frac{i}{2}
        \sum\limits_{m=0}^{N-k}    
        \twap_{m+k}
         e^{-i(1+m)x}
        \qquad
        \text{and}
        \qquad 
        g_{k,\te}^+(x):=
        (k-\te)\frac{i}{2}
        \sum\limits_{m=0}^{N-k}    
        \twap_{m+k}
         e^{imx}.
    \end{align}
    
    For each $k=1,\ldots,N$, set
    \begin{align}
        \left.\begin{aligned}
            h^{-J^+}_{k}&:=
                J^+ g^-_{k,\te^-},&\qquad
            h^{+J^+}_{k}&:=
                J^+ g^+_{k,\te^+},\\
            h^{-J^-}_{k}&:=
                J^- g^-_{k,\te^+},&\qquad
            h^{+J^-}_{k}&:=
                J^- g^+_{k,\te^-}.
        \end{aligned}\right\}
        \label{def:hkplusminu}
    \end{align}

    Finally, define the square matrix $\MStabTor(\la)$ of order $2N+2$ by
    \begin{align}
        \big(\MStabTor(\la)\big)_{jk}=
            \begin{cases}
                (\ha{h^{-J^+}_{k}})_{j-2} - \delta_{j-1,k},
                & 1 \leq j \leq N+1,\quad
                 1\leq k\leq N\\
                %%%%%%%%%%%%%%%%%%%%%%%%%%%%%%%%%%%
                (\ha{h^{+J^+}_{k-N}})_{j-2} ,
                & 1 \leq j \leq N+1,\quad
                 N+1\leq k\leq 2N\\
                %%%%%%%%%%%%%%%%%%%%%%%%%%%%%%%%%%%
                (\ha{\Pi^+})_{j-2},
                & 1 \leq j \leq N+1,\quad
                 k=2N+1
                \\
                %%%%%%%%%%%%%%%%%%%%%%%%%%%%%%%%%%%
                (\ha{h^{+J^-}_{k}})_{j-N-3},
                & N+2 \leq j \leq 2N+2,\quad
                 1\leq k\leq N\\
                %%%%%%%%%%%%%%%%%%%%%%%%%%%%%%%%%%%
                (\ha{h^{-J^-}_{k-N}})_{j-N-3} - \delta_{j-2,k},
                & N+2 \leq j \leq 2N+2,\quad
                 N+1\leq k\leq 2N
                \\
                %%%%%%%%%%%%%%%%%%%%%%%%%%%%%%%%%%%
                (\ha{\Pi^-})_{j-N-3},
                & N+2 \leq j \leq 2N+2,\quad
                 k=2N+2
                 \\
                %%%%%%%%%%%%%%%%%%%%%%%%%%%%%%%%%%%
                0,&\text{in other case}
            \end{cases}.\label{def:MStabTor}
    \end{align}
    Then
    \begin{align*}
        \dim\big(\ker\big(\MStabHol(\la)\big)\big)
        \leq 
        \dim\big(\ker\big(\MStabTor(\la)\big)\big).
    \end{align*}
    
\end{lemma}
\begin{example}\label{example:matstabtor}
    With $N=2$ the matrix $\MStabTor$ would be of order $6$. In this case,
    \begin{align}
        \MStabTor=
            \begin{pmatrix}
            (\ha{h_1^{-J^+}})_{-1}&(\ha{h_2^{-J^+}})_{-1}&
            (\ha{h_1^{+J^+}})_{-1}&(\ha{h_2^{+J^+}})_{-1}&
            (\ha{\Pi^+})_{-1}&0\\
            %%%%%%%%%%%%%%%%%%%%%%%%%%%%%%%%%%%%%%
            (\ha{h_1^{-J^+}})_{0}-1&(\ha{h_2^{-J^+}})_{0}&
            (\ha{h_1^{+J^+}})_{0}&(\ha{h_2^{+J^+}})_{0}&
            (\ha{\Pi^+})_{0}&0\\
            %%%%%%%%%%%%%%%%%%%%%%%%%%%%%%%%%%%%%%
            (\ha{h_1^{-J^+}})_{1}&(\ha{h_2^{-J^+}})_{1}-1&
            (\ha{h_1^{+J^+}})_{1}&(\ha{h_2^{+J^+}})_{1}&
            (\ha{\Pi^+})_{1}&0\\
            %%%%%%%%%%%%%%%%%%%%%%%%%%%%%%%%%%%%%%
            (\ha{h_1^{+J^-}})_{-1}&(\ha{h_2^{+J^-}})_{-1}&
            (\ha{h_1^{-J^-}})_{-1}&(\ha{h_2^{-J^-}})_{-1}&
            0&(\ha{\Pi^-})_{-1}\\
            %%%%%%%%%%%%%%%%%%%%%%%%%%%%%%%%%%%%%%
            (\ha{h_1^{+J^-}})_{0}&(\ha{h_2^{+J^-}})_{0}&
            (\ha{h_1^{-J^-}})_{0}-1&(\ha{h_2^{-J^-}})_{0}&
            0&(\ha{\Pi^-})_{0}\\
            %%%%%%%%%%%%%%%%%%%%%%%%%%%%%%%%%%%%%%
            (\ha{h_1^{+J^-}})_{1}&(\ha{h_2^{+J^-}})_{1}&
            (\ha{h_1^{-J^-}})_{1}&(\ha{h_2^{-J^-}})_{1}-1&
            0&(\ha{\Pi^-})_{1}\\
        \end{pmatrix}.
    \end{align}
\end{example}
\begin{proof}
    The argument is essentially the same as in the proof of Lemma~\ref{lemma:exis_tor}. Let
    \begin{align*}
        d:=\dim\big(\ker(\MStabHol(\la))\big).
    \end{align*}
    Fix a basis of $\ker(\MStabHol(\la))$. By Proposition~\ref{prop:stab_hol}, there exist $d$ linearly independent pairs $(F^+,F^-)$ such that $\LOpStabHol{\la}(F^+,F^-)=0$.
    
    Fix such a pair $(F^+,F^-)$. Then \eqref{system:stab_tor} is satisfied by $F^+|_\T$, $F^-|_\T$ with $g^+=g^-=0$. Applying Duhamel's formula to both equations yields
    \begin{subequations}\label{eqn:recoFstabTor}
        \begin{align}
            F^+|_\T(x) &= \ha{F^+|_\T}(0) \Pi^+(x)
            +\sum_{k=1}^N (\ha{F^+|_\T})_{k-1} h_{k}^{-J^+}(x)
            +\sum_{k=1}^N (\ha{F^-|_\T})_{k-1} h_{k}^{+J^+}(x),
            \label{eqn:recoFstabTor:plus}
            \\
            F^-|_\T(x) &= \ha{F^-|_\T}(0) \Pi^-(x)
            +\sum_{k=1}^N (\ha{F^+|_\T})_{k-1} h_{k}^{+J^-}(x)
            +\sum_{k=1}^N (\ha{F^-|_\T})_{k-1} h_{k}^{-J^-}(x).
            \label{eqn:recoFstabTor:minu}
        \end{align}
    \end{subequations}
    Projecting onto $e^{ijx}$ for $j=-1,\ldots,N-1$ and collecting the unknown Fourier coefficients, we obtain a linear system of the form
    \begin{align*}
        \MStabTor {\bf c}(F^+,F^-)=0,
    \end{align*}
    where 
    \begin{align}\label{def:cplumin}
        \mathbf{c}(F^+,F^-)
        :=
        \begin{pmatrix}
            (\ha{F^+|_\T})_0&\!
            \cdots&\!
            (\ha{F^+|_\T})_{N-1}&\!
            (\ha{F^-|_\T})_0&\!
            \cdots&\!
            (\ha{F^-|_\T})_{N-1}&\!
            F^+|_\T(0)&\!
            F^-|_\T(0)
        \end{pmatrix}^t.
    \end{align}
    
    Since the $d$ pairs $(F^+,F^-)$ are linearly independent, \eqref{eqn:recoFstabTor} imply that the vectors $\mathbf{c}(F^+,F^-)$ are linearly independent as well. Hence $\dim\ker(\MStabTor)\geq  d$.
\end{proof}

\begin{corollary}\label{coro:stab_tor}
    Let $\te\in(0,1)$ and $\la\in\C$. Assume Condition~\ref{cond:Proots} holds and
    \begin{align*}
        \det\big(\MStabTor(\la)\big)\neq 0.
    \end{align*}

    Then, for every holomorphic pair of functions $(G^+,G^-)\in\XHholo\oplus\XHholo$, that are continuous on $\overline{\D}$, there exists a unique pair of holomorphic functions $(F^+,F^-)\in\XHholo\oplus\XHholo$ with $\pa_z F^\pm$ continuous on $\overline{\D}$, such that \eqref{eqn:stab_finite_hol_main} holds.
\end{corollary}

\begin{proof}
    Since the dimension of $\ker\big(\MStabTor(\la)\big)$ is $0$, the same holds for $\ker\big(\MStabHol(\la)\big)$ and Proposition~\ref{prop:stab_hol} applies.
\end{proof}

\subsubsection{Restricted invertibility of $\LOpStabHol{\laap}$}

The main idea of this section is to construct a suitable \textit{pseudoinverse} that allows us to invert $\LOpStabHol{\laap}$ controlling the norm.

\begin{lemma}\label{lemma:detMstabholAn}
    The dependence of $\det(\MStabTor)$ on $\la$ is analytic.
\end{lemma}

\begin{proof}
    We show that each entry of $\MStabTor$, defined in \eqref{def:MStabTor}, is analytic in $\la$. Let $g:\T\to\C$ be continuous, let $C_0\in\C$, and consider
    \begin{align*}
    f(x) := \Pi^+(x) C_0 + J^+ g (x),
    \end{align*}
    where $\Pi^+$ and $J^+$ are defined in \eqref{def:Piplusminu} and \eqref{def:Jplusminu}, respectively.
    
    The function $\Pi^+$ is analytic in $\la$ and never vanishes, since it is an exponential function. As for $J^+ g$, observe that it is the product of $\Pi^+$ with an antiderivative. Define
    \begin{align*}
        \Phi(x,\la):=\int_0^x
        \frac{\beta(y)}{\Pi^+(y)}g(y)\dy.
    \end{align*}
    Then, for every $x\in[a,b]$, the map $\Phi(x,\cdot)$ is holomorphic on $\C$. Indeed, for any triangular contour $\partial\Delta\subset \C$, continuity and boundedness on $[0,x]\times\partial\Delta$ allow us to interchange the order of integration, yielding
    \begin{align*}
        \int_{\partial\Delta}\Phi(x,\la) \,d\la
        =
        \int_0^x
        \int_{\partial\Delta}
        \frac{\beta(y)}{\Pi^+(y)}g(y) \,d\la \dy
        =0.
    \end{align*}
    Hence, by Morera's theorem, $\Phi(x,\cdot)$ is holomorphic.
    
    Consequently, the function $f$ depends analytically on $\la$. Applying Morera's theorem once more, we obtain
    \begin{align*}
        \int_{\partial\Delta}\ha{f}_k \,d\la
        =
        \frac{1}{2\pi}
        \int_{\partial\Delta}
            \intpi e^{-ikx}
            f(x) \dx \,d\la
        =
        \frac{1}{2\pi}
        \intpi
        e^{-ikx}
        \int_{\partial\Delta}
            f(x) \,d\la \dx
        =0.
    \end{align*}
    Therefore, $\ha{f}_k$ is holomorphic. 
\end{proof}

\begin{lemma}\label{lemma:stab_is_regu}
    Let $\te=\frac{1}{3}$. Then there exists $\la\in\C$ such that
    \begin{align*}
        \det\big(\MStabTor(\la)\big)\neq 0.
    \end{align*}
\end{lemma}

\proofcompassi

\begin{corollary}\label{coro:isozeros}
    The function $\det(\MStabTor):\C\to\C$ has only isolated zeros for $\te=\frac{1}{3}$.
\end{corollary}
\begin{proof} 
    By Lemma~\ref{lemma:detMstabholAn}, the function $\det(\MStabTor)$ is holomorphic. If it had a non-isolated zero, then by the Identity Theorem for holomorphic functions it would vanish identically on $\C$. However, there exists $\la\in\C$ such that $\det\big(\MStabTor(\la)\big)\neq 0$, a contradiction.
\end{proof}

\begin{defi}\label{def:Psigma1}
    Let $\te\in(0,1)$ and $\la\in\C$. Let $\sigma_1$ be the smallest singular value of $\MStabTor(\la)$ and assume that $\sigma_1$ has multiplicity one. Let $u_1,v_1\in\C^{2N+2}$ be left and right singular vectors corresponding to $\sigma_1$, normalized so that $\norm{u_1}=\norm{v_1}=1$. Then
    \begin{align*}
        \MStabTor v_1=\sigma_1 u_1,
        \qquad
        \MStabTor^\ast u_1=\sigma_1 v_1.
    \end{align*}
    Let $\mathcal{U}_1:=\spa(u_1)$, $\mathcal{V}_1:=\spa(v_1)$ and 
    let $\PU:\C^{2N+2}\to\mathcal{U}_1$, 
    $\PV:\C^{2N+2}\to\mathcal{V}_1$ be the orthogonal projections. In particular, it holds that
    \begin{align*}
        \PU = u_1u_1^\ast,\qquad
        \PV = v_1v_1^\ast
    \end{align*}

    For any pair $(G^+,G^-)\in\XHhol{0}\oplus\XHhol{0}$, define the vector ${\bf b}(G^+,G^-)\in\C^{2N+2}$ by
    \begin{align}\label{def:bplumin}
        {\bf b}(G^+,G^-):=
        \begin{pmatrix}
            (\ha{J^+ G^+|_\T})_{-1},\! &\!
            \cdots\!&\!,
            (\ha{J^+ G^+|_\T})_{N-1},\!&\!
            -(\ha{J^- G^-|_\T})_{-1},\!&\!
            \cdots\!&\!,
            -(\ha{J^- G^-|_\T})_{N-1}
        \end{pmatrix}^t.
    \end{align}

    Let $(G_{u_1}^+,G_{u_1}^-)\in\XHholo\oplus\XHholo$ be such that
    \begin{align}\label{gauge}
        u_1^\ast {\bf b}(G_{u_1}^+,G_{u_1}^-)\neq0.
    \end{align}
    
    We define the operator
    \begin{align*}
        \Pcal_\la[G_{u_1}^+,G_{u_1}^-]:
        \XHholo\oplus\XHholo
        \longrightarrow
        \XHholo\oplus\XHholo
    \end{align*}
    by
    \begin{align*}
        \Pcal_\la[G_{u_1}^+,G_{u_1}^-](G^+,G^-)
        :=
        (G^+,G^-)
        -
        \frac{u_1^\ast {\bf b}(G^+,G^-)}{u_1^\ast {\bf b}(G_{u_1}^+,G_{u_1}^-)}
        (G_{u_1}^+,G_{u_1}^-).
    \end{align*}
\end{defi}

\begin{rem}
    We write $\Pcal_\la[G_{u_1}^+,G_{u_1}^-]$ to emphasize the dependence on $\la$. In fact, both $u_1$ and ${\bf b}$ also depend on $\la$.
\end{rem}

\begin{rem}
    Let $\la\in\C$ and let $\sigma_1(\la)$ denote the smallest singular value of $\MStabTor(\la)$, and assume that $\sigma_1(\la)$ has multiplicity $1$.
    The quantity $\sigma_1(\la)$ may be very small. In fact, we do not prove that it is nonzero for $\la=\laap$.
\end{rem}

\begin{lemma}\label{lemma:svd:stable}
    Let $\la_0\in\C$ and let $\overline{\sigma}_1,\underline{\sigma}_2>0$, with $\overline{\sigma}_1<\underline{\sigma}_2$. Let $V^{\la_0}$ be a neighborhood of $\la_0$ and let $M: V^{\la_0}\to \C^{r\times r}$ be a continuous map.

    Let 
    \begin{align*}
        \sigma_1(M(\la)) \leq  \dots \leq  \sigma_r(M(\la))
    \end{align*}
    denote the singular values of $M(\la)$, counted with multiplicity. Assume that
    \begin{align*}
        \sigma_1(M(\la)) \leq  \overline{\sigma}_1,
        \qquad
        \sigma_2(M(\la)) \geq  \underline{\sigma}_2
    \end{align*}
    for all $\la \in V^{\la_0}$.

    Then the maps
    \begin{align*}
        \la\mapsto \PUl{\la},\qquad \la\mapsto \PVl{\la}
    \end{align*}
    are continuous on $V^{\la_0}$.
    
    Moreover, the map defined by
    \begin{align}\label{def:Minvreg}
        M^\dagger_{\rm reg}(\la):=
        M(\la)^\ast
        \bigl(M(\la)M(\la)^\ast+\PUl{\la}\bigr)^{-1}
        \bigl(I-\PUl{\la}\bigr)
    \end{align}
    is well defined and continuous on $V^{\la_0}$ and satisfies, for all $\la\in V^{\la_0}$,
    \begin{align}\label{eqn:pseudoinvproj}
        \PV M^\dagger_{\rm reg} &= 0,\\
        \label{eqn:pseudoinvinve}
        MM^\dagger_{\rm reg}&=I-\PU.
    \end{align}

    Furthermore, the following estimate holds
    \begin{align}\label{eqn:pseudoinvesti}
        \norm{M^\dagger_{\rm reg}(\la) u}_{\ell^2(\C^r)}\leq 
        \frac{1}{\underline{\sigma}_2}\norm{u}_{\ell^2(\C^r)}\quad 
        \text{for all } \la\in V^{\la_0}.
    \end{align}
    
\end{lemma}

\begin{proof}
    The singular values of $M$ may be computed as the nonnegative square roots of the eigenvalues of $M^\ast M$. By Theorem 5.5 from \cite[Chapter~2, \S5.1]{kato1980perturbation},  we know that the rank-one spectral projectors $\la\mapsto \PUl{\la}$ and $\la\mapsto \PVl{\la}$ are continuous on $V^{\la_0}$.

    In addition, note that the matrix given by
    \begin{align*}
        M M^\ast + \PU
    \end{align*}
    is also continuous on $V^{\la_0}$.

    For each fixed $\la$, take $M(\la)=U\Sigma V^\ast$. Then $\PU=u_1u_1^\ast$ and $\PV=v_1v_1^\ast$ for some unit vectors $u_1,v_1\in\C^r$.

    Therefore,
    \begin{align*}
        \det\big(MM^\ast + \PU\big) = 
        \det\big(\Sigma\Sigma^\ast +E_{11}\big) \geq
        \prod_{k=2}^r \sigma_k^2 \geq \underline{\sigma}_2^{2(r-1)} > 0,
    \end{align*}
    where $E_{11}$ denotes the square matrix of order $r$ consisting of all zeros except for the $(1,1)$ entry, where it has a $1$. Hence, $M M^\ast + \PU$ is invertible for all $\la\in D_{\delta}(z_0)$, and its inverse depends continuously on $\la$. 

    Consequently, the function $M^\dagger_{\rm reg}(\la)$ from \eqref{def:Minvreg} is well defined and continuous on $V^{\la_0}$.

    We show that \eqref{eqn:pseudoinvproj} and \eqref{eqn:pseudoinvinve} hold. The idea is to conjugate so that the middle factors are diagonal. Fix $\la\in V^{\la_0}$. Then we obtain
    \begin{align*}
        \PV M^\dagger_{\rm reg}&= 
        \PV
        M^\ast
        \big(MM^\ast+\PU\big)^{-1}
        \big(I-\PU\big)\\
        &= 
        V E_{11}
        \Sigma^\ast
        \big(\Sigma\Sigma^\ast+E_{11}\big)^{-1}
        \big(I-E_{11}\big) U^\ast=0.
        \\ \\
        M M^\dagger_{\rm reg}&=
        M M^\ast
        \bigl(MM^\ast+\PU\bigr)^{-1}
        \bigl(I-\PU\bigr)\\
        &=
        U \Sigma \Sigma^\ast        
        \bigl(\Sigma \Sigma^\ast+E_{11}\bigr)^{-1}
        \bigl(I-E_{11}\bigr) U^\ast\\
        &=
        U  
        \bigl(I-E_{11}\bigr) U^\ast=I-\PU.
    \end{align*}
    
    Finally, \eqref{eqn:pseudoinvesti} follows by applying the same conjugation argument for fixed $\la$
    \begin{align*}
        \norm{
        M^\dagger_{\rm reg} u
        }_{\ell^2(\C^r)}^2
        &=
        \norm{M^\ast
        \big(MM^\ast+\PU\big)^{-1}
        \big(I-\PU\big) u}^2\\
        &=
        \norm{V 
        \Sigma^\ast
        \big(\Sigma\Sigma^\ast+E_{11}\big)^{-1}
        \big(I-E_{11}\big) U^\ast
        u}^2\\
        &=
        \norm{
        \Sigma^\ast
        \big(\Sigma\Sigma^\ast+E_{11}\big)^{-1}
        \big(I-E_{11}\big)
        U^\ast u}^2\\
        &\leq 
        \sum_{k=2}^r
        \frac{1}{\sigma_k^2}
        |(U^\ast u)_k|^2\leq 
        \frac{1}{\sigma_2^2}
        \norm{U^\ast u}_{\ell^2(\C^r)}^2
        \leq \frac{1}{\underline{\sigma}_2^2}
        \norm{u}_{\ell^2(\C^r)}^2.
    \end{align*}
\end{proof}

\begin{lemma}
\label{lemma:stab_singval}
    Let $\te=\frac{1}{3}$ and 
    let $\sigma_1 \leq  \dots \leq  \sigma_{2N+2}$ denote the singular values of $\MStabTor(\laap)$, counted with multiplicity. Then it holds that
    \begin{align*}
        \sigma_1^2 < \overline{\sigma}_1,
        \qquad
        \sigma_2^2 > \underline{\sigma}_2,
    \end{align*}
    where
    \begin{align*}
        \overline{\sigma}_1 :=
        \getvar{svd_stab1},
        \qquad
        \underline{\sigma}_2 :=
        \getvar{svd_stab2}.
    \end{align*}
\end{lemma}
\proofcompassi

Recall the approximate eigenvector $\Eigveap$ of $\LStabthe$, see \eqref{residual_stability}. We choose $(G_{u_1}^+,G_{u_1}^-)$ in Definition~\ref{def:Psigma1} to be $(F^+,F^-)$, where $(F^+,F^-)$ are the functions introduced in Definition~\ref{def:FplusFminu} with $f=\Eigveap$.

The following lemma shows that this choice satisfies \eqref{gauge} for every $\la$ sufficiently close to $\laap$. Moreover, we prove that the operator $\Pcal_\la$ is continuous with respect to~$\la$.

\begin{lemma}\label{lemma:prod_fv_u}
    Let $\te=\frac{1}{3}$ and let $(\Fvrmap^+,\Fvrmap^-)$ be the pair from Definition~\ref{def:FplusFminu} corresponding to $f=\Eigveap$, defined in \eqref{residual_stability}.
    It holds that
    \begin{align*}
        \big|u_1(\laap)^\ast{\bf b}(\Fvrmap^+,\Fvrmap^-)\big| > \getvar{Jfvu1},
    \end{align*}
    $\bf b$ is computed at $\la=\laap$.
\end{lemma}
\proofcompassi

\begin{corollary}\label{coro:P_cont_la}
    Let $\te=\frac{1}{3}$, let $(\Fvrmap^+,\Fvrmap^-)$ as in Lemma~\ref{lemma:prod_fv_u} and let $\Pcal_\la[\Fvrmap^+,\Fvrmap^-]$ be the operator from Definition~\ref{def:Psigma1}.
    
    Then there is a choice for $u_{1}$ in terms of $\la$ that makes the operator $\Pcal_\la[\Fvrmap^+,\Fvrmap^-]:\XHholo \oplus\XHholo\to\XHholo \oplus\XHholo$ continuous in a neighborhood $V^{\laap}$ of $\laap$. In particular, it holds that
    \begin{align}\label{gauge_image}
        u_1(\la)^\ast{\bf b}
        \big(\Pcal_\la[\Fvrmap^+,\Fvrmap^-](G^+,G^-)\big)=0
    \end{align}
     for all $(G^+,G^-)\in \XHholo \oplus\XHholo$ and $\la\in V^\laap$. 

\end{corollary}
\begin{proof}
    We know from Lemma~\ref{lemma:svd:stable} that the projection map $\PU$ is continuous at $\laap$. Notice that the singular vector $u_1(\la)$, even if normalized, is only unique up to multiplication by a complex scalar of unit modulus. Let us kill this invariance. Fix a singular vector $u_1(\laap)$ and take a normalized $u_1(\la)$. We define $u^\sharp_1(\la)$ as
    \begin{align*}
        u^\sharp_1(\la) := \frac{\ti{u}_1(\la)}{\norm{\ti{u}_1(\la)}_{\ell^2(\C^{2N+2})}},
    \end{align*}
    where
    \begin{align*}
        \ti{u}_1(\la) := \PUl{\la} u_1(\laap).
    \end{align*}
    For the sake of notational simplicity we will denote $u_1(\la):=u^\sharp_1(\la)$.

    By construction, the only possible obstruction to continuity is that $\ti{u}_1$ vanishes. However, we have that
    \begin{align*}
        \norm{\ti{u}_1(\la)}_{\ell^2(\C^{2N+2})} &\geq 
        \norm{\ti{u}_1(\laap)}_{\ell^2(\C^{2N+2})} -\norm{\ti{u}_1(\la)-\ti{u}_1(\laap)}_{\ell^2(\C^{2N+2})}\\
        & \geq
        1-\norm{\PUl{\la} - \PUl{\laap}}.
    \end{align*}
    Since the projection is continuous in $\la$, we can find a suitable neighborhood $V^{\laap}$ of $\laap$ such that 
    \begin{align*}
        \norm{\PUl{\la} - \PUl{\laap}} < 1 \quad \text{for all } \la\in V^{\laap}.
    \end{align*}

    Consequently, $u_1$ is continuous on $V^\laap$. Note that the dependence of ${\bf b}(G^+,G^-)$ on $\la$ is also continuous so we can shrink the neighborhood to ensure that
    \begin{align*}
        \big|u_1(\la)^\ast{\bf b}(\Fvrmap^+,\Fvrmap^-)\big|\neq 0  \quad \text{for all } \la\in V^{\laap}.
    \end{align*}

    In conclusion, the operator $\Pcal_\la[\Fvrmap^+,\Fvrmap^-]$ is continuous on $V^{\laap}$.
    
    Finally, a straightforward computation shows \eqref{gauge_image}.
\end{proof}

\begin{lemma}\label{lemma:Proots2}
    Condition~\ref{cond:Proots} holds for $\te=\frac{1}{3}$ and $\la=\laap$.
\end{lemma}

\proofcompassi

\begin{lemma} \label{lemma:stab_tor_degen}
    Let $\te=\frac{1}{3}$ and let $(\Fvrmap^+,\Fvrmap^-)$ as in Lemma~\ref{lemma:prod_fv_u}, and $\Pcal_\la[\Fvrmap^+,\Fvrmap^-]$ be the operator defined in Definition~\ref{def:Psigma1}.
    
    For every pair $(G^+,G^-)\in\XHholo\oplus\XHholo$ of functions, with $G^\pm$ continuous on $\overline{\D}$, there exists a unique pair $(F^+,F^-)\in\XHholo\oplus\XHholo$ whose derivatives are continuous on $\overline{\D}$ such that 
    \begin{align}\label{eqn:stab_degen}
        \LOpStabHol{\laap}(F^+,F^-) = \Pcal_\laap[\Fvrmap^+,\Fvrmap^-](G^+,G^-),
    \end{align}
    and
    \begin{align}
        \MStabTor(\laap) 
        {\bf c}(F^+,F^-)=-
        {\bf b}\Pcal_\laap[\Fvrmap^+,\Fvrmap^-]
        (G^+,G^-),\qquad 
        \PVl{\laap} 
        \mathbf{c}(F^+,F^-)=0.
        \label{eqn:invertregular}
    \end{align}
    Here ${\bf c}$ and ${\bf b}$ are defined at \eqref{def:cplumin} and \eqref{def:bplumin} respectively.
\end{lemma}

\begin{proof}
    By Corollary~\ref{coro:isozeros}, we know that there exists a neighborhood $V^\laap$ of $\laap$ such that $\det(\MStabTor)$ never vanishes on $V^\laap-\{\laap\}$. Moreover, from  Corollary~\ref{coro:P_cont_la} and Lemma~\ref{lemma:Proots2}, we can choose this neighborhood small enough so that $\Pcal_\la[\Fvrmap^+,\Fvrmap^-]$ is continuous and Condition~\ref{cond:Proots} is satisfied for $\la\in V^\laap$.

    Fix $\la\in V^\laap-\{\laap\}$ and apply Corollary~\ref{coro:stab_tor} to  $\Pcal_\la[\Fvrmap^+,\Fvrmap^-](G^+,G^-)$ to obtain $(F_\la^+,F_\la^-)\in\XHholo\oplus\XHholo$, with derivatives continuous on $\overline{\D}$, such that \eqref{eqn:stab_degen} holds.

    We have that
    \begin{subequations}\label{system:stab:FtiG}
        \begin{empheq}[left=\empheqlbrace]{align}
            F_{\la}^+|_\T(x)
            &= 
            F_{\la}^+|_\T(0)  \Pi^+_{\la}(x)
            +
            J^+_{\la}\ti{G}_{\la}^+|_\T(x)
            \nonumber
            \\&\qquad
            +
            \sum_{k=1}^N (\ha{F_{\la}^+|_\T})_{k-1}  h_{k}^{-J^+_{\la}}(x)
            +\sum_{k=1}^N (\ha{F_{\la}^-|_\T})_{k-1}  h_{k}^{+J^+_{\la}}(x),
        \label{system:stab:FtiG:plu}
        \\
            F_{\la}^-|_\T(x)
            &= 
            F_{\la}^-|_\T(0)  \Pi_{\la}^-(x)
            -
            J_{\la}^-\ti{G}_{\la}^-|_\T(x)
            \nonumber
            \\&\qquad
            +\sum_{k=1}^N (\ha{F_{\la}^+|_\T})_{k-1}  h_{k}^{+J^-_{\la}}(x)
            +\sum_{k=1}^N (\ha{F_{\la}^-|_\T})_{k-1}  h_{k}^{-J^-_{\la}}(x),
        \label{system:stab:FtiG:min}
        \end{empheq}
    \end{subequations}
    we are emphasizing the quantities that depend on $\la$ and $(\ti{G}^+_\la,\ti{G}^-_\la)$ is given by
    \begin{align*}
        (\ti{G}^+_\la,\ti{G}^-_\la):=
        \Pcal_\la[\Fvrmap^+,\Fvrmap^-](G^+,G^-).
    \end{align*}
    
    Projecting both equations onto $e^{ikx}$, $k=-1,\ldots,N-1$, yields
    \begin{align}\label{eqn:rest_inv_aux}
        \MStabTor(\la) 
        {\bf c}(F^+_{\la},F^-_{\la})+
        {\bf b}_\la
        (\ti{G}^+_\la,\ti{G}^-_\la)=0.
    \end{align}
    Moreover, by \eqref{gauge_image} we know that
    \begin{align*}
        \PUl{\la} {\bf b}_\la(\ti{G}^+_\la,\ti{G}^-_\la)
        =0.
    \end{align*}
    Since $\det\big(\MStabTor(\la)\big)\neq 0$, we know that ${\bf c}(F^+_{\la},F^-_{\la})$ is the only vector satisfying \eqref{eqn:rest_inv_aux}.
    
    By Lemma~\ref{lemma:svd:stable}, with $M=\MStabTor$, we know that 
    \begin{align*}
        {\bf c}(\la):=-M^\dagger_{\rm reg}(\la)  {\bf b}_\la(\ti{G}^+_\la,\ti{G}^-_\la)
    \end{align*}
    is continuous in $\la$. Moreover,
    \begin{align*}
        \PV {\bf c}(\la)=0 \qquad \text{and}\qquad 
        \MStabTor {\bf c}(\la) = - {\bf b}_\la(\ti{G}^+_\la,\ti{G}^-_\la).
    \end{align*}

    Consequently, we have that 
    \begin{align*}
        \mathbf{c}(F^+_{\la},F^-_{\la})={\bf c}(\la).
    \end{align*}

    Therefore, $(F_{\la}^+|_\T,F_{\la}^-|_\T)$ is continuous in $\la$ for all $\la \in V^{\laap}-\{\laap\}$, since \eqref{system:stab:FtiG} allows us to express these functions as a sum of continuous terms. By the dominated convergence theorem, for every $k<0$ we obtain:
    \begin{align*}
        \frac{1}{2\pi}
        \intpi
        F_{\laap}^\pm|_\T(x) e^{-ikx}\dx
        =
        \frac{1}{2\pi}
        \intpi
        \lim_{\la\to\laap}
        F_{\la}^\pm|_\T(x) e^{-ikx}\dx
        =
        \lim_{\la\to\laap}
        \frac{1}{2\pi}
        \intpi
        F_{\la}^\pm|_\T(x) e^{-ikx}\dx=0.
    \end{align*}
    Hence, the functions $F_{\la}^+|_\T$ and $F_{\la}^-|_\T$ admit a holomorphic extension to $\D$ even for $\la=\laap$.

    Finally, regarding uniqueness, assume that $(F_{\laap}^+,F_{\laap}^-)\in\XHholo\oplus\XHholo$ satisfies
    \begin{align*}
        \LOpStabHol{\laap}(F_{\laap}^+,F_{\laap}^-)=0,
        \qquad
        \big(\mathbf{c}(F_{\laap}^+,F_{\laap}^-),v_1(\laap)\big)=0.
    \end{align*}
    Projecting the corresponding identities as above yields
    \begin{align*}
        \MStabTor(\laap) \mathbf{c}(F_{\laap}^+,F_{\laap}^-)=0.
    \end{align*}
    By Lemma~\ref{lemma:stab_singval}, we know that
    \begin{align*}
        \sigma_2(\MStabTor(\laap))\geq \underline{\sigma}_2>0,
    \end{align*}
    which implies that $\dim\big(\ker(\MStabTor(\laap))\big)\leq  1$. Hence
    \begin{align*}
        {\bf c}(F_{\laap}^+,F_{\laap}^-)=\al v_1(\laap)
    \end{align*}
    for some $\al\in\C$. Taking the inner product with $v_1(\laap)$ and using the constraint 
    \begin{align*}
        \PVl{\laap}{\bf c}(F_{\laap}^+,F_{\laap}^-)=0
    \end{align*}
    yields $\al=0$,
    so ${\bf c}(F_{\laap}^+,F_{\laap}^-)=0$.
\end{proof}

\subsubsection{Estimates of the explicit functions and operators}
\label{sec:stab:bounds}

The next step is to find an explicit estimate for the restricted inverse. To this end, we establish several bounds in Lemmas~\ref{lemma:kappa3}--\ref{lemma:J_L2H1_stab}.

\begin{lemma}\label{lemma:kappa3}
    Let $\kappa_3^+$ and $\kappa_3^-$ be the functions given by
    \begin{align*}
        \kappa_3^\pm(x):=
        e^{\pm2\Re(\laap)\int_0^x \beta(y)\dy}
        \int_0^x
            \beta(y)^2 
            e^{\mp 2\Re(\laap)\int_0^y \beta(s)\,ds}        
    \dy.
    \end{align*}

    Recall \eqref{def:Piplusminu}, so that
    \begin{align*}
        |\Pi^\pm(x)|=
        e^{\pm\Re(\laap)\int_0^x \beta(y)\dy}.
    \end{align*}

    Then
    \begin{align*}
        \norm{\kappa^\pm_3}_{L^1([0,\pi])} <
        \getvar{kappa3}
        .
    \end{align*}
\end{lemma}
\proofcompassi

\begin{lemma}\label{lemma:kappa4}
    Let $\te=\frac{1}{3}.$ Given a function $\kappa:[0,\pi]\to\RR$, define the functions $\kappa_{4,\kappa}^+$ and $\kappa_{4,\kappa}^-$ as
    \begin{align*}
        \kappa_{4,\kappa}^\pm(x):=
        \big|
        (i\pm\laap)\beta(x)\mp i\te
        \big|^2 
        \kappa(x).
    \end{align*}

    Let $\kappa_3^+$ and $\kappa_3^-$ as in Lemma~\ref{lemma:kappa3}. Then the following estimates hold
    \begin{align*}
        \norm{
        \kappa_{4,\kappa_3^\pm}^+
        }_{L^1([0,\pi])}<
        \getvar{kappa4p}
        ,\qquad
        \norm{
        \kappa_{4,\kappa_3^\pm}^-
        }_{L^1([0,\pi])}<
        \getvar{kappa4m}
        .
    \end{align*}
\end{lemma}
\proofcompassi

\begin{lemma}\label{lemma:J_L2L2_stab}
    Let $g\in\XLIpi$, $\la=\laap$, and $\te=\frac{1}{3}$. Then
    \begin{align*}
        \norm{J^\pm g}_\XLIpi^2 <
        \getvar{kappa3}\,
        \norm{g}_\XLIpi^2.
    \end{align*}

    Recall that $J^\pm$ is defined in \eqref{def:Jplusminu}.
\end{lemma}

\begin{proof}
    We first note that
    \begin{align*}
        |\Pi^\pm(x)|=
        \exp\!\left(
            \pm\Re(\laap) 
            \int_0^x\beta(y)\dy
        \right)
    \end{align*}

    It follows that, on $[0,\pi]$, 
    \begin{align}
        \norm{J^+ g}_{L^2([0,\pi])}^2
        &=
        \int_0^\pi
        |\Pi^+(x)|^2 
        \left|\int_0^x\frac{\beta(y)}{\Pi^+(y)} g(y)
        \dy\right|^2
        \dx
        \nonumber
        \\ 
        &\leq 
        \int_0^\pi
        |\Pi^+(x)|^2 
        \int_0^x\frac{\beta(y)^2}{|\Pi^+(y)|^2} 
        \dy
        \int_0^x|g(y)|^2
        \dy
        \dx
        =
        \int_0^\pi
        \kappa^+_3(x) 
        \norm{g}_{L^2([0,x])}^2
        \dx.
        \label{eqn:J_L2L2_stab:aux1}
    \end{align}
    where $\kappa_3^+$ is defined in Lemma~\ref{lemma:kappa3}.

    Hence, applying Fubini's theorem to \eqref{eqn:J_L2L2_stab:aux1}, we obtain
    \begin{align*}
        \norm{J^+ g}_{L^2([0,\pi])}^2
        &\leq
        \int_{0}^\pi 
        |g(y)|^2
        \int_y^{\pi} 
        \kappa_3^+(x) 
        \dx
        \dy
        \\
        &\leq
        \norm{\kappa_3^+}_{L^1([0,\pi])} 
        \norm{g}_{L^2([0,\pi])}^2.
    \end{align*}

    An analogous argument for $x\in[-\pi,0]$ yields
    \begin{align*}
        \norm{J^+ g}_{L^2([-\pi,0])}^2
        \leq
        \norm{\kappa_3^-}_{L^1([0,\pi])} 
        \norm{g}_{L^2([-\pi,0])}^2.
    \end{align*}

    Consequently,
    \begin{align*}
        \norm{J^+ g}_\XLIpi ^2
        &\leq 
        \norm{\kappa_3^+}_{L^1([0,\pi])} 
        \norm{g}_{L^2([0,\pi])}^2
        +
        \norm{\kappa_3^-}_{L^1([0,\pi])} 
        \norm{g}_{L^2([-\pi,0])}^2\\
        &\leq 
        \max\big\{\norm{\kappa_3^+}_{L^1([0,\pi])},
        \norm{\kappa_3^-}_{L^1([0,\pi])}\big\} 
        \norm{g}_\XLIpi^2.
    \end{align*}

    Finally, for $J^-$ the roles of $\norm{\kappa_3^+}_{L^1([0,\pi])}$ and $\norm{\kappa_3^-}_{L^1([0,\pi])}$ are exchanged. Hence, we obtain the same estimate.
\end{proof}

\begin{lemma}\label{lemma:J_L2H1_stab}
    Let $g\in\XLIpi$, $\la=\laap$, and $\te=\frac{1}{3}$. Then
    \begin{align*}
        \norm{\pa_x J^+g}_\XLIpi
        &<
        \getvar{J1p}\,
        \norm{g}_\XLIpi,\\
        \norm{\pa_x J^-g + i J^-g}_\XLIpi
        &<
        \getvar{J1m}\,
        \norm{g}_\XLIpi.
    \end{align*}
\end{lemma}

\begin{proof}
    First assume that $g\in C^0(\T)$, so that $\pa_x J^\pm g$ is well defined pointwise.
    Since continuous functions are dense in $\XLIpi$, the estimate extends to all $g\in\XLIpi$.
    
    Differentiating $J^\pm g$ yields
    \begin{align*}
        \pa_x J^\pm g &= 
        \big(\big(i\pm\laap)\beta-i\te^\pm\big)
        J^\pm g+
        \beta g.
    \end{align*}

    Define 
    \begin{align*}
        I_{1^+}&:=
        \norm{
        ((i+\laap)\beta-i\te^+) 
        J^+ g
        }_\XLIpi,\qquad
        I_2:=
        \norm{
        \beta  g
        }_\XLIpi,\\
        I_{1^-}&:=
        \norm{
        ((i-\laap)\beta+i\te^+) 
        J^- g
        }_\XLIpi.
    \end{align*}

    Thus, we estimate
    \begin{align*}
        \norm{\pa_x J^+g}_\XLIpi &\leq 
        I_{1^+} + I_2,\\
        \norm{\pa_x J^-g + i J^-g}_\XLIpi &\leq 
        I_{1^-} + I_2.  
     \end{align*}

    It is straightforward that
    \begin{align}\label{eqn:J_L2H1_stab:auxI2}
        I_2 \leq 
        \norm{\beta}_{\XLinfinterval} 
        \norm{g}_\XLIpi.
    \end{align}
    Lemma~\ref{lemma:betamax} provides an estimate for $\norm{\beta}_{\XLinfinterval}$.

    To estimate $I_{1^+}$, recall \eqref{eqn:J_L2L2_stab:aux1},
    \begin{align}
        \norm{
        ((i+\laap)\beta-i\te^+) 
        J^+ g
        }_{L^2([0,\pi])}^2&=
        \int_0^\pi 
        \big|
        ((i+\laap)\beta(x)-i\te^+) 
        J^+ g(x)
        \big|^2
        \dx
        \nonumber
        \\
        &\leq 
        \int_0^\pi
        \underbrace{
        \big|
        (i+\laap)\beta(x)-i\te^+
        \big|^2 
        \kappa_3^+(x)
        }_{
        \kappa_{4,\kappa_3^+}^+(x)
        } 
        \norm{g}_{L^2([0,x])}^2
        \dx
        \nonumber
        \\
        &\leq 
        \int_0^\pi
        |g(y)|^2
        \int_y^{\pi}
        \kappa_{4,\kappa_3^+}^+(x) 
        \dx\dy
        \nonumber
        \\
        &\leq 
        \norm{g}_{L^2([0,\pi])}^2 
        \norm{
        \kappa_{4,\kappa_3^+}^+
        }_{L^1([0,\pi])}.
        \label{eqn:J_L2H1_stab:auxI11}
    \end{align}
    By Lemma~\ref{lemma:kappa4}, we have an explicit bound for \eqref{eqn:J_L2H1_stab:auxI11}.

    Now, for $x\in[-\pi,0]$,
    \begin{align}
        \norm{
        ((i+\laap)\beta-i\te^+) 
        J^+ g
        }_{L^2([-\pi,0])}^2
        \leq 
        \norm{g}_{L^2([-\pi,0])}^2 
        \norm{\kappa_{4,\kappa_3^-}^+}_{L^1([0,\pi])}.
        \label{eqn:J_L2H1_stab:auxI12}
    \end{align}

    Therefore, we get from \eqref{eqn:J_L2H1_stab:auxI11} and \eqref{eqn:J_L2H1_stab:auxI12} that
    \begin{align*}
        I_{1^+}^2\leq 
        \max\big\{
        \norm{
        \kappa_{4,\kappa_3^+}^+
        }_{L^1([0,\pi])},
        \norm{
        \kappa_{4,\kappa_3^-}^+
        }_{L^1([0,\pi])}
        \big\} 
        \norm{g}_\XLIpi^2.
    \end{align*}

    Combining these estimates yields the first bound in the statement. The second one follows in the same way
    \begin{align*}
        I_{1^-}^2 \leq 
        \max\big\{\norm{
        \kappa_{4,\kappa_3^-}^-
        }_{L^1([0,\pi])}, 
        \norm{
        \kappa_{4,\kappa_3^+}^-
        }_{L^1([0,\pi])}
        \big\} 
        \norm{g}_\XLIpi^2.
    \end{align*}
\end{proof}

\begin{lemma}\label{lemma:hk:stab}
The following estimates hold for $\la=\laap$ and $\te=\frac{1}{3}$,
    \begin{align*}
        \norm{\Pi^+}_\XLIpi^2+
        \sum_{k=1}^N
        \norm{J^+ g^-_{k,\te^-}}
        _\XLIpi^2+
        \norm{J^+ g^+_{k,\te^+}}
        _\XLIpi^2
        &\leq 
        \getvar{hkL2p}
        ,\\
        \norm{\Pi^-}_\XLIpi^2+
        \sum_{k=1}^N
        \norm{J^- g^-_{k,\te^+}}
        _\XLIpi^2+
        \norm{J^- g^+_{k,\te^-}}
        _\XLIpi^2&\leq
        \getvar{hkL2m},
        \\
        \norm{\pa_x \Pi^+}_\XLIpi^2+
        \sum_{k=1}^N
        \norm{\pa_x J^+ g^-_{k,\te^-}}
        _\XLIpi^2+
        \norm{\pa_x J^+ g^+_{k,\te^+}}
        _\XLIpi^2
        &\leq
        \getvar{hkH1p},\\
        \norm{\ti{\pa_x} \Pi^-}_\XLIpi^2+
        \sum_{k=1}^N
        \norm{
            \ti{\pa_x} J^- g^-_{k,\te^+} 
        }
        _\XLIpi^2+
        \norm{
        \ti{\pa_x} J^- g^+_{k,\te^-}
        }
        _\XLIpi^2&\leq 
        \getvar{hkH1m},
    \end{align*}
    where $\ti{\pa_x}f := \pa_x f + i f$.
\end{lemma}

\proofcompassi

We introduce a new operator, which will be useful to specify the space in which we are working. For every $s\in\RR$, let $i_H:\XH{s}\to\XHhol{s}\oplus\XHhol{s}$ be the operator given by
\begin{align}\label{def:i_H}
    i_H(g) := \Big(
    \sum_{k\geq 0} \ha{g}_k z^k, 
    \sum_{k\geq 0} \ha{g}_{-k-1} z^k
    \Big).
\end{align}

Define $\mathcal{W}_{\laap}\subset \XHhol{0}\oplus\XHhol{0}$ by
\begin{align}\label{def:W_s}
    \mathcal{W}_\laap&:=
    \Pcal_\laap[\Fvrmap^+,\Fvrmap^-](\XHhol{0}\oplus\XHhol{0}).\\
    W_\laap&:=
    i_H^{-1}\mathcal{W}_{\laap}.\label{def:W_s2}
\end{align}

\begin{prop}\label{prop:stab:lininv}
    Let $g\in W_\laap$. Then there exists a unique $f\in\XH{1}$, with
    \begin{align*}
        \Pcal_{\mathcal{V}_1(\laap)}
        {\bf c}(i_Hf) = 0,
    \end{align*}
    such that
    \begin{align*}
        \LStabthe f - \laap f = g.
    \end{align*}

    In addition, the following estimate holds
    \begin{align*}
        \norm{f}_\XH{1} \leq 
        \getvar{Lstab_inv}
        \,\norm{g}_\XH{0}.
    \end{align*}
\end{prop}
\begin{proof}
    Since $\XH{1}$ is dense in $\XH{0}$, we will assume that $g\in\XH{1}$ and extend the result by density to all $\XH{0}$. Hence, we have $(G^+,G^-):=i_H(g)\in W_\laap \cap (\XHhol{1}\oplus\XHhol{1}).$
    Then, $G^\pm$ is continuous on $\overline{\D}$ and we can apply Lemma~\ref{lemma:stab_tor_degen}:
    
    There exists a unique pair $(F^+,F^-)\in\XHholo\oplus\XHholo$ of holomorphic functions, with derivatives continuous on $\overline{\D}$, such that
    \begin{align*}
        \LOpStabHol{\laap}(F^+,F^-) =
        \Pcal_\laap[\Fvrmap^+,\Fvrmap^-](G^+,G^-)=(G^+,G^-).
    \end{align*}

    Now we seek for an explicit estimate. Set $(F^+,F^-):=i_H(f)$ and observe that
    \begin{align*}
        \norm{f}_{\XH{1}}^2 &=
        \sum_{k\in\Z}
        (1+k^2)
        |\ha{f}_k|^2\\&=
        \sum_{k\geq 0}
        (1+k^2)
        |F^+_k|^2
        +
        \sum_{k\geq 0}
        (1+(k+1)^2)
        |F^-_{k}|^2.
    \end{align*}

    Note that
    \begin{align*}
        \pa_x\big(
        e^{ix} F^-|_\T(x)
        \big) =
        \sum_{k\geq 0}
        F_k^{-}
        \pa_x\big( e^{i(k+1)x} \big)\ 
        \Rightarrow\
        \sum_{k\geq 0}
        (k+1)^2
        |F^-_{k}|^2 &=  
        \frac{1}{2\pi}
        \intpi 
        \big|\pa_x\big(
        e^{ix} F^-(x)
        \big)\big|^2\dx\\
         &=  
        \frac{1}{2\pi}
        \intpi 
        \big|
        \pa_x F^-(x)+
        i F^-(x)
        \big|^2\dx.
    \end{align*}

    Thus, we obtain
    \begin{align}
        \norm{f}_{\XH{1}}^2 =
        \norm{F^+|_\T}_{\XL}^2+
        \norm{\pa_x F^+|_\T}_{\XL}^2+
        \norm{F^-|_\T}_{\XL}^2+
        \norm{\pa_x F^-|_\T + i F^-|_\T}_{\XL}^2
        \label{eqn:decomp_stab_lin}
    \end{align}

    We only explain how to work with $F^+$ using \eqref{system:stab:FtiG:plu}, since $F^-$ follows in the same way using \eqref{system:stab:FtiG:min}.
    Expression \eqref{system:stab:FtiG:plu} reads as follows:
    \begin{align*}
        F^+|_\T(x)
        &= 
        F^+|_\T(0)  
        \Pi^+(x)
        +
        J^+G^+|_\T(x)
        \nonumber
        \\&\qquad
        +
        \sum_{k=1}^N (\ha{F^+|_\T})_{k-1}  h_{k}^{-J^+}(x)
        +\sum_{k=1}^N (\ha{F^-|_\T})_{k-1}  h_{k}^{+J^+}(x).
    \end{align*}

    Define
    \begin{align*}
        {\bf h}^+(x):= \begin{pmatrix}
            h_1^{-J^+}(x)&
            \cdots&
            h_N^{-J^+}(x)&
            h_1^{+J^+}(x)&
            \cdots&
            h_N^{+J^+}(x)&
            \Pi^+(x)&
            0
        \end{pmatrix}
    \end{align*}
    Hence, we obtain
    \begin{align*}
        F^+|_\T(x)
        = 
        {\bf h}^+(x) 
        {\bf c}(F^+,F^-)
        +
        J^+G^+|_\T(x),
    \end{align*}
    where we know that
    \begin{align*}
        {\bf c}(F^+,F^-)=
        -M^\dagger_{\rm reg} 
        {\bf b}(G^+,G^-).
    \end{align*}
    Recall that ${\bf c}(F^+,F^-)$ is defined in \eqref{def:cplumin}, ${\bf b}$ is defined in \eqref{def:bplumin}, and $M^\dagger_{\rm reg}$ is the matrix given in Lemma~\ref{lemma:svd:stable} with $M=\MStabTor$.

    Using the same strategy from \eqref{eqn:exis_final_step2}, we have that
    \begin{align}
        \scalednorm{
        {\bf h}^+ 
        M^\dagger_{\rm reg}  
        {\bf b}(G^+,G^-)
        }_\XLIpi 
        \leq
        \frac{1}{\sigma_2(\laap)}
        C_{0^+}
        \norm{{\bf b}(G^+,G^-)}_{\ell^2(\C^{2N+2})}.
        \label{lemma:stab:lininv:aux1}
    \end{align}
    where 
    \begin{align*}
        C_{0^+}^2:=
        \norm{\Pi^+}_\XLIpi^2
        +
        \sum_{k=1}^N 
        \norm{h_k^{-J^+}}_\XLIpi^2
        + 
        \norm{h_k^{+J^+}}_\XLIpi^2.
    \end{align*}

    Additionally, note that
    \begin{align*}
        \norm{{\bf b}(G^+,G^-)}_{\ell^2(\C^{2N+2})}^2&\leq 
        \norm{J^+ G^+}_\XL^2 + 
        \norm{J^- G^-}_\XL^2\\
        &\leq 
        \max\{\norm{J^+}_{\XL\to\XL},\norm{J^-}_{\XL\to\XL}\}^2 
        \big(\underbrace{
            \norm{G^+}_\XL^2 + 
            \norm{G^-}_\XL^2}_{\norm{g}_\XL^2}
        \big).
    \end{align*}
    Moreover, by Lemma~\ref{lemma:J_L2L2_stab}, we have the same explicit bound for $\norm{J^+}_{\XL\to\XL}$ and $\norm{J^-}_{\XL\to\XL}$. Let $C_{J_0}$ be such that
    \begin{align}
        \norm{{\bf b}(G^+,G^-)}_{\ell^2(\C^{2N+2})}\leq 
        C_{J_0} \norm{g}_\XL
        .\label{eqn:stab:lininv:aux2}
    \end{align}
    
    Consequently, we have from \eqref{lemma:stab:lininv:aux1}--\eqref{eqn:stab:lininv:aux2} that
    \begin{align}
        \norm{F^+|_\T}_\XL
        &\leq 
        \frac{1}{\sqrt{2\pi}}\Big(
        \frac{C_{0^+}}{\sigma_2(\laap)}
        \norm{{\bf b}(G^+,G^-)}
        _{\ell^2(\C^{2N+2})}
        +
        \norm{J^+ G^+|_\T}_\XLIpi
        \Big)
        \nonumber
        \\
        &\leq \frac{1}{\sqrt{2\pi}} 
        \frac{C_{0^+}C_{J_0}}{\sigma_2(\laap)} 
        \norm{g}_\XL
        +
        C_{J_0}
        \norm{G^+|_\T}_\XL.
        \label{eqn:stab:lininv:aux3}
    \end{align}

    Let $C_{0^-},C_{1^+},C_{1^-},C_{J_1^+}$, and $C_{J_1^-}$ be defined by
    \begin{align*}
        C_{0^-}^2&:=
        \norm{\Pi^-}_\XLIpi^2
        +
        \sum_{k=1}^N 
        \norm{h_k^{+J^-}}_\XLIpi^2
        +
        \norm{h_k^{-J^-}}_\XLIpi^2
        ,\\
        C_{1^+}^2&:=
        \norm{\pa_x\Pi^+}_\XLIpi^2
        +
        \sum_{k=1}^N 
        \norm{\pa_x h_k^{-J^+}}_\XLIpi^2
        + 
        \norm{\pa_x h_k^{+J^+}}_\XLIpi^2
        ,\\
        C_{1^-}^2&:=
        \norm{\pa_x\Pi^-+i\Pi^-}_\XLIpi^2
        +
        \sum_{k=1}^N 
        \norm{\pa_x h_k^{+J^-} + ih_k^{+J^-}}_\XLIpi^2
        +
        \norm{\pa_x h_k^{-J^-} + ih_k^{-J^-}}_\XLIpi^2,\\
        C_{J_{1^+}}&:=
        \sup_{\norm{g}_\XLIpi=1}
        \big\{ 
            \norm{\pa_x J^+ g}_\XLIpi
        \big\},\qquad
        C_{J_{1^-}}:=
        \sup_{\norm{g}_\XLIpi=1}
        \big\{ 
            \norm{\pa_x J^- g + i J^- g}_\XLIpi
        \big\}.
    \end{align*}

    A similar computation to \eqref{eqn:stab:lininv:aux3} shows
    \begin{align*}
        \norm{F^-|_\T}_\XL
        &\leq 
        \frac{1}{\sqrt{2\pi}}
        \frac{C_{0^-}C_{J_0}}{\sigma_2(\laap)} 
        \norm{g}_\XL
        +
        C_{J_0} 
        \norm{G^-|_\T}_\XL,\\
        %%%%%%%%%%%%%%%%%%%%%%%%%%
        \norm{\pa_x F^+|_\T}_\XL & \leq
        \frac{1}{\sqrt{2\pi}}
        \frac{C_{1^+} C_{J_0}}{\sigma_2(\laap)}
        \norm{g}_\XL
        +
        C_{J_{1^+}}
        \norm{G^+|_\T}_\XLIpi,
        \\
        %%%%%%%%%%%%%%%%%%%%%%%%%%
        \norm{\pa_x F^-|_\T + i F^-|_\T}_\XL
        &\leq
        \frac{1}{\sqrt{2\pi}}
        \frac{C_{1^-} C_{J_0}}{\sigma_2(\laap)}
        \norm{g}_\XL
        +
        C_{J_{1^-}}
        \norm{G^-|_\T}_\XLIpi.
    \end{align*}

    Combining these estimates into \eqref{eqn:decomp_stab_lin} yields
    \begin{align*}
        \norm{f}_{\XH{1}}^2 &\leq 
        a\norm{g}_\XL^2 
        +
        b^+\norm{G^+|_\T}_\XL^2
        +
        b^-\norm{G^-|_\T}_\XL^2\\
        &\qquad+
        2c^+\norm{g}_\XL\norm{G^+|_\T}_\XL
        +
        2c^-\norm{g}_\XL\norm{G^-|_\T}_\XL.
    \end{align*}
    where $a,b^+,b^-,c^+$, and $c^-$ are given by
    \begin{align*}
        a&:=
        \frac{1}{2\pi}
        \frac{C_{J_0}^2}{\sigma_2(\laap)^2}\big(
        C_{0^+}^2 + 
        C_{0^-}^2 + 
        C_{1^+}^2 + 
        C_{1^-}^2 \big),\\
        b^\pm&:= C_{J_0}^2 + C_{J_{1^\pm}}^2,\qquad 
        c^{\pm}:= 
        \frac{1}{\sqrt{2\pi}}
        \frac{C_{J_0}}{\sigma_2(\laap)}
        \big(
        C_{0^\pm}C_{J_0} + C_{1^\pm}C_{J_{1^\pm}}
        \big).
    \end{align*}

    Finally, define the quantities $\vartheta_1,\vartheta_2$ by
    \begin{align*}
        \vartheta_+:= \getvar{vthep},\qquad 
        \vartheta_-:= \getvar{vthem}.
    \end{align*}
    
    Consequently, we have that\begin{align*}
        \norm{f}_{\XH{1}}^2 &\leq 
        (a+
        c^+\vartheta_++
        c^-\vartheta_-)
        \norm{g}_\XL^2 
        +
        (b^++\vartheta_+^{-1}c^+)\norm{G^+|_\T}_\XL^2
        +
        (b^-+\vartheta_-^{-1}c^-)\norm{G^-|_\T}_\XL^2\\
        &\leq \big(
        a+
        c^+\vartheta_++
        c^-\vartheta_-
        + 
        \max\{
        b^++\vartheta_+^{-1}c^+,
        b^-+\vartheta_-^{-1}c^-
        \}\big)\norm{g}_\XL^2.
    \end{align*}
    Here, every quantity has an explicit upper bound. We substitute them and check the inequality in the statement.
\end{proof}

\subsection{The choice for $\Eigdiff$}\label{sec:eta_estimates}

In this section we fix $\te=\frac{1}{3}$ and $\la=\laap$. We describe the construction of the function $\Eigdiff$ from \eqref{def:ContOpStab} and compute some estimates required to close our argument.

\begin{defi}\label{def:eta_func}
    Let $f\in\XH{1}$ and $i_H:\XH{1}\to\XHhol{1}\oplus\XHhol{1}$ as in \eqref{def:i_H}. Define the operator $\Procomp$ as the composition given by
    \begin{align}
        \Procomp :=  u_1^\ast(\laap)\circ {\bf b}\circ i_H.
        \label{def:Procomp}
    \end{align}
    
    Then we define $\Eigdiff(f)$ by
    \begin{align*}
        \Eigdiff(f) := \frac{
        \Procomp(
        \Dcal_\te(f+\Eigveap)+
        \Eigres
        )
        }{
        \Procomp(
        f+\Eigveap
        )
        }.
    \end{align*}

\end{defi}
\begin{rem}
    With this choice of $\Eigdiff(f)$, we have that
    \begin{align*}
        \Procomp
        \big[
        (\Eigdiff(f) I-\DiffOpExpNonExp)(f+\Eigveap)-\Eigres
        \big] = 0.
    \end{align*}

    Moreover,
    \begin{align*}
        (\Eigdiff(f) I-\DiffOpExpNonExp)(f+\Eigveap)-\Eigres
        \in  W_\laap.
    \end{align*}
    Recall the definition of $W_\laap$ in \eqref{def:W_s2}.
\end{rem}

\begin{lemma}\label{lemma:eta:bound}
    Recall Lemma~\ref{lemma:prod_fv_u} and
    let $x_{\rm max}>0$ such that
    \begin{align*}
        x_{\rm max} <
        \frac{\big|
            \Procomp(\Eigveap
            )
        \big|}{C_{J_0}},
    \end{align*}
    where $C_{J_0}\geq\max\big\{\norm{J^+}_{\XL\to\XL},\norm{J^-}_{\XL\to\XL}\big\}$. Let $B_{x_{\rm max}}$ be defined as
    \begin{align*}
        B_{x_{\rm max}}:=\{f\in\XH{1}:\ \norm{f}_\XH{1} < x_{\rm max}\}.
    \end{align*}

    Then, the operator $\Eigdiff|_{B_{x_{\rm max}}}$ is well defined.
    Moreover, for every $f\in B_{x_{\rm max}}$, it holds that
    \begin{align*}
        |\Eigdiff(f)|\leq 
        q_\Eigdiff(\norm{f}_\XH{1}),
    \end{align*}
    where $q_\Eigdiff:[0,x_{\rm max})\to\RR$ is given by
    \begin{align}\label{def:q_eta}
        q_\Eigdiff(x):=
        \frac{(x+
        \norm{\Eigveap}_\XH{1}) 
        \norm{\DiffOpExpNonExp}
        _{\XH{1}\to\XL}+
        \norm{\Eigres}_\XL}{
        x_{\rm max}-
        x}.
    \end{align}
    
\end{lemma}

\begin{rem}
    We can compute an upper bound for $x_{\rm max}$ using Lemma~\ref{lemma:J_L2L2_stab} and Lemma~\ref{lemma:prod_fv_u}. Moreover, in \eqref{def:q_eta} we also have explicit bounds for every term by Corollary~\ref{coro:DiffOpExpNonExp} and Lemma~\ref{lemma:norm_fv}.
\end{rem}

\begin{proof}
    We start by finding an upper bound for the numerator appearing in $\Eigdiff$. For each function $f\in\XH{1}$, denote by $(F^+,F^-)=i_H(f)$, it follows that
    \begin{align*}
        |\Procomp(f)|^2
       &
        \leq
        \norm{
        {\bf b}(F^+,F^-)
        }_{\ell^2(\C^{2N+2})}^2
        \\&
        \leq
        \norm{J^+}_{\XL\to\XL}^2
        \norm{F^+}_\XL^2
        +
        \norm{J^-}_{\XL\to\XL}^2
        \norm{F^-}_\XL^2
        \\&
        \leq
        \max\big\{
        \norm{J^+}_{\XL\to\XL},
        \norm{J^-}_{\XL\to\XL}\}^2 
        \norm{f}_\XL^2.
    \end{align*}
    Here, we know that $\norm{J^-}_{\XL\to\XL}$ is finite since Lemma~\ref{lemma:J_L2L2_stab} provides an explicit estimate.

    Then there exists $C_{J_0}>0$ such that
    \begin{align}\label{eqn:Pbihbound}
        |\Procomp(f)|
        \leq 
        C_{J_0} \norm{f}_\XL.
    \end{align}

    Hence, the numerator is estimated as follows
    \begin{align}
        \big|
            \Procomp(
            \Dcal_\te(f+\Eigveap)+
            \Eigres
            )
        \big|
        &\leq 
        C_{J_0} 
        \norm{
            \Dcal_\te(f+\Eigveap)
            +
            \Eigres
        }_\XL
        \nonumber
        \\
        &\leq 
        C_{J_0} 
        \big(
        (\norm{f}_\XH{1} 
        +
        \norm{\Eigveap}_\XH{1} ) 
        \norm{\DiffOpExpNonExp}_{\XH{1}\to\XL}+
        \norm{
            \Eigres
        }_\XL
        \big).
        \label{eqn:q_eta:nume}
    \end{align}

    Now, we get a lower bound for the denominator. Assume $\norm{f}_\XH{1}<x_{\rm max}$. Then
    \begin{align}
        \big|\Procomp(f+\Eigveap)\big|
        &\geq 
        \big|\Procomp(\Eigveap)\big|
        - 
        \big|\Procomp(f)\big|
        \nonumber
        \\
        &\geq 
        \big|\Procomp(\Eigveap)\big|
        - 
        C_{J_0} 
        \norm{f}_\XL
        \nonumber
        \\
        &\geq 
        \big|\Procomp(\Eigveap)\big|
        - 
        C_{J_0} 
        \norm{f}_\XH{1}
        \nonumber
        \\
        &>
        C_{J_0} \big(
        x_{\rm max}
        - 
        \norm{f}_\XH{1}
        \big).
        \label{eqn:q_eta:deno}
    \end{align}

    Finally, \eqref{def:q_eta} follows from \eqref{eqn:q_eta:nume} and \eqref{eqn:q_eta:deno}.
\end{proof}

We will also need an estimate for the derivative of $\Eigdiff$.
\begin{lemma}\label{lemma:eta:bound_dx}
    Let $x_{\rm max}$ as in Lemma~\ref{lemma:eta:bound}. Then, for every $f\in\XH{1}$ with $\norm{f}_\XH{1}<x_{\rm max}$, the linear operator $d_f\Eigdiff$ satisfies that
    \begin{align*}
        |d_f\Eigdiff(h)| \leq q_{d\Eigdiff}\big(\norm{f}_\XH{1}\big)
        \norm{h}_\XH{1},\quad\text{for all } h\in \XH{1},
    \end{align*}
    where the function $q_{d\Eigdiff}:[0,x_{\rm max}]\to\RR^+$ is given by
    \begin{align*}
        q_{d\Eigdiff}(x):=
        \frac{
            \big(
                2x
                +
                \norm{\Eigveap}_\XL
                +
                \norm{\Eigveap}_\XH{1}
            \big)
            \norm{\DiffOpExpNonExp}
            _{\XH{1}\to\XL}
            +
            \norm{\Eigres}_\XL
        }
        {(x_{\rm max}-x)^2}.
    \end{align*}
    
\end{lemma}

\begin{proof}
    The computation is straightforward
    \begin{align*}
        d_f\Eigdiff(h)
        &=
        \frac{
        \Procomp (\Dcal_\te(h))
        }{
        \Procomp(f+\Eigveap)
        }
        -
        \frac{
        \Procomp(\Dcal_\te(f+\Eigveap)
        +
        \Eigres)
        }{
        \big[
        \Procomp(f+\Eigveap)
        \big]^2
        } 
        \Procomp (h)\\ 
        %%%%%%%%%%%%%%%%%%%%%%%%%%%
        &=
        \frac{
        \big[
        \Procomp (\Dcal_\te(h))
        \big]\big[
        \Procomp(f+\Eigveap)
        \big]
        -
        \big[
        \Procomp(\Dcal_\te(f+\Eigveap)
        +
        \Eigres)
        \big]
        \big[
        \Procomp(h)
        \big]
        }{\big[
        \Procomp(f+\Eigveap)
        \big]^2}.
    \end{align*}

    Hence, taking the absolute value, and using \eqref{eqn:Pbihbound} and \eqref{eqn:q_eta:deno} yields
    \begin{align*}
        |d_f\Eigdiff(h)|
        &\leq
        \frac{
        \norm{\Dcal_\te(h)}_\XL
        \norm{f+\Eigveap}_\XL
        +
        \norm{\Dcal_\te(f+\Eigveap)
        +
        \Eigres}_\XL
        \norm{h}_\XL
        }{
        \big[x_{\rm max}
        -
        \norm{f}_\XH{1}\big]^2
        }\\
        &\leq
        \frac{
        \big(2\norm{f}_{\XH{1}}
        +
        \norm{\Eigveap}_\XL
        +
        \norm{\Eigveap}_\XH{1}\big) 
        \norm{\Dcal_\te}_{\XH{1}\to\XL}
        +
        \norm{\Eigres}_\XL
        }{\big[x_{\rm max}
        -
        \norm{f}_\XH{1}\big]^2} 
        \norm{h}_{\XH{1}}.
    \end{align*}
\end{proof}

\subsection{Proofs of Lemmas \ref{lemma:stab_ballintoitself}--\ref{lemma:real_eigen_control}}\label{sec:stab_lemmas}

Fix $f\in\XH{1}$, with $\norm{f}_\XH{1}\leq\getvar{rad_stab}$. In particular, we have that
\begin{align*}
    \norm{f}_\XH{1}<x_{\rm max}.
\end{align*}

Moreover, the function $q_\Eigdiff$ defined in \eqref{def:q_eta} is increasing on $[0,x_{\rm max})$. Hence, we obtain
\begin{align*}
    |\Eigdiff(f)| \leq 
    q_\Eigdiff(\norm{f}_\XH{1}) \leq  
    q_\Eigdiff(\getvar{rad_stab}).
\end{align*}
This estimate allows us to prove Lemma~\ref{lemma:real_eigen_control} since we have explicit estimates for every quantity appearing in \eqref{def:q_eta}. Therefore, we check that
\begin{align*}
    \Re(\laap) - q_\Eigdiff(\getvar{rad_stab}) > \getvar{lare}.
\end{align*}

We now turn to Lemma~\ref{lemma:stab_ballintoitself}. Recall \eqref{def:ContOpStab_bis}, i.e.,
\begin{align}\label{def:ContOpStab_bis_bis}
    \TopStab f :=
    (\LStabthe - 
    \laap I)^{-1}\!
    \big[(\Eigdiff(f) I -\DiffOpExpNonExp)
    (\Eigveap+f) - \Eigres\big].
\end{align}

Taking norms in \eqref{def:ContOpStab_bis_bis} and using the estimate in Proposition~\ref{prop:stab:lininv}, we obtain
\begin{align}
    \norm{\TopStab(f)}_\XH{1}
    &\leq 
    \getvar{Lstab_inv} 
    \norm{
    (\Eigdiff(f) I -\DiffOpExpNonExp)
    (\Eigveap+f) - \Eigres
    }_\XL
    \nonumber
    \\
    &\leq
    \getvar{Lstab_inv} 
    \Big(
    q_\Eigdiff({\norm{f}_{\XH{1}}})
    \norm{\Eigveap+f}_\XL
    \nonumber
    \\&\qquad
    +
    \norm{\Dcal_\te}_{\XH{1}\to\XL}
    \norm{\Eigveap+f}_\XH{1}
    + 
    \norm{\Eigres}_\XL
    \Big).
    \label{eqn:stab_lemmas:aux1}
\end{align}
Here, $q_\Eigdiff$ is the increasing function defined in Lemma~\ref{lemma:eta:bound}, which we have applied in the last step. Hence, inserting the estimates in Corollary~\ref{coro:DiffOpExpNonExp} and Lemma~\ref{lemma:norm_fv} into \eqref{eqn:stab_lemmas:aux1} yields for all $f\in X$:
\begin{align*}
    0\leq \norm{\TopStab(f)}_\XH{1} \leq
    \getvar{rad_stab}
    \ \Rightarrow\
    \TopStab(X)\subset X.
\end{align*}

Finally, we prove Lemma~\ref{lemma:stab_lipschitz}. We begin by deriving an estimate for $d_f\TopStab$, with $f\in X$. Differentiating \eqref{def:ContOpStab_bis_bis}, we obtain
\begin{align}\label{eqn:stab_lemmas:auxdT}
    d_f\TopStab(h)=
    (\LStabthe-\laap I)^{-1}
        [(f+\Eigveap)d_f\Eigdiff(h)+
        \Eigdiff(f)h-
        \Dcal_\te h].
\end{align}

To apply Proposition~\ref{prop:stab:lininv}, it remains to check that the term inside the brackets belongs to $W_\laap$. Observe that the map
\begin{align*}
    (\Eigdiff(f) I -\DiffOpExpNonExp)(\Eigveap+f) - \Eigres:\XH{1}\to W_\laap,
\end{align*}
takes values in $W_\laap$. Therefore, its Fréchet derivative at $f$ also takes values in $W_\laap$. Consequently,
\begin{align}
    \norm{
    d_f\TopStab(h)
    }_\XH{1}&\leq 
    \getvar{Lstab_inv} 
    \bnorm{
    [(f+\Eigveap)d_f\Eigdiff(h)+
    \Eigdiff(f)h-
    \Dcal_\te h]}_\XL
    \nonumber
    \\
    &\leq 
    \getvar{Lstab_inv} 
    \Big(
    (\norm{f}_\XH{1}+\norm{\Eigveap}_\XL)
    q_{d\Eigdiff}\big(\norm{f}_\XH{1}\big)
    \\&\qquad+
    q_\Eigdiff\big(\norm{f}_\XH{1}\big)
    +
    \norm{\Dcal_\te}_{\XH{1}\to\XL}
    \Big) 
    \norm{h}_\XH{1}.
    \label{eqn:stab_lemmas:aux2}
\end{align}
Here, $q_{d\Eigdiff}$ is the increasing function from Lemma~\ref{lemma:eta:bound_dx}. We have applied this lemma and Lemma~\ref{lemma:eta:bound} again. Inserting the estimates in Corollary~\ref{coro:DiffOpExpNonExp} and Lemma~\ref{lemma:norm_fv} into \eqref{eqn:stab_lemmas:aux2} yields
\begin{align*}
    \norm{
        d_f\TopStab(h)
    }_\XH{1}
    &\leq 
    \getvar{lip_stab} 
    \norm{h}_\XH{1}.
\end{align*}

\appendix

\section{Classical inequalities}\label{sec:clasineq}
In this section, we establish some classical a priori estimates.
\begin{lemma}\label{lemma:apriori:LInfty_H1odd}
Let $f\in \XHevod{1}$. Then the function $f$ is continuous on $\T$ and
\begin{align*}
    \norm{f}_{\XLinf}^2\leq 4\zeta(2) \norm{f}_\XHevod{1}^2.
\end{align*}

\end{lemma}
\begin{proof}
    Continuity of $f$ follows from Morrey's inequality $\mathcal{C}^{s-\frac{1}{2}}\supset\XH{s},\ s>\frac{1}{2}.$ 

    Applying the Cauchy--Schwarz inequality yields
    \begin{align*}
        \norm{f}_{\XLinf}^2
        \leq 
        \Big(\sum_{k\neq0}|\ha{f}_k|\Big)^2
        \leq 
        \Big(\sum_{k\neq0}
        \frac{1}{k^2}\Big)
        \Big(\sum_{k\in\Z}
        k^2|\ha{f}_k|^2\Big)=
        2\zeta(2) \norm{\pa_x f}_{\XL}^2=
        4\zeta(2) \norm{f}_\XHevod{1}^2.
    \end{align*}
\end{proof}

\begin{corollary}
\label{coro:apriori:BanachAlg:H1Even}
    Let $f,g\in\XHeven{1}.$ Then $fg\in\XHeven{1}$ and
    \begin{align*}
        \norm{fg}_{\XHeven{1}}^2\leq\ 8\zeta(2) 
        \norm{f}_{\XHeven{1}}^2 
        \norm{g}_{\XHeven{1}}^2.
    \end{align*}
\end{corollary}

\begin{proof}
    We use
    \begin{align*}
        \norm{fg}_{\XHeven{1}}^2
        &\leq 
        2\norm{\pa_x f g}_\XHodd{0}^2
        +
        2\norm{f \pa_x g}_\XHodd{0}^2\\
        &\leq
        2\norm{g}_{\XLinf}^2\norm{f}_{\XHeven{1}}^2+
        2\norm{f}_{\XLinf}^2\norm{g}_{\XHeven{1}}^2\\
        &\leq8\zeta(2) 
        \norm{f}_{\XHeven{1}}^2 
        \norm{g}_{\XHeven{1}}^2.
    \end{align*}
    
    The last step is a consequence of Lemma~\ref{lemma:apriori:LInfty_H1odd}.
\end{proof}

\begin{lemma}\label{lemma:apriori:Infty_H1:hol_mean0}
    Let $F\in \XHholZMi{1}$. Then $f:=F|_\T$ is continuous on $\T$ and
    \begin{align*}
        \norm{f}_{\XLinf}^2\leq 
        \frac{\pi}{12} 
        \norm{\pa_x f}_\XLIpi^2.
    \end{align*}
\end{lemma} % pi^2/3 \norm{\pa_x f}_\XL^2 \leq  
\begin{proof}
    We proceed as in the proof of Lemma~\ref{lemma:apriori:LInfty_H1odd}. Since
    $f=F|_\T$ has only positive Fourier modes, we have
    \begin{align*}
        \norm{f}_{\XLinf}^2
        \leq 
        \Big(\sum_{k > 0}|\ha{f}_k|\Big)^2
        \leq 
        \Big(\sum_{k > 0}
        \frac{1}{k^2}\Big)
        \Big(\sum_{k>}
        k^2|\ha{f}_k|^2\Big)=
        \zeta(2) \norm{\pa_x f}_{\XL}^2=
        \frac{\pi}{12} \norm{\pa_x f}_\XLIpi^2.
    \end{align*}
    Here we used the Basel identity $\zeta(2)=\frac{\pi^2}{6}$ and the relation
    \begin{align*}
        \norm{f}_\XL^2=\frac{1}{2\pi}\norm{ f}_\XLIpi^2\qquad \text{for all }f\in \XL.
    \end{align*}

\end{proof}

\section{Commutator of Hilbert transform and cosine}
\label{sec:BracketHiCo}

In this section, we do some of the computations required to prove Lemmas~\ref{lemma:exis:td} and \ref{lemma:equiv_stab}.
\begin{lemma}\label{lemma:BracketHiCo}
    Let $f\in\XH{s}$, with $s>\frac{1}{2}.$ We have that
    \begin{align*}
        [H,\cos(j\cdot)] f(x)=
        \frac{i}{2}
        \ha{f}_0
        \big(e^{-ijx}
        -
        e^{ijx}\big)
        +
        \frac{i}{2}\big(
            \ha{f}_j
            -
            \ha{f}_{-j}
        \big)
        +
        i\sum_{m=1}^{j-1}
            \big(
                \ha{f}_me^{-i(j-m)x}
                -
                \ha{f}_{-m}e^{i(j-m)x}
            \big).
    \end{align*}
\end{lemma}
\begin{proof}
    We expand $f$ in Fourier series:
    \begin{align*}
        f:=\sum_{m\in\Z}\ha{f}_m e^{imx}.
    \end{align*}
    
    A direct computation yields
    \begin{align*}
        &[H,\cos(j\cdot)]f(x)
        =
        \frac{1}{2}[H,e^{ij\cdot}+e^{-ij\cdot}]f(x)
        =
        \frac{1}{2}\sum_{m\in\Z}\ha{f}_m[H,e^{ij\cdot}+e^{-ij\cdot}]e^{imx}\\
        &=
        -\frac{i}{2}\sum_{m\in\Z}
        \ha{f}_m\Big[(\sgn(j+m)-\sgn(m))e^{i(j+m)x}
        +
        (\sgn(m-j)-\sgn(m))e^{-i(j-m)x}\Big]
        \\
        &=
        -
        \frac{i}{2}\sum_{m=-j}^0
        \ha{f}_m(\sgn(j+m)-\sgn(m))e^{i(j+m)x}
        +
        \frac{i}{2}\sum_{m=0}^j
        \ha{f}_m(\sgn(j-m)+\sgn(m))e^{-i(j-m)x}
        \\
        &=
        -\frac{i}{2}
        \sum_{m=0}^j
        \ha{f}_{-m}(\sgn(j-m)+\sgn(m))e^{i(j-m)x}
        +
        \frac{i}{2}
        \sum_{m=0}^j
        \ha{f}_m(\sgn(j-m)+\sgn(m))e^{-i(j-m)x}
        \\
        &=
        \frac{i}{2}
        \sum_{m=0}^j
        (\sgn(j-m)+\sgn(m))
        (\ha{f}_me^{-i(j-m)x}
        -
        \ha{f}_{-m}e^{i(j-m)x})
        \\
        &=
        \frac{i}{2}
        (e^{-ijx} - e^{ijx})
        \ha{f}_0
        +
        \frac{i}{2}
        (\ha{f}_j - \ha{f}_{-j})
        +
        i\sum_{m=1}^{j-1}
        (\ha{f}_me^{-i(j-m)x}
        -
        \ha{f}_{-m}e^{i(j-m)x}).
    \end{align*}
\end{proof}

\begin{corollary}\label{coro:BracketHiTw:df}
    Let $f$ be a function in $\XH{s},\ s>\frac{3}{2}$. We have that
    \begin{align*}
        [H,\twap]\pa_x f(x)=
        \sum_{j=1}^N
        \twap_j
        \Big(
            -\frac{j}{2}\big(
                \ha{f}_j+\ha{f}_{-j}
            \big)
            -
            \sum_{m=1}^{j-1} m \big(
                    \ha{f}_m e^{-i(j-m)x}
                    +
                    \ha{f}_{-m} e^{i(j-m)x}
                \big)
        \Big).
    \end{align*}
\end{corollary}
\begin{proof}
    $\pa_x f$ is zero-mean. Then, substituting $\pa_x f$ into Lemma~\ref{lemma:BracketHiCo} we obtain
    \begin{align*}
        [H,\cos(j\cdot)]\pa_x f(x)
        &=
        -\frac{j}{2}(\ha{f}_j+\ha{f}_{-j})
        -\sum_{m=1}^{j-1}
        m(\ha{f}_me^{-i(j-m)x}+\ha{f}_{-m}e^{i(j-m)x}).
    \end{align*}
    Therefore, we get that
    \begin{align*}
        [H,\twap]\pa_x f(x)
        &=
        \sum_{j=1}^N
            \twap_j
            [H,\cos(j\cdot)]
            \pa_x f(x)\\
        &=
        \sum_{j=1}^N
            \twap_j
            \Big(
                -\frac{j}{2}\big(
                \ha{f}_j
                +
                \ha{f}_{-j}
                \big)
                -\sum_{m=1}^{j-1} m \big(
                    \ha{f}_me^{-i(j-m)x}
                    +
                    \ha{f}_{-m}e^{i(j-m)x}
                \big)
            \Big).
    \end{align*}
\end{proof}

\begin{corollary}\label{coro:BracketHiTw:f}
    Let $f$ be a function in $\XH{s},\ s>\frac{3}{2}$. We have that
    \begin{align*}
        [H,\twap](\pa_x f+i\te f)(x)&=
        -\frac{\te}{2}
        \ha{f}_0
        \sum_{j=1}^N
        \twap_j
        \big(e^{-ijx}
        -
        e^{ijx}\big)
        -
        \frac{1}{2}
        \sum_{j=1}^N
        \twap_j
        \Big(
            (j+\te)
            \ha{f}_j
            +
            (j-\te)
            \ha{f}_{-j}
        \Big)
        \\
        &\qquad -\sum_{j=1}^N
        \twap_j
        \sum_{m=1}^{j-1}
            \big(
                (m+\te)
                \ha{f}_m e^{-i(j-m)x}
                +
                (m-\te)
                \ha{f}_{-m} e^{i(j-m)x}
            \big)
    \end{align*}
\end{corollary}
\begin{proof}
    A direct substitution yields the desired formula:
    \begin{align*}
        &[H,\twap](\pa_x f+i\te f)(x)=
        \sum_{j=1}^N
        \twap_j
        \Big(
            -\frac{j}{2}\big(
                \ha{f}_j+\ha{f}_{-j}
            \big)
            -
            \sum_{m=1}^{j-1} m \big(
                    \ha{f}_m e^{-i(j-m)x}
                    +
                    \ha{f}_{-m} e^{i(j-m)x}
                \big)
        \Big)+\\
        &
        \qquad i\te \sum_{j=1}^N 
        \twap_j
        \Big(
        \frac{i}{2}
        \ha{f}_0
        \big(e^{-ijx}
        -
        e^{ijx}\big)
        +
        \frac{i}{2}\big(
        \ha{f}_j
        -
        \ha{f}_{-j}\big)
        +
        i\sum_{m=1}^{j-1}
            \big(\ha{f}_me^{-i(j-m)x}
            -
            \ha{f}_{-m}e^{i(j-m)x}\big)
        \Big)
        \\
        &=
        -\frac{\te}{2}
        \ha{f}_0
        \sum_{j=1}^N
        \twap_j
        \big(e^{-ijx}
        -
        e^{ijx}\big)
        -
        \frac{1}{2}
        \sum_{j=1}^N
        \twap_j
        \big(
            (j+\te)
            \ha{f}_j
            +
            (j-\te)
            \ha{f}_{-j}
        \big)
        \\
        &\quad -\sum_{j=1}^N
        \twap_j
        \sum_{m=1}^{j-1}
            \big(
                (m+\te)
                \ha{f}_m e^{-i(j-m)x}
                +
                (m-\te)
                \ha{f}_{-m} e^{i(j-m)x}
            \big).
    \end{align*}
\end{proof}

\section{Palindromic polynomial}\label{sec:polpaly}
In this section, we explain how to work rigorously with the roots of the polynomial $z^N P(z)$, where $P$ is defined at \eqref{def:P_}. Furthermore, we also prove Lemma~\ref{lemma:Proots1} and Lemma~\ref{lemma:Proots2}.

\begin{rem}\label{rem:poly_trav}
    The polynomial $z^N P(z)$ is related to $c$ and $\twap$ by the following identity on $\T$
    \begin{align*}
        P(e^{ix}) =
        c+\twap(x).
    \end{align*}
\end{rem}
A first observation about $z^N P(z)$ is that it is palindromic, since
\begin{align*}
    z^N P(z) =
    c z^N + \frac{1}{2}\sum_{j=1}^N \twap_j \big(z^{N+j}+z^{N-j}\big).
\end{align*}
Such polynomials have an interesting property, as the next lemma states.
\begin{lemma}\label{lemma:pal_root_reduc}
    Let $d\in\N$ and $Q(z)=\sum_{k=0}^d Q_k z^k$ be a polynomial with complex coefficients. Assume $Q$ is palindromic, i.e.
    \begin{align*}
        Q_k = Q_{d-k}\ \forall k=0,\ldots,d.
    \end{align*}

    Then, for every root $\al\in\C$ of $Q$, the inverse $\al^{-1}$ is also a root of $Q$.
\end{lemma}
\begin{proof}
    Since the degree of $Q$ is $d$, we know that $0\neq Q_d = Q_0$. Consequently, $0$ is never a root of $Q$.
    
    We now compute $z^d Q(z^{-1})$ for $z\neq 0$:
    \begin{align*}
        z^d Q(z^{-1})=z^d \sum_{k=0}^d Q_k z^{-k}
        = \sum_{k=0}^d Q_k z^{d-k}
        = \sum_{k=0}^d Q_{d-k} z^{k}
        = \sum_{k=0}^d Q_{k} z^{k}=Q(z).
    \end{align*}

    Thus, if $\al$ is a root of $Q$, we have 
    $\al^d Q(\al^{-1})=0$, so $\al^{-1}$ is again a root of $Q$.
\end{proof}

In the supplementary material we attach a file with approximations of $N$ roots of $z^N P(z)$ inside the unit disk. The next lemma shows how accurate these approximations are.

\begin{lemma}\label{lemma:enclose_poly_roots}
    For every approximated root $\ti{\al}$ of $z^N P(z)$ lying inside the unit complex disk, and whose explicit expression is given in the supplementary material, there exists a unique root $\al\in D_r(\ti{\al})$, where $\overline{D}_r(\ti{\al})$ denotes the closed disk with center $\ti{\al}$ and radius $r=\getvar{radii_roots}$.

    Furthermore, given any other approximated root $\ti{\al}'$, the disks $D_r(\ti{\al})$ and $D_r(\ti{\al}')$ are disjoint.

    In addition, we also know that no disk $\overline{D}_r(\ti{\al})$ intersects the unit circle.
\end{lemma}

\begin{proof}
    Although this proof is computer-assisted, we include it here rather than in Appendix~\ref{sec:compassisted}. We denote by $Q(z):=z^N P(z)$ and
    \begin{align*}
        Q(z) =: \sum_{j=0}^{2N} Q_j z^j.
    \end{align*}
    Let $\ti{\al}\in\C$ be one of the numerical
    approximations provided in the supplementary material. The goal is to prove that there exists a unique root $\al$ of $Q$ in the closed disk $\overline{D}_r(\ti{\al})$, with $r=\getvar{radii_roots}$.
    
    This verification is performed by the routine \verb|enclose_roots()| in \verb|lemmas.py|. The method proceeds by looking for a correction $\delta\in\C$ such that
    \begin{align*}
        Q(\ti{\al}+\delta)=0.
    \end{align*}
    
    Expanding $Q$ in Taylor series at $\ti{\al}$ yields
    \begin{align*}
        0
        = Q(\ti{\al})
        + \delta \pa_z Q(\ti{\al})
        + R(\ti{\al},\delta),
    \end{align*}
    where the remainder is
    \begin{align*}
        R(z,\delta)
        := \sum_{k=2}^{2N} \frac{\delta^k}{k!} \pa_z^k Q(z).
    \end{align*}
    
    The code defines the operator $T_{\ti{\al}}:\overline{D}_r(0)\to\C$, given by
    \begin{align*}
        T_{\ti{\al}}(\delta)
        := -\frac{1}{\pa_z Q(\ti{\al})}
        \bigl(Q(\ti{\al}) + R(\ti{\al},\delta)\bigr),
    \end{align*}
    so that fixed points of $T_{\ti{\al}}$, when added to $\ti{\al}$, correspond exactly to roots of
    $Q$ near $\ti{\al}$.
    
    To verify that $T_{\ti{\al}}$ is a contraction, the code first computes coefficient-based upper bounds for the derivatives $\pa_z^k Q$ on the unit disk. These bounds are implemented using the explicit formula
    \begin{align*}
        |\pa_z^k Q(z)|
        \leq
        \sum_{j=0}^{2N} \frac{j!}{(j-k)!} |Q_j|,
        \qquad |z|\leq 1,
    \end{align*}
    which depends only on the coefficients of $Q$.
    
    Using these bounds, the routine constructs explicit enclosures for both the remainder $R(\ti{\al},\delta)$ and its $\delta$--derivative, valid for all $|\delta|\leq r$.
    The two contraction conditions
    \begin{align*}
        T_{\ti{\al}}\big(\overline{D}_r(0)\big) \subset \overline{D}_r(0),
        \qquad
        |\pa_z T_{\ti{\al}}\big(\overline{D}_r(0)\big)| < 1
    \end{align*}
    are then verified numerically using ball arithmetic, ensuring full rigor.
    
    Once these conditions are certified, the contraction mapping theorem guarantees the existence of a unique fixed point $\delta\in \overline{D}_r(0)$, and hence a unique root $\al=\ti{\al}+\delta$ of $Q$ inside $\overline{D}_r(\ti{\al})$.
    
    Finally, the same routine checks that the disks corresponding to distinct approximations $\ti{\al}$ are pairwise disjoint and strictly contained in the open unit disk.
\end{proof}

\begin{corollary}\label{coro:Nrootsdisk}
    The polynomial $z^N P(z)$ has exactly $N$ roots inside the unit complex open disk $\D$. Recall this set is denoted by $Z$.
\end{corollary}
\begin{proof}
    We know that there are at least $N$ simple roots inside the unit complex disk. Then, applying Lemma~\ref{lemma:pal_root_reduc} we find other $N$ different roots with modulus greater than one. Since the degree of $z^NP(z)$ is $2N$, we have found every root and there are exactly $N$ contained in $\D$.
\end{proof}

\section{Treatment of the solutions of the differential equations}\label{sec:ode}
Several parts of the proof require working with the solutions of linear first-order ODEs on $[-\pi,\pi]$. In particular, we refer to the `$h$ functions` from \eqref{def:hk}, \eqref{def:Piplusminu}, and \eqref{def:hkplusminu}. Even though these solutions are explicit, they involve heavy computational operations, which make it difficult to work directly with them.

To overcome this problem, we provide high-accuracy approximations, which are more manageable, of the solutions in the supplementary data. 
The error between the exact and the approximate solution will be measured by evaluating explicit residuals and applying Duhamel-type representations.

\subsection{Existence argument}\label{sec:ode_exis}
Recall the definition of $h_k$ with $k=1,\ldots,N+1$ in \eqref{def:hk}. These functions arise in the construction of the approximate inverse operator used in the existence argument, and their accurate numerical evaluation is a key ingredient in our proof.

In this section, we explain how to construct continuous approximations, $h_k^{\rm ap}$, of $h_k$ and provide rigorous bounds for the errors
\begin{align*}
    \tau_k := h_k - h_k^{\rm ap}.
\end{align*}

More precisely, we prove the following lemma.
\begin{lemma}\label{lemma:ode_exis}
    For each $k=1,\ldots,N+1$, the approximation error satisfies
    \begin{align*}
        \norm{\tau_k}_\XLIpi^2
        < 
        \rho_k.
    \end{align*}
    Here, the quantities $\rho_k$ are explicit and can be found in the file \verb|exis_resis.txt| from the supplementary material.
\end{lemma}

In order to construct the approximations $h_k^{\rm ap}$, we partition the domain
$[-\pi,\pi]$ into subintervals of varying length, and denote this partition by
$\mathcal{J}$. Let $J=[x_0-r,x_0+r]\in\mathcal{J}$ be one of these intervals. For every $k=1,\ldots,N+1$, we provide in the supplementary material \verb|exis_data.txt| the coefficients $h_{N+1,l}^{{\rm ap},J},$ for $l=1,\ldots,5$. We define with them the function $h_{k}^{{\rm ap},J}$, given by
\begin{align*}
h_{k}^{{\rm ap},J}(x):=\delta_{k,N+1} + \sum_{l=1}^{4} h_{k,l}^{{\rm ap},J} (x-x_0)^l.
\end{align*}

Observe that, by definition, each $h_k$ solves 
\begin{align}\label{eqn:exis_global}
    \left.\begin{aligned}
        \big(c+\twap(x)\big) \partial_x h_{k}(x)  &= i h_{k}(x) +g_k(x) ,\\
        h_{k}(0) &= \delta_{k,N+1},
    \end{aligned}\right\}
\end{align}
where $g_k$ is given by \eqref{def:gk_exis} for $k=1,\ldots,N$ and $g_{N+1}=0$.
The function $h_{N+1}$ corresponds to a homogeneous equation, while the remaining solutions include an affine term. 

The data defining the functions $h_k^{\rm ap,J}$ have been constructed to approximate the functions $h_{k}^{J}:J\to\C$, solving
\begin{align}\label{eqn:exis_local}
    \left.\begin{aligned}
        \big(c+\twap(x)\big) \partial_x h_{k}^{J}(x) &= i h_{k}^{J}(x) + g_k(x),\\
        h_{k}^{J}(x_0) &= \delta_{k,N+1},
    \end{aligned}\right\}
\end{align}
The coefficients $h_{k,l}^{{\rm ap},J}$ have been chosen so that they minimize the function
\begin{align*}
    (
    h_{k,1}^{{\rm ap},J},
    h_{k,2}^{{\rm ap},J},
    h_{k,3}^{{\rm ap},J},
    h_{k,4}^{{\rm ap},J}
    )
    \mapsto 
    \norm{
        (c+\twap)  \pa_x h_{k}^{{\rm ap},J} 
        -
        i h_{k}^{{\rm ap},J} 
        -
        g_k
    }_{L^2(J)}^2.
\end{align*}

With these pieces, we are able to construct continuous approximations of $h_k$, namely $h_k^{\rm ap}$, for every $k=1,\ldots,N+1$. 

The construction of $h_k$ from $h_k^{\rm ap}$ is done using the following idea.
Let $J_1,J_2\in\mathcal{J}$ be two consecutive intervals with a common endpoint $x_\ast := J_1\cap J_2$. On each interval $J$, we have:
\begin{itemize}
    \item A local approximation $h^{{\rm ap},J}_{N+1}$ on $J$ of the function $h^{J}_{N+1}$ solving \eqref{eqn:exis_local};
    \item For each $k=1,\dots,N$, a local approximation $h^{{\rm ap},J}_k$ of the function $h^{J}_{k}$ solving \eqref{eqn:exis_local}.
\end{itemize}

Since the equation is linear, any other smooth solution $\ti{h}_k^J$ on $J$ to \eqref{eqn:exis_local} is written as 
\begin{align*}
    \ti{h}_k^J = C  h^{J}_{N+1} + h^{J}_k,
\end{align*}
for a suitable constant $C\in\mathbb{C}$. This allows us to choose $C$ on each subinterval so that $h_k^{\rm ap}$ is continuous and satisfies the same initial condition as $h_k$, that is $h_k(0)=\delta_{k,N+1}$. 

Fix $k=1,\ldots,N+1$ and assume that we have already constructed on $J_1$ the approximation $h_k^{\rm ap}|_{J_1}$, of the function $h_k$ solving \eqref{eqn:exis_global}. To extend it to $J_2$ by continuity, we choose $C$ so that the value at the interface $x_\ast$ coincides:
\begin{align*}
    h_k^{\rm ap}|_{J_1}(x_\ast)=
    C  h^{{\rm ap},J_2}_{N+1}(x_\ast) + h^{{\rm ap},J_2}_k(x_\ast).
\end{align*}

Equivalently,
\begin{align*}
    C
    = \frac{h_k^{\rm ap}|_{J_1}(x_\ast) - h^{{\rm ap},J_2}_k(x_\ast)}{h^{{\rm ap},J_2}_{N+1}(x_\ast)}.
\end{align*}

Hence, we can continuously extend $h_k^{\rm ap}|_{J_1}$ to $h_k^{\rm ap}|_{J_1\cup J_2}$
\begin{align}
    h_k^{\rm ap}|_{J_2}(x)
    := C  h^{{\rm ap},J_2}_{N+1}(x) + h^{{\rm ap},J_2}_k(x).
    \label{def:hjexisap}
\end{align}

We apply this procedure iteratively from $x=0$, imposing
\begin{align}
    h_k^{\rm ap}(0)=\delta_{k,N+1}
    \label{def:hjapexis:iniconds}
\end{align}
thereby obtaining a continuous approximation on the whole interval $[-\pi,\pi]$.

The routine \verb|_ode_values()| from \verb|load_data.py| splits the domain into positive and negative parts using \verb|_split_blocks_by_domain()|, and constructs the solutions on each domain using \verb|_cons_sol_0_x|.

\paragraph{Working with the approximations.}
Fix one of the intervals $J=[x_0-r,x_0+r]\in\mathcal{J}$ and $k=1,\ldots,N+1$. On $J$, the approximating function $h^{\rm ap}:=h_k^{\rm ap}$ is a polynomial of degree $N_r=5$ of the form
\begin{align}\label{eqn:hk_ap_local}
    h^{\rm ap}(x)
    =
    \sum_{l=0}^{N_r} \al^J_l (x-x_0)^l,
\end{align}
where the coefficients $\al^J_l$ have been determined from the stored data \verb|exis_data.txt| and the iterative process from \eqref{def:hjexisap} and \eqref{def:hjapexis:iniconds}.

\begin{defi}\label{def:CSEfuncs}
    Let $J=[x_0-r,x_0+r]$. Fix $l\geq 0$ and $j\in\Z$, define:
    \begin{align*}
        C_J^{l,j} :=
        \int_J (x-x_0)^l  \cos(jx)\dx,\qquad
        S_J^{l,j} :=
        \int_J (x-x_0)^l  \sin(jx)\dx,\qquad
        \overline{E}_J^{l,j} :=
        \int_J (x-x_0)^l  e^{-ijx}\dx.
    \end{align*}
\end{defi}

With the notation introduced in Definition~\ref{def:CSEfuncs}, $j$-th Fourier coefficient of $h^{\rm ap}(x)$ is computed as follows:
\begin{align*}
    (\ha{h^{\rm ap}})_j 
    =
    \frac{1}{2\pi}
    \sum_{J\in\mathcal{J}}
    \sum_{l=0}^{N_r}
    \al_l^J
     \overline{E}_J^{l,j}.
\end{align*}

\begin{note}
    The integrals from Definition~\ref{def:CSEfuncs} are thought to be precomputed. For example, the routine \verb|build_CS_for_jj()| from \verb|auxiliar_funcs.py| does this job. In particular constructs $C_J^{l,j}$ and $S_J^{l,j}$ with $l$ from $0$ up to a specified order, and $j$ fixed. In addition, the function \verb|precompute_expo_vals()|, which builds, for $l$ from $0$ up to a specified order, and $j=$ \verb|n_min|,...,\verb|n_min| $+$ \verb|total|,  the quantities
    \begin{align*}
        \overline{E}_J^{l,j} = 
        C_J^{l,j} - i S_J^{l,j}.
    \end{align*}
\end{note}

\begin{proofsth}{Lemma~\ref{lemma:ode_exis}}
    
    Although this proof contains some computer-assisted parts, we present it here rather than in Appendix~\ref{sec:compassisted}.
    
    The goal of this section is to provide a rigorous $L^2$ bound for the approximation error
    \begin{align*}
        \tau_{k} := 
        h_k - h_k^{\rm ap}, \qquad k=1,\ldots,N+1. 
    \end{align*}
    
    More precisely, we prove that for every $k=1,\ldots,N+1$,
    \begin{align*}
        \norm{\tau_{k}}_\XLIpi^2 < 
        \rho_k,
    \end{align*}
    where the corresponding quantities $\rho_k$ are listed in the file \verb|exis_resis.txt|.
    
    Notice that the function $\tau_{k}$ satisfies a new ODE on each $J$,
    \begin{align}
        \big(c+\twap(x)\big) 
        \pa_x\tau_{k}(x) &= 
        \big(c+\twap(x)\big) 
        \pa_x h_k (x)
        - 
        \big(c+\twap(x)\big) 
        \pa_x h_k^{\rm ap}(x)
        \nonumber\\
        &= 
        i h_k (x) + g_k(x) -
        \big(c+\twap(x)\big) 
        \pa_x h_k^{\rm ap}(x)
        \nonumber\\
        &= 
        i (\tau_{k} + h_k^{\rm ap}(x))
        + 
        g_k (x)
        -
        \big(c+\twap(x)\big) 
        \pa_x h_k^{\rm ap}(x)
        \nonumber\\
        &= 
        i \tau_{k}(x) - \xi_k(x),
        \label{eqn:ode:tauexisj}
    \end{align}
    where the function $\xi_k$ is explicit and defined by 
    \begin{align}\label{def:xik_exis}
        \xi_k :=
        (c+\twap) \pa_x h_k^{\rm ap} - i h_k^{\rm ap} - g_k.
    \end{align}
    
    Since $h_k$ is smooth and $h_{k}^{\rm ap}$ is piecewise smooth, the derivative of $\tau_{k}$ will not be continuous on every point. However, we can obtain $\tau_k$ from $\xi_k(x)$ using Duhamel's formula,
    \begin{align}
        \tau_{k}(x) & = 
        \underbrace{\tau_{k}(0)}_0 - 
        \PiExis(x) \int_0^x \frac{\beta(y)}{\PiExis(y)}  
        \xi_k(y)\dy
        \nonumber
        \\
        & =
        - J \xi_k(x),
        \label{eqn:tauexis_Jxi}
    \end{align}
    where $J$ is defined in \eqref{def:Jexis}.
    \begin{rem}
        Continuity of $\tau_{k}$ is needed to apply Duhamel's formula. In our construction, $h_k^{\rm ap}$ is continuous on $[-\pi,\pi]$, and therefore $\tau_{k}=h_k-h_k^{\rm ap}$ is continuous as well.
    \end{rem}
    
    Therefore, we apply Lemma~\ref{lemma:J_L2L2_exis} to obtain a $\XLIpi$ bound of $\tau_{k}$ in terms of $\xi_k$. 
    
    The whole routine is carried out by the function \verb|ode_exis_L2()| from \verb|lemmas.py|.
\end{proofsth}

\paragraph{Decomposition of $\norm{\xi_k}_\XLIpi^2$.}
This section is devoted to explain the formulas we use to estimate $\norm{\xi_k}_\XLIpi$ in the proof of Lemma~\ref{lemma:ode_exis}.

We decompose the $\XLIpi$-norm into a sum of local $L^2$ contributions over the subintervals $J\in \mathcal{J}$. Accordingly, for each fixed interval $J$ we expand $\norm{\xi_k}_{L^2(J)}$ as follows:
\begin{align*}
    \norm{\xi_k}_{L^2(J)}^2
    =
    \sum_{l=1}^5 \Iexis_l(h_k^{\rm ap};J) + 
    \norm{g_k}_{L^2(J)}^2
\end{align*}
where
\begin{subequations}\label{def:Iexis}
    \begin{empheq}{align}
        \Iexis_1(h;J)
        &:=
        \norm{h}_{L^2(J)}^2,
        \\
        \Iexis_2(h;J)
        &:=
            \norm{
                (c+\twap)
                \pa_x h
            }
            _{L^2(J)}^2,
        \\
        \Iexis_3(h;J)
        &:=
        2\Re
        \big(
            (c+\twap)\pa_x h, 
            -i h
        \big)_{L^2(J)}
        ,\\
        \Iexis_{4,k}(h;J)
        &:=
        2\Re
        \big(
            i h, g_k
        \big)
        _{L^2(J)}
        \\
        \Iexis_{5,k}(h;J)
        &:=
        2\Re
        \big(
            (c+\twap)\pa_x h
            , 
            - g_k
        \big)
        _{L^2(J)}.
    \end{empheq}
\end{subequations}

Recall that $g_k$ is given by
\begin{align*}
    g_k(x):=\frac{i}{2} 
    k \twap_k+
    k\sum_{m=1}^{N-k}
    \twap_{k+m} 
    \sin(m x).
\end{align*}
Notice that we can compute $\norm{g_k}_\XLIpi$ using Parseval's identity,
\begin{align}
    \frac{1}{2\pi}\norm{g_k}_\XLIpi^2 
    = 
    \sum_{j\in\Z} |(\ha{g_k})_j|^2
    = 
    \frac{k^2}{4} 
    (\twap_k)^2 + 
    \frac{k^2}{2}
    \sum_{m=1}^{N-k}
    (\twap_{k+m})^2,\label{eqn:I6exis}
\end{align}
We compute $\norm{g_k}_\XLIpi^2$ for all $k=1,\ldots,N+1$, where $g_{N+1}=0$, in \verb|gk_exis_norm_sq()|.

In order to obtain formulas for \eqref{def:Iexis}, we will write $h$ on $J$ as 
\begin{align}
    h(x):=\sum_{l=0}^{N_r}
    \al_l 
    (x-x_0)^l.
    \label{def:localhalpha}
\end{align}

The quantity $\Iexis_{1}(h;J)$ is computed as follows:
\begin{align}  
    \Iexis_1(h;J)&=
    \norm{h}_{L^2(J)}^2
    =
    \sum_{\substack{l_1=0\\l_2=0}}^{N_r}
        \al_{l_1} 
        \overline{\al_{l_2}} 
        \IIexis_1(l_1,l_2;J),
        \nonumber\\
    \IIexis_1(l_1,l_2;J)&:=
        C_{J}^{l_1+l_2,0},\label{def:II1exis}
\end{align}
where $C_{J}^{l,j}$ is defined in Definition~\ref{def:CSEfuncs} and 
\begin{align*}
    C_{J}^{l_1+l_2,0} = \int_J (x-x_0)^{l_1+l_2} \dx = 
    \frac{r^{l_1+l_2+1}}{l_1+l_2+1} (1+(-1)^{l_1+l_2}).
\end{align*}

$\Iexis_2(h;J)$ and $\Iexis_3(h;J)$ are computed in a similar way.
Define $\twap_0:=c$ and
\begin{align}\label{def:DSexis}
    D[v,v](d) &:= \sum_{j_1-j_2=d}
    v_{j_1} 
    \overline{v_{j_2}},
    \qquad 
    \qquad 
    S[v,v](s) := \sum_{j_1+j_2=s} 
    v_{j_1} 
    \overline{v_{j_2}},\\
    \IIexis_{2}(l_1,l_2;J)
    &:=
    \frac{l_1 l_2}{2}
    \sum_{\substack{j_1=0\\j_2=0}}^N 
    \twap_{j_1}
    \twap_{j_2} 
    \big(
        C_J^{l_1+l_2-2,j_1-j_2}
        +
        C_J^{l_1+l_2-2,j_1+j_2}
    \big)\nonumber
    \\
    &=
    \frac{l_1 l_2}{2}\Big(
    \sum_{d=-N}^N 
        D[\twap,\twap](d) 
        C_J^{l_1+l_2-2,d}
    +
    \sum_{s=0}^{2N} 
        S[\twap,\twap](s) 
        C_J^{l_1+l_2-2,s}
        \Big)
        \label{def:II2exis}.
\end{align}
\begin{rem}
    We only have defined $C_{J}^{l,j}$ for $l\geq 0$. In \eqref{def:II2exis}, $l_1+l_2-2 < 0$ implies $l_1l_2=0$, hence $\IIexis_{2}(l_1,l_2;J)=0$. For simplicity, we adopt the convention of allowing $C_{J}^{l,j}$ with negative $l$  whenever it is multiplied by a zero coefficient throughout the rest of the section.
\end{rem}

Then, for $\Iexis_2(h;J)$ we obtain
\begin{align*}
    \Iexis_2(h;J)
    &=
    \norm{
        (c+\twap) 
        \pa_x h
    }_{L^2(J)}^2
    \nonumber
    \\
    &=
    \sum_{\substack{l_1=0\\l_2=0}}^{N_r}
    \sum_{\substack{j_1=0\\j_2=0}}^N 
    \int_J
        l_1 l_2 
        \twap_{j_1}
        \twap_{j_2} 
        \al_{l_1}
        \overline{\al_{l_2}} 
        \cos(j_1x)\cos(j_2x) 
        (x-x_0)^{l_1+l_2-2}
    \dx
    \nonumber
    \\
    &=
    \sum_{\substack{l_1=0\\l_2=0}}^{N_r}
    \al_{l_1} 
    \overline{\al_{l_2}} 
    \frac{l_1 l_2}{2}
    \sum_{\substack{j_1=0\\j_2=0}}^N 
        \twap_{j_1}
        \twap_{j_2} 
        \big(
            C_J^{l_1+l_2-2,j_1-j_2}
            +
            C_J^{l_1+l_2-2,j_1+j_2}
        \big)
    \nonumber
    \\
    &=
    \sum_{\substack{l_1=0\\l_2=0}}^{N_r}
    \al_{l_1} 
    \overline{\al_{l_2}} 
    \IIexis_{2}(l_1,l_2;J).
\end{align*}

For $\Iexis_3(h;J)$ we have that
\begin{align*}
    \Iexis_3(h;J)&=
    2\Re
    \big(
        (c+\twap)\pa_x h, 
        -i h
    \big)_{L^2(J)}
    =
    2\Re
    \int_J 
        i \big(c+\twap(x)\big) 
        \pa_x h(x) 
        \overline{h(x)}
    \dx
    \\&=
    2\Re
    \sum_{\substack{l_1=0\\l_2=0}}^{N_r}
    \sum_{j=0}^N 
        il_1
        \twap_{j} 
        \al_{l_1} 
        \overline{\al_{l_2}} 
        C_J^{l_1+l_2-1,j}
    =
    \sum_{\substack{l_1=0\\l_2=0}}^{N_r}
        \Re\big(\al_{l_1} 
        \overline{\al_{l_2}} 
        \IIexis_{3}(l_1,l_2;J)\big).
\end{align*}
where $\IIexis_{3}(l_1,l_2;J)$ is defined as
\begin{align}
    \IIexis_{3}(l_1,l_2;J) :=
    2il_1\sum_{j=0}^N 
        \twap_{j} 
        C_J^{l_1+l_2-1,j},\label{def:II3exis}
\end{align}

Finally, we proceed to compute $\Iexis_{4,k}(h;J)$ and $\Iexis_{5,k}(h;J)$. Here, we will also need $S_{J}^{l,j}$ from Definition~\ref{def:CSEfuncs}. Then,
\begin{align*}
    \Iexis_{4,k}(h;J)&=
    2\Re
    \big(
        i h, g_k
    \big)_{L^2(J)}
    =
    2\Re
    \int_J
        \big(
        ih(x)\overline{g_k(x)}
        \big)\dx
    \\
    &=
    2k
    \sum_{l=0}^{N_r}
    \Re
    \int_J
    i\al_l 
    \Big(
        -\frac{i}{2}
        \twap_k
        +
        \sum_{m=1}^{N-k} 
            \twap_{m+k}
            \sin(mx)
    \Big) 
    (x-x_0)^l
    \dx
    \\
    &=
    \sum_{l=0}^{N_r}
    \Re\big(
    \al_l
    \IIexis_{4}(l,k;J)
    \big),
\end{align*}
where
\begin{align}
    \IIexis_{4}(l,k;J)
    :=2k\Big(
            \frac{1}{2}
            \twap_k 
            C_J^{l,0}
            +
            i\sum_{m=1}^{N-k} 
                \twap_{m+k} 
                S_J^{l,m}
        \Big).
    \label{def:II4exis}
\end{align}

Finally, we have for $\Iexis_{5,k}(h;J)$ the computations below
\begin{align*}
    \Iexis_{5,k}(h;J)&=
    -2\Re
    \int_J
        (c+\twap(x)) 
        \pa_x h(x) 
        \overline{g_k(x)}
    \dx
    \\
    &=
    \sum_{l=0}^{N_r}
    \Re(
    \al_l
    \IIexis_{51,k}(l;J)
    ),
\end{align*}
where
\begin{subequations}\label{def:II5exis}
    \begin{empheq}{align}
    \IIexis_{51,k}(l;J)&:=
        \sum_{j=0}^N
        lk 
        \twap_{j} 
        \Big(i
            C_J^{l-1,j} 
            \twap_k
            -
            \sum_{m=1}^{N-k}
                \twap_{m+k}
                \big(
                    S_J^{l-1,m+j}
                    +
                    S_J^{l-1,m-j}
                \big)
        \Big)
        \nonumber
        \\
        &=
        ilk \twap_k\sum_{j=0}^N
        \twap_{j} 
            C_J^{l-1,j}
        -
        lk
        \sum_{m=1}^{N-k}
            \twap_{m+k}
            \IIIexis_{5,m}(l;J),
        \\
    \IIIexis_{5,m}(l;J)
        &:=
        \sum_{j=0}^N
            \twap_{j} 
            \big(
                S_J^{l-1,m+j}
                +
                S_J^{l-1,m-j}
            \big).
    \end{empheq}
\end{subequations}

\begin{note}
    In order to compute \eqref{def:Iexis} for every $k=1,\ldots,N+1$, we precompute some data. In particular, there is a different routine for each \eqref{def:II1exis}--\eqref{def:II5exis} and most of them are precomputed in \verb|pre_resis_exis()| for each different interval $J$. We then use these precomputed data to compute the sum of the terms of \eqref{def:Iexis} for a fixed interval $J$, this part is carried out by the routine \verb|I_sols_resi_exis_from_precomp()|.
\end{note}

\subsection{Stability argument}\label{sec:ode_stab}
This section is analogous to Appendix~\ref{sec:ode_exis}, but now working with  \eqref{def:Piplusminu} and \eqref{def:hkplusminu} instead of \eqref{def:hk}. Indeed, in \eqref{def:Piplusminu} and \eqref{def:hkplusminu} we construct $4N+2$ functions depending on the parameters $\la$ and $\te$. We introduce the following notation, which fixes an ordering of these functions:
\begin{align}
    h_k=\begin{cases}
        h_{k}^{+J^{+}},&1\leq k\leq N\\
        h_{k-N}^{-J^{+}},&N+1\leq k\leq 2N\\
        h_{k-2N}^{+J^{-}},&2N+1\leq k\leq 3N\\
        h_{k-3N}^{-J^{-}},&3N+1\leq k\leq 4N\\
        \Pi^+,&k=4N+1,\\
        \Pi^-,&k=4N+2.
    \end{cases}
    \label{def:hkstabsorted}
\end{align}

Since we have already described how we construct the continuous approximations in Appendix~\ref{sec:ode_exis}, our goal in this section is to prove an estimate for the difference between $h_k$ and the functions approximating it: $h_k^{\rm ap}$. 
The data used to construct the $h_k^{\rm ap}$ with $\la=\laap$ is found in \verb|sing_p_data.txt| and \verb|sing_m_data.txt|, while the data used to construct the approximations with $\la=\getvar{la_reg}$ is stored in \verb|regu_p_data.txt| and \verb|regu_m_data.txt|.

These functions are solutions of the following ODEs
\begin{align*}
    \big(c+\twap(x)\big)
    \pa_x h_k^{\pm J^+}(x) &=
    (i+\la) h_k^{\pm J^+}(x) -
    i\te^+
    \big(c+\twap(x)\big)
     h_k^{\pm J^+}(x) +
    g_{k,\te^\pm}^\pm,\\
    \big(c+\twap(x)\big)
    \pa_x h_k^{\pm J^-}(x) &=
    (i-\la) h_k^{\pm J^-}(x) -
    i\te^-
    \big(c+\twap(x)\big)
     h_k^{\pm J^+}(x) +
    g_{k,\te^\mp}^\pm,
\end{align*}
where $g_{k,\te}^+$ and $g_{k,\te}^-$ are given in \eqref{def:gkplusminu} for $k=1,\ldots,N$ and $g_{N+1,\te}^+=g_{N+1,\te}^-=0$.

The homogeneous equation for $h_k$ with $k=1,\ldots,2N$ is the equation that $h_{4N+1}=\Pi^+$ satisfies, and the homogeneous equation for $h_k$ with $k=2N+1,\ldots,4N$, is the equation satisfied by $h_{4N+2}=\Pi^-$. Hence, we use the routine \verb|_ode_values()| to construct the continuous approximations by joining the pieces as we did in \eqref{def:hjexisap} and \eqref{def:hjapexis:iniconds}. We then arrange them as in \eqref{def:hkstabsorted} using \verb|_stab_fixed()|.

The two lemmas we show are the following:

\begin{lemma}\label{lemma:ode_stab:regu}
    Let $\la=\getvar{la_reg}$ and $\te=\frac{1}{3}$, $h_k$ from \eqref{def:hkstabsorted}, and $h_k^{\rm ap}$ constructed from \verb|regu_p_data.txt| and \verb|regu_m_data.txt|. 
    Then it holds that
    \begin{align*}
        \norm{h_k-h_k^{\rm ap}}_\XLIpi^2
        < 
        \rho_k,
    \end{align*}
    where the quantities $\rho_k$ are explicit and can be found in the file \verb|regu_resis.txt| from the supplementary material.
\end{lemma}

\begin{lemma}\label{lemma:ode_stab:sing}
    Let $\la=\laap$ and $\te=\frac{1}{3}$, $h_k$ from \eqref{def:hkstabsorted}, and $h_k^{\rm ap}$ constructed from \verb|sing_p_data.txt| and \verb|sing_m_data.txt|. 
    Then it holds that
    \begin{align*}
        \norm{h_k-h_k^{\rm ap}}_\XLIpi^2
        <
        \rho_k,
    \end{align*}
    where the quantities $\rho_k$ are explicit and can be found in the file \verb|sing_resis.txt| from the supplementary material.
\end{lemma}

\begin{proofsth}{Lemma~\ref{lemma:ode_stab:regu} and Lemma~\ref{lemma:ode_stab:sing}}
    The argument is the same as in the proof of Lemma~\ref{lemma:ode_exis}. We estimate the $\XLIpi$-norm of the residuals and use $L^2$ estimates for $J^\pm$.
    
    Observe that $|\Pi(x)| = 1$ for $x \in \RR$ and $\la = \getvar{la_reg}$. Hence, we use Lemma~\ref{lemma:J_L2L2_exis} to estimate $J^\pm$, instead of Lemma~\ref{lemma:J_L2L2_stab}, which is used for $\la = \laap$.

    The computations are carried out by \verb|ode_stab_resis_sing()| and \verb|ode_stab_resis_regu()|. The former calls the latter with the required parameters.
\end{proofsth}

We introduce additional notation that will be useful for computing the residual terms.
Recall the definition \eqref{def:hkplusminu}. Namely,
\begin{align*}
    h_{k}^{-J^{+}} &:= J^{+}  g^{-}_{k,\te^{-}}, 
    &\qquad
    h_{k}^{+J^{+}} &:= J^{+}  g^{+}_{k,\te^{+}},\\
    h_{k}^{-J^{-}} &:= J^{-}  g^{-}_{k,\te^{+}}, 
    &\qquad
    h_{k}^{+J^{-}} &:= J^{-}  g^{+}_{k,\te^{-}},
\end{align*}
where, for each $k=1,\dots,N$ and $\te\in\mathbb{R}$,
\begin{align*}
    \te^+&:=\te,\qquad \te^-:=1-\te,\\
    g^{-}_{k,\te}(x)
    &:=(k-\te)\frac{i}{2}
    \sum_{m=0}^{N-k}\twap_{m+k}  e^{-i(1+m)x},
    \\
    g^{+}_{k,\te}(x)
    &:=(k-\te)\frac{i}{2}
    \sum_{m=0}^{N-k}\twap_{m+k}  e^{imx}.
\end{align*}

For $s\in\{\pm1\}$ set
\begin{align*}
    h_k^{(s_J)(s_g)}&:=
    \begin{cases}
        h^{+J^{+}},&s_J=1, s_g=1,\\
        h^{-J^{+}},&s_J=1, s_g=-1,\\
        h^{+J^{-}},&s_J=-1, s_g=1,\\
        h^{-J^{-}},&s_J=-1, s_g=-1,
    \end{cases}&
    J^{(s)}&:=
    \begin{cases}J^+,& s=+1,\\ 
    J^-,& s=-1.\end{cases}
    \\
    g^{(s)}_{k,\te}&:=
    \begin{cases}g^+_{k,\te},& s=+1,\\ g^-_{k,\te},& s=-1.\end{cases}
    &
    \te^{(s)} &:=
    \begin{cases}\te^+,& s=+1,\\ \te^-,& s=-1,\end{cases}
\end{align*}

Define $\xi_k^{s_J,s_g}$ by
\begin{align}
    \xi_k^{s_J,s_g}(x)
    &:=
    (c+\twap(x))\partial_x h_k^{(s_J)(s_g)}(x)\nonumber\\
    &\qquad-
    \Big((i+s_J\laap)-i \te^{(s_J)}\big(c+\twap(x)\big)\Big)h_k^{(s_J)(s_g)}(x)
    -
    g^{(s_g)}_{k,\te^{(s_J s_g)}}(x).\label{def:resi_stab}
\end{align}

The study of $\Pi^+$ and $\Pi^-$ corresponds to the particular case of \eqref{def:resi_stab} where the affine term vanishes, i.e. with $g^{(s_g)}_{k,\theta^{(s_J s_g)}} = 0$.

\paragraph{Decomposition of $\norm{\xi_k^{s_J,s_g}}_\XLIpi^2$.}

To get a lighter notation, we set $\twap_0 := c$ and
\begin{align}
    W^{(s_J)}_{j}
    :=
    \begin{cases}
        i+s_J \laap - i \te^{(s_J)}  c, & j=0,\\
       -i \te^{(s_J)} \twap_{j}, & j=1,\ldots,N.
    \end{cases}\qquad
    W^{(s_J)}(x):=\sum_{j=0}^N
    W^{(s_J)}_{j} \cos(jx).
    \label{def:ode:Wpm}
\end{align}

If we write \eqref{def:resi_stab} with this notation, we obtain
\begin{align}\label{eqn:resi_stab}
    [\xi_k^{s_J,s_g}h](x)
    =
    \big(c+\twap(x)\big) \pa_x h(x)
    -
    W^{(s_J)} h(x)
    -
    g^{(s_g)}_{k,\te^{(s_J s_g)}}(x).
\end{align}

Hence, taking the $L^2(J)$-norm yields
\begin{align*}
    \norm{\xi_k^{s_J,s_g}}_{L^2(J)}^2
    &=
    \Istab_{1,s_J}(h_k^{(s_g)J^{(s_J)}};J)+
    \Istab_{2}(h_k^{(s_g)J^{(s_J)}};J)+
    \Istab_{3,s_J}(h_k^{(s_g)J^{(s_J)}};J)\nonumber\\&\qquad + 
    \sum_{l=4}^5 \Istab_{l,s_J,s_g,k}(h_k^{(s_g)J^{(s_J)}};J)+
    \norm{g_{k,\te^{(s_Js_g)}}^{(s_g)}}_{L^2(J)}^2,
\end{align*}
where
\begin{subequations}\label{def:Istab}
    \begin{empheq}{align}
        \Istab_{1,s_J}(h;J)&:=
        \norm{
            W^{(s_J)}h
        }_{L^2(J)}^2,
        \label{def:I1stab}
        \\
        \Istab_2(h;J)&:=
        \norm{
            (c+\twap) \pa_x h
        }_{L^2(J)}^2=\Iexis_2(h;J),\\
        \Istab_{3,s_J}(h;J)&:=
        2 \Re \big(
            -W^{(s_J)}h
            , 
            (c+\twap) \pa_x h
        \big)_{L^2(J)}
        \label{def:I3stab}
        ,\\
        \Istab_{4,s_J,s_g,k}(h;J)
        &:=
        2\Re \big(
            W^{(s_J)} h
            , 
            g_{k,\te^{(s_Js_g)}}^{(s_g)}
        \big)_{L^2(J)},
        \\
        \Istab_{5,s_J,s_g,k}(h;J)&:=
        2\Re\big(
            (c+\twap) \pa_x h
            , 
            -g_{k,\te^{(s_Js_g)}}^{(s_g)}
        \big)_{L^2(J)}.
    \end{empheq}
\end{subequations}

As we did in \eqref{eqn:I6exis}, we can compute directly $\norm{g_{k,\te^{(s_Js_g)}}^{(s_g)}}_{\XLIpi}^2$ using Parseval's identity. To this end, note that
\begin{align*}
    g_{k,\te^{(s_Js_g)}}^{(s_g)} 
    = 
        (k-\te^{(s_Js_g)})\frac{i}{2}
        \sum_{m=0}^{N-k}\twap_{m+k}  e^{i\frac{(2m+1)s_g-1}{2}x}.
\end{align*}

Therefore, for every $k=1,\ldots,N$ and $s_J,s_g\in\{-1,+1\}$, we obtain
\begin{align}
    \frac{1}{2\pi}\norm{g_{k,\te^{(s_Js_g)}}^{(s_g)}}_{\XLIpi}^2
    &=(k-\te^{(s_Js_g)})^2
    \frac{1}{4}
    \sum_{m=0}^{N-k}|\twap_{m+k}|^2.
    \label{eqn:I6stab}
\end{align}

The other terms are computed using the local expression \eqref{def:localhalpha} of $h$ on $J$. Observe that $\Istab_2(h;J)=\Iexis_2(h;J)$ so we have that
\begin{align*}
    \Istab_{2}(h;J)=
    \sum_{\substack{l_1=0\\l_2=0}}^{N_r}
    \al_{l_1} 
    \overline{\al_{l_2}} 
    \IIexis_{2}(l_1,l_2;J).
\end{align*}
 
For $\Istab_{1}(h;J)$, we obtain
\begin{align*}
    \Istab_{1,s_J}(h;J)=
    \sum_{\substack{l_1=0\\l_2=0}}^{N_r}
        \al_{l_1}
        \overline{\al_{l_2}} 
        \IIstab_{1,s_J}(l_1,l_2;J),
\end{align*}
where
\begin{align}
    \IIstab_{1,s_J}(l_1,l_2;J)&=
    \int_J
    \sum_{\substack{j_1=0\\j_2=0}}^N
    W^{(s_J)}_{j_1}
    \overline{W^{(s_J)}_{j_2}} 
    \cos(j_1 x)
    \cos(j_2 x) (x-x_0)^{l_1+l_2}
    \dx
    \nonumber
    \\
    &=
    \frac{1}{2}
    \sum_{\substack{j_1=0\\j_2=0}}^N
    W^{(s_J)}_{j_1}
    \overline{W^{(s_J)}_{j_2}} 
    \big(
        C_{J}^{l_1+l_2,j_1-j_2}
        +
        C_{J}^{l_1+l_2,j_1+j_2}
    \big)
    \nonumber
    \\
    &=
    \frac{1}{2}\Big(
    \sum_{d=-N}^N 
        D[W^{(s_J)},W^{(s_J)}](d) 
        C_J^{l_1+l_2,d}
    +
    \sum_{s=0}^{2N} 
        S[W^{(s_J)},W^{(s_J)}](s) 
        C_J^{l_1+l_2,s}
        \Big)
    \label{def:II1stab}
\end{align}
Recall that $D[W^{(s_J)},W^{(s_J)}]$ and $S[W^{(s_J)},W^{(s_J)}]$ are defined using \eqref{def:DSexis}.

For $\Istab_{3,s_J}(h;J)$, we find that
\begin{align*}
    \Istab_{3,s_J}(h;J)
    &=2 \Re \big(
        -W^{(s_J)}h
        , 
        (c+\twap) \pa_x h
    \big)_{L^2(J)}
    =
    \sum_{\substack{l_1=0\\l_2=0}}^{N_r}
        \Re\big(\al_{l_1} 
        \overline{\al_{l_2}}
        \IIstab_{3,s_J}(l_1,l_2;J)\big),
\end{align*}
where
\begin{align}
    \IIstab_{3,s_J}(l_1,l_2;J)
    &:=
    -2l_2\int_J
    \sum_{\substack{j_1=0\\j_2=0}}^N
    W^{(s_J)}_{j_1} 
    \twap_{j_2} 
    \cos(j_1 x) 
    \cos(j_2 x) 
    (x-x_0)^{l_1+l_2-1}
    \dx
    \nonumber
    \\
    &=
    -l_2
    \sum_{\substack{j_1=0\\j_2=0}}^N
    W^{(s_J)}_{j_1} 
    \twap_{j_2} 
    \big(
        C_{J}^{l_1+l_2-1,j_1-j_2}
        +
        C_{J}^{l_1+l_2-1,j_1+j_2}
    \big)
    \nonumber
    \\
    &=
    -l_2
    \Big(
    \sum_{d=-N}^N 
        D[W^{(s_J)},\twap](d) 
        C_J^{l_1+l_2-1,d}
    +
    \sum_{s=0}^{2N} 
        S[W^{(s_J)},\twap](s) 
        C_J^{l_1+l_2-1,s}.
        \Big)
        \label{def:II3stab}
\end{align}

For $\Istab_{4,s_J,s_g,k}$, it yields
\begin{align*}
    \Istab_{4,s_J,s_g,k}(h;J)
    :=
    2\Re\big(
        W^{(s_J)} h
        , 
        g_{k,\te^{(s_Js_g)}}^{(s_g)}
    \big)_{L^2(J)}
    =
    \sum_{l=0}^{N_r}
    \Re\big(
    \al_l 
    \IIstab_{4,s_J,s_g,k}(l;J)
    \big)
\end{align*}
where
\begin{subequations}\label{def:II4stab}
    \begin{empheq}{align}
    \IIstab_{4,s_J,s_g,k}(l;J)
        &:=
        -i(k-\te^{(s_Js_g)})\!\int_J 
        \sum_{j=0}^N
        \sum_{m=0}^{N-k}
            W^{(s_J)}_j
            \twap_{m+k}
            \cos(jx) 
            e^{-i\frac{(2m+1)s_g-1}{2}x} 
        (x-x_0)^{l}\dx\nonumber
        \\
        &=-i\frac{k-\te^{(s_Js_g)}}{2}
        \sum_{m=0}^{N-k}
            \twap_{m+k} \IIIstab_{4,m,s_J,s_g}(l;J)
        \\
    \IIIstab_{4,m,s_J,s_g}(l;J)&:=\sum_{j=0}^N
            W^{(s_J)}_j
        \big(
            \overline{E}_J^{l,\frac{(2m+1)s_g-1}{2}-j}
            +
            \overline{E}_J^{l,\frac{(2m+1)s_g-1}{2}+j}
        \big)
    \end{empheq}
\end{subequations}

Finally, we compute $\Istab_{5,s_J,s_g,k}$, getting 
\begin{align*}
    \Istab_{5,s_J,s_g,k}(h;J):=
    2\Re\big(
        (c+\twap) \pa_x h
        , 
        -g_{k,\te^{(s_Js_g)}}^{(s_g)}
    \big)_{L^2(J)}
    \nonumber
    =
    \sum_{l=0}^{N_r}
    \Re\big(
    \al_l 
    \IIstab_{5,s_J,s_g,k}(l;J)
    \big),
\end{align*}
where
\begin{subequations}\label{def:II5stab}
    \begin{empheq}{align}
    \IIstab_{5,s_J,s_g,k}(l;J)&:=
    il(k-\te^{(s_Js_g)})
    \int_J 
    \sum_{j=0}^N
        \twap_j 
        \sum_{m=0}^{N-k}\twap_{m+k} 
        \cos(jx)
        e^{-i\frac{(2m+1)s_g-1}{2}x}
    (x-x_0)^{l-1}
    \dx
    \nonumber
    \\
    &=
    il\frac{k-\te^{(s_Js_g)}}{2}
    \sum_{m=0}^{N-k}
    \twap_{m+k}
    \IIIstab_{5,m,s_g}(l;J).
    \\
\IIIstab_{5,m,s_g}(l;J)&:=
    \sum_{j=0}^N
        \twap_j
        \big(
            \overline{E}_J^{l-1,\frac{(2m+1)s_g-1}{2}-j}
            +
            \overline{E}_J^{l-1,\frac{(2m+1)s_g-1}{2}+j}
        \big).
    \end{empheq}
\end{subequations}

\begin{note}
    In order to compute \eqref{def:Istab} for every $k=1,\ldots,N+1$, we precompute some data. In particular, there is a different routine for each \eqref{def:II1stab}--\eqref{def:II5stab} and most of them are precomputed in \verb|pre_resis_stab()| for each different interval $J$. We then use these precomputed data to compute the sum of the terms of \eqref{def:Istab} for a fixed interval $J$, this part is carried out by the routine \verb|I_sols_resi_stab_from_precomp()|.
\end{note}

\subsubsection{The functions $J^+\Fvrmap^+|_\T$ and $J^-\Fvrmap^-|_\T$}
\label{sec:ode:Jfv}

The functions $\Fvrmap^+$ and $\Fvrmap^-$ are given in Definition~\ref{def:FplusFminu}, with $f=\Eigveap$ (see \eqref{residual_stability}). From the stored data in \verb|Jfv_p_data.txt| and \verb|Jfv_m_data.txt|, we construct continuous approximations of $J^+\Fvrmap^+|_\T$ and $J^-\Fvrmap^-|_\T$, respectively.
These approximations, denoted by $h_+^{\rm ap}$ and $h_-^{\rm ap}$, are obtained following the construction introduced in \eqref{def:hjexisap} and \eqref{def:hjapexis:iniconds}.

The goal of this section is proving the following lemma.
\begin{lemma}\label{lemma:ode_stab:Fv}
    Let $\la=\laap$ and $\te=\frac{1}{3}$. Let $\Fvrmap^+$ and $\Fvrmap^-$ given in Definition~\ref{def:FplusFminu} with $f=\Eigveap$, where $\Eigveap$ is provided in the supplementary material.
    
    Then
    \begin{align*}
        \norm{J^+\Fvrmap^+|_\T-h^{\rm ap}_+}_\XLIpi^2
        < 
        \rho_+
        \qquad 
        \norm{J^-\Fvrmap^-|_\T-h^{\rm ap}_-}_\XLIpi^2
        < 
        \rho_-
        .
    \end{align*}
    The quantities $\rho_+$ and $\rho_-$ are explicit and are given in the file \verb|Jfv_resis.txt| in the supplementary material.
\end{lemma}
\begin{proof}
    The proof is the same as for Lemma~\ref{lemma:ode_stab:sing}, but the residual is computed with a different formula.
\end{proof}

From $\Eigveap$, given by
\begin{align*}
    \Eigveap(x) = \sum_{m=-N}^{N-1}(\ha{\Eigveap})_m e^{imx},
\end{align*}
define the functions
\begin{align*}
    {\Eigveap}^{+}(x):=
    \sum_{m=0}^{N-1}
        (\ha{\Eigveap})_m 
        e^{imx},
    \qquad 
    {\Eigveap}^{-}(x):=
    \sum_{m=0}^{N-1}
        (\ha{\Eigveap})_{-m-1}
        e^{imx}.
\end{align*}

Hence, the residuals $\xi^{+}$ and $\xi^{-}$ are given by
\begin{subequations}\label{def:resi:Jfv}
    \begin{align}
        \xi^{\pm}(x)
        &:=
        \big(c+\twap(x)\big)\pa_x h^{\rm ap}_\pm(x)
        -
        W^{(\pm1)} h^{\rm ap}_\pm(x)
        -
        \sum_{m=0}^{N-1}
        (\ha{{\Eigveap}^\pm})_m
        e^{imx},
    \end{align}
\end{subequations}
where $W^{(1)}$ and $W^{(-1)}$ are given in \eqref{def:ode:Wpm}.

Hence, taking the $L^2(J)$-norm yields
\begin{align*}
    \norm{\xi^\pm}_{L^2(J)}^2
    &=
    \Istab_{1,\pm 1}(h^{\rm ap}_\pm;J)+
    \Istab_{2}(h^{\rm ap}_\pm;J)+
    \Istab_{3,\pm 1}(h^{\rm ap}_\pm;J) + 
    \sum_{l=4}^5 \Istab_{l,\pm 1,\Eigveap}(h^{\rm ap}_\pm;J)+
    \norm{{\Eigveap}^\pm}_{L^2(J)}^2,
\end{align*}
The terms $\Istab_{1,\pm 1}, \Istab_{2}$, and  $\Istab_{3,\pm 1}$ have been also studied, see \eqref{def:I1stab}--\eqref{def:I3stab}. In addition, as we did in \eqref{eqn:I6exis} and \eqref{eqn:I6stab}, we can compute directly the sum over $J$ of $\norm{{\Eigveap}^\pm}_{L^2(J)}^2$ using Parseval's identity. In this case, we obtain
\begin{align}
    \frac{1}{2\pi}\norm{{\Eigveap}^+}_{\XLIpi}^2
    =
        \sum_{m=0}^{N-1} 
        |{\Eigveap_{m}}|^2,
    \qquad
    \frac{1}{2\pi}\norm{{\Eigveap}^-}_{\XLIpi}^2
    =
        \sum_{m=-N}^{-1} 
        |{\Eigveap_{m}}|^2.
    \label{eqn:I6Jfv}
\end{align}

Finally, we provide the computations of $\Istab_{l,+1,\Eigveap}$ and $\Istab_{l,-1,\Eigveap}$. Using the notation from \eqref{def:localhalpha}, we find that
\begin{align*}
    \Istab_{4,s_J,\Eigveap}(h;J)&:=
    2\Re\big(W^{(s_J)} h, 
    {\Eigveap}^{(s_J)}\big)
    =
    \sum_{l=0}^{N_r}
    \Re\big(
    \al_l 
    \IIstab_{4,s_J,\Eigveap}(l;J)
    \big),
    \\
    \Istab_{5,s_J,\Eigveap}(h;J)&:=
    2\Re\big((c+\twap) \pa_x h, 
    -{\Eigveap}^{(s_J)}\big)
    =
    \sum_{l=0}^{N_r}
    \Re\big(
    \al_l 
    \IIstab_{5,s_J,\Eigveap}(l;J)
    \big),
\end{align*}
where
\begin{align}
    \IIstab_{4,s_J,\Eigveap}(l;J)
    &:=
    \int_J
    \sum_{j=0}^{N}
    \sum_{m=0}^{N-1}
        W_j^{(s_J)}
        \overline{{\Eigveap_{m}}^{(s_J)}}
        \cos(jx) 
        e^{-imx} 
        (x-x_0)^l
    \dx
    \nonumber
    \\
    &=
    \sum_{j=0}^{N}
    \sum_{m=0}^{N-1}
        W_j^{(s_J)}
        \overline{{\Eigveap_{m}}^{(s_J)}}
        \big(
            \overline{E}_J^{l,m-j}
            +
            \overline{E}_J^{l,m+j}
        \big),
        \label{def:II4Jfv}
    \\
    %%%%%%%%%%%%%%%%%%%%%%
    \IIstab_{5,s_J,\Eigveap}(l;J)
    &:=
    -\int_J
    \sum_{j=0}^{N}
    \sum_{m=0}^{N-1}
        l\twap_j
        \overline{{\Eigveap_{m}}^{(s_J)}}
        \cos(jx) 
        e^{-imx} 
        (x-x_0)^{l-1}
    \dx
    \nonumber 
    \\
    &=
    -
    \sum_{j=0}^{N}
    \sum_{m=0}^{N-1}
        l\twap_j
        \overline{{\Eigveap_{m}}^{(s_J)}}
        \big(
            \overline{E}_J^{l-1,m-j}
            +
            \overline{E}_J^{l-1,m+j}
        \big).
        \label{def:II5Jfv}
\end{align}

\section{Computer-assisted proofs}\label{sec:compassisted}
This appendix presents the computer-assisted results used in the paper. We describe the computational procedures, algorithms, and verification strategies used in the proofs, so that the computer-assisted component is transparent and reproducible.

All computer-assisted steps are carried out using rigorous numerical methods, based on interval arithmetic, ensuring that every computed interval encloses the exact real value. 
The programs used for these computations were written in Python using the SageMath mathematical software system, which provides the interval arithmetic routines used throughout the verification. In particular, numerical approximations are never used heuristically: each algorithm is designed to produce mathematically certified bounds, which are then combined with analytic arguments to establish the stated results.

The appendix is organized as follows. We first describe the organization of the data files used by the programs, indicating the role of each file and the information it contains. Next, for every lemma that requires computer-assisted verification, we describe the corresponding verification procedure and explain how the relevant bounds are obtained.

The material in this appendix is not required for understanding the conceptual arguments of the paper, but it is essential for verifying the correctness of the computer-assisted claims. Readers interested in reproducibility or independent verification are encouraged to consult this section.

All source code and most data files used in the computations are available at \getvar{github_link}. Larger data files are provided separately on Zenodo at \getvar{zenodo_link}. To reproduce the results, the Zenodo files should be placed in the directory \verb|supplementary_data/|.

\subsection{Supplementary data}\label{sec:compassisted:suppdata}
All auxiliary numerical data used in the computer-assisted proofs are
stored in plain text files stored in the directory
\verb|supplementary_data/|. These files are loaded at runtime by
the routines defined in \verb|load_data.py|.

All numerical quantities are interpreted using interval arithmetic.
Real numbers are converted to \verb|RealBallField| instances 
(\verb|RBF|) and complex numbers to \verb|ComplexBallField| instances (\verb|CBF|), both with $200$ bits of precision. Consequently, every value used in the computation represents an enclosure of the corresponding exact quantity. Numerical values may appear either in standard decimal notation or in base-10 scientific notation (for instance \verb|1.23e-6| meaning $1.23\cdot10^{-6}$). In all cases they are parsed as elements of the real ball field (\verb|RBF|).

\paragraph{Tabulated real data.}
Some files consist of rows of real numbers separated by whitespace. These are read by the routine \verb|_load_tab_data()|, which converts each entry to an element of \verb|RBF|. In the files listed below each line contains a single real number.

\begin{itemize}
    \item \verb|tw_data.txt| contains the speed $c$ of the traveling wave and the cosine Fourier coefficients. The first line stores $c$, while the $(j+1)$-th line stores $\twap_j$ for $j=1,\ldots,N$. Recall that $N=\getvar{Nexplicit}$.

    \item \verb|exis_resis.txt| contains the $N+1$ quantities $\rho_k$ appearing in Lemma~\ref{lemma:ode_exis}.

    \item \verb|regu_resis.txt| contains the $4N+2$ quantities $\rho_k$ appearing in Lemma~\ref{lemma:ode_stab:regu}.

    \item \verb|sing_resis.txt| contains the $4N+2$ quantities $\rho_k$ appearing in Lemma~\ref{lemma:ode_stab:sing}.

    \item \verb|Jfv_resis.txt| contains the two quantities $\rho_+$ and $\rho_-$, in this order, appearing in Lemma~\ref{lemma:ode_stab:Fv}.

    \item \verb|gk_Linf_bounds.txt| contains $L^\infty$ bounds for the real part of the $g_k$ functions given in \eqref{def:gk_exis}. In particular, we use the data from this file in the proof of Lemma~\ref{lemma:hk_funcs_exis}.
\end{itemize}

\paragraph{Tabulated complex data.}
Several files contain complex numbers written in the form \verb|re,im|, representing the real and imaginary parts of a complex number. These files are read by the routine \verb|_load_tab_data_complex()|, which parses each token and constructs the corresponding element of \verb|CBF|.

\begin{itemize}
    \item \verb|eigen_data.txt| contains $2$ lines, each consisting of a complex number. The first line contains $\laap$, while the second one contains $\getvar{la_reg}$.

    \item \verb|polynomial_zeros.txt| contains $N$ lines, each containing a complex number. Each number is approximating a root of $z^N P(z)$ inside the unit complex disk. Recall that $P$ is defined in \eqref{def:P_}.

    \item \verb|sing_u1.txt| contains $2N+2$ lines, each one with a complex number. These numbers represent the coefficients of an explicit approximation of the left singular vector associated with the smallest singular value of $\MStabTor(\laap)$. In particular the loader routine \verb|eigen_u1()| also normalizes the approximation so that it has norm $1$.

    \item \verb|exis_V.txt| contains an approximation of the $V$ matrix in the singular value decomposition of $\MExisTor$.

    \item \verb|regu_V.txt| contains an approximation of the $V$ matrix in the singular value decomposition of $\MStabTor(\getvar{la_reg})$.

    \item \verb|sing_V.txt| contains an approximation of the $V$ matrix in the singular value decomposition of $\MStabTor(\laap)$. We also remark that we construct $\Eigveap$ from the last column of $V$.
    
\end{itemize}

\paragraph{Construction of $\Eigveap$.}
Let $V\in\C^{(2N+2)\times(2N+2)}$ be the complex matrix constructed from \verb|sing_V.txt|. We define the Fourier coefficients of $\Eigveap$ as
\begin{align*}
    \ha{\Eigveap}_k =\begin{cases}
        V_{N-k,2N+2}
        &-N\leq k\leq -1,\\
        V_{k+1,2N+2}&0\leq k\leq N-1.
    \end{cases}
\end{align*}

Then $\Eigveap$ is defined by
\begin{align*}
    \Eigveap = \sum_{k=-N}^{N-1}(\ha{\Eigveap})_k e^{ikx}.
\end{align*}

\paragraph{ODE approximation data.}
The polynomial approximations used in the rigorous integration of the ODE are stored in files organized in blocks. Each block corresponds to a dyadic subinterval of the $[-\pi,\pi]$ and has the form
\begin{verbatim}
    J j k
    VALUES
    ...
    END
\end{verbatim}
where \verb|j| and \verb|k| determine the interval $J_{j,k}$ defined by 
\begin{align*}
    J_{j,k} := \frac{2\pi}{2^j}[k-1, k] - \pi.
\end{align*}

The keyword \verb|VALUES| is followed by a sequence of lines, each line containing $N_r=5$ complex coefficients. The first data line corresponds to the coefficients of the homogeneous solution, while the rest of the lines correspond to the affine ones. The block is terminated by the keyword \verb|END|.

In this section we only explain how the local polynomial data on an interval $J$ is stored. The procedure used to reconstruct the approximations from the stored data was described in Appendix~\ref{sec:ode}.

\begin{itemize}
    \item \verb|exis_data.txt| contains the necessary data to construct the approximations $h_k^{\rm ap}$, $k=1,\dots,N+1$ described in Section~\ref{sec:ode_exis}. 

    Let $\mathcal{J}$ denote the collection of all subintervals used in the construction, namely the set of intervals $J_{j,k}$. Let $J\in \mathcal{J}$ be one of these intervals, and let $x_0\in\RR$ and $r>0$ be such that
    \begin{align}\label{def:Jintervals}
        J = [x_0-r, x_0+r].
    \end{align}

    On $J$, we store after \verb|VALUES| and before \verb|END| $N+1$ lines, each one corresponding to one of the functions $h_k^{{\rm ap},J}$. More precisely, the first line  stores the non-constant coefficients of the homogeneous approximation $h_{N+1}^{\mathrm{ap},J}$, while the subsequent $N$ lines store the non-constant coefficients of the affine approximations $h_k^{\mathrm{ap},J}$ for $k=1,\ldots,N$.
    In each line we can find $5$ pairs of coefficients separated by spaces. These pairs represent the coefficients $h_{k,l}^{{\rm ap},J}$, $l=1,\ldots,5$ and have the format \verb|re,im|.

    \item \verb|sing_p_data.txt| and \verb|sing_m_data.txt| contain the data required to construct the approximations $h_k^{\rm ap}$, $k=1,\dots,4N+2$, described in Section~\ref{sec:ode_stab}, with $\lambda=\lambda^{\rm ap}$ and $\theta=\frac{1}{3}$.
    
    The structure is the same as in \verb|exis_data.txt|, so we only explain how the ordering of the lines corresponds to \eqref{def:hkstabsorted}.

    The first line of \verb|sing_p_data.txt| corresponds to $h_{4N+1}$, while the remaining lines in this file correspond to the functions $h_k$, $k=1,\ldots,2N$. 
    
    Similarly, the first line of \verb|sing_m_data.txt| corresponds to $h_{4N+2}$, while the remaining lines correspond to the functions $h_k$, $k=2N+1,\ldots,4N$.

    \item \verb|regu_p_data.txt| and \verb|regu_m_data.txt|:
    As for \verb|sing_p_data.txt| and \verb|sing_m_data.txt|, we use the same framework, with $\lambda = \getvar{la_reg}$ instead of $\lambda^{\rm ap}$.

    \item \verb|Jfv_p_data.txt| and \verb|Jfv_m_data.txt|:
    These files contain approximations of $J^+ F_v^{\rm ap,+}\big|_\T$ and $J^- F_v^{\rm ap,-}\big|_\T$, respectively. The framework is the same as described above, but in this case each block consists of a single line.

\end{itemize}

\paragraph{Bounds.}
Finally, the file \verb|bounds.txt| contains the remaining numerical
bounds used throughout the paper. These are read by the routine
\verb|load_bounds()|, which converts the numerical values into \verb|RBF| elements
and stores them in a dictionary indexed by their labels.

The file is organized as sequence of entries, one per line. Each line
contains several fields separated by the symbol \verb!|!. Lines that
are empty or start with the character \verb|#| are ignored. From each
entry the program extracts three pieces of information: a label
identifying the bound, the numerical value of the bound, and a tag
indicating the side of the inequality. Thus a typical entry has the form
\begin{verbatim}
    label|...|value|...|side|...|description
\end{verbatim}
where \verb|label| is the identifier used in the code, \verb|value| is
the numerical bound written in decimal notation, and \verb|side|
indicates whether the bound corresponds to the lower or upper side of
the corresponding inequality. In addition, \verb|description| is a brief description of the bound to improve readability.

\subsection{Objects, methods and routines}
We briefly describe the main classes and routines used in the implementation.
\paragraph{Classes.}
The following classes are defined in \verb|classes.py|.
\begin{itemize}
    \item \verb|FourierRealSeries|.
    This class represents real Fourier series and implements basic operations with them, such as evaluation, differentiation, and algebraic manipulations of the coefficients.
    
    \item \verb|Functions_1D|. 
    A class representing one-dimensional real-valued functions by storing callable black-box evaluators for the function and a finite number of its derivatives. It also supports basic algebraic manipulations, with the corresponding derivatives computed automatically. In particular, products are handled through Leibniz rule, and inverses through a Faà di Bruno formula. We also include a constructor that creates a \verb|Functions_1D| instance from a \verb|FourierRealSeries|.
    
\end{itemize}

\paragraph{Methods.}
The following classes are defined in \verb|methods.py|.
\begin{itemize}
    \item \verb|integral_1D| constructs a callable instance for one-dimensional integration, which encloses integrals of \verb|Functions_1D| objects over an interval $I$. 
    
    We only include the method \verb|gauss_lege_2dots()|, which implements a Gauss--Legendre quadrature. In particular, it uses the enclosure
    \begin{align}
        \int_I f(x)\dx
        \in 
        r \big(
            f(x_0+r\om)
            +
            f(x_0-r\om)
        \big)
        +\frac{r^5}{135}
        \pa_x^4 f(I),
        \label{nummethod:gaussleg}
    \end{align}
    where $I:=[x_0-r,x_0+r]$ and $\om:=\frac{\sqrt{3}}{3}$.

    \item \verb|image_1D| constructs a callable object that encloses the images of functions represented as \verb|Functions_1D| objects. We only include the method \verb|taylor_method()|, which, for a fixed order $k$, implements a Taylor enclosure. In particular, if $k=2$, we use
    \begin{align}\label{compassi:secord}
        f(I) \subset f(x_0) + [-r,r]\pa_x f(x_0) + \frac{1}{2}[0,r^2] \pa_x^2 f(I).
    \end{align}
    where $I:=[x_0-r,x_0+r]$.
    
\end{itemize}

\paragraph{Routines.}
We describe a selection of the routines from \verb|auxiliar_funcs.py|. We do not cover every function, but focus on those that are more elaborated.

\begin{itemize}

\item \verb|max_prev_esti()|. 
    Let $[a,b]\subset\RR$ be a compact interval, $B\in\RR$, and let $f:[a,b]\to\RR$. 
    This routine verifies whether 
    \begin{align}
        f(x) \leq B,\quad \text{for all } x\in[a,b].\label{eqn:r:max_prev_esti}
    \end{align}
    
    The function $f$ is represented as a \verb|Functions_1D| instance, the domain as a list containing the endpoints of the interval, and $B$ as a \verb|RealBallField| instance.

    In order to prove \eqref{eqn:r:max_prev_esti}, assume we have a routine, which we call \verb|method()|, such that given a compact interval $I$ and a function $f:I\to\RR$, it returns an interval $J$ satisfying
    \begin{align*}
        J \supset f(I).
    \end{align*}
    Smaller intervals $I$ are expected to produce tighter enclosures $J$.
    
    The algorithm starts with the list \verb|intervals| containing only $[a,b]$. While \verb|intervals| is nonempty, an interval $I$ is removed from the list and $f(I)$ is enclosed using \verb|method()|. If $f(I) > B$, the routine stops and returns \verb|False|. If $J\ni B$, the interval $I$ is bisected and both subintervals are appended to the list. Otherwise the interval is accepted.

    If the list becomes empty, the inequality holds on the whole domain $[a,b]$ and the routine returns \verb|True|. 
    Two additional stopping criteria control the total number of iterations and the number of intervals stored in the list.

\item \verb|integ_adaptive_1D()|. 
    Let $[a,b]\subset\RR$ be a compact interval and $f:[a,b]\to\RR$. This routine produces an enclosure of the integral
    \begin{align*}
        \int_{[a,b]} f(x)\dx .
    \end{align*}
    
    The function $f$ is represented as a \verb|Functions_1D| instance and the domain as a list containing the endpoints of the interval.
    
    Assume we have a routine \verb|method()| which, given a compact interval $I$ and a function $f:I\to\RR$, returns an interval $J$ such that
    \begin{align*}
        J \supset \int_I f(x)\dx .
    \end{align*}
    As before, smaller intervals are expected to yield tighter enclosures.
    
    The adaptive subdivision strategy is the same as in \verb|max_prev_esti()|. Starting from the list \verb|intervals| containing $[a,b]$, intervals are removed, evaluated using \verb|method()|, and either accepted or bisected according to prescribed tolerances.\footnote{Rigor does not come from the tolerances. They are only used as a criterion to accept or reject an interval.}
    
    Accepted intervals contribute their enclosure to the total integral. When the list becomes empty, the whole domain has been processed and the routine returns the final enclosure. As before, two stopping conditions limit the total number of iterations and the number of stored intervals.

\item \verb|triangular_integral()|.
    Let $f,g:[0,\pi]\to\RR$, and assume that $f(x)\geq 0$ for all $x\in[0,\pi]$. This routine produces an enclosure of the triangular integral
    \begin{align*}
        \int_0^\pi \int_0^x f(x)g(y)\dy\dx.
    \end{align*}
    
    The input \verb|integrand| represents the function $f$, while \verb|mult_term| represents the function $g$. Both are given as \verb|Functions_1D| instances and are assumed to provide derivatives up to order four.
    
    The routine is based on the identity
    \begin{align*}
        \int_0^\pi \int_0^x f(x)g(y)\dy\dx
        =
        \int_0^\pi g(x)\iota(x)\dx,
    \end{align*}
    where
    \begin{align*}
        \iota(x)=\int_x^\pi f(t)\,dt.
    \end{align*}
    In particular,
    \begin{align*}
        \iota(\pi)=0,
        \qquad
        \pa_x\iota(x)=-f(x).
    \end{align*}
    
    The algorithm starts from the interval $[0,\pi]$ and maintains a stack of subintervals, processed from right to left. For each interval $I:=[x_0-r,x_0+r]$, the values of $\iota$ at the relevant quadrature nodes are enclosed by integrating $f$ over suitable subintervals. We use the Gauss--Legendre enclosure formula \eqref{nummethod:gaussleg} to enclose
    \begin{align*}
        \int_{I} \iota(y) g(y) \dy.
    \end{align*}
    Hence, we need to evaluate $\iota$ at $x_0-r\om$ and $x_0+r\om$ with $\om=\frac{\sqrt{3}}{3}$ and enclose $\pa_x^4 (\iota g) (I)$. To this end, we need to know how to enclose the function $\iota$ up to four derivatives since
    \begin{align*}
        \pa_x^4 (\iota g)  = 
        \pa_x^4 \iota g
        +
        4\pa_x^3 \iota \pa_x g
        +
        6\pa_x^2 \iota \pa_x^2 g
        +
        4\pa_x \iota \pa_x^3 g
        +
        \iota  \pa_x^4 g.
    \end{align*}
    Since $\iota$ is an antiderivative of $-f$, its derivatives can be enclosed directly from those of $f$. Because $f$ is nonnegative, $\iota$ is decreasing, and therefore
    \begin{align*}
    \iota(I)=[\iota(x_0+r),\iota(x_0-r)].
    \end{align*}
    
    Each evaluation of $\iota$ is obtained by a cumulative integral, computed using the enclosure \eqref{nummethod:gaussleg}.
    
    If the resulting local error enclosure satisfies a prescribed tolerance, the interval is accepted and its contribution is added to the total integral. Otherwise, the interval is bisected and both subintervals are returned to the stack, with the right one processed first so that the reconstruction of $\iota$ remains consistent.
    
    When all intervals have been accepted, the routine returns the final enclosure of the triangular integral. As before, the total number of iterations is bounded by a safety criterion.

\item \verb|gersh_smallest_eig()| implements Gershgorin's theorem to produce a lower bound for the real parts of the eigenvalues of a given matrix.

\item \verb|gersh_gap_one_two()| implements Gershgorin's theorem to compute an upper bound for the real part of the eigenvalue associated with the Gershgorin disk corresponding to the last row. It also returns a lower bound for the real parts of the remaining eigenvalues.

\end{itemize}

\subsection{Details of the proofs}

The routine \verb|substituting_estimates()| can be found in the file \verb|lemmas.py|. Although it does not prove any lemma, it collects all substitutions used throughout the paper. The function returns \verb|True| if no errors are detected.

\begin{proofsth}{Lemma~\ref{lemma:residue_exis_L2}}
    The coefficients of $\twap$ are provided in the supplementary material.
    The function $\twxi$ is defined as the residual obtained by substituting $\twap$ into the Burgers--Hilbert equation, namely
    \begin{align*}
        \twxi(x)
        = c\pa_x\twap(x)
        + H\twap(x)
        + \twap(x) \pa_x\twap(x).
    \end{align*}

    We represent $\twap$ as a \verb|FourierRealSeries| object, which stores Fourier coefficients and implements algebraic operations directly at the coefficient level, including differentiation \verb|dx()|, multiplication \verb|__mul__()|, and the Hilbert transform \verb|hx()|.
    
    An instance representing $\twxi$ is constructed by the routine \verb|equ_tw_eval_symbolic()| from the speed $c$ and the Fourier coefficients of $\twap$. Its $\XHodd{0}$-norm is then obtained from these Fourier coefficients with \verb|norm_homogeneous()|. Since this routine returns the full $H^0(\T)$ norm, the result is finally divided by $\sqrt{2}$.
    
    Because all operations are performed on interval-valued coefficients and the norm is evaluated algebraically, the resulting enclosure of $\norm{\twxi}_{\XHodd{0}}$ is fully rigorous. All these steps are carried out by the routine \verb|norm_xi_twap_sq()| in
    \verb|lemmas.py|.
    
    The verified enclosure is strictly smaller than the constant stated in the lemma, which completes the proof.
\end{proofsth}

\begin{proofsth}{Lemmas~\ref{lemma:norm_fv} and \ref{lemma:eignot0}}
    The proof of this lemmas is done in \verb|norm_xi_fvap_sq()| from \verb|lemmas.py|.
    
    The computation of the residual is carried out coefficient by coefficient, following the same strategy used in the proof of Lemma~\ref{lemma:residue_exis_L2}. In that lemma we used the class \verb|FourierRealSeries| in order to represent the action of the operator on the Fourier coefficients.
    
    In the present case we proceed in the same spirit, but without introducing an auxiliary class. Instead, the computation is implemented directly by manipulating the Fourier coefficients of the approximate eigenvector.
    
    Let $f^{\mathrm{ap}}$ denote the approximate eigenvector with Fourier coefficients stored in the vector \verb|fvap|. The $\XL$ norm is computed directly as
    \begin{align*}
        \norm{\Eigveap}_{\XL}^2
        =
        \sum_{k=-N}^{N-1} |(\ha{\Eigveap})_k|^2 ,
    \end{align*}
    which in the code corresponds to the function \verb|norm_vector_sq()|, from \verb|auxiliar_funcs.py|, applied to \verb|fvap|.
    
    Similarly, the $\XH{1}$ norm is evaluated using the Fourier representation
    \begin{align*}
        \norm{\Eigveap}_\XH{1}^2
        =
        \sum_{k=-N}^{N-1} (1+k^2) |(\ha{\Eigveap})_k|^2 .
    \end{align*}
    This is implemented by multiplying each Fourier coefficient by the weight $1+k^2$ and summing the resulting contributions.
    
    Finally, the residual is computed from its definition, that is
    \begin{align*}
        \Eigres
        =
        \LStabthe \Eigveap
        -
        \laap \Eigveap .
    \end{align*}
    The operator $\LStabthe$ is applied directly at the level of Fourier coefficients: the derivative, the Hilbert transform, and the convolution with the Fourier coefficients of the wave profile are computed mode by mode. The resulting coefficients are split into those that remain inside the truncated Fourier window and those that fall outside it. The $\XL$ norm of the residual is then obtained as the sum of the squared norms of these two contributions.
\end{proofsth}

\begin{proofsth}{Lemma~\ref{lemma:twapprox_nonzero_coeffs}}
    The Fourier coefficients of $\twap$ are read from the file
    \verb|tw_data.txt| and represented using interval arithmetic. After reading the file, we also verify that none of these intervals contains $0$. The routine checking this is \verb|coeffs_are_not_zero()| from \verb|lemmas.py|.
\end{proofsth}

\begin{proofsth}{Lemma~\ref{lemma:Proots1}}
    Let $\al\in Z$. Since $\al$ is a simple root, we may write
    \begin{align*}
        A(\al)\prod_{\substack{\al'\in Z\\ \al'\neq \al}}
            (\al-\al')^{-1}&=
            -2\frac{z^{N-1}}{\twap_N}
            \prod_{\al\in Z}(z-\al^{-1})^{-1}
            \prod_{\substack{\al'\in Z\\ \al'\neq \al}}
            (\al-\al')^{-1}\\&=
        \lim_{z\to\al}
        -(z-\al)\frac{z^{N-1}}{z^N P(z)}=
        -\frac{\al^{N-1}}{\pa_z (z^N P(z))(\al)}.
    \end{align*}
    
    This quantity is evaluated using the interval enclosure of $\al$ provided by
    Lemma~\ref{lemma:enclose_poly_roots} together with interval arithmetic for the evaluation of $\pa_z(z^N P(z))$. The corresponding verification is implemented in the routine \verb|proots_exis()| in \verb|lemmas.py|, which checks that the resulting complex interval cannot contain any integer.
    
    The computation certifies that the above expression is never an integer for any
    $\al\in Z$.
\end{proofsth}

\begin{proofsth}{Lemma~\ref{lemma:exis_singval}}
    Recall from \eqref{def:Mexistor} that $\MExisTor$ is defined in terms of finitely many Fourier coefficients of the functions $h_k$ introduced in \eqref{def:hk}. For each $k=1,\ldots,N+1$, we have constructed an explicit approximation $h_k^{\rm ap}$. The corresponding numerical data are stored in the file \verb|exis_data.txt| and described in Appendix~\ref{sec:ode_exis}, where we explain how we compute the Fourier modes of $h_k^{\rm ap}$ and provide explicit $\XLIpi$ bounds for $\tau_{k}$, defined by
    \begin{align*}
        \tau_{k}
        :=
        h_k - h_k^{\rm ap},
        \qquad k=1,\ldots,N+1.
    \end{align*}
    Lemma~\ref{lemma:ode_exis} provides rigorous bounds for $\norm{\tau_{k}}_\XLIpi^2$ in terms of the stored quantities $\rho_k$, listed in \verb|exis_resis.txt|.

    Using the $h_k^{\rm ap}$ functions, we construct the matrix
    \begin{align*}
        \MExisTorApprox \in \C^{(N+1)\times(N+1)},
    \end{align*}
    with entries
    \begin{align}
        (\MExisTorApprox)_{jk}=
            (\ha{h_k^{\rm ap}})_{j-1} -\delta_{k,j-1},
            \qquad 1\leq j, k\leq N+1.
    \end{align}
    By construction,
    \begin{align*}
        \MExisTor = \MExisTorApprox + R,
    \end{align*}
    where the error matrix $R$ is given by
    \begin{align*}
        R_{jk}=(\ha{\tau_{k}})_{j-1},
        \qquad 1\leq j, k\leq N+1.
    \end{align*}

    A direct expansion shows that
    \begin{align}\label{eqn:gram_bound}
        \MExisTor^\ast \MExisTor
        -
        \MExisTorApprox^\ast \MExisTorApprox
        &=R^\ast
        \MExisTorApprox+
        \MExisTorApprox^\ast
        R+
        R^\ast
        R
    \end{align}

    For each $k=1,\ldots,N+1$, we set
    \begin{align*}
        r_k := \Big(\frac{\rho_k}{2\pi}\Big)^{1/2},\qquad
        c_k := 
        \Big(\sum_{j=1}^{N+1}
        |(\MExisTorApprox)_{jk}|^2
        \Big)^{1/2}.
    \end{align*}
    By Parseval’s identity and Lemma~\ref{lemma:ode_exis},
    \begin{align*}
        \sum_{j=1}^{N+1}
        |R_{jk}|^2 \leq 
        \frac{1}{2\pi}
        \norm{\tau_{k}}_\XLIpi^2\leq 
        \frac{1}{2\pi}\rho_k= r_k^2.
    \end{align*}

    Applying Cauchy--Schwarz, we obtain for every $j,k$
    \begin{align*}
        |(R^\ast \MExisTorApprox)_{jk}|
        &\leq r_j  c_k,
        \\
        |(\MExisTorApprox^\ast R)_{jk}|
        &\leq 
         c_j r_k,
         \\
        |(R^\ast R)_{jk}|
        &\leq 
        r_j r_k.
    \end{align*}

    Hence, combining these estimates into \eqref{eqn:gram_bound} yields
    \begin{align}\label{eqn:entrywise_gram_bound}
        \big|(\MExisTor^\ast \MExisTor
        -
        \MExisTorApprox^\ast \MExisTorApprox)_{jk}\big|
        \leq 
        r_j  c_k + c_j  r_k + r_j  r_k.
    \end{align}

    We define the matrix $B\in\RR^{(N+1)\times(N+1)}$ by
    \begin{align*}
        B_{jk} := r_j c_k + r_k c_j + r_j r_k,
    \end{align*}
    and its Frobenius norm
    \begin{align*}
        \norm{B}_F:=
        \Big(\sum_{j,k=1}^{N+1} B_{jk}^2\Big)^{1/2}.
    \end{align*}

    Since $\norm{E}_2\leq \norm{E}_F$ for any matrix $E$, we conclude from \eqref{eqn:entrywise_gram_bound} that
    \begin{align}
        \norm{\MExisTor^\ast \MExisTor
        -
        \MExisTorApprox^\ast \MExisTorApprox}_2
        \leq 
        \norm{B}_F.
    \end{align}

    By Weyl’s theorem on eigenvalue perturbations for Hermitian matrices, see for example \cite[Corollary~III.2.6]{bhatia2013matrix}, we deduce that
    \begin{align*}
        \la_{\min}\big(\MExisTor^\ast \MExisTor\big)
        &\geq
        \la_{\min}\big(\MExisTorApprox^\ast \MExisTorApprox\big) - \|\MExisTor^\ast \MExisTor-\MExisTorApprox^\ast \MExisTorApprox\|_2\\
        &\geq
        \la_{\min}\big(\MExisTorApprox^\ast \MExisTorApprox\big) - \norm{B}_F.
    \end{align*}
    Hence, a lower bound on $\la_{\min}(\MExisTorApprox^\ast \MExisTorApprox)$ and $B_F$ immediately yields a rigorous lower bound for \begin{align*}\sigma_1^2 = \la_{\min}(\MExisTor^\ast\MExisTor).\end{align*}

    To obtain a bound for $\la_{\min}(\MExisTorApprox^\ast \MExisTorApprox)$, we apply Gershgorin’s theorem \cite{Gershgorin31} to a nearly diagonalized representation of $\MExisTorApprox^\ast \MExisTorApprox$. To this end, we use the matrix
    \begin{align*}
        V_{\rm exis}^{\rm ap}\in\C^{(N+1)\times(N+1)},
    \end{align*}
    provided in the supplementary material, and define
    \begin{align*}
        \ti{S} :=
        {V_{\rm exis}^{\rm ap}}^\ast
        \MExisTorApprox^\ast
        \MExisTorApprox
        V_{\rm exis}^{\rm ap}.
    \end{align*}
    The matrix $\ti{S}$ is nearly diagonal and Hermitian. However, its spectrum is not the same as $\MExisTorApprox^\ast \MExisTorApprox$ because $V_{\rm exis}^{\rm ap}$ is not unitary. However, it is close to a unitary one $V$. This is proved in Lemma~2.4 from \cite{gomez2021any} for orthogonal matrices, the argument for the unitary case follows the same argument with very small changes. Moreover, immediately after the lemma, the authors provide a method to rigorously construct an enclosure of
    \begin{align*}
        S:=V^\ast\MExisTorApprox^\ast \MExisTorApprox V.
    \end{align*}
    from the computed enclosure of $\ti{S}$. Applying Gershgorin's theorem to $S$ provides a lower bound for its smallest eigenvalue. 
    
    The computer-assisted part of the argument is done in \verb|svd_exis()| from \verb|lemmas.py| and proceeds as follows. The matrix $\MExisTorApprox$ is constructed directly from the data stored in \verb|exis_data.txt|, this is done in \verb|construct_exis_mat()|. Using the certified $\XLIpi$-bounds $\rho_k$, stored in \verb|exis_resis.txt|, the routine \verb|mat_exis_add_radii()| constructs \verb|to_rest|, an upper bound for $\norm{B}_F$, 

    To obtain a lower bound on $\la_{\min}(\MExisTorApprox^\ast \MExisTorApprox)$ we first compute $\ti{S}$. To this end, we use $V_{\rm exis}^{\rm ap}$, stored in \verb|exis_V.txt|. Then the function \verb|S_from_ti_S_for_Gersh()| constructs $S$, and finally
    the routine \verb|gersh_smallest_eig()| computes \verb|min_val|, a lower bound  for the eigenvalues of $S$, using Gershgorin’s theorem.
    
    Finally, we compute the difference between \verb|min_val| and \verb|to_rest|, yielding 
    \begin{align*}
        \sigma_1^2 > \getvar{svd1_exis}.
    \end{align*}
\end{proofsth}

\begin{proofsth}{Lemma \ref{lemma:betamax}}
    The function $\beta$ is represented as the \verb|Functions_1D| object with derivatives up to order $2$, constructed from the cosine Fourier representation of $\twap$ and the parameter $c$, by \verb|beta_func_constructor()|. 
    
    The maximum bound $B:=19.9056$ of $\beta$ is then verified by the adaptive routine \verb|max_prev_esti|, we also show that $\beta$ is positive applying \verb|max_prev_esti()| to $-\beta$ with $B=0$. The method we choose is based on a second order Taylor, see \eqref{compassi:secord}.
    
    The whole computation is carried out in \verb|beta_max()| from \verb|lemmas.py|.
\end{proofsth}

\begin{proofsth}{Lemma~\ref{lemma:kappa1}}
    Observe that
    \begin{align*}
        \norm{\kappa_1}_{L^1([0,\pi])}=
        \int_0^\pi 
            \int_0^x \beta(y)^2
        \dy\dx
        =
        \int_0^\pi \int_y^\pi 
        \beta(y)^2
        \dx\dy
        =
        \int_0^\pi 
        \beta(y)^2 (\pi - y)
        \dy.
    \end{align*}

    The function \verb|aux_kappa1_func_constructor()| builds $(\pi-x) \beta(x)^2$ as a \verb|Functions_1D| with derivatives up to order $4$. In order to construct $\beta(x)^2$, the routine \verb|aux_kappa1_func_constructor()| uses \verb|beta_sq_func_constructor()|, where we use a Fourier representation of $(c+\twap)^2$.

    We then compute an enclosure of the integral
    \begin{align*}
        \int_0^\pi (\pi-x) \beta(x)^2 \dx
    \end{align*}
    using \verb|integ_adaptive_1D()| and $\verb|gauss_lege_2dots()|$.
    
    All these steps are implemented in the routine \verb|kappa1()| from \verb|lemmas.py|. Finally, the resulting enclosure is compared with the bound in the statement. 
\end{proofsth}

\begin{proofsth}{Lemma~\ref{lemma:betamod_L2}}
    Since $\beta$ is $\RR$-valued, we have that
    \begin{align*}
        \bigl|\partial_x \beta(x) + i \beta(x)^2\bigr|^2
        = \bigl(\partial_x\beta(x)\bigr)^2 + \beta(x)^4,
    \end{align*}
    and therefore,
    \begin{align*}
        \norm{\partial_x \beta + i\beta^2}_{\XLIpi}^2
        = \int_{-\pi}^{\pi} \bigl(\partial_x\beta(x)\bigr)^2 + \beta(x)^4\dx.
    \end{align*}
    
    The computer-assisted verification proceeds as follows. We construct a \verb|Functions_1D| object, with derivatives up to order $4$, for the nonnegative integrand
    \begin{align*}
        x \longmapsto \bigl(\partial_x\beta(x)\bigr)^2 + \beta(x)^4
    \end{align*}
    using \verb|betamod_func_sq_constructor()|. Internally, this constructor builds the equivalent expression
    \begin{align*}
        \bigl(\partial_x\beta(x)\bigr)^2 + \beta(x)^4 = \frac{1 + (\pa_x\twap(x))^2}{(c+\twap(x))^4},
    \end{align*}
    and creates \verb|FourierRealSeries| instances for the numerator and the denominator.
    
    We then compute an enclosure of the integral
    \begin{align*}
        \int_0^\pi \bigl(\partial_x\beta(x)\bigr)^2 + \beta(x)^4 \dx
    \end{align*}
    using \verb|integ_adaptive_1D()| and $\verb|gauss_lege_2dots()|$.
    
    Finally, since the integrand is even, the integral over $[-\pi,\pi]$ equals twice the integral over $[0,\pi]$, and hence
    \begin{align*}
        \norm{\partial_x \beta + i\beta^2}_{\XL}^2
        = 2 \int_0^\pi \bigl(\partial_x\beta(x)\bigr)^2 + \beta(x)^4 \dx.
    \end{align*}

    All these steps are implemented in the routine \verb|betamod_L2_sq()| from \verb|lemmas.py|. Finally, the resulting enclosure is compared with the bound in the statement. 
\end{proofsth}

\begin{proofsth}{Lemma~\ref{lemma:kappa2}}
    Our computer-assisted proof evaluates the integral
    \begin{align*}
        \norm{\kappa_2}_{L^1([0,\pi])}&=
        \int_0^\pi \int_0^x
            \big[(\pa_x\beta(x))^2+\beta(x)^4\big],
    \end{align*}
    using \verb|triangular_integral()|.
    
    We then compare the resulting enclosure of the integral with the bound in the statement. The whole routine is called \verb|kappa2()| and is found in \verb|lemmas.py|.
\end{proofsth}

\begin{proofsth}{Lemma~\ref{lemma:hk_funcs_exis}}
    For every $k=1,\ldots,N+1$, the function $h_k$ solves
    \begin{align*}
        \pa_x h_k(x)=  \beta(x) \big(ih_k(x) + g_k(x)\big),
    \end{align*}
    where $g_k$ is given by \eqref{def:gk_exis} for $k=1,\ldots,N$ and $g_{N+1}=0$. Hence,
    \begin{align*}
        \pa_x^2 h_k=
        \beta
        \pa_x g_k+
        (\pa_x\beta+i\beta^2)
        g_k+
        (i\pa_x\beta-\beta^2)
        h_k.
    \end{align*}

    First, we have for $h_{N+1}=\PiExis$ that
    \begin{align*}
        \norm{\pa_x^2 \PiExis}_\XLIpi=
        \norm{
        (i\pa_x\beta-\beta^2) 
        \PiExis}_\XLIpi
        =
        \norm{
        i\pa_x\beta-\beta^2}_\XLIpi,
    \end{align*}
    Lemma~\ref{lemma:betamod_L2} provides an explicit upper bound for $\norm{
        i\pa_x\beta-\beta^2}_\XLIpi^2$.

    We now estimate the remaining terms. We have for every $k=1,\ldots,N$ that
    \begin{align*}
        \norm{\pa_x^2 h_k}_\XLIpi
        \leq I_1 + I_2 + I_3,
    \end{align*}
    where
    \begin{align*}
        I_1&:=
        \norm{\beta}_\XLinfinterval
        \norm{\pa_x g_k}_\XLIpi,\\
        I_2&:=
        \norm{\pa_x\beta+i\beta^2}_\XLIpi
        \norm{g_k}_\XLinfinterval,\\
        I_3&:=
        \norm{\kappa_2}_{L^1([0,\pi])}^{1/2}
        \norm{g_k}_\XLIpi.
    \end{align*}

    The quantities $\norm{\beta}_\XLinfinterval$, $\norm{\pa_x\beta+i\beta^2}_\XLIpi$ and $\norm{\kappa_2}_{L^1([0,\pi])}$ are bounded using Lemmas~\ref{lemma:betamax}, \ref{lemma:betamod_L2}, and \ref{lemma:kappa2}, respectively.

    Next, we compute explicit expressions for the norms of $g_k$. A direct computation yields
    \begin{align*}
        \norm{\pa_x g_k}_\XLIpi^2
        &=
        k^2\pi
        \sum_{m=1}^{N-k}
        m^2 (\twap_{k+m})^2,\\
        \norm{g_k}_\XLIpi^2
        &=
        k^2\pi
        \Big(
        \frac{1}{2}(\twap_k)^2
        +
        \sum_{m=1}^{N-k}
        (\twap_{k+m})^2
        \Big).
    \end{align*}

    To bound $\norm{g_k}_\XLinfinterval$, we use the estimates stored in \verb|gk_Linf_bounds.txt|. More precisely, we verify using \verb|max_prev_esti()| that the values $C_k$, $k=1,\ldots,N$, satisfy
    \begin{align*}
        \norm{\Re(g_k)}_{L^\infty([0,\pi])} < 
        C_k.
    \end{align*}

    Hence, we have
    \begin{align*}
        \norm{g_k}_\XLinfinterval^2
        &\leq 
        \norm{\Im(g_k)}_\XLinfinterval^2
        +
        \norm{\Re(g_k)}_\XLinfinterval^2
        \leq 
        \Big(\frac{k}{2}
        \twap_k\Big)^2
        +
        C_k^2.
    \end{align*}

    Combining the above estimates yields the claimed bound. The full verification is carried out in the routine \verb|hk_norm_exis_H2sq()| in \verb|lemmas.py|.
    
\end{proofsth}

\begin{proofsth}{Lemma~\ref{lemma:stab_is_regu}}
    We fix $\la=\getvar{la_reg}$ and show that
    \begin{align*}
        \det\big(\MStabTor(\la)^\ast\MStabTor(\la)\big) > 0.
    \end{align*}
    We proceed as in the proof of Lemma~\ref{lemma:exis_singval}. The entries of $\MStabTor(\la)$ are the Fourier coefficients of certain functions, namely $h_k$ from \eqref{def:hkstabsorted}. With this notation, the matrix $\MStabTor(\la)$ from \eqref{def:MStabTor} can be written as
    \begin{align}
        {\MStabTor(\la)}_{jk}=
        \begin{cases}
            (\ha{h_{k+N}})_{j-2}-\delta_{j-1,k} ,&
            1\leq j\leq N+1,\ 1\leq k\leq N,\\
            (\ha{h_{k}})_{j-2} ,&
            1\leq j\leq N+1,\ N+1\leq k\leq 2N,\\
            (\ha{h_{4N+1}})_{j-2} ,&
            1\leq j\leq N+1,\ k=2N+1,
            \\
            (\ha{h_{k+2N}})_{j-N-3} ,&
            N+1\leq j\leq 2N+2,\ 1\leq k\leq N,\\
            (\ha{h_{k+3N}})_{j-N-3}-\delta_{j-2,k} ,&
            N+1\leq j\leq 2N+2,\ N+1\leq k\leq 2N,\\
            (\ha{h_{4N+2}})_{j-N-3} ,&
            N+1\leq j\leq 2N+2,\ k=2N+2,
        \end{cases}
        \label{def:Mstab:sort}
    \end{align}

    Hence, using the notation from Lemma~\ref{lemma:ode_stab:regu}, we have 
    \begin{align*}
        \norm{h_k}_{\XLIpi} \leq \norm{h_k^{\rm ap}}_{\XLIpi} + \rho_k^{1/2}.
    \end{align*}
    
    Consequently, following the same strategy as in the proof of Lemma~\ref{lemma:exis_singval}, we obtain a rigorous lower bound on the smallest singular value of $\MStabTor(\la)$, which is strictly positive. The only routines that differ are \verb|construct_stab_mat()| and \verb|mat_stab_add_radii()|. 

    The routine \verb|construct_stab_mat()| constructs $\MStabTorApprox$ by substituting $h_k$ by $h_k^{\rm ap}$ in \eqref{def:Mstab:sort}. We define
    \begin{align*}
        R:=\MStabTor - \MStabTorApprox,
    \end{align*}
    so the routine \verb|mat_stab_add_radii()| uses the fact that
    \begin{align}\label{eqn:mat_stab_resi_sum}
        2\pi\sum_{j=1}^{2N+2} |R_{jk}|^2 \leq 
        \begin{cases}
            \rho_{k+N} + \rho_{k+2N}
            ,& 1\leq k\leq N,\\
            \rho_{k} + \rho_{k+3N}
            ,& N+1\leq k\leq 2N,\\
            \rho_{4N+1}
            ,& k=2N+1,\\
            \rho_{4N+2}
            ,& k=2N+2,
        \end{cases}
    \end{align}
    to bound the following quantity
    \begin{align*}
        \norm{\MStabTorApprox^\ast R + R\MStabTorApprox^\ast + R^\ast R}_F.
    \end{align*}
    
    The full routine is implemented in \verb|svd_regu()| from \verb|lemmas.py|.
\end{proofsth}

\begin{proofsth}{Lemma~\ref{lemma:stab_singval}}
    As in the proof of Lemma~\ref{lemma:stab_is_regu}, we adopt the notation from \eqref{def:hkstabsorted}. However, here we work with $\la=\laap$ instead of $\la = \getvar{la_reg}$, and we use Lemma~\ref{lemma:ode_stab:sing} rather than Lemma~\ref{lemma:ode_stab:regu}. These are not the only differences from the argument used in the proof of Lemma~\ref{lemma:stab_is_regu}. In the present proof, we seek rigorous bounds for the two smallest singular values of a matrix. Consequently, we run the routine \verb|gersh_gap_one_two()| instead of \verb|gersh_smallest_eig()|. The full routine is \verb|svd_sing()| from \verb|lemmas.py|.
\end{proofsth}

\begin{proofsth}{Lemma~\ref{lemma:prod_fv_u}}
    We load an explicit approximation of $u_1$ stored in \verb|sing_u1.txt|, after a renormalization, we denote it by $u_1^{\rm ap}$, i.e. $\norm{u_1^{\rm ap}}_{\ell^2(\C^{2N+2})}=1$. 
    Let $\om\in\C$ with $|\om|=1$ and define $\delta:=u_1-\om^{-1}  u_1^{\rm ap}$, so we get 
    \begin{align*}
        \big|u_1(\la)^\ast{\bf b}(\Fvrmap^+,\Fvrmap^-)\big|
        &=
        \big|\big(
        {\bf b}(
            \Fvrmap^+,\Fvrmap^-
        ), 
        u_1(\la)
        \big)_{\ell^2(\C^{2N+2})}\big|
        \nonumber
        \\
        &\geq
        \big|\big({\bf b}(
            \Fvrmap^+,\Fvrmap^-
        ), 
        \om^{-1} u_1^{\rm ap}\big)_{\ell^2(\C^{2N+2})}\big|
        -\big|
        \big({\bf b}(
            \Fvrmap^+,\Fvrmap^-
        ), 
        \delta\big)_{\ell^2(\C^{2N+2})}\big|
        \nonumber
        \\
        &\geq
        \big|\big({\bf b}(
            \Fvrmap^+,\Fvrmap^-
        ), 
        u_1^{\rm ap}\big)_{\ell^2(\C^{2N+2})}\big|
        -
        \norm{{\bf b}(
            \Fvrmap^+,\Fvrmap^-
        )}_{\ell^2(\C^{2N+2})}
        \norm{\delta}_{\ell^2(\C^{2N+2})},
    \end{align*}
    where we have used Cauchy--Schwarz inequality in the last step. We can use Lemma~\ref{lemma:norm_fv} and Lemma~\ref{lemma:J_L2L2_stab} to obtain an upper estimate on $\norm{{\bf b}(\Fvrmap^+,\Fvrmap^-)}_{\ell^2(\C^{2N+2})}$. Using these lemmas, we obtain
    \begin{align*}
            \norm{{\bf b}(
            \Fvrmap^+,\Fvrmap^-
        )}_{\ell^2(\C^{2N+2})}^2
        &\leq 
        \norm{J^+ \Fvrmap^+}_{\XL}^2
        +
        \norm{J^- \Fvrmap^-}_{\XL}^2\\
        &\leq 
        \max\big\{
        \norm{J^+}_{\XL\to\XL},
        \norm{J^-}_{\XL\to\XL}
        \big\}^2 
        \norm{\Eigveap}_\XL^2.
    \end{align*}

    We now bound $\norm{\delta}_{\ell^2(\C^{2N+2})}$. Set 
    \begin{align}\label{def:om}
        \om := \frac{(u_1^{\rm ap}, u_1)_{\ell^{2}(\C^{2N+2})}}{|(u_1^{\rm ap}, u_1)_{\ell^{2}(\C^{2N+2})}|}.
    \end{align}
    We claim that the denominator $|(u_1^{\rm ap}, u_1)_{\ell^{2}(\C^{2N+2})}|$ does not vanish, we will show this after \eqref{eqn:prod_fv_u:aux4}.
  
    For this choice of $\om$, we find that
    \begin{align}
        \norm{\delta}_{\ell^2(\C^{2N+2})}^2 &= \norm{u_1-\om^{-1} u_1^{\rm ap}}_{\ell^{2}(\C^{2N+2})}^2
        \nonumber
        \\&= \norm{u_1}_{\ell^{2}(\C^{2N+2})}^2+
        \norm{u_1^{\rm ap}}_{\ell^{2}(\C^{2N+2})}^2
        -2\Re(\om^{-1} u_1^{\rm ap}, u_1)_{\ell^{2}(\C^{2N+2})}
        \nonumber
        \\
        &= 2
        -2|(u_1^{\rm ap}, u_1)_{\ell^{2}(\C^{2N+2})}|
        \nonumber
        \\
        &\leq 2\big(1
        -|(u_1^{\rm ap}, u_1)_{\ell^{2}(\C^{2N+2})}|^2\big),
        \label{eqn:prod_fv_u:aux3}
    \end{align}
    where we have used that $|(u_1^{\rm ap}, u_1)_{\ell^{2}(\C^{2N+2})}|\leq1$ in the last step.

    Define $\xi_{u_1^{\rm ap}}\in\C^{2N+2}$ by
    \begin{align}
        \xi_{u_1^{\rm ap}} &:= 
        \MStabTor \MStabTor^\ast  u_1^{\rm ap} - \sigma_1^2  u_1^{\rm ap}
        \nonumber
        \\
        &=
        \sum_{k=1}^{2N+2} (\MStabTor \MStabTor^\ast - \sigma_1^2I) (u_1^{\rm ap},u_k)_{\ell^{2}(\C^{2N+2})} u_k
        \nonumber
        \\
        &=
        \sum_{k=1}^{2N+2} (\sigma_k^2 - \sigma_1^2) (u_1^{\rm ap},u_k)_{\ell^{2}(\C^{2N+2})} u_k
        =\sum_{k=2}^{2N+2} (\sigma_k^2 - \sigma_1^2) (u_1^{\rm ap},u_k)_{\ell^{2}(\C^{2N+2})} u_k.\label{eqn:prod_fv_u:auxaux}
    \end{align}
    
    Taking the norm in \eqref{eqn:prod_fv_u:auxaux} yields
    \begin{align}
        \norm{\xi_{u_1^{\rm ap}}}_{\ell^{2}(\C^{2N+2})}^2
        \geq (\sigma_2^2-\sigma_1^2)^2\sum_{k=2}^{2N+2} |(u_1^{\rm ap},u_k)_{\ell^{2}(\C^{2N+2})}|^2 =
         (\sigma_2^2-\sigma_1^2)^2 \big(1-|(u_1^{\rm ap},u_1)_{\ell^{2}(\C^{2N+2})}|^2\big).
        \label{eqn:prod_fv_u:aux4}
    \end{align}
    The denominator $|(u_1^{\rm ap}, u_1)_{\ell^{2}(\C^{2N+2})}|$ in \eqref{def:om} does not vanish. If it did we would have that
    \begin{align*}
        \norm{\xi_{u_1^{\rm ap}}}_{\ell^{2}(\C^{2N+2})}
        \geq
         \sigma_2^2-\sigma_1^2,
    \end{align*}
    so we check that this does not happen.

    From \eqref{eqn:prod_fv_u:aux3} and \eqref{eqn:prod_fv_u:aux4} we get the upper estimate
    \begin{align*}
        \norm{\delta}_{\ell^2(\C^{2N+2})}
        \leq \sqrt{2}\frac{\norm{\xi_{u_1^{\rm ap}}}_{\ell^2(\C^{2N+2})}}{\sigma_2^2-\sigma_1^2}.
    \end{align*}
    Here, Lemma~\ref{lemma:stab_singval} provides a lower bound for the denominator $\sigma_2^2 - \sigma_1^2$.

    In order to bound $\xi_{u_1^{\rm ap}}$, we will use Corollary 4.15 of Chapter 4 in \cite{stewart1990matrix} with $A=\MStabTor\MStabTor^\ast$ and $X=u_1^{\rm ap}$. Denoting by $a:=\norm{\MStabTor^\ast u_1^{\rm ap}}_{\ell^{2}(\C^{2N+2})}\in\RR$,\footnote{The quantity $a$ corresponds to $M$ in \cite{stewart1990matrix}. Since we only have one column we also have that $\mu_1=M$.} we have that
    \begin{align*}
        |a^2 - \sigma_1^2| &\leq 
        \norm{\MStabTor\MStabTor^\ast u_1^{\rm ap} - 
        a^2 u_1^{\rm ap}}_{\ell^2(\C^{2N+2})}\\
        &\leq 
        \norm{\MStabTorApprox\MStabTorApprox^\ast u_1^{\rm ap} - 
        a^2 u_1^{\rm ap}}_{\ell^2(\C^{2N+2})}
        +
        \norm{
        (\MStabTorApprox R^\ast 
        +
        R\MStabTorApprox^\ast 
        +
        RR^\ast )
        u_1^{\rm ap}
        }_{\ell^2(\C^{2N+2})},
    \end{align*}
    where we have written the matrix $\MStabTor$ as
    \begin{align*}
        \MStabTor = \MStabTorApprox + R,
    \end{align*}
    (this same decomposition, with $\la=\getvar{la_reg}$ instead of $\laap$ is found in the proof of Lemma~\ref{lemma:stab_is_regu}.)
    
    Hence, we find that
    \begin{align*}
        \norm{\xi_{u_1^{\rm ap}}}_{\ell^2(\C^{2N+2})} &= 
        \norm{
            \MStabTor \MStabTor^\ast  u_1^{\rm ap} 
            -
            \sigma_1^2  u_1^{\rm ap}
        }_{\ell^2(\C^{2N+2})}\\
        &\leq 
        \norm{
            \MStabTor \MStabTor^\ast  u_1^{\rm ap} 
            -
            a^2  u_1^{\rm ap}
        }_{\ell^2(\C^{2N+2})}
        +
        |
            a^2 
            - 
            \sigma_1^2 
        |\\
        &\leq 
        2\norm{\MStabTorApprox\MStabTorApprox^\ast u_1^{\rm ap} - 
        a^2 u_1^{\rm ap}}_{\ell^2(\C^{2N+2})}
        +
        2\norm{
        (\MStabTorApprox R^\ast 
        +
        R\MStabTorApprox^\ast 
        +
        RR^\ast )
        u_1^{\rm ap}
        }_{\ell^2(\C^{2N+2})}
    \end{align*}

    We bound $\rho:=\norm{
    (\MStabTorApprox R^\ast 
    +
    R\MStabTorApprox^\ast 
    +
    RR^\ast )
    u_1^{\rm ap}
    }_{\ell^2(\C^{2N+2})}$ as follows. Let $r_k$ and $m_k$ denote the columns of $R$ and $\MStabTorApprox$ respectively. 
    Hence,
    \begin{align*}
        \rho
        &\leq 
        \norm{\MStabTorApprox R^\ast 
        }_2
        +
        \norm{
        R\MStabTorApprox^\ast 
        u_1^{\rm ap}
        }_{\ell^2(\C^{2N+2})}
        +
        \norm{
        RR^\ast
        }_2
        \\
        &\leq 
        \sum_{l=1}^{2N+2}
        \norm{
        m_l r_l^\ast
        }_2
        +
        \norm{
        r_l m_l^\ast
        u_1^{\rm ap}
        }_{\ell^2(\C^{2N+2})}
        +
        \norm{
        r_l r_l^\ast
        }_2
    \end{align*}
    Since $ab^\ast$ is a rank-one matrix for any vectors $a,b\in\C^{2N+2}$, we have that $\norm{ab^\ast}_2 = \norm{ab^\ast}_F \leq \norm{a}_F\norm{b}_F$, so we get that
    \begin{align*}
        \rho
        &\leq 
        \sum_{l=1}^{2N+2}\Big(
        \norm{
        m_l
        }_{\ell^2(\C^{2N+2})}
        +
        |
        m_l^\ast
        u_1^{\rm ap}
        |
        +
        \norm{
        r_l
        }_{\ell^2(\C^{2N+2})}
        \Big)
        \norm{
        r_l
        }_{\ell^2(\C^{2N+2})}.
    \end{align*}
        
    The computation of
    \begin{align*}
    \big|\big({\bf b}(
            \Fvrmap^+,\Fvrmap^-
        ), 
        u_1^{\rm ap}\big)_{\ell^2(\C^{2N+2})}\big|,
    \end{align*}
    is done using the approximations of $J^+\Fvrmap^+$ and $J^-\Fvrmap^-$ described in Appendix~\ref{sec:ode:Jfv}, namely $h_+^{\rm ap}$ and $h_-^{\rm ap}$. Recall \eqref{def:bplumin}, so we have that
    \begin{align*}
        {\bf b}(
            \Fvrmap^+,\Fvrmap^-
        )&=
        \begin{pmatrix}
            (\ha{J^+\Fvrmap^+})_{-1},\! &\!
            \cdots\!&\!,
            (\ha{J^+\Fvrmap^+})_{N-1},\!&\!
            -(\ha{J^-\Fvrmap^-})_{-1},\!&\!
            \cdots\!&\!,
            -(\ha{J^-\Fvrmap^-})_{N-1}
        \end{pmatrix}^t.
    \end{align*}

    Consequently, applying Lemma~\ref{lemma:ode_stab:Fv} and Parseval's identity, we obtain
    \begin{align*}
        \scalednorm{{\bf b}(
            \Fvrmap^+,\Fvrmap^-
        )- 
        \begin{pmatrix}
            (\ha{h_+^{\rm ap}})_{-1} ,\! &\!
            \cdots\!&\!,
            (\ha{h_+^{\rm ap}})_{N-1},\!&\!
            -(\ha{h_-^{\rm ap}})_{-1} ,\! &\!
            \cdots\!&\!,
            -(\ha{h_-^{\rm ap}})_{N-1}
        \end{pmatrix}^t}_{\ell^2(\C^{2N+2})}^2\leq 
        \frac{\rho_+ + \rho_-}{2\pi}.
    \end{align*}
    
    Then, by the Cauchy--Schwarz inequality, we have that
    \begin{align*}
        \big|\big({\bf b}(
            \Fvrmap^+,\Fvrmap^-
        ), 
        u_1^{\rm ap}\big)_{\ell^2(\C^{2N+2})}\big|
        \geq 
        \Big|
            \sum_{j=-1}^{N-1}
                (\ha{h_+^{\rm ap}})_j (\overline{u_1^{\rm ap}})_{j+2}
            -
            \sum_{j=-1}^{N-1}
                (\ha{h_-^{\rm ap}})_j (\overline{u_1^{\rm ap}})_{j+N+3}
        \Big|
        -
        \frac{(\rho_+ + \rho_-)^{1/2}}{\sqrt{2\pi}}.
    \end{align*}

    The full routine is done in \verb|prod_fv_u1()| from \verb|lemmas.py|.
\end{proofsth}

\begin{proofsth}{Lemma~\ref{lemma:Proots2}}
For every $\la\in\C$, we have that
\begin{align*}
    \lim_{z\to \al} i\frac{z-\al}{z}
    \Big(
    (i\pm\la)
    \frac{z^{N}}{z^N P(z)}-
    i\te^{\pm}\Big)
    &=
    \lim_{z\to \al} i(z-\al)
    (i\pm\la)
    \frac{z^{N-1}}{z^N P(z)}\\
    &=
    (\pm i \la-1)
    \frac{\al^{N-1}}{\pa_z (z^N P(z))(\al)}.
\end{align*}

Using the interval enclosures of $\al$ provided by Lemma~\ref{lemma:enclose_poly_roots}, this expression is evaluated with rigorously. The routine \verb|proots_stab()| in \verb|lemmas.py| verifies that, for every $\al\in Z$ and for both choices of sign, the resulting enclosure cannot contain an integer.
\end{proofsth}

\begin{proofsth}{Lemma~\ref{lemma:kappa3}}
    We know by Lemma~\ref{lemma:betamax} that $\beta$ is a positive and bounded function. Hence,
    \begin{align*}
        0<\int_0^\pi \beta(y)\dy <\infty.
    \end{align*}

    We then know that
    \begin{align}\label{eqn:kappa3:brute} 
        I_\beta:=e^{2\Re(\laap)\int_0^\pi \beta(y)\dy}>1.
    \end{align}
    The quantity $\int_0^\pi \beta(y)\dy$ is computed by the routine \verb|integ_adaptive_1D()|.

    Using \eqref{eqn:kappa3:brute} and $\kappa_1$ in Lemma~\ref{lemma:kappa1}, we obtain for every $x\in[0,\pi]$,
    \begin{align*}
    |\kappa_3^\pm(x)|
    &=
        e^{\pm2\Re(\laap)\int_0^x \beta(y)\dy}
        \int_0^x
            \beta(y)^2 
            e^{\mp 2\Re(\laap)\int_0^y \beta(s)\,ds}        
        \dy \\
    &\leq 
        e^{2\Re(\laap)\int_0^\pi \beta(y)\dy}
        \int_0^x
            \beta(y)^2      
        \dy
        =
        I_\beta 
        \kappa_1(x).
    \end{align*}

    Finally, we apply Lemma~\ref{lemma:kappa1}. The whole computation is done in \verb|kappa3()| from \verb|lemmas.py|.
\end{proofsth}

\begin{proofsth}{Lemma~\ref{lemma:kappa4}}
    We construct the \verb|Functions_1D| instances \verb|integ_p| and \verb|integ_m| with the routine \verb|betamod_laap_sq_func_constructor()|. Here, \verb|integ_p| represents $I^+(x):=(i+\laap)\beta(x)- i\te$, while \verb|integ_m| represents $I^-(x):=(i-\laap)\beta(x)+ i\te$. 
    
    We also construct an instance representing $\beta^2$ with \verb|beta_sq_func_constructor()|.

    Recall $I_\beta$ from \eqref{eqn:kappa3:brute} and notice that
    \begin{align*}
        \norm{\kappa_{4,\kappa_3^\pm}^+}_{L^1([0,\pi])} 
        &= 
        \int_0^\pi 
        \big|I^+(x)
        \big|^2 
        \kappa_3^\pm(x)
        \dx\\
        &\leq 
        I_\beta
        \int_0^\pi 
        \big|I^+(x)
        \big|^2 
        \kappa_1(x)
        \dx=
        I_\beta
        \int_0^\pi 
        \int_0^x
        \big|I^+(x)
        \big|^2 
        \beta(y)^2
        \dy
        \dx.
    \end{align*}
    The computation for $\kappa_{4,\kappa_3^\pm}^-$ follows in the same way. 
    
    Finally, the integrals
    \begin{align*}
        \int_0^\pi 
        \int_y^\pi
        \big|I^+(x)
        \big|^2
        \beta(y)^2
        \dx
        \dy
        \qquad 
        \text{and}
        \qquad
        \int_0^\pi 
        \int_y^\pi
        \big|I^-(x)
        \big|^2
        \beta(y)^2
        \dx
        \dy
    \end{align*}
    are evaluated with \verb|triangular_integral()|. The whole computation is realized in \verb|kappa4()| from \verb|lemmas.py|.
\end{proofsth}

\begin{proofsth}{Lemma~\ref{lemma:hk:stab}}
    We have to compute four different sums. We explain only how the third and fourth inequalities work, since the other two are analogous. Recall \eqref{def:hkstabsorted}, that is $h_{4N+1}=\Pi^+$, and $h_{4N+2}=\Pi^-$ and
    \begin{align*}
        h_k=J^+ g^+_{k,\te^+},\ h_{k+N}=J^+ g^-_{k,\te^-},\ h_{k+2N}=J^- g^+_{k,\te^-}\ \text{and } h_{k+3N}=J^- g^-_{k,\te^+}\ \text{for all } k=1,\ldots,N.
    \end{align*}
    Consequently, if we define $\mathfrak{S}:=\{4N+1\}\cup\{k\in\Z:\ 1\leq k\leq 2N\}$, the left member of the third inequality of the statement is
    \begin{align*}
        \sum_{k\in\mathfrak{S}}
        \norm{\pa_x h_k}_\XLIpi^2\leq 
        \sum_{k\in\mathfrak{S}}
        \big(
            \norm{\pa_x h^{\rm ap}_k}_\XLIpi
            +
            \norm{\pa_x (h_k-h_k^{\rm ap})}_\XLIpi
        \big)^2.
    \end{align*}
    Now, we evaluate $\norm{\pa_x h^{\rm ap}_k}_\XLIpi$ with the local routine \verb|_hk_norms_p()|, and use the
    Lemmas~\ref{lemma:J_L2L2_stab} and \ref{lemma:J_L2H1_stab}, and~\ref{lemma:ode_stab:sing} to get the upper bound $\rho_k^1$ for $\norm{\pa_x \tau_k}_\XLIpi$. Observe that in the proof of Lemma~\ref{lemma:ode_stab:sing} we enclose every residual quantity $\norm{\xi_k^{(s_J)(s_g)}}_\XLIpi^2$, multiply it by the constant $C_0$ in Lemma~\ref{lemma:J_L2L2_stab}, and compare it with the quantity $\rho_k$ in \verb|sing_resis.txt|:
    \begin{align*}
        \norm{h_k-h_k^{\rm ap}}_\XLIpi^2< C_0 \norm{\xi_k^{(s_J)(s_g)}}_\XLIpi^2 < \rho_k.
    \end{align*}

    Now, let $C_1^+$ be the constant from the first inequality in Lemma~\ref{lemma:J_L2H1_stab}, we obtain
    \begin{align*}
        \norm{\pa_x(h_k-h_k^{\rm ap})}_\XLIpi^2<
        C_1^+ \norm{\xi_k^{(s_J)(s_g)}}_\XLIpi^2 <
        \frac{C_1^+}{C_0} \rho_k.
    \end{align*}

    The fourth inequality is analogous, if we define $\mathfrak{D}:=\{4N+2\}\cup\{k\in\Z:\ 2N+1\leq k\leq 4N\}$, we have to estimate now 
    \begin{align*}
        \sum_{k\in\mathfrak{D}}
        \norm{\ti{\pa_x} h_k}_\XLIpi^2\leq 
        \sum_{k\in\mathfrak{D}}
        \big(
            \norm{\ti{\pa_x} h^{\rm ap}_k}_\XLIpi
            +
            \norm{\ti{\pa_x} \tau_k}_\XLIpi
        \big)^2.
    \end{align*}
    The term $\norm{\ti{\pa_x} h^{\rm ap}_k}_\XLIpi$ is evaluated with the local routine \verb|_hk_norms_m()|, while $\norm{\ti{\pa_x} \tau_k}_\XLIpi$ is bounded using again Lemmas~\ref{lemma:J_L2L2_stab}, \ref{lemma:J_L2H1_stab}, and~\ref{lemma:ode_stab:sing}. Notice that the constant we take now in Lemma~\ref{lemma:J_L2H1_stab} is the one in the second inequality. Denoting it by $C_1^-$ we obtain
    \begin{align*}
        \norm{\ti{\pa_x} (h_k-h_k^{\rm ap})}_\XLIpi^2< C_1^- \norm{\xi_k^{(s_J)(s_g)}}_\XLIpi^2 < \frac{C_1^-}{C_0}\rho_k.
    \end{align*}
\end{proofsth}

\subsection*{Acknowledgments}

AC and MMGPC acknowledge financial support from the Severo Ochoa Programme for Centers of Excellence (Grants \nolinkurl{CEX2019-000904-S} and \nolinkurl{CEX-2023-001347-S}), funded by 
\nolinkurl{MCIN/AEI/10.13039/501100011033}

AC and MMGPC were also supported by the grants \nolinkurl{PID2020-114703GB-I00}, \nolinkurl{PID2024-158418NB-I00}, and \nolinkurl{RED2024-153842-T}, all funded by \nolinkurl{MICIU/AEI/10.13039/501100011033}.

JGS has been partially supported by the MICINN (Spain) research grant number \nolinkurl{PID2021–125021NA–I00}, by NSF under Grants \nolinkurl{DMS-2245017}, \nolinkurl{DMS-2247537} and \nolinkurl{DMS-2434314}, by the AGAUR project \nolinkurl{2021-SGR-0087} (Catalunya) and by a Simons Fellowship. 

\printbibliography

@article{hur2014modulational,
  title        = {Modulational instability in the {Whitham} equation for water waves},
  author       = {Hur, Vera Mikyoung and Johnson, Mathew A.},
  journal      = {Studies in Applied Mathematics},
  volume       = {134},
  number       = {1},
  pages        = {120--143},
  year         = {2015},
  doi          = {10.1111/sapm.12061}
}

@article{hur2015modulational,
  title        = {Modulational instability in the {Whitham} equation with surface tension and vorticity},
  author       = {Hur, Vera Mikyoung and Johnson, Mathew A.},
  journal      = {Nonlinear Analysis: Theory, Methods \& Applications},
  volume       = {129},
  pages        = {104--118},
  year         = {2015},
  doi          = {10.1016/j.na.2015.08.019}
}

@article{hur2016modulational,
  title        = {Modulational instability in nonlinear nonlocal equations of regularized long wave type},
  author       = {Hur, Vera Mikyoung and Pandey, Ashish Kumar},
  journal      = {Physica D: Nonlinear Phenomena},
  volume       = {325},
  pages        = {98--112},
  year         = {2016},
  doi          = {10.1016/j.physd.2016.03.005}
}

@article{angulo2017stability,
  title        = {Stability properties of periodic traveling waves for the intermediate long wave equation},
  author       = {Pava, Jaime Angulo and Cardoso, Eleomar Jr. and Natali, F{\'a}bio},
  journal      = {Revista Matem\'atica Iberoamericana},
  volume       = {33},
  number       = {2},
  pages        = {417--448},
  year         = {2017},
  doi          = {10.4171/RMI/943}
}

@article{johnson2013stability,
  title        = {Stability of small periodic waves in fractional {KdV}-type equations},
  author       = {Johnson, Mathew A.},
  journal      = {SIAM Journal on Mathematical Analysis},
  volume       = {45},
  number       = {5},
  pages        = {3168--3193},
  year         = {2013},
  doi          = {10.1137/120894397}
}

@article{MasperoRadakovic2025,
  author  = {Maspero, Alberto and Radakovic, Antonio Milosh},
  title   = {Full description of {B}enjamin--{F}eir instability for generalized {K}orteweg--de {V}ries equations},
  journal = {SIAM J. Math. Anal.},
  volume  = {57},
  year    = {2025},
  number  = {3},
  pages   = {3030--3070},
  doi     = {10.1137/24M1652076},
}

@misc{RaduStevenson2025,
      title={Desingularization of nondegenerate rotating vortex patches},
      author={Radu, R\u{a}zvan-Octavian and Stevenson, Noah},
      year={2025},
      eprint={2511.18592},
      archivePrefix={arXiv},
      primaryClass={math.AP},
      url={https://arxiv.org/abs/2511.18592}
}

@article{GarciaHmidiMateu2024,
  author  = {Garc\'{\i}a, Claudia and Hmidi, Taoufik and Mateu, Joan},
  title   = {Time periodic solutions close to localized radial monotone profiles for the {2D} {E}uler equations},
  journal = {Ann. PDE},
  volume  = {10},
  year    = {2024},
  number  = {1},
  pages   = {Paper No. 1, 75},
  doi     = {10.1007/s40818-023-00166-5},
}

@article{garcia2025dynamicsvortexcapsolutions,
      title={Dynamics of vortex cap solutions on the rotating unit sphere},
      author={Garc\'{\i}a, Claudia and Hassainia, Zineb and Roulley, Emeric},
      journal={Journal of Differential Equations},
      volume={417},
      pages={1--63},
      year={2025},
      doi={10.1016/j.jde.2024.11.012}
}

@misc{garcia2026timeperiodicleapfroggingvortexrings,
      title={Time-periodic leapfrogging vortex rings in the 3D Euler equations},
      author={Garc\'{\i}a, Claudia and Hassainia, Zineb and Hmidi, Taoufik},
      year={2026},
      eprint={2603.21644},
      archivePrefix={arXiv},
      primaryClass={math.AP},
      url={https://arxiv.org/abs/2603.21644}
}

@misc{HassainiaHmidiRoulley2024desingularization,
      title={Desingularization of time-periodic vortex motion in bounded domains via {KAM} tools},
      author={Hassainia, Zineb and Hmidi, Taoufik and Roulley, Emeric},
      year={2024},
      eprint={2408.16671},
      archivePrefix={arXiv},
      primaryClass={math.AP},
      url={https://arxiv.org/abs/2408.16671}
}

@article{CrouseillesFaou2013,
    author  = {Crouseilles, Nicolas and Faou, Erwan},
    title   = {Quasi-periodic solutions of the {2D} {E}uler equations},
    journal = {Asymptot. Anal.},
    volume  = {81},
    year    = {2013},
    number  = {1},
    pages   = {31--34},
    doi     = {10.3233/ASY-2012-1117},
}

@article{EncisoPeraltaSalasTorres2023,
    author  = {Enciso, Alberto and Peralta-Salas, Daniel and Torres de Lizaur, Francisco},
    title   = {Quasi-periodic solutions to the incompressible {E}uler equations in dimensions two and higher},
    journal = {J. Differential Equations},
    volume  = {354},
    year    = {2023},
    pages   = {170--182},
    doi     = {10.1016/j.jde.2023.01.013},
}

@article{BaldiMontalto2021,
    author  = {Baldi, Pietro and Montalto, Riccardo},
    title   = {Quasi-periodic incompressible {E}uler flows in {3D}},
    journal = {Adv. Math.},
    volume  = {384},
    year    = {2021},
    pages   = {Paper No. 107730},
    doi     = {10.1016/j.aim.2021.107730},
}

@article{BertiHassainiaMasmoudi2023,
    author  = {Berti, Massimiliano and Hassainia, Zineb and Masmoudi, Nader},
    title   = {Time quasi-periodic vortex patches of {E}uler equation in the plane},
    journal = {Invent. Math.},
    volume  = {233},
    year    = {2023},
    number  = {3},
    pages   = {1279--1391},
    doi     = {10.1007/s00222-023-01195-4},
}

@article{HassainiaRoulley2025,
    author  = {Hassainia, Zineb and Roulley, Emeric},
    title   = {Boundary effects on the emergence of quasi-periodic solutions for {E}uler equations},
    journal = {Nonlinearity},
    volume  = {38},
    year    = {2025},
    number  = {1},
    pages   = {015016},
    doi     = {10.1088/1361-6544/ad9ba7},
}

@article{HassainiaHmidiRoulley2024,
    author  = {Hassainia, Zineb and Hmidi, Taoufik and Roulley, Emeric},
    title   = {Invariant {KAM} tori around annular vortex patches for {2D} {E}uler equations},
    journal = {Comm. Math. Phys.},
    volume  = {405},
    number  = {11},
    year    = {2024},
    pages   = {Paper No. 270},
    doi     = {10.1007/s00220-024-05141-0},
}

@article{HassainiaHmidiMasmoudi2025,
    author  = {Hassainia, Zineb and Hmidi, Taoufik and Masmoudi, Nader},
    title   = {Rigorous derivation of the leapfrogging motion for planar {E}uler equations},
    journal = {Invent. Math.},
    volume  = {242},
    year    = {2025},
    number  = {3},
    pages   = {725--825},
    doi     = {10.1007/s00222-025-01368-3},
}

@article{CastroGomezSerrano2025NonConvex,
    author  = {Gerard Castro-L\'opez and Javier G\'omez-Serrano},
    title   = {Existence of analytic non-convex V-states of the 2D Euler equation},
    journal = {Communications in Mathematical Physics},
    volume  = {406},
    number  = {9},
    pages   = {Paper No. 217},
    year    = {2025},
    doi     = {10.1007/s00220-025-05382-7},
}

@article{GuoHallstromSpirn2004,
    author  = {Guo, Yan and Hallstrom, Chris and Spirn, Daniel},
    title   = {Dynamics near an unstable {K}irchhoff ellipse},
    journal = {Comm. Math. Phys.},
    volume  = {245},
    year    = {2004},
    number  = {2},
    pages   = {297--354},
    doi     = {10.1007/s00220-003-1017-z},
}

@misc{WangXuZhou2022,
      title={Degenerate bifurcations of two-fold doubly-connected uniformly rotating vortex patches},
      author={Yuchen Wang and Xin Xu and Maolin Zhou},
      year={2022},
      eprint={2212.01869},
      archivePrefix={arXiv},
      primaryClass={math.AP},
      url={https://arxiv.org/abs/2212.01869}
}

@article{HmidiMateu2016Kirchhoff,
    author  = {Taoufik Hmidi and Joan Mateu},
    title   = {Bifurcation of rotating patches from Kirchhoff vortices},
    journal = {Discrete and Continuous Dynamical Systems},
    volume  = {36},
    number  = {10},
    year    = {2016},
    pages   = {5401--5422},
    doi = {10.3934/dcds.2016038},
}

@book{Kirchhoff1874,
    author    = {Kirchhoff, Gustav},
    title     = {Vorlesungen {\"u}ber mathematische {P}hysik},
    publisher = {Teubner},
    address   = {Leipzig},
    year      = {1874},
}

@misc{ElgindiJo2025,
      title={Cusp formation in vortex patches},
      author={Elgindi, Tarek M. and Jo, Min Jun},
      year={2025},
      eprint={2504.02705},
      archivePrefix={arXiv},
      primaryClass={math.AP},
      url={https://arxiv.org/abs/2504.02705}
}

@book{Fraenkel2000,
    author    = {Fraenkel, L. E.},
    title     = {An introduction to maximum principles and symmetry in elliptic problems},
    series    = {Cambridge Tracts in Mathematics},
    volume    = {128},
    publisher = {Cambridge University Press},
    address   = {Cambridge},
    year      = {2000},
    pages     = {x+340},
    doi       = {10.1017/CBO9780511569203},
}

@article{GomezSerranoParkShiYao2021,
    author  = {G\'{o}mez-Serrano, Javier and Park, Jaemin and Shi, Jia and Yao, Yao},
    title   = {Symmetry in stationary and uniformly-rotating solutions of active scalar equations},
    journal = {Duke Math. J.},
    volume  = {170},
    year    = {2021},
    number  = {13},
    pages   = {2957--3038},
    doi     = {10.1215/00127094-2021-0002},
}

@article{Hmidi2015trivial,
    author  = {Hmidi, Taoufik},
    title   = {On the trivial solutions for the rotating patch model},
    journal = {J. Evol. Equ.},
    volume  = {15},
    year    = {2015},
    number  = {4},
    pages   = {801--816},
    doi     = {10.1007/s00028-015-0281-7},
}

@article{HassainiasMasmoudiWheeler2020,
  author  = {Hassainia, Zineb and Masmoudi, Nader and Wheeler, Miles H.},
  title   = {Global bifurcation of rotating vortex patches},
  journal = {Comm. Pure Appl. Math.},
  volume  = {73},
  year    = {2020},
  number  = {9},
  pages   = {1933--1980},
  doi     = {10.1002/cpa.21855},
}

@article{delaHozHassainiaHmidiMateu2016,
  author  = {de la Hoz, Francisco and Hassainia, Zineb and Hmidi, Taoufik and Mateu, Joan},
  title   = {An analytical and numerical study of steady patches in the disc},
  journal = {Anal. PDE},
  volume  = {9},
  year    = {2016},
  number  = {7},
  pages   = {1609--1670},
  doi     = {10.2140/apde.2016.9.1609},
}

@article{delaHozHmidiMateuVerdera2016,
  author  = {de la Hoz, Francisco and Hmidi, Taoufik and Mateu, Joan and Verdera, Joan},
  title   = {Doubly connected {$V$}-states for the planar {E}uler equations},
  journal = {SIAM J. Math. Anal.},
  volume  = {48},
  year    = {2016},
  number  = {3},
  pages   = {1892--1928},
  doi     = {10.1137/140992801},
}

@article{HmidiMateu2016degenerate,
  author  = {Hmidi, Taoufik and Mateu, Joan},
  title   = {Degenerate bifurcation of the rotating patches},
  journal = {Adv. Math.},
  volume  = {302},
  year    = {2016},
  pages   = {799--850},
  doi     = {10.1016/j.aim.2016.07.022},
}

@article{CaoLaiZhan2021,
  author  = {Cao, Daomin and Lai, Shanfa and Zhan, Weicheng},
  title   = {Traveling vortex pairs for {2D} incompressible {E}uler equations},
  journal = {Calc. Var. Partial Differential Equations},
  volume  = {60},
  year    = {2021},
  number  = {5},
  pages   = {Paper No. 190},
  doi     = {10.1007/s00526-021-02068-5},
}

@article{CaoWanWangZhan2021,
  author  = {Cao, Daomin and Wan, Jie and Wang, Guodong and Zhan, Weicheng},
  title   = {Rotating vortex patches for the planar {E}uler equations in a disk},
  journal = {J. Differential Equations},
  volume  = {275},
  year    = {2021},
  pages   = {509--532},
  doi     = {10.1016/j.jde.2020.11.027},
}

@article{ChoiJeong2022,
  author  = {Choi, Kyudong and Jeong, In-Jee},
  title   = {Stability and instability of {K}elvin waves},
  journal = {Calc. Var. Partial Differential Equations},
  volume  = {61},
  year    = {2022},
  number  = {6},
  doi     = {10.1007/s00526-022-02334-0},
}

@article{ElgindiJeong2023,
  author  = {Elgindi, Tarek M. and Jeong, In-Jee},
  title   = {On singular vortex patches, {I}: {W}ell-posedness issues},
  journal = {Mem. Amer. Math. Soc.},
  volume  = {283},
  year    = {2023},
  number  = {1400},
  pages   = {1--102},
  doi     = {10.1090/memo/1400},
}

@article{ElgindiJeong2020,
  author  = {Elgindi, Tarek M. and Jeong, In-Jee},
  title   = {On singular vortex patches, {II}: long-time dynamics},
  journal = {Trans. Amer. Math. Soc.},
  volume  = {373},
  year    = {2020},
  number  = {9},
  pages   = {6757--6775},
  doi     = {10.1090/tran/8134},
}

@article{Garcia2020karman,
  author  = {Garc\'{\i}a, Claudia},
  title   = {K\'{a}rm\'{a}n vortex street in incompressible fluid models},
  journal = {Nonlinearity},
  volume  = {33},
  year    = {2020},
  number  = {4},
  pages   = {1625--1676},
  doi     = {10.1088/1361-6544/ab6309},
}

@article{Garcia2021choreography,
  author  = {Garc\'{\i}a, Claudia},
  title   = {Vortex patches choreography for active scalar equations},
  journal = {J. Nonlinear Sci.},
  volume  = {31},
  year    = {2021},
  number  = {5},
  pages   = {Paper No. 75, 31},
  doi     = {10.1007/s00332-021-09729-x},
}

@article{GarciaHaziot2023,
  author  = {Garc\'{\i}a, Claudia and Haziot, Susanna V.},
  title   = {Global bifurcation for corotating and counter-rotating vortex pairs},
  journal = {Comm. Math. Phys.},
  volume  = {402},
  year    = {2023},
  number  = {2},
  pages   = {1167--1204},
  doi     = {10.1007/s00220-023-04741-6},
}

@article{GarciaHmidiSoler2020,
  author  = {Garc\'{\i}a, Claudia and Hmidi, Taoufik and Soler, Juan},
  title   = {Non uniform rotating vortices and periodic orbits for the two-dimensional {E}uler equations},
  journal = {Arch. Ration. Mech. Anal.},
  volume  = {238},
  year    = {2020},
  number  = {2},
  pages   = {929--1085},
  doi     = {10.1007/s00205-020-01561-z},
}

@article{HassainiaHmidi2021,
  author  = {Hassainia, Zineb and Hmidi, Taoufik},
  title   = {Steady asymmetric vortex pairs for {E}uler equations},
  journal = {Discrete Contin. Dyn. Syst.},
  volume  = {41},
  year    = {2021},
  number  = {4},
  pages   = {1939--1969},
  doi     = {10.3934/dcds.2020348},
}

@article{HassainiaWheeler2022,
  author  = {Hassainia, Zineb and Wheeler, Miles H.},
  title   = {Multipole vortex patch equilibria for active scalar equations},
  journal = {SIAM J. Math. Anal.},
  volume  = {54},
  year    = {2022},
  number  = {6},
  pages   = {6054--6095},
  doi     = {10.1137/21M1415339},
}

@article{HmidiMateu2017,
  author  = {Hmidi, Taoufik and Mateu, Joan},
  title   = {Existence of corotating and counter-rotating vortex pairs for active scalar equations},
  journal = {Comm. Math. Phys.},
  volume  = {350},
  year    = {2017},
  number  = {2},
  pages   = {699--747},
  doi     = {10.1007/s00220-016-2784-7},
}

@article{Turkington1985,
  author  = {Turkington, Bruce},
  title   = {Corotating steady vortex flows with {$n$}-fold symmetry},
  journal = {Nonlinear Anal.},
  volume  = {9},
  year    = {1985},
  number  = {4},
  pages   = {351--369},
  doi     = {10.1016/0362-546X(85)90079-8},
}

@article{Park2022,
  author  = {Park, Jaemin},
  title   = {Quantitative estimates for uniformly-rotating vortex patches},
  journal = {Adv. Math.},
  volume  = {411},
  year    = {2022},
  pages   = {Paper No. 108779, 44},
  doi     = {10.1016/j.aim.2022.108779},
}

@article{GomezSerranoParkShi2025,
  author  = {G\'{o}mez-Serrano, Javier and Park, Jaemin and Shi, Jia},
  title   = {Existence of non-trivial non-concentrated compactly supported stationary solutions of the {2D} {E}uler equation with finite energy},
  journal = {Mem. Amer. Math. Soc.},
  volume  = {311},
  year    = {2025},
  number  = {1577},
  pages   = {82},
  doi     = {10.1090/memo/1577},
}

@article{CastroCordobaGomezSerrano2016,
  author  = {\'Angel Castro and Diego C\'ordoba and Javier G\'omez-Serrano},
  title   = {Uniformly rotating analytic global patch solutions for active scalars},
doi = {10.1007/s40818-016-0007-3},
  journal = {Annals of PDE},
  volume  = {2},
  number  = {1},
  year    = {2016},
  pages   = {1--34},
}

@article{HmidiMateuVerdera2013,
  author  = {Taoufik Hmidi and Joan Mateu and Joan Verdera},
  title   = {Boundary regularity of rotating vortex patches},
  journal = {Archive for Rational Mechanics and Analysis},
  volume  = {209},
  number  = {1},
  year    = {2013},
  pages   = {171--208},
  doi = {10.1007/s00205-013-0618-8},
}

@article{Burbea1982,
  author  = {Jacob Burbea},
  title   = {Motions of vortex patches},
  journal = {Letters in Mathematical Physics},
  volume  = {6},
  number  = {1},
  doi={10.1007/BF02281165},
  year    = {1982},
  pages   = {1--16},
}

@article{DeemZabusky1978,
  author  = {George S. Deem and Norman J. Zabusky},
  title   = {Vortex waves: Stationary ``V-states'', interactions, recurrence, and breaking},
  journal = {Physical Review Letters},
  volume  = {40},
  number  = {13},
  year    = {1978},
  pages   = {859--862},
  doi = {10.1103/PhysRevLett.40.859}
}

@article{BertozziConstantin1993,
  author  = {Andrea L. Bertozzi and Peter Constantin},
  title   = {Global regularity for vortex patches},
  journal = {Communications in Mathematical Physics},
  doi     = {10.1007/BF02097055},
  volume  = {152},
  number  = {1},
  year    = {1993},
  pages   = {19--28},
}

@book{Chemin1995,
  author    = {Jean-Yves Chemin},
  title     = {Fluides parfaits incompressibles},
  series    = {Ast\'erisque},
  volume    = {230},
  publisher = {Soci\'et\'e Math\'ematique de France},
  year      = {1995},
}

@article{Chemin1993,
  author  = {Jean-Yves Chemin},
  title   = {Persistance de structures g\'eom\'etriques dans les fluides incompressibles bidimensionnels},
  journal = {Annales de l'\'Ecole Normale Sup\'erieure},
  volume  = {26},
  doi = {10.24033/asens.1679},
  year    = {1993},
  pages   = {517--542},
}

@article{HunterIfrimTataruWong2015,
  author  = {John K. Hunter and Marius Ifrim and Daniel Tataru and Thomas K. Wong},
  title   = {Long time solutions for a Burgers--Hilbert equation via a modified energy method},
  journal = {Proceedings of the American Mathematical Society},
  volume  = {143},
  year    = {2015},
  pages   = {3407--3412},
  doi = {10.1090/proc/12215}
}

@article{DahGose2023,
   title={Highest Cusped Waves for the Burgers–Hilbert Equation},
   volume={247},
   ISSN={1432-0673},
   url={http://dx.doi.org/10.1007/s00205-023-01904-6},
   DOI={10.1007/s00205-023-01904-6},
   number={5},
   journal={Archive for Rational Mechanics and Analysis},
   publisher={Springer Science and Business Media LLC},
   author={Dahne, Joel and G\'omez-Serrano, Javier},
   year={2023},
   month=aug
}

@article{HunterIfrim2012,  author  = {John K. Hunter and Marius Ifrim},  title   = {Enhanced life span of smooth solutions of a Burgers--Hilbert equation},  journal = {SIAM Journal on Mathematical Analysis},  volume  = {44},  year    = {2012},  pages   = {2039--2052},doi={10.1137/110849791}}

@article{KrupaVasseur2020,
  author  = {Sam G. Krupa and Alexis F. Vasseur},
  title   = {Stability and uniqueness for piecewise smooth solutions to a nonlocal scalar conservation law with applications to Burgers--Hilbert equation},
  journal = {SIAM Journal on Mathematical Analysis},
  volume  = {52},
  number  = {3},
  year    = {2020},
  pages   = {2491--2530},
  doi = {10.1137/19M1257883}
}

@article{BressanZhang2017,  author  = {Alberto Bressan and Tianyou Zhang},  title   = {Piecewise smooth solutions to the Burgers--Hilbert equation}, doi={10.4310/CMS.2017.v15.n1.a7},  journal = {Communications in Mathematical Sciences},  volume  = {15},  number  = {1},  year    = {2017},  pages   = {165--184},}

@article{BressanNguyen2014,  author  = {Alberto Bressan and Khai T. Nguyen},  title   = {Global existence of weak solutions for the Burgers--Hilbert equation},  journal = {SIAM Journal on Mathematical Analysis}, doi = {10.1137/140957536},  volume  = {46},  number  = {4},  year    = {2014},  pages   = {2884--2904},}

@article{Yang,
  author  = {Yang, Ruoxuan},
  title   = {Shock Formation of the Burgers--Hilbert Equation},
  journal = {SIAM Journal on Mathematical Analysis},
  volume  = {53},
  number  = {5},
  year    = {2021},
  pages   = {5756--5802},
  doi     = {10.1137/21M1399348}
}

@article {SautWang,
    AUTHOR = {Saut, Jean-Claude and Wang, Yuexun},
     TITLE = {The wave breaking for {W}hitham-type equations revisited},
   JOURNAL = {SIAM J. Math. Anal.},
  FJOURNAL = {SIAM Journal on Mathematical Analysis},
    VOLUME = {54},
      YEAR = {2022},
    NUMBER = {2},
     PAGES = {2295--2319},
      ISSN = {0036-1410,1095-7154},
   MRCLASS = {76B15 (35A20 35S30 76B03)},
  MRNUMBER = {4409228},
MRREVIEWER = {Peter\ N.\ Zhevandrov},
       DOI = {10.1137/20M1345207},
       URL = {https://doi.org/10.1137/20M1345207},
}

@article{CastroCordobaGancedo2010,  author  = {{\'A}ngel Castro and Diego C{\'o}rdoba and Francisco Gancedo},  title   = {Singularity formations for a surface wave model},
  doi = {10.1088/0951-7715/23/11/006},
journal = {Nonlinearity},  volume  = {23},  year    = {2010},  pages   = {2835--2847},
}

@article{KleinSaut2015,
  author  = {Christoph Klein and Jean-Claude Saut},
  title   = {A numerical approach to blow-up issues for dispersive perturbations of Burgers' equation},
  journal = {Physica D: Nonlinear Phenomena},
  volume  = {295--296},
  year    = {2015},
  pages   = {46--65},
  doi = {10.1016/j.physd.2014.12.004}
}

@incollection{Hunter2016,
  author    = {John K. Hunter},
  title     = {The Burgers--Hilbert equation},
  booktitle = {Theory, Numerics and Applications of Hyperbolic Problems II},
  series    = {Springer Proceedings in Mathematics and Statistics},
  volume    = {237},
  publisher = {Springer},
  address   = {Cham},
  year      = {2018},
  pages     = {41--57},
  doi = {10.1007/978-3-319-91548-7_3},
}

@article {HMRSJZ,
    AUTHOR = {Hunter, John K. and Moreno-Vasquez, Ryan C. and Shu, Jingyang
              and Zhang, Qingtian},
     TITLE = {On the approximation of vorticity fronts by the
              {B}urgers-{H}ilbert equation},
   JOURNAL = {Asymptot. Anal.},
  FJOURNAL = {Asymptotic Analysis},
    VOLUME = {129},
      YEAR = {2022},
    NUMBER = {2},
     PAGES = {141--177},
      ISSN = {0921-7134,1875-8576},
   MRCLASS = {35Q31},
  MRNUMBER = {4465915},
       DOI = {10.3233/asy-211724},
       URL = {https://doi.org/10.3233/asy-211724},
}

@article {BielloHunter,
    AUTHOR = {Biello, Joseph and Hunter, John K.},
     TITLE = {Nonlinear {H}amiltonian waves with constant frequency and
              surface waves on vorticity discontinuities},
   JOURNAL = {Comm. Pure Appl. Math.},
  FJOURNAL = {Communications on Pure and Applied Mathematics},
    VOLUME = {63},
      YEAR = {2010},
    NUMBER = {3},
     PAGES = {303--336},
      ISSN = {0010-3640,1097-0312},
   MRCLASS = {35Q53 (37K05 37N10 76D17)},
  MRNUMBER = {2599457},
MRREVIEWER = {Allen\ Parker},
       DOI = {10.1002/cpa.20304},
       URL = {https://doi.org/10.1002/cpa.20304},
}

@incollection {MW,
    AUTHOR = {Marsden, Jerrold and Weinstein, Alan},
     TITLE = {Coadjoint orbits, vortices, and {C}lebsch variables for
              incompressible fluids},
      NOTE = {Order in chaos (Los Alamos, N.M., 1982)},
   JOURNAL = {Phys. D},
  FJOURNAL = {Physica D. Nonlinear Phenomena},
    VOLUME = {7},
      YEAR = {1983},
    NUMBER = {1-3},
     PAGES = {305--323},
      ISSN = {0167-2789,1872-8022},
   MRCLASS = {58F05 (58D25 76C05)},
  MRNUMBER = {719058},
MRREVIEWER = {Joseph\ Klein},
       DOI = {10.1016/0167-2789(83)90134-3},
       URL = {https://doi.org/10.1016/0167-2789(83)90134-3},
}

@book{ince1956ordinary,
  title={Ordinary Differential Equations},
  author={Ince, Edward Lindsay},
  year={1956},
  publisher={Dover Publications},
  address={New York},
  pages={558},
  isbn={0486603490},
  note={Reprint of the 1944 edition}
}

@book{bhatia2013matrix,
  title={Matrix analysis},
  author={Bhatia, Rajendra},
  series={Graduate Texts in Mathematics},
  volume={169},
  year={1997},
  publisher={Springer Science \& Business Media}
}

@article{Gershgorin31,
 URL = {http://mi.mathnet.ru/eng/im5235},
 author = {Gerschgorin, Semyon Aranovich},
 journal = {Bulletin de l'Acad\'emie des Sciences de l'URSS. Classe des sciences math\'ematiques et na},
 number = {6},
 pages = {749--754},
 title = {\"Uber die Abgrenzung der Eigenwerte einer Matrix},
 urldate = {2026-02-04},
 issue = {6},
 year = {1931}
}

@book{ahlfors1979complex,
  title={Complex analysis},
  author={Ahlfors, Lars Valerian},
  edition={3},
  year={1979},
  publisher={McGraw-Hill New York}
}

@book{whittaker1996modern,
  title={A Course of Modern Analysis},
  author={Whittaker, Edmund Taylor and Watson, George Neville},
  year={1996},
  publisher={Cambridge University Press},
  address={Cambridge},
  series={Cambridge Mathematical Library},
  isbn={9780521588072},
  pages={608},
  note={Reprint of the 4th edition (1927)}
}

@article {CCZ2,
    AUTHOR = {Castro, \'Angel and C\'ordoba, Diego and Zheng, Fan},
     TITLE = {Stability of traveling waves for the {B}urgers-{H}ilbert
              equation},
   JOURNAL = {Anal. PDE},
  FJOURNAL = {Analysis \& PDE},
    VOLUME = {16},
      YEAR = {2023},
    NUMBER = {9},
     PAGES = {2109--2145},
      ISSN = {2157-5045,1948-206X},
   MRCLASS = {35F25 (35C07 76B47)},
  MRNUMBER = {4668089},
MRREVIEWER = {Weicheng\ Zhan},
       DOI = {10.2140/apde.2023.16.2109},
       URL = {https://doi.org/10.2140/apde.2023.16.2109},
}

@article{gomez2021any,
  title={Any three eigenvalues do not determine a triangle},
  author={G{\'o}mez-Serrano, Javier and Orriols, Gerard},
  journal={Journal of Differential Equations},
  volume={275},
  pages={920--938},
  year={2021},
  publisher={Elsevier},
  doi = {10.1016/j.jde.2020.11.002}
}

@book{stewart1990matrix,
  address = {Boston},
  author = {Stewart, Gilbert W. and Sun, Ji-guang},
  isbn = {0126702306},
  publisher = {Academic},
  refid = {864877974},
  series = {Computer Science and Scientific Computing},
  title = {Matrix Perturbation Theory},
  url = {https://www.worldcat.org/title/matrix-perturbation-theory/oclc/908946968},
  year = {1990}
}

@book{kato1980perturbation,
  title     = {Perturbation Theory for Linear Operators},
  author    = {Kato, Tosio},
  year      = {1980},
  edition   = {2nd},
  publisher = {Springer-Verlag},
  address   = {Berlin, Heidelberg, New York},
  isbn      = {3-540-58661-X}
}

@article{Arioli-Gazzola-Koch:uniqueness-bifurcation-ns,
  author  = {Arioli, G. and Gazzola, F. and Koch, H.},
  title   = {Uniqueness and bifurcation branches for planar steady {N}avier--{S}tokes equations under {N}avier boundary conditions},
  journal = {Journal of Mathematical Fluid Mechanics},
  year    = {2021},
  volume  = {23},
  number  = {3},
  pages   = {49},
  doi     = {10.1007/s00021-021-00572-4}
}

@article{Arioli-Koch:cap-stationary-ks,
  author  = {Arioli, G. and Koch, H.},
  title   = {Computer-assisted methods for the study of stationary solutions in dissipative systems, applied to the {K}uramoto-{S}ivashinski equation},
  journal = {Arch. Ration. Mech. Anal.},
  year    = {2010},
  volume  = {197},
  number  = {3},
  pages   = {1033--1051},
  doi = {10.1007/s00205-010-0309-7},
}

@article{Bedrossian-PunshonSmith:chaos-stochastic-2d-galerkin-ns,
  author  = {Bedrossian, J. and Punshon-Smith, S.},
  title   = {Chaos in stochastic 2d {G}alerkin-{N}avier-{S}tokes},
  journal = {Communications in Mathematical Physics},
  year    = {2024},
  volume  = {405},
  number  = {4},
  pages   = {107},
  doi     = {10.1007/s00220-024-04949-0}
}

@article{Buckmaster-CaoLabora-GomezSerrano:implosion-compressible,
  author  = {Buckmaster, T. and Cao-Labora, G. and G\'{o}mez-Serrano, J.},
  title   = {Smooth imploding solutions for 3d compressible fluids},
  journal = {Forum of Mathematics, Pi},
  year    = {2025},
  volume  = {13},
  pages   = {e6},
  doi     = {10.1017/fmp.2024.12}
}

@article{Cadiot:proofs-existence-stability-capillary-gravity-whitham,
  author  = {Cadiot, M.},
  title   = {Constructive proofs of existence and stability of solitary waves in the {W}hitham and capillary--gravity {W}hitham equations},
  journal = {Nonlinearity},
  year    = {2025},
  volume  = {38},
  number  = {3},
  pages   = {035021},
  doi     = {10.1088/1361-6544/adb5e8}
}

@article{Castelli-Gameiro-Lessard:rigorous-numerics-ill-posed-pde,
  author  = {Castelli, R. and Gameiro, M. and Lessard, J.-P.},
  title   = {Rigorous numerics for ill-posed {PDE}s: periodic orbits in the {B}oussinesq equation},
  journal = {Arch. Ration. Mech. Anal.},
  year    = {2018},
  volume  = {228},
  number  = {1},
  pages   = {129--157},
  doi     = {10.1007/s00205-017-1186-0}
}

@article{Castro-Cordoba-GomezSerrano:global-smooth-solutions-sqg,
  author  = {Castro, A. and C\'{o}rdoba, D. and G\'{o}mez-Serrano, J.},
  title   = {Global smooth solutions for the inviscid {S}{Q}{G} equation},
  journal = {Memoirs of the American Mathematical Society},
  year    = {2020},
  volume  = {266},
  number  = {1292},
  doi     = {10.1090/memo/1292}
}

@article{Chen-Hou-Huang:blowup-degregorio,
  author  = {Chen, J. and Hou, T. Y. and Huang, D.},
  title   = {On the {F}inite {T}ime {B}lowup of the {D}e {G}regorio {M}odel for the 3{D} {E}uler {E}quations},
  journal = {Communications on Pure and Applied Mathematics},
  year    = {2021},
  volume  = {74},
  number  = {6},
  pages   = {1282--1350},
  doi     = {10.1002/cpa.21991}
}

@article{Day-Lessard-Mischaikow:validated-continuation-equilibria-pde,
  author  = {Day, S. and Lessard, J.-P. and Mischaikow, K.},
  title   = {Validated continuation for equilibria of {PDE}s},
  journal = {SIAM Journal on Numerical Analysis},
  year    = {2007},
  volume  = {45},
  number  = {4},
  pages   = {1398--1424},
  doi     = {10.1137/050645968}
}

@article{Figueras-DeLaLLave:cap-periodic-orbits-kuramoto,
  author  = {Figueras, J.-L. and de la Llave, R.},
  title   = {Numerical computations and computer assisted proofs of periodic orbits of the {K}uramoto-{S}ivashinsky equation},
  journal = {SIAM Journal on Applied Dynamical Systems},
  year    = {2017},
  volume  = {16},
  number  = {2},
  pages   = {834--852},
  doi     = {10.1137/16M1073790}
}

@article{Figueras-Gameiro-Lessard-DeLaLLave:framework-cap-invariant-objects,
  author  = {Figueras, J.-L. and Gameiro, M. and Lessard, J.-P. and de la Llave, R.},
  title   = {A framework for the numerical computation and a posteriori verification of invariant objects of evolution equations},
  journal = {SIAM Journal on Applied Dynamical Systems},
  year    = {2017},
  volume  = {16},
  number  = {2},
  pages   = {1070--1088},
  doi     = {10.1137/16M1073777}
}

@article{Gameiro-Lessard:periodic-orbits-ks,
  author  = {Gameiro, M. and Lessard, J.-P.},
  title   = {A posteriori verification of invariant objects of evolution equations: periodic orbits in the {K}uramoto-{S}ivashinsky {PDE}},
  journal = {SIAM Journal on Applied Dynamical Systems},
  year    = {2017},
  volume  = {16},
  number  = {1},
  pages   = {687--728},
  doi     = {10.1137/16M1073789}
}

@article{Gameiro-Lessard-Mischaikow:validated-continuation-large-parameter-ranges-pde,
  author  = {Gameiro, M. and Lessard, J.-P. and Mischaikow, K.},
  title   = {Validated continuation over large parameter ranges for equilibria of {PDE}s},
  journal = {Mathematics and Computers in Simulation},
  year    = {2008},
  volume  = {79},
  number  = {4},
  pages   = {1368--1382},
  doi     = {10.1016/j.matcom.2008.03.014}
}

@article{Guo-Hadzic-Jang-Schrecker:gravitational-collapse-stars-self-similar,
  author  = {Guo, Y. and Had\v{z}i\'{c}, M. and Jang, J. and Schrecker, M.},
  title   = {Gravitational collapse for polytropic gaseous stars: self-similar solutions},
  journal = {Arch. Ration. Mech. Anal.},
  year    = {2022},
  volume  = {246},
  number  = {2},
  pages   = {957--1066},
  doi     = {10.1007/s00205-022-01827-8}
}

@article{Kobayashi:global-uniqueness-stokes,
  author  = {Kobayashi, K.},
  title   = {On the global uniqueness of {S}tokes' wave of extreme form},
  journal = {IMA Journal of Applied Mathematics},
  year    = {2010},
  volume  = {75},
  number  = {5},
  pages   = {647--675},
  doi     = {10.1093/imamat/hxq037}
}

@article{vandenBerg-Breden-Lessard-vanVeen:periodic-orbits-ns,
  author  = {van den Berg, J. B. and Breden, M. and Lessard, J.-P. and van Veen, L.},
  title   = {Spontaneous periodic orbits in the {N}avier--{S}tokes flow},
  journal = {Journal of Nonlinear Science},
  year    = {2021},
  volume  = {31},
  number  = {2},
  pages   = {41},
  doi     = {10.1007/s00332-021-09695-4}
}

@article{vandenBerg-Lessard:chaotic-braided-solutions-swift-hohenberg,
  author  = {van den Berg, J. B. and Lessard, J.-P.},
  title   = {Chaotic braided solutions via rigorous numerics: chaos in the {S}wift-{H}ohenberg equation},
  journal = {SIAM Journal on Applied Dynamical Systems},
  year    = {2008},
  volume  = {7},
  number  = {3},
  pages   = {988--1031},
  doi     = {10.1137/070709128}
}

@article{Zgliczynski:periodic-orbit-kuramoto,
  author  = {Zgliczy\'nski, P.},
  title   = {Rigorous numerics for dissipative partial differential equations. {II}. {P}eriodic orbit for the {K}uramoto-{S}ivashinsky {PDE}---a computer-assisted proof},
  journal = {Foundations of Computational Mathematics},
  year    = {2004},
  volume  = {4},
  number  = {2},
  pages   = {157--185},
  doi     = {10.1007/s10208-002-0080-8}
}

@article{Zgliczynski-Mischaikow:rigorous-numerics-kuramoto,
  author  = {Zgliczy{\'n}ski, P. and Mischaikow, K.},
  title   = {Rigorous numerics for partial differential equations: the {K}uramoto-{S}ivashinsky equation},
  journal = {Foundations of Computational Mathematics},
  year    = {2001},
  volume  = {1},
  number  = {3},
  pages   = {255--288},
  doi     = {10.1007/s002080010010}
}

@article{ElgindiPasqualotto2025_invertibility_boussinesq,
  author    = {Elgindi, Tarek M. and Pasqualotto, Federico},
  title     = {Invertibility of a Linearized {B}oussinesq Flow: A Symbolic Approach},
  journal   = {Communications in Mathematical Physics},
  year      = {2025},
  volume    = {406},
  number    = {11},
  pages     = {261},
  doi       = {10.1007/s00220-025-05367-6}
}

@misc{Chen-Hou:nearly-self-similar-blowup-boussinesq-euler-analysis,
      title={Stable nearly self-similar blowup of the {2D} {B}oussinesq and {3D} {E}uler equations with smooth data {I}: {A}nalysis},
      author={Chen, Jiajie and Hou, Thomas Y.},
      year={2022},
      eprint={2210.07191},
      archivePrefix={arXiv},
      primaryClass={math.AP},
      url={https://arxiv.org/abs/2210.07191}
}

@article{Chen-Hou:nearly-self-similar-blowup-boussinesq-euler-numerics,
  author  = {Chen, Jiajie and Hou, Thomas Y.},
  title   = {Stable nearly self-similar blowup of the {2D} {B}oussinesq and {3D} {E}uler equations with smooth data {II}: {R}igorous numerics},
  journal = {Multiscale Model. Simul.},
  year    = {2025},
  volume  = {23},
  number  = {1},
  pages   = {25--130},
  doi     = {10.1137/23M1580395},
}

@article{Chen-Hou:singularity-formation-3d-euler-smooth-data,
  author  = {Chen, Jiajie and Hou, Thomas Y.},
  title   = {Singularity formation in {3D} {E}uler equations with smooth initial data and boundary},
  journal = {Proc. Natl. Acad. Sci. USA},
  volume  = {122},
  number  = {27},
  pages   = {e2500940122},
  year    = {2025},
  doi     = {10.1073/pnas.2500940122},
}

\begin{flushleft}
	\bigskip
	\'Angel Castro\\
	\textsc{Instituto de Ciencias Matemáticas\\
		28049 Madrid, Spain}\\
	\textit{E-mail address:} \nolinkurl{angel_castro@icmat.es}
\end{flushleft}

\begin{flushleft}
	\bigskip
	Javier G\'{o}mez-Serrano\\
	\textsc{Department of Mathematics, Brown University\\
		Providence, RI 02912, USA}\\
	\textit{E-mail address:} \nolinkurl{javier_gomez_serrano@brown.edu}
\end{flushleft}

\begin{flushleft}
	\bigskip
	Miguel M.G. Pascual-Caballo\\
	\textsc{Instituto de Ciencias Matemáticas\\
		28049 Madrid, Spain}\\
	\textit{E-mail address:} \nolinkurl{miguel.martinez@icmat.es}
\end{flushleft}

\end{document}